\documentclass{amsart}

\usepackage{amssymb,amsmath,ifthen}
\usepackage{stmaryrd}
\usepackage{yhmath}
\usepackage{latexsym}
\usepackage{amsthm}
\usepackage{amscd}
\usepackage{mathrsfs}
\usepackage[all]{xy}
\usepackage{mathtools}
\usepackage{arydshln}
\usepackage{url}
\usepackage{bbm}

\newcommand{\ur}{\mathrm{ur}}
\newcommand{\rs}{\mathrm{rs}}
\newcommand{\trs}{\theta\mathchar`-\mathrm{rs}}
\newcommand{\srs}{\mathrm{srs}}

\newcommand{\strs}{\theta\mathchar`-\mathrm{srs}}

\newcommand{\ad}{\mathrm{ad}}
\newcommand{\red}{\mathrm{red}}

\newcommand{\aff}{\mathrm{aff}}
\newcommand{\der}{\mathrm{der}}
\newcommand{\LHS}{\mathrm{LHS}}
\newcommand{\RHS}{\mathrm{RHS}}
\newcommand{\lan}{\langle}
\newcommand{\ran}{\rangle}
\newcommand{\ol}{\overline}

\newcommand{\mcO}{\mathcal{O}}
\newcommand{\mfp}{\mathfrak{p}}

\newcommand{\tPhi}{\widetilde{\Phi}}
\newcommand{\Z}{\mathbb{Z}}
\newcommand{\F}{\mathbb{F}}
\newcommand{\R}{\mathbb{R}}

\newcommand{\C}{\mathbb{C}}
\newcommand{\Gm}{\mathbb{G}_{\mathrm{m}}}
\newcommand{\G}{\mathbf{G}}
\newcommand{\T}{\mathbf{T}}
\newcommand{\bfH}{\mathbf{H}}
\newcommand{\bfU}{\mathbf{U}}
\newcommand{\bfS}{\mathbf{S}}
\newcommand{\bfZ}{\mathbf{Z}}
\newcommand{\bfN}{\mathbf{N}}
\newcommand{\w}{\widetilde}

\newcommand{\ra}{\rightarrow}

\DeclareMathOperator{\tr}{tr}
\DeclareMathOperator{\sgn}{sgn}

\DeclareMathOperator{\spin}{spin}

\DeclareMathOperator{\GL}{GL}
\DeclareMathOperator{\rmO}{O}
\DeclareMathOperator{\SO}{SO}

\DeclareMathOperator{\SL}{SL}
\DeclareMathOperator{\Sp}{Sp}

\DeclareMathOperator{\U}{U}
\DeclareMathOperator{\Lie}{Lie}
\DeclareMathOperator{\cInd}{c-Ind}
\DeclareMathOperator{\Ind}{Ind}
\DeclareMathOperator{\nInd}{n-Ind}
\DeclareMathOperator{\Hom}{Hom}
\DeclareMathOperator{\Aut}{Aut}
\DeclareMathOperator{\Ker}{Ker}
\DeclareMathOperator{\Gal}{Gal}
\DeclareMathOperator{\Ad}{Ad}
\DeclareMathOperator{\Int}{Int}
\DeclareMathOperator{\Kl}{Kl}
\DeclareMathOperator{\Tr}{Tr}
\DeclareMathOperator{\Nr}{Nr}
\DeclareMathOperator{\Cent}{Cent}
\DeclareMathOperator{\diag}{diag}
\DeclareMathOperator{\depth}{depth}

\DeclareMathOperator{\Out}{Out}
\DeclareMathOperator{\Jord}{Jord}
\DeclareMathOperator{\Swan}{Swan}
\DeclareMathOperator{\Artin}{Artin}
\DeclareMathOperator{\Frob}{Frob}
\DeclareMathOperator{\SSC}{SSC}

\DeclareMathOperator{\val}{val}

\theoremstyle{plain}
\newtheorem{thm}{Theorem}[section]
\newtheorem*{thm*}{Theorem}
\newtheorem{prop}[thm]{Proposition}
\newtheorem{lem}[thm]{Lemma}
\newtheorem{cor}[thm]{Corollary}
\newtheorem{conj}[thm]{Conjecture}

\theoremstyle{definition}
\newtheorem{defn}[thm]{Definition}

\theoremstyle{remark}
\newtheorem{rem}[thm]{Remark}
\newtheorem*{claim*}{Claim}

\SelectTips{cm}{11}

\title{Simple supercuspidal $L$-packets of quasi-split classical groups}
\author{Masao Oi}
\address{Department of Mathematics (Hakubi center), Kyoto University, Kitashirakawa, Oiwake-cho, Sakyo-ku, Kyoto 606-8502, Japan.}
\email{masaooi@math.kyoto-u.ac.jp}

\begin{document}

\begin{abstract}
In this paper, for quasi-split classical groups over $p$-adic fields, we determine the $L$-packets consisting of simple supercuspidal representations and their corresponding $L$-parameters, under the assumption that $p$ is not equal to $2$.
The key is an explicit computation of characters of simple supercuspidal representations and the endoscopic character relation, which is a characterization of the local Langlands correspondence for quasi-split classical groups.
\end{abstract}


\maketitle

\section{Introduction}
The aim of this paper is to determine the structures of $L$-packets consisting of simple supercuspidal representations and their corresponding $L$-parameters for quasi-split classical groups over $p$-adic fields.

To be more precise, we start from recalling the local Langlands correspondence.
Let $F$ be a $p$-adic field, $\G$ a connected reductive group over $F$, and $G:=\G(F)$ the set of $F$-valued points of $\G$.
Then $\G$ defines the $L$-group ${}^{L}\G$ of $\G$, which is a semi-direct product of the Langlands dual group $\widehat{\G}$ of $\G$ and the Weil group $W_{F}$. 
We denote the set of equivalence classes of irreducible smooth representations of $G$ by $\Pi(\G)$, and the set of $\widehat{\G}$-conjugacy classes of $L$-parameters of $\G$ by $\Phi(\G)$.
Here we recall that an $L$-parameter of $\G$ is an admissible homomorphism from the local Langlands group $W_{F}\times\SL_{2}(\C)$ to the $L$-group ${}^{L}\G$ of $\G$.
Then the \textit{local Langlands correspondence for $\G$} predicts that there exists a ``natural'' map from the set $\Pi(\G)$ to the set $\Phi(\G)$ with finite fibers (called $L$-\textit{packets}).
In other words, the local Langlands correspondence asserts that there exists a natural partition of the set $\Pi(\G)$ into finite sets parametrized by $L$-parameters:
\[
\Pi(\G)=\bigsqcup_{\phi\in\Phi(\G)} \Pi^{\G}_{\phi}.
\]

For general linear groups, the local Langlands correspondence was established by Harris and Taylor (\cite{MR1876802}).
For quasi-split classical groups (i.e., special orthogonal groups, symplectic groups, and unitary groups), the correspondence was established by Arthur (\cite{MR3135650}) and Mok (\cite{MR3338302}) recently.
However, it is quite difficult to figure out the correspondences explicitly by observing their constructions.
Therefore it is a natural attempt to describe explicitly the local Langlands correspondence for these groups.

Our aim in this paper is to achieve this by investigating the ``naturality'' in detail.
Thus we next recall the naturality (namely, the characterization) of the local Langlands correspondence for the above groups more precisely.

In the case of general linear groups, the naturality of the correspondence is formulated in terms of $\varepsilon$-factors and $L$-factors of representations.

In the cases of quasi-split classical groups, the naturality of the correspondence is formulated by using the \textit{endoscopic character relation}.
To explain this, let $\G$ be a quasi-split classical group over $F$.
Then we can regard $\G$ as an \textit{endoscopic group} of a (twisted) general linear group $\GL_{N}$ over $F$.
In particular, we have an embedding $\iota$ from the $L$-group of $\G$ to that of $\GL_{N}$.
Here the size $N$ of the general linear group depends on each classical group.
For example, for
\begin{itemize}
\item the symplectic group $\Sp_{2n}$ of size $2n$,
\item the quasi-split special orthogonal group $\SO_{2n}^{\mu}$ which is of size $2n$ and corresponds to a ramified quadratic character $\mu$ of $F^{\times}$,
\item the split special orthogonal group $\SO_{2n+2}$ of size $2n+2$, and
\item the quasi-split special orthogonal group $\SO_{2n+2}^{\ur}$ which is of size $2n+2$ and corresponds to the nontrivial unramified quadratic character $\mu_{\ur}$ of $F^{\times}$
\end{itemize}
(these are groups which will be treated in this paper), their dual groups and corresponding general linear groups are given by the following:
\[
    \begin{tabular}{|c|c|c|c|} \hline
      $\G$ & $\Sp_{2n}$ & $\SO_{2n}^{\mu}$ & $\SO_{2n+2}^{(\ur)}$ \\ \hline
      $\widehat{\G}$ & $\SO_{2n+1}(\C)$ & $\SO_{2n}(\C)$ & $\SO_{2n+2}(\C)$ \\ \hline
      $\GL_{N}$ & $\GL_{2n+1}$ & $\GL_{2n}$ & $\GL_{2n+2}$ \\ \hline
    \end{tabular}
\]
Now let us take a tempered $L$-parameter $\phi$ of $\G$.
Then, by noting that $\phi$ is a homomorphism from $W_{F}\times\SL_{2}(\C)$ to ${}^{L}\G$, we obtain an $L$-parameter of $\GL_{N}$ by composing $\phi$ with the embedding $\iota$.
From these $L$-parameters, we get representations of two different groups.
One is the representation $\pi_{\phi}^{\GL_{N}}$ of $\GL_{N}(F)$ corresponding to $\iota\circ\phi$ under the local Langlands correspondence for $\GL_{N}$ (note that, for $\GL_{N}$, each $L$-packet is a singleton).
The other is an $L$-packet $\Pi_{\phi}^{\G}$, which is a finite set of representations of $G$, corresponding to $\phi$ under the local Langlands correspondence for $\G$.
\[
\xymatrix{
\Pi(\GL_{N}) \ni \pi_{\phi}^{\GL_{N}}& \ar@{<~>}[r]^-{\text{LLC for $\GL_{N}$}} &&& {}^{L}\!\GL_{N}\\
\Pi(\G) \supseteq \Pi_{\phi}^{\G}\ar@{~>}[u]^-{\text{endoscopic lifting}}& \ar@{<~>}[r]^-{\text{LLC for $\G$}} &&W_F\times\SL_2(\C) \ar[r]_-{\phi} \ar[ru]^-{\iota\circ\phi} & {}^{L}\G \ar@{^{(}->}[u]_-{\iota} \\
}
\]
In this situation, we say that $\pi_{\phi}^{\GL_{N}}$ is the \textit{endoscopic lift} of $\Pi_{\phi}^{\G}$ from $\G$ to $\GL_{N}$.
Then the endoscopic character relation is the equality between the twisted character $\Theta_{\phi,\theta}^{\GL_{N}}$ of $\pi_{\phi}^{\GL_{N}}$ and the characters $\Theta_{\pi}$ of representations $\pi$ belonging to $\Pi_{\phi}^{\G}$, and given by the following:
\[
\Theta_{\phi,\theta}^{\GL_{N}}(g)
=\sum_{h\leftrightarrow g/\sim}\frac{D_{\G}(h)^{2}}{D_{\GL_{N},\theta}(g)^{2}} \Delta_{\G,\GL_{N}}(h,g) 
\sum_{\pi\in\Pi_{\phi}^{\G}}\Theta_{\pi}(h).
\]
Here we do not explain the precise meaning of each term in this equality (see Section \ref{sec:Arthur} for the details).
However, we emphasize that the relation between $\pi_{\phi}^{\GL_{N}}$ and $\Pi_{\phi}^{\G}$ is characterized by the endoscopic character relation since we have linear independence of the characters of representations.
Then the naturality of the local Langlands correspondence for $\G$ is formulated as follows:
for every tempered $L$-parameter $\phi$, $\Pi_{\phi}^{\G}$ and $\pi_{\phi}^{\GL_{N}}$ satisfy the endoscopic character relation.

Now we wish to describe the local Langlands correspondence for $\G$ explicitly.
Then, by the above formulation, we can divide the problem of explicit description of the local Langlands correspondence for $\G$ into the following two problems:
\begin{enumerate}
\item
Describe the endoscopic lifting from $\G$ to $\GL_{N}$ explicitly.
\item
Describe the local Langlands correspondence for $\GL_{N}$ explicitly.
\end{enumerate}

In this paper, we consider these problems for \textit{simple supercuspidal representations}, which were introduced by Gross and Reeder in \cite{MR2730575} (and also by Reeder and Yu in \cite{MR3164986}), of quasi-split classical groups.
From now on, we assume that the residual characteristic $p$ is not equal to $2$.
Simple supercuspidal representations are supercuspidal representations obtained by the compact induction of \textit{affine generic characters} of the pro-unipotent radical of the Iwahori subgroup, and characterized as the representations having the \textit{minimal positive depth}.
Here recall that every irreducible smooth representation of $G$ has a numerical invariant which is called the \textit{depth}.
The depth is a non-negative rational number, and the minimal positive depth for each group is given by the following:
\[
    \begin{tabular}{|c|c|c|c|c|} \hline
      $\G$ & $\GL_{N}$ & $\Sp_{2n}$ & $\SO_{2n}^{\mu}$ & $\SO_{2n+2}^{(\ur)}$ \\ \hline
      minimal positive depth & $\frac{1}{N}$ & $\frac{1}{2n}$ & $\frac{1}{2n}$ & $\frac{1}{2n}$\\ \hline
    \end{tabular}
\]
Since the construction of simple supercuspidal representations is very explicit, after making some choices (for example, fixing a uniformizer of $F$), we can easily parametrize the equivalence classes of them by a concrete set, which is denoted by $\SSC(\G)$ in this paper.
Roughly speaking, an element of $\SSC(\G)$ consists of data of
\begin{itemize}
\item a central character, 
\item an affine generic character on the pro-unipotent radical of the Iwahori subgroup of $G$, and 
\item a way to extend the affine generic character to its intertwining subgroup,
\end{itemize}
and described as follows (see Section \ref{sec:ssc} for details):
\[
    \begin{tabular}{|c|c|c|c|c|} \hline
      group $\G$ & parametrizing set $\SSC(\G)$\\ \hline
      $\GL_{N}$ & $(k^{\times})^{\vee}\times k^{\times}\times\C^{\times}$\\
      $\Sp_{2n}$ & $\mu_{2}\times\{0,1\}\times k^{\times}$\\ 
      $\SO_{2n}^{\mu}$ & $\mu_{2}\times k^{\times}$\\
      $\SO_{2n+2}^{(\ur)}$ & $\mu_{2}\times\{0,1\}\times k^{\times}\times\mu_{2}$\\\hline
    \end{tabular}
\]
Here $\mu_{2}$ is the set $\{\pm1\}$ of signs, and $(k^{\times})^{\vee}$ is the set of characters of the multiplicative group $k^{\times}$ of the residue field $k$ of $F$.
For $X\in\SSC(\G)$, we denote the corresponding simple supercuspidal representation of $G$ by $\pi_{X}^{\G}$.
We write $\omega_{X}^{\G}$ for the central character of $\pi_{X}^{\G}$.

Now we state our main theorem.
\begin{thm}[Main theorem]\label{thm:main}
We assume that $p$ is not equal to $2$.
Let $\pi^{\G}_{X}$ be the simple supercuspidal representation of $G$ corresponding to an element $X$ of $\SSC(\G)$.
Let $\phi\in\Phi(\G)$ be the $L$-parameter of $\pi^{\G}_{X}$ (namely, the $L$-packet $\Pi_{\phi}^{\G}$ of $\phi$ contains $\pi^{\G}_{X}$).
\begin{description}
\item[The case where $\G=\Sp_{2n}$ (Theorems \ref{thm:packetSp} and \ref{thm:liftSptoGL})]
The order of the $L$-packet $\Pi_{\phi}^{\G}$ is two and $\Pi_{\phi}^{\G}$ equals the orbit of $\pi^{\G}_{X}$ with respect to the action of the adjoint group of $\G$.
Moreover the endoscopic lift $\pi_{\phi}^{\GL_{2n+1}}$ of $\Pi_{\phi}^{\G}$ to $\GL_{2n+1}$ is given by the parabolic induction of 
\[
\pi_{Y}^{\GL_{2n}}\boxtimes\omega_{Y}^{\GL_{2n}}
\]
for some $Y\in\SSC(\GL_{2n})$ which is explicitly described in terms of $X\in\SSC(\Sp_{2n})$.

\item[The case where $\G=\SO_{2n}^{\mu}$ (Theorems \ref{thm:packetSOram} and \ref{thm:liftSOramtoGL})]
The $L$-packet $\Pi_{\phi}^{\G}$ is a singleton.
Moreover the endoscopic lift $\pi_{\phi}^{\GL_{2n}}$ of $\Pi_{\phi}^{\G}$ to $\GL_{2n}$ is a simple supercuspidal representation $\pi_{Y}^{\GL_{2n}}$ corresponding to an element $Y\in\SSC(\GL_{2n})$, which is explicitly described in terms of $X\in\SSC(\SO_{2n}^{\mu})$.
%

\item[The case where $\G=\SO_{2n+2}^{(\ur)}$ (Theorems \ref{thm:packetSOspl}, \ref{thm:packetSOspl2}, and \ref{thm:liftSOspltoGL})]
The order of the $L$-packet $\Pi_{\phi}^{\G}$ is two and $\Pi_{\phi}^{\G}$ equals the orbit of $\pi^{\G}_{X}$ with respect to the action of the adjoint group of $\G$.
Moreover the endoscopic lift of $\Pi_{\phi}^{\G}$ to $\GL_{2n+2}$ is given by the parabolic induction of 
\[
\begin{cases}
\pi_{Y}^{\GL_{2n}}\boxtimes(\omega_{Y}^{\GL_{2n}}\otimes\mu_{\G})\boxtimes\mathbbm{1} & \text{if } \zeta=1,\\
\pi_{Y}^{\GL_{2n}}\boxtimes(\omega_{Y}^{\GL_{2n}}\otimes\mu_{\ur}\otimes\mu_{\G})\boxtimes\mu_{\ur} & \text{if } \zeta=-1.
\end{cases}
\]
Here $\mu_{\G}$ is the quadratic character of $F^{\times}$ corresponding to $\G$, $\zeta$ is the fourth parameter of the data $X\in\SSC(\SO_{2n+2}^{(\ur)})$, and $Y\in\SSC(\GL_{2n})$ is an element described explicitly in terms of $X$.

\end{description}
\end{thm}

We note that, in the case of even special orthogonal groups, the local Langlands correspondence has been established modulo the action of the outer automorphism.
However, in this introduction, we ignore the difference between the set of irreducible smooth representations and the set of their orbits under the action of the outer automorphism for simplicity.

Theorem \ref{thm:main} gives a complete answer to the problem (1) for simple supercuspidal representations of $\Sp_{2n}(F)$, $\SO_{2n}^{\mu}(F)$, and $\SO_{2n+2}^{(\ur)}(F)$.
Moreover, since the $L$-parameters of simple supercuspidal representations of general linear groups have been explicitly determined by the works of Bushnell--Henniart (\cite{MR2148193}) and Imai--Tsushima (\cite{Imai:2015aa}), by combining Theorem \ref{thm:main} with them, we also get an answer to the problem (2) for the lifted representations.


Before we explain the outline of the proof of the main theorem, we remark on several preceding works:
\begin{itemize}
\item
In our previous works (\cite{MR3904769, MR4025003}), we considered the same problem for split special orthogonal groups of odd degrees and unramified quasi-split unitary groups, and got results of the same type. 
\item
In the case of split special orthogonal groups of odd degrees, the endoscopic lifts of simple supercuspidal representations had already been determined by Adrian (\cite{MR3518182}) before our work (\cite{MR3904769}), under some assumption on the residual characteristic $p$.
His method is based on a computation of the twisted local $\gamma$-factors of simple supercuspidal representations, and totally different from our one.
\item
For tamely ramified connected reductive groups, Kaletha gave an explicit construction of $L$-packets consisting of \textit{epipelagic} representations in \cite{MR3402796}.
Since a simple supercuspidal representation is a special case of epipelagic representations, the $L$-packets constructed by him includes our ones.
Moreover he showed various expected properties of the local Langlands correspondence for his $L$-packets.
In particular, he proved the stability of $L$-packets under some assumption on the residual characteristic.
Thus we can say that, in the cases of quasi-split classical groups, his $L$-packets coincide with those of Arthur and Mok.
On the other hand, the endoscopic character relation for twisted endoscopy has not been checked for his $L$-packets yet.
In other words, the endoscopic lifts of his $L$-packets to general linear groups have not been determined yet.
Therefore our main results (Theorem \ref{thm:main}) do not follow from his results.
We also emphasize that, in our method, we only have to assume that the residual characteristic $p$ is odd.
\end{itemize}

Now we first explain the rough idea of the proof of Theorem \ref{thm:main}.
The starting point is a computation of characters of simple supercuspidal representations.
By using the character formula for supercuspidal representations, we can write the character of a simple supercuspidal representation as a group-theoretical sum of the values of an affine generic character.
In particular, at ``shallowest'' elements of pro-$p$ Iwahori subgroups (which we will call \textit{affine generic} elements), we can write the characters of simple supercuspidal representations in terms of \textit{Kloosterman sums}.
Our basic strategy is to combine such a computation with the endoscopic character relation.
That is, we get the values of the twisted character of $\pi_{\phi}^{\GL_{N}}$ by combining such a computation with the endoscopic character relation, and then recover $\pi_{\phi}^{\GL_{N}}$ from its twisted character.
However, in carrying out such a procedure, there are several difficulties.

First, a priori, there is a possibility that $\Pi_{\phi}^{\G}$ contains a representation which is totally different from simple supercuspidal representations.
Therefore we first have to determine the structure of $\Pi_{\phi}^{\G}$.
To accomplish this, we utilize various properties of the local Langlands correspondence.

Second, we do not have a full character formula for the twisted characters of the representations which are obtained by the parabolic induction from ``non-$\theta$-stable'' parabolic subgroups (here $\theta$ is an involution of $\G$ used to define the twisted character).
In particular, we do not have a way to compute the twisted characters of the lifted representations of $\GL_{2n+1}(F)$ in the case where $\G=\Sp_{2n}$ in Theorem \ref{thm:main}.
To resolve this difficulty, we first study the standard endoscopy of $\Sp_{2n}$ and reduce the problem to the case where $\G=\SO_{2n}^{\mu}$.

Let us explain the more detailed outline of the proof of Theorem \ref{thm:main}.
From now on, we put $\G:=\Sp_{2n}$.
Let $\pi_{X}^{\G}$ and $\phi$ be as in Theorem \ref{thm:main}.
Then the proof of main theorem (Theorem \ref{thm:main}) can be divided into four parts as follows:

\textbf{Step 1. Determine the structure of $\Pi_{\phi}^{\G}$:}\quad
The first step is to determine the structure of $\Pi_{\phi}^{\G}$.
To do this, we first note that, by the \textit{stability} of $L$-packets, $\Pi_{\phi}^{\G}$ consists of orbits with respect to the action of the adjoint group $\G_{\mathrm{ad}}$ of $\G$.
According to a result of Kaletha in \cite{MR3001735}, every $G_{\ad}$-orbit of simple supercuspidal representations consists of exactly two simple supercuspidal representations, only one of which is generic (with respect to a fixed Whittaker datum).
Thus $\Pi_{\phi}^{\G}$ contains at least two simple supercuspidal representations, one of which is generic.
Moreover, by combining this observation with a result of M{\oe}glin and Xu (\cite{MR2767522,MR3713922}), we know that every member of $\Pi_{\phi}^{\G}$ is supercuspidal.

On the other hand, by using a constancy and a non-vanishing property of the characters of simple supercuspidal representations at affine generic elements of the Iwahori subgroup, we can show that if the order of $\Pi_{\phi}^{\G}$ is greater than two, then $\Pi_{\phi}^{\G}$ contains either an irreducible depth-zero supercuspidal representation or another simple supercuspidal representation.
However we can eliminate these possibilities as follows.
First, by the above Kaletha's result and the uniqueness of a generic representation in an $L$-packet (uniqueness part of so-called \textit{Shahidi's generic packet conjecture}, which is established by \cite{MR3592596} and \cite{Atobe:2015aa}), $\Pi_{\phi}^{\G}$ has no more  simple supercuspidal representation.
Second, by the \textit{constancy of the formal degree} of representations in each $L$-packet (which is proved in \cite{MR1070599}), in order to show that $\Pi_{\phi}^{\G}$ does not have an irreducible depth-zero supercuspidal representation, it suffices to show that the formal degree of a simple supercuspidal representation is not equal to those of irreducible depth-zero supercuspidal representations.
However, since irreducible depth-zero supercuspidal representations are obtained by the compact induction of irreducible cuspidal representations of the reductions of maximal parahoric subgroups of $G$ (\cite{MR1371680}), we can check this easily by studying the dimensions of irreducible cuspidal representations of reductive groups over finite fields.

\textbf{Step 2. Reduce the case of $\Sp_{2n}$ to the case of $\SO_{2n}^{\mu}$:}\quad
The second step is to study the relationship between simple supercuspidal $L$-packets of $\Sp_{2n}$ and those of ramified even special orthogonal groups $\SO_{2n}^{\mu}$.
Since the order of $\Pi_{\phi}^{\G}$ is two, we know that $\Pi_{\phi}^{\G}$ is the endoscopic lift of an $L$-packet $\Pi_{\phi}^{\mathbf{H}}$ of a proper endoscopic group $\mathbf{H}$ of $\G$ (i.e., $\mathbf{H}\neq\G$).
In fact, by the argument in Step 1, we can show that this group $\mathbf{H}$ is the special orthogonal group $\SO_{2n}^{\mu}$ of degree $2n$ corresponding to a ramified quadratic character $\mu$ of $F^{\times}$ which is determined by the data $X\in\SSC(\Sp_{2n})$.

On the other hand, we note that $(\G,\mathbf{H})$ is a \textit{dual pair} (strictly speaking, we should consider the orthogonal group rather than the special orthogonal group, but we ignore this difference in this introduction).
Then $\Pi_{\phi}^{\G}$ can also be regarded as (a character twist of) the \textit{theta lift} of $\Pi_{\phi}^{\mathbf{H}}$, by the compatibility of the theta lifting and the endoscopic lifting, which is known by \cite{MR3166215} (or a special case of \cite{MR3714507}).
As the theta correspondence preserves the depth of representations (\cite{MR1909608}), we can conclude that $\Pi_{\phi}^{\mathbf{H}}$ consists of only one representation, which is simple supercuspidal.
Let us denote this representation by $\pi_{X'}^{\mathbf{H}}$, for an element $X'$ of $\SSC(\mathbf{H})$.

The final task in this step is to determine the data $X'$ by computing the endoscopic character relation for $\Pi_{\phi}^{\G}$ and $\Pi_{\phi}^{\mathbf{H}}$.
Then we know that every simple supercuspidal representation of $H$ can be obtained by such a ``descent'' of simple supercuspidal representations of $G$, and the case of $\Sp_{2n}$ of Theorem \ref{thm:main} is reduced to the case of $\SO_{2n}^{\mu}$.
 
\textbf{Step 3. Determine the endoscopic lift of $\Pi_{\phi}^{\mathbf{H}}$ to $\GL_{2n}$:}\quad
The third step is to determine the endoscopic lift $\pi_{\phi}^{\GL_{2n}}$ of $\Pi_{\phi}^{\mathbf{H}}$ to $\GL_{2n}$ and complete the proof of the case of $\SO_{2n}^{\mu}$ in the main theorem (Theorem \ref{thm:main}).

We first show that the depth of $\pi^{\GL_{2n}}_{\phi}$ is not greater than $\frac{1}{2n}$ (in particular $\pi_{\phi}^{\GL_{2n}}$ is either simple supercuspidal or depth-zero supercuspidal).
Since the representation $\pi^{\GL_{2n+1}}_{\phi}$ is obtained by the parabolic induction of the tensor product of $\pi^{\GL_{2n}}_{\phi}$ and its central character, the depth of $\pi^{\GL_{2n}}_{\phi}$ is equal to that of $\pi^{\GL_{2n+1}}_{\phi}$.
Thus it suffices to show that the depth of $\pi^{\GL_{2n+1}}_{\phi}$ is not greater than $\frac{1}{2n}$.
In order to evaluate the depth of $\pi^{\GL_{2n+1}}_{\phi}$, it is enough to show that its twisted character is constant on a sufficiently large open compact set.
By using the endoscopic character relation for $(\Pi_{\phi}^{\G},\pi_{\phi}^{\GL_{2n+1}})$, we can reduce it to constancy of the characters of representations belonging to $\Pi_{\phi}^{\mathbf{G}}$ on an open compact coset of the Iwahori subgroup of $G$, and we can check it easily.
The key point of this argument is to consider the endoscopic character relation for $(\Pi_{\phi}^{\G},\pi_{\phi}^{\GL_{2n+1}})$, not for $(\Pi_{\phi}^{\mathbf{H}}, \pi_{\phi}^{\GL_{2n}})$.
The reason why we consider the pair $(\GL_{2n+1}, \G)$ rather than $(\GL_{2n},\mathbf{H})$ is that the Kottwitz--Shelstad transfer factor, which appears in the endoscopic character relation, for $(\GL_{2n+1}, \G)$ is trivial while that for $(\GL_{2n},\mathbf{H})$ is not trivial.
Namely, for the pair $(\GL_{2n},\mathbf{H})$, the above argument fails because the Kottwitz--Shelstad transfer factor may not be constant on a coset of an Iwahori subgroup.

After we evaluate the depth of $\pi^{\GL_{2n}}_{\phi}$, we eliminate the possibility that $\pi^{\GL_{2n}}_{\phi}$ is depth-zero and determine $\pi_{\phi}^{\GL_{2n}}$ by computing the endoscopic character relation for $(\GL_{2n},\mathbf{H})$ at affine generic elements.

\[
\xymatrix{
\GL_{2n+1}&&\G\ar@{~>}[ll]_-{\text{twisted}}^-{\text{endoscopy}}&&\mathbf{H}\ar@{~>}[ll]_-{\text{standard}}^{\text{endoscopy}}\ar@{~>}[rr]^-{\text{twisted}}_-{\text{endoscopy}}&&\GL_{2n}\\
\pi_{\phi}^{\GL_{2n+1}}&&\Pi_{\phi}^{\G}\ar@{~>}[ll]_-{\text{endoscopic}}^-{\text{lifting}} &&\Pi_{\phi}^{\mathbf{H}}\ar@{~>}[ll]_-{\text{endoscopic}}^{\text{lifting}}\ar@{~>}[rr]^-{\text{endoscopic}}_-{\text{lifting}} &&\pi^{\GL_{2n}}_{\phi}
}
\]

\textbf{Step 4. Deduce the case of $\SO_{2n+2}^{(\ur)}$ from the case of $\SO_{2n}^{\mu}$:}\quad
The final step is to construct the simple supercuspidal $L$-packets of split or unramified quasi-split even special orthogonal groups, and determine their endoscopic lifts to general linear groups.
For simplicity, here we consider only the split case.

In order to construct simple supercuspidal $L$-packets, we consider the theta lifting.
More precisely, if we take a ramified quadratic character $\mu$ of $F^{\times}$, then $(\SO_{2n}^{\mu}, \Sp_{2n})$ and $(\Sp_{2n}, \SO_{2n+2})$ are dual pairs, so we can construct representations of $\SO_{2n+2}$ from those of $\SO_{2n}^{\mu}$ by considering the theta lifting twice.
On the other hand, $\SO_{2n}^{\mu}\times\SO_{2}^{\mu}$ is an endoscopic group of $\SO_{2n+2}$, and the theta lift of $\Pi_{\phi}^{\SO_{2n}^{\mu}}$ to $\SO_{2n+2}$ coincides with the endoscopic lift of the $L$-packet $\Pi_{\phi}^{\SO_{2n}^{\mu}}\times\{\mathbbm{1}\}$ to $\SO_{2n+2}$ up to a character twist, by the compatibility of the theta lifting and the endoscopic lifting.
Again by using the depth-preserving property of the theta lifting, we can show that the lifted $L$-packet consists of two simple supercuspidal representations.
Finally, by computing the endoscopic character relation for $(\SO_{2n}^{\mu}\times\SO_{2}^{\mu},\SO_{2n+2})$, we can determine the lifted $L$-packet.

In fact, this construction gives only the half of the simple supercuspidal representations of $\SO_{2n+2}(F)$.
However, we can get the other half by twisting these simple supercuspidal $L$-packets by the \textit{spinor norm} character of $\SO_{2n+2}(F)$.
This completes the proof of Theorem \ref{thm:main}.
\[
\xymatrix{
&&\SO_{2n+2}\\
&&\Sp_{2n}\ar@{<~>}[u]_-{\text{theta lifting}}\\
\SO_{2n}^{\mu}\times\SO_{2}^{\mu}\ar@{~>}[uurr]^-{\text{endoscopic lifting}}&&\SO_{2n}^{\mu}\ar@{<~>}[u]_-{\text{theta lifting}}\ar@{~>}[ll]_-{\text{inflation}}^-{\text{via }\{\mathbbm{1}\}}
}
\]

Finally, we comment on applications of our results.
It is expected that the local Langlands correspondence satisfies a lot of properties other than the endoscopic character relation.
We can use the results in this paper as a touchstone in verifying such expectations.

For example, we can check the \textit{formal degree conjecture} for our $L$-packets.
The formal degree conjecture was formulated by Hiraga--Ichino--Ikeda in \cite{MR2350057} and asserts that there is an explicit relation between the special value of the adjoint $\gamma$-factor of a discrete $L$-parameter $\phi$ and the formal degree of the representations in the $L$-packet of $\phi$.
This conjecture is proved for several groups, for example, general linear groups (\cite{MR2350057}), odd special orthogonal groups (\cite{MR3649356}), and some small classical groups such as the unitary group of degree $3$ (\cite{MR3166215}).
However, for other groups such as symplectic groups, the formal degree conjecture has still been open.
We consider this conjecture for simple supercuspidal representations.
Since simple supercuspidal representations are obtained by the compact induction, we can compute their formal degree quite easily.
On the other hand, as we mentioned before, we have an explicit description of the $L$-parameters of simple supercuspidal representations as a consequence of our main results (Theorem \ref{thm:main}) and the works of Bushnell--Henniart and Imai--Tsushima.
By using this description, a computation of the special value of the adjoint $\gamma$-factors of such $L$-parameters is reduced to a simple problem of representation theory of finite groups.
In the last part of this paper, we carry out such a computation and confirm that the formal degree conjecture holds for simple supercuspidal $L$-packets of the quasi-split classical groups, under some assumption on $p$ (including some cases of ``bad'' primes).

Another example is the \textit{depth preserving property} of the local Langlands correspondence.
We can also define the \textit{depth} for $L$-parameters by using the ramification filtration of the Weil group $W_{F}$.
Then, it is known that the local Langlands correspondence for general linear groups preserves the depth of representations (e.g., see \cite{MR2508720} or \cite{MR3618046}).
Thus it is natural to investigate the relationship between the local Langlands correspondence and the depth for other groups.
It is worth pointing out that, by Theorem \ref{thm:main}, the depth-preserving property holds for simple supercuspidal $L$-packets.

\medbreak
\noindent{\bfseries Organization of this paper.}\quad
We explain the organization of this paper.
In Section \ref{sec:ssc}, we recall the definition of simple supercuspidal representations and describe the parametrizing set $\SSC(\G)$ of equivalence classes of simple supercuspidal representations explicitly for each $\G$ (a general linear group, a symplectic group, and a special orthogonal group of even degree).
In Section \ref{sec:char}, we compute the characters of simple supercuspidal representations at some special elements such as affine generic elements.
The results in this section are used throughout this paper.
In Section \ref{sec:Arthur}, we recall Arthur's result, namely the local Langlands correspondence for quasi-split classical groups, and the formulation of the endoscopic character relation.
Moreover we also explain some other properties of the local Langlands correspondence.
In Section \ref{sec:Sp}, we determine the structure of simple supercuspidal $L$-packets of symplectic groups.
In Section \ref{sec:tran}, we compute the Kottwitz--Shelstad transfer factors for classical groups by using the formula of Waldspurger established in \cite{MR2672539}.
The results in this section are the key to compute the endoscopic character relations explicitly.
In Section \ref{sec:SO-ram}, we determine the structure of simple supercuspidal $L$-packets of ramified special orthogonal groups of even degrees.
Furthermore, we determine the endoscopic lifts of them to the general linear groups.
From this, we also get a description of the $L$-parameters of simple supercuspidal $L$-packets of symplectic groups.
In Section \ref{sec:SO-ur}, we determine the structure of simple supercuspidal $L$-packets of unramified special orthogonal groups of even degrees, and determine their $L$-parameters.
In Section \ref{sec:FDC}, as an application of our results, we check the formal degree conjecture for simple supercuspidal representations of the above groups.
In Section \ref{sec:Kl}, we collect some properties of the Kloosterman sums, which arise as the character values of simple supercuspidal representations at affine generic elements.
In Section \ref{sec:depth}, we characterize simple supercuspidal representations in terms of the depth.
In Section \ref{sec:spin}, we show the compatibility of the local Langlands correspondences for special orthogonal groups of even degrees and the spinor twist.
This result is used in the final step in Section \ref{sec:SO-ur}.

\medbreak
\noindent{\bfseries Acknowledgment.}\quad
The author expresses his gratitude to his advisor Yoichi Mieda for his encouragement and a lot of valuable advice.
He also wishes to thank Naoki Imai for a comment on the $L$-parameters of simple supercuspidal representations and Hiraku Atobe for a comment on the theta correspondence.
Finally, the author is deeply grateful to the anonymous referee for his/her thorough reading of this paper and for providing a huge number of constructive and fruitful suggestions.
All of the advice given by the referee were very essential and quite helpful in improving this paper.
Especially, the proof of the results in Section C has been drastically simplified thanks to the referee's suggestion; the present version of the proof is completely due to him/her.
The author would like to emphasize that this work would never have been completed without the referee's help.

This work was carried out with the support from the Program for Leading Graduate Schools MEXT, JSPS Research Fellowship for Young Scientists, and KAKENHI Grant Number 17J05451.

\setcounter{tocdepth}{2}
\tableofcontents

\noindent{\bfseries Notation.}\quad
\begin{description}
\item[$p$-adic field]
Let $p$ be an odd prime number.
We fix a $p$-adic field $F$.
We denote its ring of integers, its maximal ideal, and its residue field by $\mcO$, $\mfp$, and $k$, respectively.
We fix a uniformizer $\varpi$ of $F$.
Let $q$ be the order of $k$.
For $x \in \mcO$, $\bar{x}$ denotes the image of $x$ in $k$.
We often regard an element of $k^{\times}$ as an element of $F^{\times}$ by the Teichm{\"u}ller lift.

We fix an algebraic closure $\ol{F}$ of $F$ and write $\ol{k}$ for the residue field of $\ol{F}$.
We denote the Weil group of $F$, its inertia subgroup, and its wild inertia subgroup by $W_{F}$, $I_{F}$, and $P_{F}$, respectively.

\item[quadratic extension of $F$]
For a quadratic character $\mu$ of $F^{\times}$, we denote by $E_{\mu}$ the quadratic extension of $F$ corresponding to $\mu$ under the local class field theory.
Let $\mu_{\ur}$ be the nontrivial unramified quadratic character of $F^{\times}$, and we simply write $E_{\ur}$ for $E_{\mu_{\ur}}$.

\item[quadratic extension of $k$]
We write $\tilde{k}$ for the quadratic extension of $k$.
We denote the norm and the trace with respect to the extension $\tilde{k}/k$ shortly by $\Nr$ and $\Tr$.
We set $\Nr^{1}$ to be the kernel of $\Nr$:
\[
\Nr^{1}:= \Ker(\Nr\colon \tilde{k}^{\times}\rightarrow k^{\times}).
\]

We write $\omega_{0}$ for the nontrivial quadratic character of $k^{\times}$.
Let $k^{\times2}$ be the set of square elements of $k^{\times}$ and we fix a non-square element $\epsilon\in k^{\times}\smallsetminus k^{\times2}$ and an element $\tilde{\epsilon}$ of $\tilde{k}$ satisfying $\Nr(\tilde{\epsilon})=\epsilon$.

Finally, for a character $\psi$ (resp.\ $\chi$) on $k$ (resp.\ $k^{\times}$), we write $\tilde\psi$ (resp.\ $\tilde\chi$) for the character $\psi\circ\Tr$ (resp.\ $\chi\circ\Nr$) on $\tilde{k}$ (resp. $\tilde{k}^{\times}$).

\item[additive character]
Throughout this paper, we fix an additive character $\psi$ on $F$ of level one.
Then its restriction $\psi|_{\mcO}$ to $\mcO$ induces a nontrivial additive character on $k$.
We denote it by $\psi$ again.

\item[algebraic group]
For an algebraic variety $\G$ over $F$, we denote the set of its $F$-valued points by $G$.
When $\G$ is an abelian algebraic group, we write $G(q)$ for the subgroup of $G$ consisting of the elements whose orders are finite and prime to $p$.
When $\G$ is a connected reductive group, we write $\widehat{\G}$ and ${}^{L}\G=\widehat{\G}\rtimes W_{F}$ for its Langlands dual group and $L$-group, respectively.
For an algebraic group $\mathbf{T}$ over $F$ or $\C$, we write $X^\ast(\mathbf{T})$ for its absolute character group and $X_{\ast}(\mathbf{T})$ for its absolute cocharacter group.

\item[matrix]
We denote the identity matrix of size $N$ by $I_{N}$.
We set $J_{N}$ to be the anti-diagonal matrix whose $(i, N+1-i)$-th entry is given by $(-1)^{i-1}$:
\[
J_{N}:=
\begin{pmatrix} 
&&&1\\
&&-1&\\
&\adots&&\\
(-1)^{N-1}&&&
\end{pmatrix}.
\]
\end{description}

\section{Simple supercuspidal representations of classical groups}\label{sec:ssc}

In this section, we prove some group-theoretic properties of Iwahori subgroups and explain the definition of simple supercuspidal representations.
Our arguments are basically the same as in the case of split groups in \cite[Section 2]{MR3904769}, but we explain them for the sake of completeness.

\subsection{Iwahori subgroups}\label{subsec:Iwahori}
Let $\G$ be a quasi-split connected reductive group over $F$, and we assume that $\G$ is tamely ramified over $F$.
Note that every quasi-split classical group considered in this paper satisfies this assumption.

We denote the center of $\G$ by $\bfZ_{\G}$.
We fix a maximal $F$-split torus $\bfS_{\G}$ of $\G$, and let $\T_{\G}$ be the centralizer of $\bfS_{\G}$ in $\G$.
Note that $\T_{\G}$ is a maximal torus of $\G$ because of the assumption of the quasi-splitness of $\G$.
Then we get the set $\Phi_{\G}$ (resp.\ $\Psi_{\G}$) of relative roots (resp.\ affine roots).
We fix a Borel subgroup $\mathbf{B}_{\G}$ of $\G$ which contains $\bfS_{\G}$ and is defined over $F$, and write $\Delta_{\G}$ for the corresponding root basis of $\Phi_{\G}$.
For each root $a \in \Phi_{\G}$, we denote by $\bfU_{a}$ the corresponding root subgroup of $\G$.

Recall that the fixed split torus $\bfS_{\G}$ defines the reduced apartment $\mathcal{A}_{\red}(\G,\bfS_{\G})$ of the reduced Bruhat--Tits building $\mathcal{B}_{\red}(\G,F)$ of $\G$ (see Section \ref{subsec:BT}).
We fix an alcove $\mathcal{C}$ in $\mathcal{A}_{\red}(\G, \bfS_{\G})$ of $\bfS_{\G}$ in $\G$.
This determines an affine root basis $\Pi_{\G}$ of $\Psi_{\G}$ and the set $\Psi^{+}_{\G}$ of positive affine roots.
For any affine function $\alpha$ on $\mathcal{A}_{\red}(\G,\bfS_{\G})$, we write $U_{\alpha}$ for its associated subgroup of $\G(F)$, see Section \ref{subsubsec:filtration}.
We set the Iwahori subgroup associated to $\mathcal{C}$ and its subgroups as follows: 
\begin{align*}
I_{\G} &:= \lan T_{\G}^{0}, U_\alpha \mid \alpha \in \Psi_{\G}^+\ran,\\
I_{\G}^{+} &:= \lan T_{\G}^{1}, U_\alpha \mid \alpha \in \Psi_{\G}^+\ran \text{, and}\\
I_{\G}^{++} &:= \lan T_{\G}^{1}, U_\alpha \mid \alpha \in \Psi_{\G}^+ \smallsetminus \Pi_{\G}\ran,
\end{align*}
where $T_{\G}^{0}$ is the unique parahoric subgroup of $T_{\G}=\T_{\G}(F)$, and $T_{\G}^{1}$ is the pro-unipotent radical of $T_{\G}^{0}$.
These groups are the first three steps of the Moy--Prasad filtration of the Iwahori subgroup $I_{\G}$ associated to the barycenter of the alcove $\mathcal{C}$ (see \cite[Section 2.6]{MR3164986} and also Section \ref{subsec:depth}).
We note that $I_{\G}^{+}$ and $I_{\G}^{++}$ are normal in $I_{\G}$, and the graded quotients are finite abelian groups.

\begin{prop}\label{prop:I-quot}
\begin{enumerate}
 \item We have
 \[
 I_{\G}/I_{\G}^{+} \cong T_{\G}^{0}/T_{\G}^{1}.
 \]
 \item The quotient $V_{\G}:=I_{\G}^{+}/I_{\G}^{++}$ is a finite-dimensional $k$-vector space and $\bfS_{\G}(k)$ acts on it.
 Moreover we have
 \[
V_{\G} \cong \bigoplus_{\alpha \in \Pi_{\G}} V_{\G}(\dot{\alpha}).
 \]
 Here, $V_{\G}(\dot{\alpha})$ is the $\dot{\alpha}$-isotypic part of $V_{\G}$ with respect to the $\bfS_{\G}(k)$-action.
\end{enumerate}
\end{prop}

\begin{proof}
Since every element in $I_{\G}^{+}$ is $p$-adically unipotent, the intersection of $I_{\G}^{+}$ and $T_{\G}^{0}$ is equal to $T_{\G}^{1}$.
Hence we get (1).

We prove (2).
We first recall that the multiplication map 
\[
\T_{\G}(F)\times\prod_{a\in\Phi_{\red}}\bfU_{a}(F) \ra G
\]
is injective in any order (see \cite[2.2.3]{MR756316}), where $\Phi_{\red}$ is the set of non-divisible roots of $\bfS_{\G}$.
Moreover this induces bijections 
\begin{align*}
&T_{\G}^{0}\times\prod_{a\in\Phi_{\red}} \mathcal{U}_{a}^{+} \ra I_{\G}\\
&T_{\G}^{1}\times\prod_{a\in\Phi_{\red}} \mathcal{U}_{a}^{+} \ra I_{\G}^{+},\quad\text{and}\\
&T_{\G}^{1}\times\prod_{a\in\Phi_{\red}} \mathcal{U}_{a}^{++} \ra I_{\G}^{++}.
\end{align*}
Here we put
\[
\mathcal{U}_{a}^{+}
:=
\begin{cases}
U_{\alpha_{a}^{+}}&\text{if $2a$ is not a root},\\
U_{\alpha_{a}^{+}}\cdot U_{\alpha_{2a}^{+}} &\text{if $2a$ is a root},
\end{cases}
\]
where $\alpha_{a}^{+}$ and $\alpha_{2a}^{+}$ are the smallest elements of $\Psi_{\G}^{+}$ satisfying $\dot{\alpha}_{a}^{+}=a$ and $\dot{\alpha}_{2a}^{+}=2a$, respectively.
Similarly, we put
\[
\mathcal{U}_{a}^{++}
:=
\begin{cases}
U_{\alpha_{a}^{++}}&\text{if $2a$ is not a root},\\
U_{\alpha_{a}^{++}}\cdot U_{\alpha_{2a}^{++}} &\text{if $2a$ is a root},
\end{cases}
\]
where $\alpha_{a}^{++}$ and $\alpha_{2a}^{++}$ are the smallest elements of $\Psi_{\G}^{+}\smallsetminus\Pi_{\G}$ satisfying $\dot{\alpha}_{a}^{++}=a$ and $\dot{\alpha}_{2a}^{++}=2a$, respectively.
See Propositions 6.4.9 and 6.4.48 in \cite{MR0327923} for details of these bijections.

Here note that the difference between $\mathcal{U}_{a}^{+}$ and $\mathcal{U}_{a}^{++}$ is nontrivial only when $a$ or $2a$ is the gradient of some simple affine root.
For such a root $a\in\Phi_{\red}$, let us consider the structure of the quotient $\mathcal{U}_{a}^{+}/\mathcal{U}_{a}^{++}$.

We first consider the case where $a=\dot{\alpha}$ for some $\alpha\in\Pi_{\G}$.
Let us suppose that also $2a$ is a root.
Then the affine roots $\alpha_{a}^{+}$, $\alpha_{2a}^{+}$, $\alpha_{a}^{++}$, and $\alpha_{2a}^{++}$ are given as follows:
\begin{itemize}
\item
We have $\alpha_{a}^{+}=\alpha$.
\item
We have $\alpha_{2a}^{+}=2\alpha$.
This can be checked as follows.
First, the gradient of $2\alpha$ is given by $2a$ and $2\alpha$ is a positive affine root.
Thus it suffices to show the minimality of $2\alpha$.
Since $\alpha$ is a simple root, $\alpha$ vanishes on some wall of the fixed alcove $\mathcal{C}$.
Thus also $2\alpha$ vanishes on the wall.
Hence every positive affine root whose gradient is given by $2a$ cannot be smaller than $2\alpha$.
\item
We have $\alpha_{a}^{++}=\alpha+r$, where $r$ is the smallest positive number such that $\alpha+r$ is an affine root.
\item
We have $\alpha_{2a}^{++}=2\alpha$.
Indeed, since $2\alpha$ is not a simple affine root, it belongs to $\Psi_{\G}^{+}\smallsetminus\Pi_{\G}$.
As $2\alpha$ is the smallest positive affine root whose gradient is given by $2a$, in particular, it is the smallest one in $\Psi_{\G}^{+}\smallsetminus\Pi_{\G}$.
\end{itemize}
Therefore, for such a root $a$, the quotient $\mathcal{U}_{a}^{+}/\mathcal{U}_{a}^{++}$ is given by
\[
(U_{\alpha}\cdot U_{2\alpha}) / (U_{\alpha+r}\cdot U_{2\alpha})
=
U_{\alpha} / (U_{\alpha+r}\cdot U_{2\alpha}).
\]
Here recall that $U_{2\alpha}$ is contained in $U_{\alpha}$.
Let us show that this quotient is in fact equal to $U_{\alpha}/(U_{\alpha+\varepsilon}\cdot U_{2\alpha})=\bar{U}_{\alpha}/\bar{U}_{2\alpha}$.
Since the valuation of $F$ is discrete, the filtration $\{U_{a+s}\}_{s\in\R_{\geq0}}$ decreases discretely.
We put
\[
\alpha=\alpha+r_{0}<\alpha+r_{1}<\cdots<\alpha+r_{n}=\alpha+r
\] 
to be the points between $\alpha$ and $\alpha+r$ where the filtration jumps.
Namely, for any $r_{i}<s\leq r_{i+1}$, we have 
\[
U_{\alpha+r_{i}}
\supsetneq
U_{\alpha+s}
=
U_{\alpha+r_{i+1}}.
\]
Since $r$ is the smallest positive number such that $\alpha+r$ is an affine root, $\alpha+r_{i}$ is not an affine root for $1\leq i < n$.
Thus, for every $1\leq i < n$, we have $\bar{U}_{\alpha+r_{i}}=\bar{U}_{2(\alpha+r_{i})}$ (recall that an affine function $\beta$ is called an affine root if it satisfies $\bar{U}_{\beta}\supsetneq\bar{U}_{2\beta}$, see \cite[Section 1.6]{MR546588}).
In other words, we have $U_{\alpha+r_{i}}=U_{\alpha+r_{i}+\varepsilon}\cdot U_{2(\alpha+r_{i})}$.
As we have $U_{\alpha+r_{i}+\varepsilon}=U_{\alpha+r_{i+1}}$ and $U_{2(\alpha+r_{i})}\subset U_{2\alpha}$, we get $U_{\alpha+r_{i}}\subset U_{\alpha+r_{i+1}}\cdot U_{2\alpha}$.
By considering this inclusion for each $i$ and noting that $U_{\alpha+\varepsilon}=U_{\alpha+r_{1}}$ for any sufficiently small positive number $\varepsilon$, we finally get $U_{\alpha+\varepsilon}\subset U_{\alpha+r}\cdot U_{2\alpha}$.
Therefore we can conclude that
\[
U_{\alpha}/(U_{\alpha+r}\cdot U_{2\alpha})
=
U_{\alpha}/(U_{\alpha+\varepsilon}\cdot U_{2\alpha})
=
\bar{U}_{\alpha}/\bar{U}_{2\alpha}.
\]

On the other hand, if $2a$ is not a root, then we only have $\alpha_{a}^{+}=\alpha$ and $\alpha_{a}^{++}=\alpha+r$.
Then, by a similar, but much simpler argument as above, we can show that the quotient $\mathcal{U}_{a}^{+}/\mathcal{U}_{a}^{++}$ is equal to $\bar{U}_{\alpha}/\bar{U}_{2\alpha}$ (recall that $\bar{U}_{2\alpha}$ is trivial by definition).

We next consider the case where $a=\frac{1}{2}\dot{\alpha}$ for some $\alpha\in\Pi_{\G}$.
Then the affine roots $\alpha_{a}^{+}$, $\alpha_{2a}^{+}$, $\alpha_{a}^{++}$, and $\alpha_{2a}^{++}$ are given as follows:
\begin{itemize}
\item
We have $\alpha_{a}^{+}=\frac{1}{2}\alpha+r$.
Here $r$ is the smallest positive number such that $\frac{1}{2}\alpha+r$ is a positive affine root.
Note that $\frac{1}{2}\alpha$ is not an affine root since $\alpha$ is a simple affine root.
Thus $\alpha_{a}^{+}$ have to be strictly greater than $\frac{1}{2}\alpha$.
\item
We have $\alpha_{2a}^{+}=\alpha$.
\item
We have $\alpha_{a}^{++}=\frac{1}{2}\alpha+r$.
To show this, it suffices to check that $\frac{1}{2}\alpha+r$ is not a simple affine root.
If $\frac{1}{2}\alpha+r$ is simple, then it vanishes on some wall of the alcove.
However $\frac{1}{2}\alpha+r$ is strictly greater than $\frac{1}{2}\alpha$, which is non-negative on every wall of the alcove.
This is a contradiction, hence $\frac{1}{2}\alpha+r$ is not simple.
\item
We have $\alpha_{2a}^{++}=\alpha+r'$.
Here $r'$ is the smallest positive number such that $\alpha+r'$ is a positive affine root.
\end{itemize}
Therefore, for such a root $a$, the quotient $\mathcal{U}_{a}^{+}/\mathcal{U}_{a}^{++}$ is given by
\[
(U_{\frac{1}{2}\alpha+r}\cdot U_{\alpha}) / (U_{\frac{1}{2}\alpha+r}\cdot U_{\alpha+r'}).
\]
In particular, we have a surjection from $U_{\alpha}/U_{\alpha+r'}$ to $\mathcal{U}_{a}^{+}/\mathcal{U}_{a}^{++}$.
By a similar argument to the previous case, we can show that $U_{\alpha}/U_{\alpha+r'}$ equals $\bar{U}_{\alpha}/\bar{U}_{2\alpha}$ (note that in this case $\bar{U}_{2\alpha}$ is trivial).

In summary, for every $\alpha\in\Pi_{\G}$, we have a surjection from $\bar{U}_{\alpha}/\bar{U}_{2\alpha}$ to the quotient
\[
\mathcal{U}_{\dot{\alpha}_{\red}}^{+}/\mathcal{U}_{\dot{\alpha}_{\red}}^{++}.
\]
Here we put 
\[
\dot{\alpha}_{\red}
:=
\begin{cases}
\dot{\alpha}& \text{if $\dot{\alpha}$ is reduced},\\
\frac{1}{2}\dot{\alpha}& \text{if $\dot{\alpha}$ is not reduced}.
\end{cases}
\]
Since $\bar{U}_{\alpha}/\bar{U}_{2\alpha}$ is a finite-dimensional $k$-vector space (see \cite[1.6]{MR546588}) and $\bfS_{\G}(k)$ acts on it by $\dot{\alpha}$, so is the quotient $\mathcal{U}_{\dot{\alpha}_{\red}}^{+}/\mathcal{U}_{\dot{\alpha}_{\red}}^{++}$.

Now, for each $\alpha\in\Pi_{\G}$, we consider the homomorphism
\[
\mathcal{U}_{\dot{\alpha}_{\red}}^{+}/\mathcal{U}_{\dot{\alpha}_{\red}}^{++}
\rightarrow
I_{\G}^{+}/I_{\G}^{++}=V_{\G}
\]
which is induced by the inclusion map $\mathcal{U}_{\dot{\alpha}_{\red}}^{+}\hookrightarrow I_{\G}^{+}$.
Here recall that $I_{\G}^{+}$ and $I_{\G}^{++}$ are the second and third steps of the Moy--Prasad filtration of the Iwahori subgroup $I_{\G}$, respectively (see Section \ref{subsec:depth}).
Since each graded quotient of the positive depth part of the Moy--Prasad filtration is abelian (see \cite[Section 3.2]{MR1371680}), $V_{\G}$ is abelian.
Thus we obtain a homomorphism
\[
\bigoplus_{\alpha\in\Pi_{\G}}\mathcal{U}_{\dot{\alpha}_{\red}}^{+}/\mathcal{U}_{\dot{\alpha}_{\red}}^{++}
\rightarrow
V_{\G}.
\]
We show that this map is in fact bijective.
If we can show it, then the proof completes.
First we check the surjectivity.
Let $gI_{\G}^{++}$ be an element of $V_{\G}$ with $g\in I_{\G}^{+}$.
Then, by the bijectivity of the map $T_{\G}^{1}\times\prod_{a\in\Phi_{\red}} \mathcal{U}_{a}^{+} \rightarrow I_{\G}^{+}$ for any order on $\Phi_{\red}$, we can write $g$ as $g=t\prod_{a\in\Phi_{\red}}x_{a}$.
Then, by noting that
\begin{itemize}
\item
the quotient $I_{\G}^{+}/I_{\G}^{++}$ is abelian,
\item
$T_{\G}^{1}$ is contained in $I_{\G}^{++}$, and
\item
$\mathcal{U}_{a}^{+}=\mathcal{U}_{a}^{++}\subset I_{\G}^{++}$ if $a\neq \dot{\alpha}_{\red}$ for any simple affine root $\alpha\in\Pi_{\G}$,
\end{itemize}
 we have
\[
gI_{\G}^{++}
=\prod_{\alpha\in\Pi_{\G}}x_{\dot{\alpha}_{\red}}I_{\G}^{++}.
\]
Here the product can be taken in any order since the quotient $I_{\G}^{+}/I_{\G}^{++}$ is abelian.
As $I_{\G}^{++}$ is normal in $I_{\G}^{+}$, we may regard $\prod_{\alpha\in\Pi_{\G}}x_{\dot{\alpha}_{\red}}I_{\G}^{++}$ as
\[
\Bigl(\prod_{\alpha\in\Pi_{\G}}x_{\dot{\alpha}_{\red}}\mathcal{U}_{\dot{\alpha}_{\red}}^{++}\Bigr)I_{\G}^{++}.
\]
Thus the element 
\[
(x_{\dot{\alpha}_{\red}}\mathcal{U}_{\dot{\alpha}_{\red}}^{++})_{\alpha\in\Pi_{\G}}\in\bigoplus_{\alpha\in\Pi_{\G}}\mathcal{U}_{\dot{\alpha}_{\red}}^{+}/\mathcal{U}_{\dot{\alpha}_{\red}}^{++}
\]
maps to $gI_{\G}^{++}$ under the above homomorphism.

We next check the injectivity.
We assume that two elements 
\[
(x_{\dot{\alpha}_{\red}}\mathcal{U}_{\dot{\alpha}_{\red}}^{++})_{\alpha\in\Pi_{\G}}, 
(y_{\dot{\alpha}_{\red}}\mathcal{U}_{\dot{\alpha}_{\red}}^{++})_{\alpha\in\Pi_{\G}}
\in\bigoplus_{\alpha\in\Pi_{\G}}\mathcal{U}_{\dot{\alpha}_{\red}}^{+}/\mathcal{U}_{\dot{\alpha}_{\red}}^{++}
\]
map to the same element of $I_{\G}^{+}/I_{\G}^{++}$ under the above homomorphism.
Then, again by noting the normality of $I_{\G}^{++}$ in $I_{\G}^{+}$, we have
\[
\prod_{\alpha\in\Pi_{\G}}x_{\dot{\alpha}_{\red}}I_{\G}^{++}
=
\prod_{\alpha\in\Pi_{\G}}y_{\dot{\alpha}_{\red}}I_{\G}^{++}.
\]
As the quotient $I_{\G}^{+}/I_{\G}^{++}$ is abelian, this implies that $\prod_{\alpha\in\Pi_{\G}}y_{\dot{\alpha}_{\red}}^{-1}x_{\dot{\alpha}_{\red}}$ belongs to $I_{\G}^{++}$.
Then, by considering the bijectivities of the maps $T_{\G}^{1}\times\prod_{a\in\Phi_{\red}} \mathcal{U}_{a}^{++} \rightarrow I_{\G}^{++}$ and $T_{\G}^{1}\times\prod_{a\in\Phi_{\red}} \mathcal{U}_{a}^{+} \rightarrow I_{\G}^{+}$, we can conclude that $y_{\dot{\alpha}_{\red}}^{-1}x_{\dot{\alpha}_{\red}}$ belongs to $\mathcal{U}_{\dot{\alpha}_{\red}}^{++}$ for each $\alpha\in\Pi_{\G}$.
Namely, we have $(x_{\dot{\alpha}_{\red}}\mathcal{U}_{\dot{\alpha}_{\red}}^{++})_{\alpha\in\Pi_{\G}} =(y_{\dot{\alpha}_{\red}}\mathcal{U}_{\dot{\alpha}_{\red}}^{++})_{\alpha\in\Pi_{\G}}$.
\end{proof}

We denote the image of $x\in I_{\G}^{+}$ under the map 
\[
I_{\G}^{+} \twoheadrightarrow \bigoplus_{\alpha \in \Pi_{\G}} V_{\G}(\dot{\alpha})
\]
 by $(x_\alpha)_{\alpha \in \Pi_{\G}}$, and call each $x_\alpha$ the \textit{simple affine component} of $x$.

\begin{rem}\label{rem:sim-aff-comp}
By the proof of Proposition \ref{prop:I-quot}, the simple affine components of an element $g=t\prod_{a\in\Phi_{\red}}x_{a}\in I_{\G}^{+}$ with $t\in T_{\G}^{1}$ and $x_{a}\in\mathcal{U}_{a}^{+}$ are given by $(\overline{x}_{\dot{\alpha}_{\red}})_{\alpha\in\Pi_{\G}}$, where $\overline{x}_{\dot{\alpha}_{\red}}$ is the image of $x_{\dot{\alpha}_{\red}}$ in $\mathcal{U}_{\dot{\alpha}_{\red}}^{+}/\mathcal{U}_{\dot{\alpha}_{\red}}^{++}=V_{\G}(\dot{\alpha})$.
\end{rem}

\begin{defn}
\begin{enumerate}
 \item An element $x \in I_{\G}^{+}$ is said to be $\mathit{affine}$ $\mathit{generic}$ if $x_\alpha$ is nonzero for every $\alpha \in \Pi_{\G}$.
 \item A character $\chi \colon Z_{\G}I_{\G}^{+} \ra \C^{\times}$ is called \textit{affine generic} if its restriction $\chi|_{I_{\G}^{+}}$ to $I_{\G}^{+}$ factors through the quotient $V_{\G}$ and is nontrivial on $V_{\G}(\dot{\alpha})$ for every $\alpha \in \Pi_{\G}$.
\end{enumerate}
\end{defn}

We put $\bfN_{\G}$ to be the normalizer of $\bfS_{\G}$ in $\G$.
Let $\widetilde{W}$ be the Iwahori--Weyl group of $\bfS_{\G}$ defined by 
\[
\widetilde{W} := N_{\G}/T_{\G}^{0}.
\]
Then we have the following proposition:
\begin{prop}[\cite{MR3481263}]\label{prop:IW}
\begin{enumerate}
 \item We have $G=I_{\G}N_{\G}I_{\G}$, and the map $I_{\G}nI_{\G} \mapsto \dot{n}$ induces a bijection 
 \[
 I_{\G}\backslash G/I_{\G} \cong \widetilde{W}.
 \]
 \item We have $G=I_{\G}^{+}N_{\G}I_{\G}^{+}$, and the map $I_{\G}^{+}nI_{\G}^{+} \mapsto \dot{n}$ induces a bijection 
\[
I_{\G}^{+} \backslash G / I_{\G}^{+} \cong N_{\G}/T_{\G}^{1}.
\]
 \item There exists an exact sequence 
 \[
 1 \ra W_\mathrm{aff} \ra \widetilde{W} \xrightarrow{\kappa_G} X^{\ast}\bigl(Z_{\widehat{\G}}^{I_F}\bigr)^{\Sigma_F} \ra 1,
 \]
 where $W_\mathrm{aff}$ is the affine Weyl group of $S_{\G}$,  $\Sigma_F:=\Gal(F^\mathrm{ur}/F)$, and $\kappa_G$ is the Kottwitz homomorphism defined in \cite[Section 7]{MR1485921}.
Moreover the subgroup $\widetilde{\Omega} \subset \widetilde{W}$ consisting of the elements normalizing $I_{\G}$ maps isomorphically to $X^{\ast}(Z_{\widehat{\G}}^{I_F})^{\Sigma_F}$, and we have $\widetilde{W} \cong W_\mathrm{aff} \rtimes \widetilde{\Omega}$.
\end{enumerate}
\end{prop}

We fix a set of representatives $\Omega \subset N_{\G}$ of $\widetilde{\Omega} \subset \widetilde{W}=N_{\G}/T_{\G}^{0}$.
Let $N_G(I_{\G})$ and $N_G(I_{\G}^{+})$ be the normalizers of $I_{\G}$ and $I_{\G}^{+}$ in $G$, respectively.
The following lemma is a consequence of Proposition \ref{prop:IW}.
See \cite[Lemma 2.5]{MR3904769} for the detail of the proof (we can use the exactly same argument as in the proof of \cite[Lemma 2.5]{MR3904769}).

\begin{lem}\label{lem:NI}
We have $N_G(I_{\G})=N_G(I_{\G}^{+})=I_{\G}\Omega$.
\end{lem}

%

Now we show a lemma which is a key to compute the characters of simple supercuspidal representations.
The next lemma is a slight generalization of \cite[Lemma 2.6]{MR3904769} to the case where $\G$ is not split over $F$:

\begin{lem}\label{lem:key-lem}
Let $y \in G$.
If $y$ satisfies $ygy^{-1} \in I_{\G}$ for an affine generic element $g \in I_{\G}^{+}$, then $y \in N_{G}(I_{\G})=I_{\G}\Omega$.
\end{lem}

\begin{proof}
Let $y \in G$ satisfy $ygy^{-1} \in I_{\G}$ for an affine generic element $g \in I_{\G}^{+}$.
Since affine genericity is preserved by $I_{\G}$-conjugation, any element of $I_{\G}yI_{\G}$ satisfies the same condition as $y$.
Therefore, by Proposition \ref{prop:IW}, we may assume $y \in N_{\G}$.

As $yI_{\G}y^{-1}$ and $I_{\G}$ have the same volume for a Haar measure of $G$, we only have to show $yI_{\G}y^{-1}\subset I_{\G}$.

We recall that the multiplication map 
\[
\T_{\G}(F)\times\prod_{a\in\Phi_{\red}}\bfU_{a}(F) \ra G
\]
is injective and induces a bijection
\[
T_{\G}^{0}\times\prod_{a\in\Phi_{\red}} \mathcal{U}_{a}^{+} \ra I_{\G}
\]
in any order (see the proof of Proposition \ref{prop:I-quot}).
By using this decomposition, we write 
\[
g= t \cdot \prod_{a\in\Phi_{\red}}x_{a},\quad\text{and}\quad
ygy^{-1}=yty^{-1} \cdot \prod_{a\in\Phi_{\red}}yx_ay^{-1},
\]
where $t\in T_{\G}^{1}$ and $x_{a}\in \mathcal{U}_{a}^{+}$ for each $a\in\Phi_{\red}$. 

We first show that the $y$-conjugation maps every simple affine root whose gradient is a reduced root to a positive affine root.
Let $a\in\Phi_{\red}$ be a reduced root such that $a=\dot{\alpha}$ for some $\alpha\in\Pi_{\G}$.
By the affine genericity of $g$, the image of $x_{a}$ in the quotient $\mathcal{U}_{a}^{+}/\mathcal{U}_{a}^{++}$ is not zero (see Remark \ref{rem:sim-aff-comp}).
As we have $\mathcal{U}_{a}^{+}/\mathcal{U}_{a}^{++}=\bar{U}_{\alpha}/\bar{U}_{2\alpha}$ (see the proof of Proposition \ref{prop:I-quot}), this implies that $x_{a}$ lies in $U_{\alpha}$, but not in $U_{\alpha+\varepsilon}$.
In particular, the simple affine root $\alpha\in\Pi_{\G}$ can be characterized as the maximal affine function such that 
\begin{itemize}
\item
its gradient is $a$ and
\item
$x_{a}$ belongs to $U_{\alpha}$.
\end{itemize}
Therefore, its $y$-conjugation $y\alpha y^{-1}$ can be characterized as the maximal affine function such that 
\begin{itemize}
\item
its gradient is $yay^{-1}$ and
\item
$yx_{a}y^{-1}$ belongs to $U_{y\alpha y^{-1}}$.
\end{itemize}

On the other hand, by the uniqueness of expression and the assumption that $ygy^{-1}\in I_{\G}$, for every $a\in\Phi_{\red}$, we have $yx_{a}y^{-1}\in\mathcal{U}_{yay^{-1}}^{+}$.
Here, by the definition of $\mathcal{U}_{yay^{-1}}^{+}$, we have 
\[
\mathcal{U}_{yay^{-1}}^{+}
:=
\begin{cases}
U_{\alpha_{yay^{-1}}^{+}}&\text{if $2a$ is not a root},\\
U_{\alpha_{yay^{-1}}^{+}}\cdot U_{\alpha_{2yay^{-1}}^{+}} &\text{if $2a$ is a root}.
\end{cases}
\]
In the latter case, $2\alpha_{yay^{-1}}^{+}$ is a positive affine root whose gradient is given by $2yay^{-1}$.
Thus, by the minimality of $\alpha_{2yay^{-1}}^{+}$, we have $2\alpha_{yay^{-1}}^{+}\geq \alpha_{2yay^{-1}}^{+}$.
Hence $U_{\alpha_{yay^{-1}}^{+}}\cdot U_{\alpha_{2yay^{-1}}^{+}}$ is contained in $U_{\frac{1}{2}\alpha_{2yay^{-1}}^{+}}$.
Therefore, by the maximality of $y\alpha y^{-1}$, we get
\[
y\alpha y^{-1} \geq
\begin{cases}
\alpha_{yay^{-1}}^{+}&\text{if $2a$ is not a root},\\
\frac{1}{2}\alpha_{2yay^{-1}}^{+}&\text{if $2a$ is a root}.
\end{cases}
\]
In both cases, this implies that $y\alpha y^{-1}$ is positive since the affine function in the right-hand side is positive.
Thus the affine root $y\alpha y^{-1}$ belongs to $\Psi_{\G}^{+}$.

We next show that, in a similar way, the $y$-conjugation maps every simple affine root whose gradient is not a reduced root to a positive affine root.
Let $a\in\Phi_{\red}$ be a reduced root such that $a=\frac{1}{2}\dot{\alpha}$ for some $\alpha\in\Pi_{\G}$.
By the affine genericity of $g$, the image of $x_{a}$ in the quotient $\mathcal{U}_{a}^{+}/\mathcal{U}_{a}^{++}$ is not zero.
Namely, $x_{a}$ lies in $U_{\frac{1}{2}\alpha+r}U_{\alpha}$, but not in $U_{\frac{1}{2}\alpha+r}U_{\alpha+r'}$.
Here we use the notations in the proof of Proposition \ref{prop:I-quot}, that is, $r$ (resp.\ $r'$) is the smallest positive number such that $\frac{1}{2}\alpha+r$ (resp.\ $\alpha+r'$) is a positive affine root.
By the same argument as in the proof of Proposition \ref{prop:I-quot}, we can show that $U_{\frac{1}{2}\alpha+r}U_{\alpha+r'}$ is equal to $U_{\frac{1}{2}\alpha+\varepsilon}$ for a sufficiently small positive number $\varepsilon$.
On the other hand, $U_{\frac{1}{2}\alpha+r}U_{\alpha}$ is contained in $U_{\frac{1}{2}\alpha}$.
Thus, the simple affine root $\alpha\in\Pi_{\G}$ can be characterized as the maximal affine function such that 
\begin{itemize}
\item
its gradient is $2a$ and
\item
$x_{a}$ belongs to $U_{\frac{1}{2}\alpha}$.
\end{itemize}
Therefore, its $y$-conjugation $y\alpha y^{-1}$ can be characterized as the maximal affine function such that 
\begin{itemize}
\item
its gradient is $2yay^{-1}$ and
\item
$yx_{a}y^{-1}$ belongs to $U_{\frac{1}{2}y\alpha y^{-1}}$.
\end{itemize}

By the same argument as in the previous case, we can show that $yx_{a}y^{-1}$ is contained in $U_{\frac{1}{2}\alpha_{2yay^{-1}}^{+}}$.
Therefore, by the maximality of $y\alpha y^{-1}$, we get
\[
y\alpha y^{-1} \geq\alpha_{2yay^{-1}}^{+} \geq 0.
\]
Namely, the affine root $y\alpha y^{-1}$ belongs to $\Psi_{\G}^{+}$.

Thus we can conclude that the $y$-conjugation maps every simple affine root to some positive affine root.
Then, as every positive root can be written as a sum of simple affine roots with nonnegative integer coefficients, the $y$-conjugation preserves the set of positive affine roots.
This implies $yI_{\G}y^{-1}\subset I_{\G}$ (note that we have $yT_{\G}^{0}y^{-1} =T_{\G}^{0}$ since $y \in N_{\G}$).
\end{proof}

\subsection{General construction of simple supercuspidal representations}\label{subsec:ssc}
Let $\chi$ be an affine generic character on $Z_{\G}I_{\G}^{+}$.
For $g\in G$ and a subgroup $J$ of $G$, we set $J^{g}:=g^{-1}Jg$.
We put
\[
 N_{G}(I_{\G}^{+}; \chi) := \{n \in N_G(I_{\G}^+) \mid \chi^n=\chi\},
\]
where $\chi^n$ is the character of $(Z_{\G}I_{\G}^+)^{n}=n^{-1}(Z_{\G}I_{\G}^{+})n=Z_{\G}I_{\G}^{+}$ defined by 
\[
\chi^n(g):=\chi(ngn^{-1}).
\]
This subgroup satisfies $Z_{\G}I_{\G}^{+} \subseteq N_{G}(I_{\G}^{+}; \chi) \subseteq I_{\G}\Omega$ by Lemma \ref{lem:NI}.
Moreover, $Z_{\G}I_{\G}^{+}$ is of finite index in $N_{G}(I_{\G}^{+}; \chi)$.
(This follows from the fact that $Z_{\G}I_{\G}^{+}$ is of finite index in $I_{\G}\Omega$, which can be checked by noting that the image of $Z_{\G}$ under the Kottwitz map $\kappa_{\G}$ is of finite index in $X^{\ast}(Z_{\G}^{I_{F}})^{\Sigma_{F}}$.)
For an irreducible constituent $\tilde\chi$ of $\cInd_{Z_{\G}I_{\G}^+}^{N_{G}(I_{\G}^{+}; \chi)}\chi$, we define 
\[
\pi_{\tilde\chi} := \cInd_{N_{G}(I_{\G}^{+}; \chi)}^{G} \tilde\chi.
\]

Then the decomposition of the compact induction of $\chi$ to $G$ is described as follows:

\begin{prop}\label{prop:ssc}
\begin{enumerate}
 \item We have a decomposition 
 \[
 \cInd_{Z_{\G}I_{\G}^{+}}^{G}\chi \cong \bigoplus_{\tilde{\chi}} \dim(\tilde{\chi})\cdot\pi_{\tilde{\chi}},
 \]
 where the sum is over the set of irreducible constituents of $\cInd_{Z_{\G}I_{\G}^{+}}^{N_{G}(I_{\G}^{+};\chi)}\chi$.
 \item The representation $\pi_{\tilde{\chi}}$ is irreducible, hence supercuspidal.
 \item Let $(\chi', \tilde{\chi}')$ be another pair as above. Then, $\pi_{\tilde{\chi}}$ and $\pi_{\tilde{\chi}'}$ are equivalent if and only if $\chi^n=\chi'$ and $\tilde{\chi}^n\cong\tilde{\chi}'$ for some $n \in T_{\G}^{0}\Omega$.
\end{enumerate}
\end{prop}

\begin{proof}
The assertions follow from the same argument as in \cite[Proposition 2.8]{MR3904769} (see also \cite[Proposition 2.4]{MR3164986} or \cite[Proposition 9.3]{MR2730575}).
Note that, though we assume that $\G$ is split in \cite{MR3904769}, we can apply the exactly same argument to our situation.
\end{proof}

The irreducible supercuspidal representations $\pi_{\tilde{\chi}}$ constructed in this way are called \textit{simple supercuspidal} representations of $G$.

In the rest of this section, for each classical group considered in this paper, we explain the following:
\begin{itemize}
\item
the choice of the maximal split torus $\bfS_{\G}$ and its centralizer $\T_{\G}$,
\item
the root system $\Phi_{\G}$ and a choice of root basis $\Delta_{\G}$, and their affine versions $\Psi_{\G}$ and $\Pi_{\G}$,
\item
a description of the Iwahori subgroup $I_{\G}$,
\item
a set $\Omega$ of representatives of $\widetilde{\Omega}$,
\item
the structure of the graded quotient $V_{\G}$ of the Moy--Prasad filtration of the Iwahori subgroup $I_{\G}$,
\item
the structure of the normalizer group $N_{G}(I_{\G}^{+};\chi)$ for each affine generic character $\chi$, and
\item
a parametrization of the equivalence classes of simple supercuspidal representations.
\end{itemize}

\subsection{The case of twisted $\GL_{2n}$}\label{subsec:ssc-GL}
Let $\bfS_{\GL_{2n}}=\T_{\GL_{2n}}$ be the subgroup of $\GL_{2n}$ consisting of  diagonal matrices.
Then we can describe the set of roots $\Phi_{\GL_{2n}}$ and the set of affine roots $\Psi_{\GL_{2n}}$ as follows:
\begin{align*}
\Phi_{\GL_{2n}}&=\{\pm(e_i-e_j) \mid 1\leq i<j\leq 2n\}, \text{ and}\\
\Psi_{\GL_{2n}}&=\{a+r \mid a \in \Phi_{\GL_{2n}}, r \in \Z \}.
\end{align*}
Here we write $e_{i}$ for the $i$-th projection of the diagonal torus $\T_{\GL_{2n}}$:
\[
e_{i}\colon \T_{\GL_{2n}} \rightarrow \Gm, \quad
\diag(t_{1},\ldots, t_{2n}) \mapsto t_{i}.
\]

We take the root basis and the affine root basis to be
\begin{align*}
\Delta_{\GL_{2n}}&=\{e_1-e_2, \ldots, e_{2n-1}-e_{2n}\}, \text{ and}\\
\Pi_{\GL_{2n}}&=\{e_1-e_2, \ldots, e_{2n-1}-e_{2n}, e_{2n}-e_1+1\}.
\end{align*}
Then the corresponding alcove of the apartment determines the Iwahori subgroup $I_{\GL_{2n}}$, which consists of matrices whose reductions are upper-triangular:
\[
I_{\GL_{2n}}
=
\begin{pmatrix}
\mcO^{\times}&\mcO&\cdots&\mcO\\
\mfp&\ddots&\ddots&\vdots\\
\vdots&\ddots&\ddots&\mcO\\
\mfp&\cdots&\mfp&\mcO^{\times}
\end{pmatrix}.
\]
The normalizer of $I_{\GL_{2n}}$ in $\GL_{2n}(F)$ is given by
\[
N_{\GL_{2n}(F)}(I_{\GL_{2n}})
=
Z_{\GL_{2n}}I_{\GL_{2n}}\lan\varphi_{a}^{\GL_{2n}}\ran,
\]
for any $a\in k^{\times}$.
Here, for $a \in k^{\times}$, we put 
\[
\varphi_{a}^{\GL_{2n}} := 
\begin{pmatrix}
0 & I_{2n-1} \\
\varpi a & 0 
\end{pmatrix} \in \GL_{2n}(F)
\]
(note that $(\varphi_{a}^{\GL_{2n}})^{2n}$ equals the scalar matrix $a\varpi I_{2n}$).
We can take $\Omega$ to be $\langle\varphi_{a}^{\GL_{2n}}\rangle$ for any $a\in k^{\times}$.


The graded quotient $V_{\GL_{2n}}=I_{\GL_{2n}}^{+}/I_{\GL_{2n}}^{++}$ is isomorphic to a direct sum of the residue field $k$:
\begin{align*}
V_{\GL_{2n}} &\cong k^{\oplus 2n} \\
(x_{ij})_{ij} &\mapsto \left(\ol{x_{1 2}}, \ldots, \ol{x_{2n-1, 2n}}, \ol{x_{2n, 1}\varpi^{-1}}\right).
\end{align*}
For a character $\omega$ on $k^{\times}$, $a\in k^{\times}$, and $\zeta'\in\C^{\times}$ we define an affine generic character $\chi_{\omega,a,\zeta'}^{\GL_{2n}}$ on $Z_{\GL_{2n}}I_{\GL_{2n}}^{+}$ by
\begin{align*}
\chi_{\omega,a,\zeta'}^{\GL_{2n}}(\varpi)&:=\zeta',\\
\chi_{\omega,a,\zeta'}^{\GL_{2n}}(z)&:=\omega(\overline{z}) \text{ for $z\in \mathcal{O}^{\times}\subset Z_{\GL_{2n}}$, and}\\
\chi_{\omega,a,\zeta'}^{\GL_{2n}}(x)&:=\psi\left(\ol{x_{12}}+\cdots+\ol{x_{2n-1, 2n}}+a\ol{x_{2n,1}\varpi^{-1}}\right) \text{ for $x=(x_{ij})_{ij} \in I_{\GL_{2n}}^{+}$}.
\end{align*}
Then the normalizer of this character is given by
\[
N_{\GL_{2n}(F)}\bigl(I_{\GL_{2n}}^{+}; \chi_{\omega,a,\zeta'}^{\GL_{2n}}\bigr)
=
Z_{\GL_{2n}} I_{\GL_{2n}}^{+} \lan\varphi_{a^{-1}}^{\GL_{2n}}\ran.
\]

Since we have $(\varphi_{a^{-1}}^{\GL_{2n}})^{2n}=a^{-1}\varpi I_{2n}$, by choosing a $2n$-th root $\zeta$ of $\omega(a^{-1})\zeta'$, we can extend $\chi_{\omega,a,\zeta'}^{\GL_{2n}}$ to a character $\tilde{\chi}_{\omega,a,\zeta}^{\GL_{2n}}$ on $N_{\GL_{2n}(F)}(I_{\GL_{2n}}^{+}; \chi_{\omega,a,\zeta'}^{\GL_{2n}})$ by putting
\[
\tilde{\chi}_{\omega,a,\zeta}^{\GL_{2n}}\bigl(\varphi_{a^{-1}}^{\GL_{2n}}\bigr)
:=
\zeta.
\]
Let $\pi_{\omega,a,\zeta}^{\GL_{2n}}$ be the representation of $\GL_{2n}(F)$ defined by
\[
\pi_{\omega,a,\zeta}^{\GL_{2n}}:=\cInd^{\GL_{2n}(F)}_{Z_{\GL_{2n}}I_{\GL_{2n}}^{+}\lan\varphi_{a^{-1}}^{\GL_{2n}}\ran} \tilde{\chi}^{\GL_{2n}}_{\omega,a,\zeta}.
\]
Then, by Proposition \ref{prop:ssc}, $\pi_{\omega,a,\zeta}^{\GL_{2n}}$ is a simple supercuspidal representation of $\GL_{2n}(F)$ and we can check that the set 
\[
\SSC(\GL_{2n}):=
\{(\omega,a,\zeta) \mid \omega\in (k^{\times})^{\vee}, a \in k^{\times}, \zeta \in \C^{\times}\}
\]
parametrizes the set of equivalence classes of simple supercuspidal representations of $\GL_{2n}(F)$.

Let $\theta$ be the automorphism of $\GL_{2n}$ over $F$ defined by
\[
\theta(g)=J_{2n} {}^{t}\!g^{-1}J_{2n}^{-1}.
\]
We next consider a parametrization of $\theta$-stable simple supercuspidal representations with nontrivial central characters.

The automorphism $\theta$ preserves the subgroups $I_{\GL_{2n}}^{+}$ and $I_{\GL_{2n}}^{++}$, and induces the automorphism of $V_{\GL_{2n}} \cong k^{\oplus2n}$ defined by
\[
(x_1, \ldots, x_{2n-1}, x_{0}) \mapsto 
(x_{2n-1}, \ldots, x_1, x_{0}).
\]
Moreover, $\theta$ also preserves the center $Z_{\GL_{2n}}$ and $\theta|_{Z_{\GL_{2n}}}$ is given by $z\mapsto z^{-1}$ for $z\in Z_{\GL_{2n}}$.
Finally, by a simple computation, we can check $\theta(\varphi_{a}^{\GL_{2n}})=-(\varphi_{a}^{\GL_{2n}})^{-1}$.
Therefore we have
\begin{align*}
\bigl(\pi_{\omega,a,\zeta}^{\GL_{2n}}\bigr)^{\theta}
&\cong \cInd_{Z_{\GL_{2n}}I_{\GL_{2n}}^{+}\lan\varphi_{a^{-1}}^{\GL_{2n}}\ran}^{\GL_{2n}(F)} \bigl(\tilde{\chi}_{\omega,a,\zeta}^{\GL_{2n}}\bigr)^\theta\\ 
&\cong \cInd_{Z_{\GL_{2n}}I_{\GL_{2n}}^{+}\lan\varphi_{a^{-1}}^{\GL_{2n}}\ran}^{\GL_{2n}(F)} \tilde{\chi}_{\omega^{-1},a,\omega(-1)\zeta^{-1}}^{\GL_{2n}}
= \pi_{\omega^{-1},a,\omega(-1)\zeta^{-1}}^{\GL_{2n}},
\end{align*}
and we can conclude that $\pi_{\omega,a,\zeta}^{\GL_{2n}}$ is $\theta$-stable if and only if $\omega^{2}=\mathbbm{1}$ and $\zeta^{2}=\omega(-1)$.
Thus, when the central character of $\pi_{\omega,a,\zeta}^{\GL_{2n}}$ is nontrivial, the character $\omega$ is necessarily the unique nontrivial quadratic character $\omega_{0}$ of $k^{\times}$, and the central character of $\pi_{\omega_{0},a,\zeta}^{\GL_{2n}}$ is given by
\begin{align*}
F^{\times}\cong\lan\varpi a^{-1}\ran \times \mcO^{\times}\ra\C^{\times};\quad
\varpi a^{-1}&\mapsto\omega_{0}(-1)^{n}, \quad\text{and}\\
u&\mapsto\omega_{0}(\ol{u})\quad\text{for}\quad u\in \mcO^{\times}.
\end{align*}

In summary, the set of equivalence classes of $\theta$-stable simple supercuspidal representations of $\GL_{2n}(F)$ with nontrivial central characters is parametrized by the set
\[
\SSC^{\theta}_{\omega_{0}}(\GL_{2n})
:=\{(\omega_{0},a,\zeta)\in\SSC(\GL_{2n}) \mid \zeta^{2}=\omega_{0}(-1)\}.
\]

\subsection{The case of $\Sp_{2n}$}\label{subsec:ssc-Sp}
In this subsection, we consider the case of  
\[
\Sp_{2n} := \{g \in \GL_{2n} \mid {}^{t}\!gJ_{2n}g=J_{2n} \}. 
\]
Let $\bfS_{\Sp_{2n}}=\T_{\Sp_{2n}}$ be the subgroup of diagonal matrices in $\Sp_{2n}$:
\[
\T_{\Sp_{2n}}:=\bigl\{\diag(t_1, \ldots, t_n, t_n^{-1}, \ldots, t_1^{-1}) \mid t_i\neq0\bigr\}.
\]
Then the set of roots $\Phi_{\Sp_{2n}}$ and the set of affine roots $\Psi_{\Sp_{2n}}$ are given by
\begin{align*}
\Phi_{\Sp_{2n}}&=\{\pm e_i\pm e_j \mid 1\leq i<j\leq n\} \cup \{ \pm 2e_i \mid 1\leq i \leq n\}, \text{ and}\\
\Psi_{\Sp_{2n}}&=\{a+r \mid a \in \Phi_{\Sp_{2n}}, r \in \Z \}.
\end{align*}

We take the root basis and the affine root basis to be
\begin{align*}
\Delta_{\Sp_{2n}}&=\{e_1-e_2, \ldots, e_{n-1}-e_n, 2e_n\}, \text{ and}\\
\Pi_{\Sp_{2n}}&=\{e_1-e_2, \ldots, e_{n-1}-e_n, 2e_n, -2e_1+1\}.
\end{align*}
Then the corresponding alcove of the apartment determines the Iwahori subgroup $I_{\Sp_{2n}}$, which consists of elements of $\Sp_{2n}(F)$ belonging to the following set of matrices:
\[
\begin{pmatrix}
\mcO^{\times}&\mcO&\cdots&\mcO\\
\mfp&\ddots&\ddots&\vdots\\
\vdots&\ddots&\ddots&\mcO\\
\mfp&\cdots&\mfp&\mcO^{\times}
\end{pmatrix}.
\]
The normalizer of $I_{\Sp_{2n}}$ in $\Sp_{2n}(F)$ is given by $I_{\Sp_{2n}}$ itself, hence $\widetilde{\Omega}$ and $\Omega$ are trivial in this case.

The graded quotient $V_{\Sp_{2n}}$ is isomorphic to a direct sum of the residue field $k$:
\begin{align*}
V_{\Sp_{2n}} &\cong k^{\oplus n+1} \\
(y_{ij})_{ij} &\mapsto \Bigr(\ol{y_{1 2}}, \ldots, \ol{y_{n, n+1}}, \ol{y_{2n, 1}\varpi^{-1}}\Bigl).
\end{align*}

For $\xi \in \{\pm1\}$, $\kappa\in\{0,1\}$, and $a \in k^{\times}$, we define an affine generic character $\chi^{\Sp_{2n}}_{\xi,\kappa,a}$ on $Z_{\Sp_{2n}}I_{\Sp_{2n}}^{+}=\pm I_{\Sp_{2n}}^{+}$ by
\begin{align*}
 \chi^{\Sp_{2n}}_{\xi,\kappa,a}(-1) &:= \xi, \text{ and}\\
 \chi^{\Sp_{2n}}_{\xi,\kappa,a}(y) &:= \psi\Bigl(\ol{y_{12}}+\cdots+\ol{y_{n-1, n}}+\epsilon^{\kappa}\ol{y_{n, n+1}}+a\ol{y_{2n, 1}\varpi^{-1}}\Bigr) \text{ for $y=(y_{ij})_{ij} \in I_{\Sp_{2n}}^{+}$}.
\end{align*}
Then the normalizer of this character is given by 
\[
N_{\Sp_{2n}(F)}\bigl(I_{\Sp_{2n}}^{+}; \chi^{\Sp_{2n}}_{\xi,\kappa,a}\bigr)
=
\pm I_{\Sp_{2n}}^{+}.
\]
Thus, by Proposition \ref{prop:ssc}, the representation
\[
\pi^{\Sp_{2n}}_{\xi,\kappa,a}:=\cInd^{\Sp_{2n}(F)}_{\pm I_{\Sp_{2n}}^{+}} \chi^{\Sp_{2n}}_{\xi,\kappa,a}
\]
is simple supercuspidal and we can check that the set
\[
\SSC(\Sp_{2n})
:=\bigl\{(\xi,\kappa,a) \mid \xi\in\{\pm1\}, \kappa\in\{0,1\}, a \in k^{\times}\bigr\}
\]
parametrizes the set of equivalence classes of simple supercuspidal representations of $\Sp_{2n}(F)$.

\subsection{The case of $\SO_{2n}^{\mu}$}\label{subsec:ssc-SO-ram}
Let $\mu$ be a ramified quadratic character of $F^{\times}$.
In this subsection, we consider the case of  
\[
\SO_{2n}^{\mu}:=\{g\in\GL_{2n} \mid {}^{t}\!g J_{\mu} g=J_{\mu},\, \det(g)=1\}, \text{ where}
\]
\[
J_{\mu} := \begin{pmatrix}
 &&&&&&&1\\
 &&&&&&\adots&\\
 &&&&&1&&\\
 &&&1&&&&\\
 &&&&\ast&&&\\
 &&1&&&&&\\
 &\adots&&&&&&\\
 1&&&&&&&
\end{pmatrix},\quad
\ast=
\begin{cases}
 -\varpi&\text{if } E_{\mu}\cong F(\sqrt{\varpi})\\
 -\varpi\epsilon &\text{if } E_{\mu}\cong F(\sqrt{\varpi\epsilon}).
\end{cases}
\]
We note that the group $\SO_{2n}^{\mu}$ is a non-split and quasi-split connected reductive group over $F$, and splits over $E_{\mu}$.
We also note that when $n=1$, $\SO_{2n}^{\mu}=\SO_{2}^{\mu}$ is a torus.
Thus we only treat the case where $n\geq2$ in this paper.

Let $\bfS_{\SO_{2n}^{\mu}}$ be the maximal $F$-split torus of $\SO_{2n}^{\mu}$ consisting of diagonal matrices in $\SO_{2n}^{\mu}$:
\[
\bfS_{\SO_{2n}^{\mu}}:=\bigl\{\diag(t_{1}, \ldots, t_{n-1}, 1, 1, t_{n-1}^{-1}, \ldots, t_{1}^{-1}) \mid t_i\neq0\bigr\}.
\]
Then we can describe the set of relative roots $\Phi_{\SO_{2n}^{\mu}}$ and the set of affine roots $\Psi_{\SO_{2n}^{\mu}}$ as follows (see, e.g., \cite[Section 10.1]{MR0327923} or \cite[1.16]{MR546588} for the details):
\begin{align*}
\Phi_{\SO_{2n}^{\mu}}&=\{\pm e_i\pm e_j \mid 1\leq i<j\leq n-1\} \cup \{ \pm e_i \mid 1\leq i \leq n-1\}, \text{ and}\\
\Psi_{\SO_{2n}^{\mu}}&=\{\pm e_{i}\pm e_{j}+r \mid 1\leq i<j\leq n-1, r\in\Z\} \cup \biggl\{ \pm e_i+\frac{r}{2} \,\bigg\vert\, 1\leq i \leq n-1, r\in\Z\biggr\}.\end{align*}

We take the root basis and the affine root basis to be
\begin{align*}
\Delta_{\SO_{2n}^{\mu}}&=\{e_1-e_2, \ldots, e_{n-2}-e_{n-1}, e_{n-1}\},\text{ and}\\
\Pi_{\SO_{2n}^{\mu}}&=\biggl\{e_1-e_2, \ldots, e_{n-2}-e_{n-1}, e_{n-1}, -e_{1}+\frac{1}{2}\biggr\}.
\end{align*}
Then the corresponding alcove of the apartment determines the Iwahori subgroup $I_{\SO_{2n}^{\mu}}$, which consists of elements of $\SO_{2n}^{\mu}(F)$ belonging to the following set of matrices:
\[
\begin{pmatrix}
\mcO^{\times}&&\multicolumn{1}{c:}{\mcO}&\mcO&\multicolumn{1}{c:}{\mfp}&&&\\
 &\ddots&\multicolumn{1}{c:}{}&\vdots&\multicolumn{1}{c:}{\vdots}&&\mcO&\\
 \mfp&&\multicolumn{1}{c:}{\mcO^{\times}}&\mcO&\multicolumn{1}{c:}{\mfp}&&&\\
 \cdashline{1-10}
 \mfp&\cdots&\multicolumn{1}{c:}{\mfp}&\mathcal{O}&\multicolumn{1}{c:}{\mathfrak{p}}&\mcO&\cdots&\mcO\\
 \mcO&\cdots&\multicolumn{1}{c:}{\mcO}&\mathcal{O}&\multicolumn{1}{c:}{\mathcal{O}}&\mcO&\cdots&\mcO\\
 \cdashline{1-10}
 &&\multicolumn{1}{c:}{}&\mfp&\multicolumn{1}{c:}{\mfp}&\mcO^{\times}&&\mcO\\
 &\mfp&\multicolumn{1}{c:}{}&\vdots&\multicolumn{1}{c:}{\vdots}&&\ddots&\\
 &&\multicolumn{1}{c:}{}&\mfp&\multicolumn{1}{c:}{\mfp}&\mfp&&\mcO^{\times}\\
\end{pmatrix}.
\]
The normalizer of $I_{\SO_{2n}^{\mu}}$ in $\SO_{2n}^{\mu}(F)$ is given by $\pm I_{\SO_{2n}^{\mu}}$.
Thus we can take $\Omega$ to be $\{\pm1\}$.

Since the centralizer $\T_{\SO_{2n}^{\mu}}$ of $\bfS_{\SO_{2n}^{\mu}}$ is given by
\[
\bfS_{\SO_{2n}^{\mu}} \times \SO_{2}^{\mu},
\]
the first graded quotient of $I_{\SO_{2n}^{\mu}}$ is given by $S_{\SO_{2n}^{\mu}}(q)$.
Here we note that the reduction of the unique parahoric subgroup of $\SO_{2}^{\mu}(F)$ is trivial.
The second graded quotient $V_{\SO_{2n}^{\mu}}$ is isomorphic to a direct sum of the residue field $k$:
\begin{align*}
V_{\SO_{2n}^{\mu}} &\cong k^{\oplus n-2}\oplus k\oplus k \\
(y_{ij})_{ij} &\mapsto (\ol{y_{1 2}}, \ldots, \ol{y_{n-2, n-1}}, \ol{y_{n-1, n}}, \ol{y_{n+1,1}}).
\end{align*}

For $\xi \in \{\pm1\}$ and $a \in k^{\times}$, we define an affine generic character $\chi^{\SO_{2n}^{\mu}}_{\xi,a}$ on $Z_{\SO_{2n}^{\mu}}I_{\SO_{2n}^{\mu}}^{+}=\pm I_{\SO_{2n}^{\mu}}^{+}$ by
\begin{align*}
 \chi^{\SO_{2n}^{\mu}}_{\xi,a}(-1) &= \xi \text{ and}\\
 \chi^{\SO_{2n}^{\mu}}_{\xi,a}(y) &= \psi(\ol{y_{12}}+\cdots+\ol{y_{n-2, n-1}}+\ol{y_{n-1, n}}+a\ol{y_{n+1,1}}) \text{ for $y=(y_{ij})_{ij} \in I_{\SO_{2n}^{\mu}}^+$}.
\end{align*}
Then the normalizer of this character is given by 
\[
N_{\SO_{2n}^{\mu}(F)}\bigl(I_{\SO_{2n}^{\mu}}^{+}; \chi^{\SO_{2n}^{\mu}}_{\xi,a}\bigr)
=
\pm I_{\SO_{2n}^{\mu}}^{+}.
\]

Thus, by Proposition \ref{prop:ssc}, the representation
\[
\pi^{\SO_{2n}^{\mu}}_{\xi,a}:=\cInd^{\SO_{2n}^{\mu}(F)}_{\pm I_{\SO_{2n}^{\mu}}^{+}} \chi^{\SO_{2n}^{\mu}}_{\xi,a}
\]
is simple supercuspidal and the set
\[
\SSC(\SO_{2n}^{\mu})
:=\bigl\{(\xi, a) \mid \xi\in\{\pm1\}, a\in k^{\times}\bigr\}
\]
parametrizes the set of equivalence classes of simple supercuspidal representations of $\SO_{2n}^{\mu}(F)$.

Finally we comment on the outer automorphism of $\SO_{2n}^{\mu}$.
We set
\[
\rmO_{2n}^{\mu}:=\{g\in\GL_{2n} \mid {}^{t}\!g J_{\mu} g=J_{\mu}\}, \text{ and}
\]
\[
w_{\mu}:=
\begin{pmatrix}
 I_{n-1}&&& \\
 &1&& \\
 &&-1& \\
 &&&I_{n-1} \\
\end{pmatrix}
\in \rmO_{2n}^{\mu}(F)\smallsetminus \SO_{2n}^{\mu}(F).
\]
Then the conjugation via this element in $\rmO_{2n}^{\mu}$ preserves $\SO_{2n}^{\mu}$ and defines an outer automorphism of $\SO_{2n}^{\mu}$.
We denote this automorphism again by $w_{\mu}$.

This automorphism $w_{\mu}$ preserves $I_{\SO_{2n}^{\mu}}$, $I_{\SO_{2n}^{\mu}}^{+}$, and $I_{\SO_{2n}^{\mu}}^{++}$.
The induced action on the quotient $V_{\SO_{2n}^{\mu}}$ is given by
\begin{align*}
w_{\mu} \colon k^{\oplus n-1}\oplus k &\ra k^{\oplus n-1}\oplus k\\
(h_{1}, \ldots, h_{n-1}, h_{0})&\mapsto(h_{1}, \ldots, h_{n-1}, -h_{0}).
\end{align*}
Therefore we can describe the $w_{\mu}$-twist of the simple supercuspidal representations of $\SO_{2n}^{\mu}(F)$ as follows:
\[
\Bigl(\pi^{\SO_{2n}^{\mu}}_{\xi,a}\Bigr)^{w_{\mu}}\cong\pi^{\SO_{2n}^{\mu}}_{\xi,-a}.
\]

\subsection{The case of $\SO_{2n+2}$}\label{subsec:ssc-SO}
In this subsection, we consider the case of  
\[
\SO_{2n+2} := \{g \in \GL_{2n+2} \mid {}^{t}\!gJ_{\mathbbm{1}}g=J_{\mathbbm{1}},\, \det(g)=1 \}, 
\]
with
\[
J_{\mathbbm{1}} := \begin{pmatrix}
 &&1\\
 &\adots&\\
 1&&
\end{pmatrix}.
\] 
We note that the group $\SO_{2n+2}$ is split over $F$.
We also note the following:
\begin{itemize}
\item
If $n=0$, then $\SO_{2n+2}=\SO_{2}$ is a torus.
\item
If $n=1$, then we have an accidental isomorphism
\[
\SO_{4}\cong(\SL_{2}\times\SL_{2})/ \mu_{2}
\]
as algebraic groups.
Thus, by considering the long exact sequence associated to the short exact sequence
\[
1
\rightarrow
\mu_{2}
\rightarrow
\SL_{2}\times\SL_{2}
\rightarrow
\SO_{4}
\rightarrow
1,
\]
we know that $\SO_{4}(F)$ has an open normal subgroup isomorphic to $(\SL_{2}(F)\times\SL_{2}(F))/\{\pm1\}$ such that the quotient is given by the finite abelian group $F^{\times}/F^{\times2}$.
Therefore, representation theory of $\SO_{4}(F)$ is fairly understandable via representation theory of $(\SL_{2}(F)\times\SL_{2}(F))/\{\pm1\}$ and $F^{\times}/F^{\times2}$ according to the Clifford theory (see, for example, \cite[Section 2]{MR2918491}).
\end{itemize}
For these reasons, we only treat the case where $n\geq2$ in this paper.

Let $\bfS_{\SO_{2n+2}}=\T_{\SO_{2n+2}}$ be the subgroup of diagonal matrices in $\SO_{2n+2}$:
\[
\T_{\SO_{2n+2}}:=\bigl\{\diag(t_1, \ldots, t_{n+1}, t_{n+1}^{-1}, \ldots, t_1^{-1}) \mid t_i\neq0\bigr\}.
\]
Then the set of roots $\Phi_{\SO_{2n+2}}$ and the set of affine roots $\Psi_{\SO_{2n+2}}$ are given by
\begin{align*}
\Phi_{\SO_{2n+2}}&=\{\pm e_i\pm e_j \mid 1\leq i<j\leq n+1\}, \text{ and}\\
\Psi_{\SO_{2n+2}}&=\{a+r \mid a \in \Phi, r \in \Z \}.
\end{align*}

We take the root basis and the affine root basis to be
\begin{align*}
\Delta_{\SO_{2n+2}}&=\{e_{1}-e_{2}, \ldots, e_{n}-e_{n+1}, e_{n}+e_{n+1}\}, \text{ and}\\
\Pi_{\SO_{2n+2}}&=\{e_1-e_2, \ldots, e_{n}-e_{n+1}, e_{n}+e_{n+1}, -e_{1}-e_{2}+1\}.
\end{align*}
Then the corresponding alcove of the apartment determines the Iwahori subgroup $I_{\SO_{2n+2}}$, which consists of elements of $\SO_{2n+2}(F)$ belonging to the following set of matrices:
\[
\begin{pmatrix}
\mcO^{\times}&&\multicolumn{1}{c:}{\mcO}&\mcO&\multicolumn{1}{c:}{\mcO}&&&\\
 &\ddots&\multicolumn{1}{c:}{}&\vdots&\multicolumn{1}{c:}{\vdots}&&\mcO&\\
 \mfp&&\multicolumn{1}{c:}{\mcO^{\times}}&\mcO&\multicolumn{1}{c:}{\mcO}&&&\\
 \cdashline{1-10}
 \mfp&\cdots&\multicolumn{1}{c:}{\mfp}&\mcO&\multicolumn{1}{c:}{\mcO}&\mcO&\cdots&\mcO\\
\mfp&\cdots&\multicolumn{1}{c:}{\mfp}&\mcO&\multicolumn{1}{c:}{\mcO}&\mcO&\cdots&\mcO\\
 \cdashline{1-10}
 &&\multicolumn{1}{c:}{}&\mfp&\multicolumn{1}{c:}{\mfp}&\mcO^{\times}&&\mcO\\
 &\mfp&\multicolumn{1}{c:}{}&\vdots&\multicolumn{1}{c:}{\vdots}&&\ddots&\\
 &&\multicolumn{1}{c:}{}&\mfp&\multicolumn{1}{c:}{\mfp}&\mfp&&\mcO^{\times}\\
\end{pmatrix}.
\]
The normalizer of $I_{\SO_{2n+2}}$ in $\SO_{2n+2}(F)$ is given by
\[
N_{\SO_{2n+2}(F)}(I_{\SO_{2n+2}})
=
 I_{\SO_{2n+2}}\lan\varphi_{\alpha,\beta}^{\SO_{2n+2}}\ran,
\]
for any $\alpha\in k^{\times}$ and $\beta\in k^{\times}$.
Here we put
\[
\varphi_{\alpha,\beta}^{\SO_{2n+2}}:=
\begin{pmatrix}
&&&&&(\beta\varpi)^{-1}\\
&I_{n-1}&&&&\\
&&&\alpha^{-1}&&\\
&&\alpha&&&\\
&&&&I_{n-1}&\\
\beta\varpi&&&&&
\end{pmatrix}.
\]
We can take $\Omega$ to be $\langle\varphi_{\alpha,\beta}^{\SO_{2n+2}}\rangle$ for any $\alpha,\beta$.

The graded quotient $V_{\SO_{2n+2}}$ is isomorphic to a direct sum of the residue field $k$:
\begin{align*}
V_{\SO_{2n+2}} &\cong k^{\oplus n+2} \\
(y_{ij})_{ij} &\mapsto \Bigr(\ol{y_{1 2}}, \ldots, \ol{y_{n, n+1}}, \ol{y_{n, n+2}}, \ol{y_{2n+1, 1}\varpi^{-1}}\Bigl).
\end{align*}

For $\xi\in\{\pm1\}$, $\kappa\in\{0,1\}$ and $a \in k^{\times}$, we define an affine generic character $\chi^{\SO_{2n+2}}_{\xi,\kappa,a}$ on $Z_{\SO_{2n+2}}I_{\SO_{2n+2}}^{+}=\pm I_{\SO_{2n+2}}^{+}$ by
\begin{align*}
 \chi^{\SO_{2n+2}}_{\xi,\kappa,a}(-1) &= \xi, \text{ and}\\
 \chi^{\SO_{2n+2}}_{\xi,\kappa,a}(y) &= \psi\Bigr(\ol{y_{1 2}}+\ldots+\ol{y_{n, n+1}}+\epsilon^{\kappa}\ol{y_{n, n+2}}+a\ol{y_{2n+1, 1}\varpi^{-1}}\Bigl) \text{ for } y \in I_{\SO_{2n+2}}^{+}.
\end{align*}
Then the normalizer of this character is given by
\[
N_{\SO_{2n+2}(F)}\bigl(I_{\SO_{2n+2}}^{+}; \chi_{\xi,\kappa,a}^{\SO_{2n+2}}\bigr)
=
\pm I_{\SO_{2n+2}}^{+} \lan\varphi_{\epsilon^{\kappa},-a^{-1}}^{\SO_{2n+2}}\ran.
\]
Thus, if we take $\zeta\in\{\pm1\}$, then we can extend $\chi^{\SO_{2n+2}}_{\xi,\kappa,a}$ to a character $\tilde{\chi}_{\xi,\kappa,a,\zeta}^{\SO_{2n+2}}$ on $N_{\SO_{2n+2}(F)}(I_{\SO_{2n+2}}^{+}; \chi_{\xi,\kappa,a}^{\SO_{2n+2}})$ by putting
\[
\tilde{\chi}_{\xi,\kappa,a,\zeta}^{\SO_{2n+2}}\bigl(\varphi_{\epsilon^{\kappa},-a^{-1}}^{\SO_{2n+2}}\bigr)
:=
\zeta.
\]

Then, by Proposition \ref{prop:ssc}, the representation
\[
\pi^{\SO_{2n+2}}_{\xi,\kappa,a,\zeta}
:=\cInd^{\SO_{2n+2}}_{\pm I_{\SO_{2n+2}}^{+}\lan\varphi_{\epsilon^{\kappa},-a^{-1}}^{\SO_{2n+2}}\ran} \tilde{\chi}^{\SO_{2n+2}}_{\xi,\kappa,a,\zeta}
\]
is simple supercuspidal and the set
\[
\SSC(\SO_{2n+2})
:=\bigl\{(\xi,\kappa,a,\zeta) \mid \xi\in\{\pm1\}, \kappa\in\{0,1\}, a\in k^{\times}, \zeta\in\{\pm1\}\bigr\}
\]
parametrizes the set of equivalence classes of simple supercuspidal representations of $\SO_{2n+2}(F)$.

Finally we comment on the outer automorphism of $\SO_{2n+2}$.
We set
\[
\rmO_{2n+2} := \{g \in \GL_{2n+2} \mid {}^{t}\!gJ_{\mathbbm{1}}g=J_{\mathbbm{1}}\}, \text{ and}
\]
\[
w:=
\begin{pmatrix}
 I_{n}&&& \\
 &0&1& \\
 &1&0& \\
 &&&I_{n} \\
\end{pmatrix}
\in \rmO_{2n+2}(F)\smallsetminus \SO_{2n+2}(F).
\]
Then the conjugation via this element in $\rmO_{2n+2}$ preserves $\SO_{2n+2}$ and defines an outer automorphism of $\SO_{2n+2}$.
We denote this automorphism again by $w$.

This automorphism $w$ preserves $I_{\SO_{2n+2}}$, $I_{\SO_{2n+2}}^{+}$, and $I_{\SO_{2n+2}}^{++}$, and the induced action on the quotient $V_{\SO_{2n+2}}$ is given by
\begin{align*}
w \colon k^{\oplus n}\oplus k\oplus k &\ra k^{\oplus n}\oplus k\oplus k\\
(h_{1}, \ldots, h_{n-1}, h_{n}, h_{n+1}, h_{0})&\mapsto(h_{1}, \ldots, h_{n-1}, h_{n+1}, h_{n}, h_{0}).
\end{align*}
Hence, if we put 
\[
t:=\diag(\underbrace{1,\ldots,1}_{n},\epsilon^{\kappa},\epsilon^{-\kappa},\underbrace{1,\ldots,1}_{n})
\in\SO_{2n+2}(F),
\]
then we have
\[
\bigl((\chi^{\SO_{2n+2}}_{\xi,\kappa,a})^{w}\bigr)^{t}
=\chi^{\SO_{2n+2}}_{\xi,\kappa,a}.
\]
On the other hand, we have
\[
\bigl((\varphi_{\epsilon^{\kappa},-a^{-1}}^{\SO_{2n+2}})^{w}\bigr)^{t}
=\varphi_{\epsilon^{\kappa},-a^{-1}}^{\SO_{2n+2}}.
\]
Therefore we can conclude that
\[
(\pi^{\SO_{2n+2}}_{\xi,\kappa,a,\zeta})^{w}
\cong
\bigl((\pi^{\SO_{2n+2}}_{\xi,\kappa,a,\zeta})^{w}\bigr)^{t}
\cong\pi^{\SO_{2n+2}}_{\xi,\kappa,a,\zeta}.
\]
In particular, each simple supercuspidal representation of $\SO_{2n+2}(F)$ is stable under the action of the outer automorphism.

\subsection{The case of $\SO_{2n+2}^{\ur}$}\label{subsec:ssc-SO-ur}
In this subsection, we consider the case of  
\[
\SO_{2n+2}^{\ur} := \{g \in \GL_{2n+2} \mid {}^{t}\!gJ_{\ur}g=J_{\ur},\, \det(g)=1 \}, 
\]
with
\[
J_{\ur} := \begin{pmatrix}
 &&&&&&&1\\
 &&&&&&\adots&\\
 &&&&&1&&\\
 &&&1&&&&\\
 &&&&-\epsilon&&&\\
 &&1&&&&&\\
 &\adots&&&&&&\\
 1&&&&&&&
\end{pmatrix}.
\] 
We note that the group $\SO_{2n+2}^{\ur}$ is a non-split and quasi-split connected reductive group over $F$, and splits over $E_{\ur}$.
Similarly to the split case, we also note that
\begin{itemize}
\item
if $n=0$, then $\SO_{2n+2}^{\ur}=\SO_{2}^{\ur}$ is a torus, and
\item
if $n=1$, then we have an accidental isomorphism
\[
\SO_{4}^{\ur}\cong(\mathrm{Res}_{E_{\ur}/F}\SL_{2})/ \mu_{2}
\]
as algebraic groups.
Thus, for the same reason as in the split case, representation theory of $\SO_{4}^{\ur}(F)$ is understandable via representation theory of $\SL_{2}(E_{\ur})/\{\pm1\}$ and $F^{\times}/F^{\times2}$.
\end{itemize}
Hence we only treat the case where $n\geq2$ in this paper.

Let $\bfS_{\SO_{2n+2}^{\ur}}$ be the diagonal maximal $F$-split torus in $\SO_{2n+2}^{\ur}$:
\[
\bfS_{\SO_{2n+2}^{\ur}}:=\bigl\{\diag(t_1, \ldots, t_{n}, 1, 1, t_{n}^{-1}, \ldots, t_1^{-1}) \mid t_i\neq0\bigr\}.
\]
Then the set of relative roots $\Phi_{\SO_{2n+2}^{\ur}}$ and the set of affine roots $\Psi_{\SO_{2n+2}^{\ur}}$ are given as follows (see, e.g., \cite[Section 10.1]{MR0327923} or \cite[1.16]{MR546588} for the details):
\begin{align*}
\Phi_{\SO_{2n+2}^{\ur}}&=\{\pm e_i\pm e_j \mid 1\leq i<j\leq n\} \cup \{ \pm e_i \mid 1\leq i \leq n\}, \text{ and}\\
\Psi_{\SO_{2n+2}^{\ur}}&=\{a+r \mid a \in \Phi, r \in \Z \}.
\end{align*}

We take the root basis and the affine root basis to be
\begin{align*}
\Delta_{\SO_{2n+2}^{\ur}}&=\{e_1-e_2, \ldots, e_{n-1}-e_n, e_n\}, \text{ and}\\
\Pi_{\SO_{2n+2}^{\ur}}&=\{e_1-e_2, \ldots, e_{n-1}-e_n, e_n, -e_1-e_{2}+1\}.
\end{align*}
Then the corresponding alcove of the apartment determines the Iwahori subgroup $I_{\SO_{2n+2}^{\ur}}$, which consists of elements of $\SO_{2n+2}^{\ur}(F)$ belonging to the following set of matrices:
\[
\begin{pmatrix}
\mcO^{\times}&&\multicolumn{1}{c:}{\mcO}&\mcO&\multicolumn{1}{c:}{\mcO}&&&\\
 &\ddots&\multicolumn{1}{c:}{}&\vdots&\multicolumn{1}{c:}{\vdots}&&\mcO&\\
 \mfp&&\multicolumn{1}{c:}{\mcO^{\times}}&\mcO&\multicolumn{1}{c:}{\mcO}&&&\\
 \cdashline{1-10}
 \mfp&\cdots&\multicolumn{1}{c:}{\mfp}&\mcO&\multicolumn{1}{c:}{\mcO}&\mcO&\cdots&\mcO\\
\mfp&\cdots&\multicolumn{1}{c:}{\mfp}&\mcO&\multicolumn{1}{c:}{\mcO}&\mcO&\cdots&\mcO\\
 \cdashline{1-10}
 &&\multicolumn{1}{c:}{}&\mfp&\multicolumn{1}{c:}{\mfp}&\mcO^{\times}&&\mcO\\
 &\mfp&\multicolumn{1}{c:}{}&\vdots&\multicolumn{1}{c:}{\vdots}&&\ddots&\\
 &&\multicolumn{1}{c:}{}&\mfp&\multicolumn{1}{c:}{\mfp}&\mfp&&\mcO^{\times}\\
\end{pmatrix}.
\]
The normalizer of $I_{\SO_{2n+2}^{\ur}}$ in $\SO_{2n+2}^{\ur}(F)$ is given by
\[
I_{\SO_{2n+2}^{\ur}} \lan\varphi_{\alpha,\beta}^{\SO_{2n+2}^{\ur}}\ran,
\]
for any $\alpha\in\Nr^{1}$ and $\beta\in k^{\times}$.
Here we identified $\Nr^{1}$ with $\SO_{2}^{\ur}(q)$ via
\[
\Nr^{1}
\cong
\SO_{2}^{\ur}(q)
=\biggl\{ \begin{pmatrix}x&\epsilon y \\ y&x\end{pmatrix}\in\GL_{2}(k)
 \,\bigg\vert\, x^{2}-\epsilon y^{2}=1 \biggr\}
\colon \quad
x+\sqrt{\epsilon}y
\mapsto
\begin{pmatrix}x&\epsilon y\\ y&x\end{pmatrix},
\]
and put
\[
\varphi_{\alpha,\beta}^{\SO_{2n+2}^{\ur}}:=
\begin{pmatrix}
&&&&(\beta\varpi)^{-1}\\
&I_{n-1}&&&\\
&&\alpha&&\\
&&&I_{n-1}&\\
\beta\varpi&&&&\\
\end{pmatrix}
\begin{pmatrix}
 I_{n}&&& \\
 &1&0& \\
 &0&-1& \\
 &&&I_{n} \\
\end{pmatrix}.
\]
We can take $\Omega$ to be $\langle\varphi_{\alpha,\beta}^{\SO_{2n+2}^{\ur}}\rangle$ for any $\alpha,\beta$.

Since the centralizer $\T_{\SO_{2n+2}^{\ur}}$ of $\bfS_{\SO_{2n+2}^{\ur}}$ is given by
\[
\bfS_{\SO_{2n+2}^{\ur}} \times \SO_{2}^{\ur},
\]
the first graded quotient of $I_{\SO_{2n+2}^{\ur}}$ is given by
\begin{align*}
T_{\SO_{2n+2}^{\ur}}(q)
&=
S_{\SO_{2n+2}^{\ur}}(q)\times\SO_{2}^{\ur}(q)\\
&=\{\diag(t_{1},\ldots,t_{n},\tilde{t}_{n+1},t_{n}^{-1},\ldots,t_{1}^{-1})\mid t_{i}\in k^{\times},\, \tilde{t}_{n+1}\in\SO_{2}^{\ur}(q) \}.
\end{align*}
Here note that $\SO_{2}^{\ur}(F)$ is isomorphic to the kernel of the norm map $\Nr_{E_{\ur}/F}\colon E_{\ur}^{\times}\rightarrow F^{\times}$, and that the reduction of its unique parahoric subgroup is given by the kernel of the norm map $\Nr \colon \tilde{k}^{\times}\rightarrow k^{\times}$ of residue fields and isomorphic to $\SO_{2}^{\ur}(q)$.
The second graded quotient $V_{\SO_{2n+2}^{\ur}}$ is isomorphic to a direct sum of the residue field $k$ and its quadratic extension $\tilde{k}$:
\begin{align*}
V_{\SO_{2n+2}^{\ur}} &\cong k^{\oplus n-1}\oplus \tilde{k}\oplus k \\
(y_{ij})_{ij} &\mapsto \Bigr(\ol{y_{1 2}}, \ldots, \ol{y_{n-1, n}}, \ol{y_{n,n+1}}+\ol{y_{n,n+2}}\sqrt{\epsilon}^{-1}, \ol{y_{2n+1, 1}\varpi^{-1}}\Bigl)
\end{align*}
(here we fix a square root $\sqrt{\epsilon}$ of $\epsilon$).
We note that if we use the above identification of $\SO_{2}^{\ur}(q)$ with $\Nr^{1}$, then the action of $T_{\SO_{2n+2}^{\ur}}(q)$ on $V_{\SO_{2n+2}^{\ur}}$ is given by
\[
\diag(t_{1},\ldots,t_{n},\tilde{t}_{n+1},t_{n}^{-1},\ldots,t_{1}^{-1})\cdot(y_{1},\ldots, y_{n},y_{0})
\]
\[
=
\biggl(\frac{t_{1}}{t_{2}}y_{1}, \ldots, \frac{t_{n-1}}{t_{n}}y_{n-1}, \frac{t_{n}}{\tilde{t}_{n+1}}y_{n}, \frac{1}{t_{1}t_{2}}y_{0}\biggr).
\]

For $\xi\in\{\pm1\}$, $\kappa\in\{0,1\}$, and $a \in k^{\times}$, we define an affine generic character $\chi^{\SO_{2n+2}^{\ur}}_{\xi,\kappa,a}$ on $Z_{\SO_{2n+2}^{\ur}}I_{\SO_{2n+2}^{\ur}}^{+}=\pm I_{\SO_{2n+2}^{\ur}}^{+}$ by
\begin{align*}
 \chi^{\SO_{2n+2}^{\ur}}_{\xi,\kappa,a}(-1) &:= \xi, \text{ and}\\
 \chi^{\SO_{2n+2}^{\ur}}_{\xi,\kappa,a}(y) &:= \psi\bigl(y_{1}+\cdots+y_{n-1}+\Tr(\tilde{\epsilon}^{\kappa} y_{n})+ay_{0}\bigr)
 \text{ for $y\in I_{\SO_{2n+2}^{\ur}}^{+}$},
\end{align*}
where $(y_{1},\ldots,y_{n},y_{0})\in k^{\oplus n-1}\oplus\tilde{k}\oplus k$ is the image of $y$ in $V_{\SO_{2n+2}^{\ur}}\cong k^{\oplus n-1}\oplus\tilde{k}\oplus k$.
Then the normalizer of the character $\chi_{\xi,\kappa,a}^{\SO_{2n+2}^{\ur}}$ is given by
\[
N_{\SO_{2n+2}^{\ur}(F)}\bigl(I_{\SO_{2n+2}^{\ur}}^{+}; \chi_{\xi,\kappa,a}^{\SO_{2n+2}^{\ur}}\bigr)
=
\pm I_{\SO_{2n+2}^{\ur}}^{+} \lan\varphi_{(\tilde{\epsilon}/c(\tilde{\epsilon}))^{\kappa},-a^{-1}}^{\SO_{2n+2}^{\ur}}\ran.
\]
Here $c$ is the nontrivial element of $\Gal(\tilde{k}/k)$.
Thus, if we take $\zeta\in\{\pm1\}$, then we can extend $\chi^{\SO_{2n+2}^{\ur}}_{\xi,\kappa,a}$ to a character $\tilde{\chi}_{\xi,\kappa,a,\zeta}^{\SO_{2n+2}^{\ur}}$ on $N_{\SO_{2n+2}^{\ur}(F)}(I_{\SO_{2n+2}}^{+}; \chi_{\xi,\kappa,a}^{\SO_{2n+2}^{\ur}})$ by putting
\[
\tilde{\chi}_{\xi,\kappa,a,\zeta}^{\SO_{2n+2}^{\ur}}\bigl(\varphi_{(\tilde{\epsilon}/c(\tilde{\epsilon}))^{\kappa},-a^{-1}}^{\SO_{2n+2}^{\ur}}\bigr)
:=
\zeta.
\]

Then, by Proposition \ref{prop:ssc}, the representation
\[
\pi^{\SO_{2n+2}^{\ur}}_{\xi,\kappa,a,\zeta}:=\cInd^{\SO_{2n+2}^{\ur}(F)}_{\pm I_{\SO_{2n+2}^{\ur}}^{+}\lan\varphi_{(\tilde{\epsilon}/c(\tilde{\epsilon}))^{\kappa},-a^{-1}}^{\SO_{2n+2}^{\ur}}\ran} \tilde{\chi}^{\SO_{2n+2}^{\ur}}_{\xi,\kappa,a,\zeta}
\]
is simple supercuspidal and the set
\[
\SSC(\SO_{2n+2}^{\ur})
:=\bigl\{(\xi,\kappa,a,\zeta) \mid \xi\in\{\pm1\}, \kappa\in\{0,1\}, a \in k^{\times}, \zeta\in\{\pm1\} \bigr\}
\]
parametrizes the set of equivalence classes of simple supercuspidal representations of $\SO_{2n+2}^{\ur}(F)$.

Finally we comment on the outer automorphism of $\SO_{2n+2}^{\ur}$.
We set
\[
\rmO_{2n+2}^{\ur} := \{g \in \GL_{2n+2} \mid {}^{t}\!gJ_{\ur}g=J_{\ur}\}, \text{ and}
\]
\[
w_{\ur}:=
\begin{pmatrix}
 I_{n}&&& \\
 &1&0& \\
 &0&-1& \\
 &&&I_{n} \\
\end{pmatrix}
\in \rmO_{2n+2}^{\ur}(F)\smallsetminus \SO_{2n+2}^{\ur}(F).
\]
Then the conjugation via this element in $\rmO_{2n+2}^{\ur}$ preserves $\SO_{2n+2}^{\ur}$ and defines an outer automorphism of $\SO_{2n+2}^{\ur}$.
We denote this automorphism again by $w_{\ur}$.

This automorphism $w_{\ur}$ preserves $I_{\SO_{2n+2}^{\ur}}$, $I_{\SO_{2n+2}^{\ur}}^{+}$, and $I_{\SO_{2n+2}^{\ur}}^{++}$, and the induced action on the quotient $V_{\SO_{2n+2}^{\ur}}^{+}$ is given by
\begin{align*}
w_{\ur} \colon k^{\oplus n-1}\oplus \tilde{k}\oplus k &\ra k^{\oplus n-1}\oplus \tilde{k}\oplus k\\
(h_{1}, \ldots, h_{n-1}, h_{n}, h_{0})&\mapsto\bigl(h_{1}, \ldots, h_{n-1}, c(h_{n}), h_{0}\bigr).
\end{align*}

Hence, if we put 
\[
t:=\diag(\underbrace{1,\ldots,1}_{n},(\tilde{\epsilon}/c(\tilde{\epsilon}))^{-\kappa},\underbrace{1,\ldots,1}_{n}),
\]
then we have
\[
\Bigl(\bigl(\chi^{\SO_{2n+2}^{\ur}}_{\xi,\kappa,a}\bigr)^{w_{\ur}}\Bigr)^{t}
=\chi^{\SO_{2n+2}^{\ur}}_{\xi,\kappa,a}.
\]
On the other hand, we have
\[
\Bigl(\bigl(\varphi_{(\tilde{\epsilon}/c(\tilde{\epsilon}))^{\kappa},-a^{-1}}^{\SO_{2n+2}^{\ur}}\bigr)^{w_{\ur}}\Bigr)^{t}
=\varphi_{(\tilde{\epsilon}/c(\tilde{\epsilon}))^{\kappa},-a^{-1}}^{\SO_{2n+2}^{\ur}}.
\]
Therefore we can conclude that
\[
\bigl(\pi^{\SO_{2n+2}^{\ur}}_{\xi,\kappa,a,\zeta}\bigr)^{w_{\ur}}
\cong\Bigl(\bigl(\pi^{\SO_{2n+2}^{\ur}}_{\xi,\kappa,a,\zeta}\bigr)^{w_{\ur}}\Bigr)^{t}
\cong\pi^{\SO_{2n+2}^{\ur}}_{\xi,\kappa,a,\zeta}.
\]
In particular, each simple supercuspidal representation of $\SO_{2n+2}^{\ur}(F)$ is stable under the action of the outer automorphism.

\section{Characters of simple supercuspidal representations}\label{sec:char}

We recall that, for a connected reductive group $\G$ over $F$ and an irreducible smooth representation $\pi$ of $G$, by the theorem of Harish-Chandra (\cite{MR0414797}), we have its \textit{character} $\Theta_{\pi}$, which is a locally constant function on the set $G^{\rs}$ of regular semisimple elements of $G$.

Let $\G$ be one of the following groups:
\[
\GL_{2n},\quad \Sp_{2n},\quad \SO_{2n}^{\mu},\quad \SO_{2n+2},\quad \SO_{2n+2}^{\ur},
\]
where $\mu$ is a ramified quadratic character of $F^{\times}$.
In the previous section, we parametrize the set of equivalence classes of simple supercuspidal representations of $G$ by the set $\SSC(\G)$.
In this paper, for $X\in\SSC(\G)$, we denote the character $\Theta_{\pi_{X}^{\G}}$ of the corresponding simple supercuspidal representation $\pi_{X}^{\G}$ simply by $\Theta_{X}^{\G}$.

We next recall the notion of a \textit{twisted character}.
Let $\pi$ be an irreducible smooth representation of $\GL_{2n}(F)$.
Since we have the automorphism $\theta$ of $\GL_{2n}$ defined over $F$, we can argue whether $\pi$ is $\theta$-stable (i.e., $\pi^{\theta}$ is isomorphic to $\pi$) or not.
If $\pi$ is $\theta$-stable, then, by fixing an isomorphism $I_{\theta}\colon\pi\cong\pi^{\theta}$, we get its $\theta$-twisted character $\Theta_{\pi,\theta}$, which is a locally constant function on the set $\GL_{2n}^{\trs}(F)$ of $\theta$-regular $\theta$-semisimple elements of $\GL_{2n}(F)$ (see \cite{MR889110} and \cite{MR3632513} for details).
Note that the $\theta$-twisted character $\Theta_{\pi,\theta}$ depends on the choice of an isomorphism $I_{\theta}$.

Now we let $\pi_{X}^{\GL_{2n}}$ be a $\theta$-stable simple supercuspidal representation of $\GL_{2n}$.
Then, by the observations in Section \ref{subsec:ssc-GL}, $\pi_{X}^{\GL_{2n}}$ is obtained by the compact induction of a $\theta$-stable character $\tilde{\chi}_{X}^{\GL_{2n}}$.
We choose $I_{\theta}$ to be an isomorphism from $\pi$ to $\pi^{\theta}$ obtained by the compact induction of $\mathrm{id}\colon\tilde{\chi}_{X}^{\GL_{2n}}\cong(\tilde{\chi}_{X}^{\GL_{2n}})^{\theta}$.
Then we have a $\theta$-twisted character $\Theta_{\pi_{X}^{\GL_{2n}}, \theta}$ of $\pi_{X}^{\GL_{2n}}$.
Similar to the abbreviation in the untwisted case, we denote the $\theta$-twisted character $\Theta_{\pi_{X}^{\GL_{2n}}, \theta}$ normalized in this way by $\Theta_{X,\theta}^{\GL_{2n}}$.

\begin{rem}\label{rem:normalization}
By using a $\theta$-stable Whittaker datum of $\GL_{2n}$, we can normalize the $\theta$-twisted character of each $\theta$-stable irreducible smooth representation in a natural way depending only on the Whittaker datum (see \cite[Section 2.2]{MR3135650}).
In fact, for a simple supercuspidal representation of $\GL_{2n}(F)$, such a normalization coincides with the above one (see \cite[Proposition 5.1]{MR3904769} for the details).
\end{rem}

Since simple supercuspidal representations are constructed by the compact induction from open compact subgroups, we can apply the following character formulas to $\Theta_{X}^{\G}$ and $\Theta_{X,\theta}^{\GL_{2n}}$:

\begin{thm}[character formula, \cite{MR1039842}]\label{thm:CF}
Let $K$ be an open subgroup of $G$ such that $K$ contains $Z_{\G}$ and $K/Z_{\G}$ is compact.
Let $\rho$ be a finite-dimensional irreducible smooth representation of $K$.
We assume that the representation $\pi:=\cInd_{K}^{G} \rho$ is irreducible, hence supercuspidal.
Then, for $g \in G^{\rs}$, we have
\[
\Theta_{\pi}(g)
=\sum_{y\in K\backslash G/K}\sum_{\begin{subarray}{c} x\in K\backslash KyK \\ xgx^{-1} \in K \end{subarray}} \tr \rho(xgx^{-1}).
\]
In particular, we have
\[
\Theta_{\pi}(g)
=\sum_{\begin{subarray}{c} x\in K\backslash G \\ xgx^{-1} \in K \end{subarray}} \tr \rho(xgx^{-1}),
\]
provided that the sum is finite.
\end{thm}

\begin{thm}[{twisted character formula, \cite[Partie 1, Th\'eor\`eme 6.2.1]{MR3632513}}]\label{thm:TCF}
Let $K$ be a $\theta$-stable open subgroup of $\GL_{2n}(F)$ such that $K$ contains $Z_{\GL_{2n}}$ and $K/Z_{\GL_{2n}}$ is compact.
Let $\rho$ be a $\theta$-stable finite-dimensional irreducible smooth representation of $K$, and we fix an isomorphism $I_{\theta}\colon\rho\cong\rho^{\theta}$.
We assume that the representation $\pi:=\cInd_{K}^{G} \rho$ is irreducible, hence supercuspidal.
Then, for $g \in \GL_{2n}^{\trs}(F)$, we have
\[
\Theta_{\pi, \theta}(g)
=\sum_{y\in K\backslash\GL_{2n}(F)/K}\sum_{\begin{subarray}{c} x\in K\backslash KyK\\ xg\theta(x)^{-1} \in K \end{subarray}} \tr\bigl(\rho(xg\theta(x)^{-1})\circ I_{\theta}\bigr).
\]
In particular, we have
\[
\Theta_{\pi, \theta}(g)
=\sum_{\begin{subarray}{c} x\in K\backslash\GL_{2n}(F)\\ xg\theta(x)^{-1} \in K \end{subarray}} \tr\bigl(\rho(xg\theta(x)^{-1})\circ I_{\theta}\bigr),
\]
provided that the sum is finite.
Here the left-hand side is the $\theta$-twisted character of $\pi$ normalized with respect to $\cInd I_{\theta}\colon\pi\cong\pi^{\theta}$.
\end{thm}

In this section, we compute the (twisted) characters of simple supercuspidal representations of $G$ at some special regular semisimple elements, such as affine generic elements.

\subsection{The case of twisted $\GL_{2n}$}\label{subsec:char-GL}
For $g \in \GL_{2n}(F)$, we put 
\[
\mathcal{N}(g) := g\theta(g) \in \GL_{2n}(F).
\]

Let $(\omega_{0},a,\zeta)\in\SSC_{\omega_{0}}^{\theta}(\GL_{2n})$ and $\pi_{\omega_{0},a,\zeta}^{\GL_{2n}}$ the $\theta$-stable simple supercuspidal representation corresponding to the data $(\omega_{0},a,\zeta)$ (see Section \ref{subsec:ssc-GL}).

First, we compute the $\theta$-twisted character $\Theta_{\omega_{0},a,\zeta,\theta}^{\GL_{2n}}$ of $\pi_{\omega_{0},a,\zeta}^{\GL_{2n}}$ at $g \in I_{\GL_{2n}}^{+}\cap \GL_{2n}^{\trs}(F)$ such that $\mathcal{N}(g)\in I_{\GL_{2n}}^{+}$ is affine generic by using the twisted character formula (Theorem \ref{thm:TCF}).
To do this, we show the following lemma on the index set of the sum in the twisted character formula.

\begin{lem}[{\cite[Lemmas 3.8 and 3.9]{MR3904769}}]\label{lem:sumTGL}
Let $g \in I_{\GL_{2n}}^{+}$ be an element such that $\mathcal{N}(g)$ is affine generic.
Then a system of representatives of the set 
\[
\big\{x \in Z_{\GL_{2n}}I_{\GL_{2n}}^{+}\lan\varphi_{a^{-1}}^{\GL_{2n}}\ran\big\backslash \GL_{2n}(F) \,\big\vert\, xg\theta(x)^{-1} \in Z_{\GL_{2n}}I_{\GL_{2n}}^{+}\lan\varphi_{a^{-1}}^{\GL_{2n}}\ran\big\}
\]
is finite and given by 
\[
T'_{\GL_{2n}}(q)
:=\{\diag(t_1, \ldots, t_{2n})\in T_{\GL_{2n}}(q) \mid t_{1}t_{2n}=\cdots=t_{n}t_{n+1}, t_n=1\}.
\]
\end{lem}

\begin{proof}
This follows from the assumption and Lemma \ref{lem:key-lem}.
See \cite[Lemmas 3.8 and 3.9]{MR3904769} for the details.
%
\end{proof}

In \cite[Proposition 3.10]{MR3904769}, by using this lemma, we computed the twisted characters of simple supercuspidal representations of $\GL_{2n}(F)$ with trivial central character.
The same computation works for the case where the central character is nontrivial.
For the sake of completeness, we explain our computation:
\begin{prop}\label{prop:charTGL}
Let $g \in I_{\GL_{2n}}^{+}\cap \GL_{2n}^{\trs}(F)$ be an element such that $\mathcal{N}(g)$ is affine generic.
Let $(g_1, \ldots, g_{2n-1},g_{0})$ be the simple affine components of $g$.
Then we have 
\[
\Theta_{\omega_{0},a,\zeta,\theta}^{\GL_{2n}}(g) = \omega_{0}(ag_{0})\cdot\Kl_{\alpha}^{n+1}(\psi; w_{\GL_{2n},\theta}, \chi_{\GL_{2n},\theta}),
\]
where $\Kl_{\alpha}^{n+1}(\psi; w_{\GL_{2n},\theta}, \chi_{\GL_{2n},\theta})$ on the right-hand side is the Kloosterman sum in Definition \ref{defn:Kl} with
\begin{align*}
w_{\GL_{2n},\theta}&:=(2,\ldots,2,1,1),\quad
\chi_{\GL_{2n},\theta}:=(\mathbbm{1},\ldots,\mathbbm{1},\omega_{0}),\\
\alpha&:=a(g_1 +g_{2n-1})^2 \cdots (g_{n-1}+g_{n+1})^2 g_{n}g_{0}.
\end{align*}
\end{prop}

\begin{proof}
Note that the simple affine components of $\mathcal{N}(g)$ are given by 
\[
(g_1+g_{2n-1}, \ldots, g_{2n-1}+g_1, 2g_{0}),
\]
and the affine genericity of $\mathcal{N}(g)$ means that none of them is zero.
Thus $\alpha$ is not zero.

By combining Theorem \ref{thm:TCF} with Lemma \ref{lem:sumTGL}, we have
\begin{align*}
&\phantom{{}={}}\Theta_{\omega_{0},a,\zeta,\theta}^{\GL_{2n}}(g)\\
&= \sum_{x\in T'_{\GL_{2n}}(q)} \tilde{\chi}^{\GL_{2n}}_{\omega_{0},a,\zeta}\bigl(xg\theta(x)^{-1}\bigr) \\
&= \sum_{z \in k^{\times}} \sum_{\begin{subarray}{c} t_1, \ldots, t_n \in k^{\times}\\ t_i t_{2n+1-i} =z, t_n=1 \end{subarray}} 
\omega_{0}(z)\psi\biggl(\frac{t_1 t_{2n-1} g_1}{z} + \cdots + \frac{t_{2n-1} t_1 g_{2n-1}}{z} + a\cdot\frac{t_{2n}^2 g_{0}}{z}\biggr) \\
&= \sum_{z \in k^{\times}} \sum_{t_1, \ldots, t_{n-1} \in k^{\times}} 
\omega_{0}(z)\psi \biggl( \frac{t_1}{t_2}(g_1+g_{2n-1}) + \cdots +\frac{t_{n-1}}{1}(g_{n-1}+g_{n+1}) + \frac{1}{z}g_n + \frac{z}{t_1^2}ag_{0} \biggr) \\
&=\sum_{\begin{subarray}{c}s_{1}, \ldots, s_{n+1} \in k^{\times} \\ s_{1}^{2}\cdots s_{n-1}^{2}s_{n}s_{n+1}=\alpha\end{subarray}} 
\omega_{0}(s_{n+1}a^{-1}g_{0}^{-1}) \cdot \psi \left( s_{1} + \cdots + s_{n} + s_{n+1} \right)\\
&= \omega_{0}(ag_{0})\cdot\Kl_{\alpha}^{n+1}\bigl(\psi; (2, \ldots, 2, 1, 1), (\mathbbm{1},\ldots,\mathbbm{1}, \omega_{0})\bigr). 
\end{align*}
Here, in the 4th equality, we replaced the $i$-th term in $\psi$ with $s_{i}$.
\end{proof}


We next consider an element $\varphi_{u}^{\GL_{2n}} g$, where $g \in I_{\GL_{2n}}^{+}$ and $u \in k^{\times}$ such that $-\mathcal{N}(\varphi_{u}^{\GL_{2n}} g)=\varphi_{u}^{\GL_{2n}} g(\varphi_{u}^{\GL_{2n}})^{-1}\theta(g) \in I_{\GL_{2n}}^{+}$ is affine generic.
In a similar way to the proof of Lemma \ref{lem:sumTGL}, we can prove the following lemma (see Lemmas 3.11 and 3.12 of \cite{MR3904769} for the proof):

\begin{lem}\label{lem:sumTGL2}
Let $g \in I_{\GL_{2n}}^{+}$ be an element such that $-\mathcal{N}(\varphi^{\GL_{2n}}_u g) \in I_{\GL_{2n}}^{+}$ is affine generic.
Then a system of representatives of the set 
\[
\bigl\{ y \in Z_{\GL_{2n}}I_{\GL_{2n}}^+\lan\varphi^{\GL_{2n}}_{a^{-1}}\ran\backslash \GL_{2n}(F) \mid y\varphi^{\GL_{2n}}_u g\theta (y)^{-1} \in Z_{\GL_{2n}} I_{\GL_{2n}}^{+}\lan\varphi^{\GL_{2n}}_{a^{-1}}\ran\bigr\}
\]
is finite and given by 
\[
T''_{\GL_{2n}}(q) := \{ \diag(t_1, \ldots, t_{2n})\in T_{\GL_{2n}}(q) \mid t_1 t_{2n-1}= \cdots = t_{2n-1} t_1 = au, t_{2n}=1 \}.
\]
\end{lem}

By using this lemma, we can show the following, which is a variant of \cite[Proposition 3.13]{MR3904769} for the case where the central character of a simple supercuspidal representation is nontrivial.
\begin{prop}\label{prop:charTGL2}
Let $g \in I_{\GL_{2n}}^{+}$ be an element such that $\varphi_{u}^{\GL_{2n}}g \in \GL_{2n}^{\trs}(F)$ and $-\mathcal{N}(\varphi^{\GL_{2n}}_u g)$ is affine generic.
Let $(g_1, \ldots, g_{2n-1},g_{0})$ be the simple affine components of $g$.
\begin{enumerate}
\item If $u\notin a^{-1}k^{\times2}$, then we have 
\[
\Theta_{\omega_{0},a,\zeta,\theta}^{\GL_{2n}} \bigl(\varphi^{\GL_{2n}}_u g\bigr)=0.
\]
\item If $u=a^{-1}v^2$ for some $v \in k^{\times}$, then we have 
\[
\Theta_{\omega_{0},a,\zeta,\theta}^{\GL_{2n}} \bigl(\varphi^{\GL_{2n}}_u g\bigr) = \zeta\cdot \left( \Kl_{\beta}^{n}(\psi) + \Kl_{-\beta}^{n}(\psi) \right),
\]
where $\beta:=(v^2g_1+ag_{0})(g_{2}+g_{2n-1})\cdots(g_n+g_{n+1})/v$.
\end{enumerate}
\end{prop}

\begin{proof}
We use the twisted character formula (Theorem \ref{thm:TCF}) and Lemma \ref{lem:sumTGL2}.
If $u\notin a^{-1}k^{\times2}$, then the set $T''_{\GL_{2n}}(q)$ is empty.
Hence the sum in the twisted character formula is zero.

We next consider the case where $u=a^{-1}v^2$ for some $v\in k^{\times}$.
Note that the simple affine components of $-\mathcal{N}(\varphi^{\GL_{2n}}_u g)$ are given by 
\[
(g_2+g_{2n-1}, g_3+g_{2n-2}, \ldots, g_{2n-1}+g_2, u^{-1}g_{0}+g_1, ug_1+g_{0}).
\]
Thus affine genericity of $-\mathcal{N}(\varphi^{\GL_{2n}}_u g)$ means that $\beta$ is not zero.
We can compute the twisted character as follows: 
\begin{align*}
& \phantom{{}={}}\Theta_{\omega_{0},a,\zeta,\theta}^{\GL_{2n}} (\varphi^{\GL_{2n}}_{a^{-1}v^{2}} g)\\
&= \sum_{y \in T''_{\GL_{2n}}(q)} \tilde{\chi}_{\omega_{0},a,\zeta}(\varphi^{\GL_{2n}}_{a^{-1}})\tilde{\chi}_{\omega_{0},a,\zeta}\bigl((\varphi^{\GL_{2n}}_{a^{-1}})^{-1}y\varphi^{\GL_{2n}}_{a^{-1}v^2} g\theta(y)^{-1}\bigr) \\
&= \zeta\omega_{0}(v^{2})\sum_{\begin{subarray}{c} t_1, \ldots, t_{2n} \in k^{\times}\\ t_i t_{2n-i} =v^2\\ t_{2n} = 1 \end{subarray}} 
\psi \left( \frac{t_{2n}t_{2n-1}v^2g_1}{v^2} + \frac{t_1t_{2n-2}g_2}{v^2} + \cdots + \frac{t_{2n-2}t_1g_{2n-1}}{v^2} + a\cdot\frac{t_{2n-1}t_{2n}g_{0}}{v^2} \right) \\
&= \zeta \sum_{\begin{subarray}{c} t_1, \ldots, t_{n} \in k^{\times}\\ t_n=\pm v \end{subarray}} 
\psi \left( \frac{1}{t_1}(v^2g_1+ag_{0}) + \frac{t_1}{t_2}(g_2+g_{2n-1}) + \cdots + \frac{t_{n-1}}{t_n}(g_n+g_{n+1}) \right) \\
&= \zeta \cdot\bigl( \Kl_{\beta}^{n}(\psi) + \Kl_{-\beta}^{n}(\psi) \bigr).
\end{align*}
\end{proof}

\subsection{The case of $\Sp_{2n}$}\label{subsec:char-Sp}
Next, we consider the case of $\Sp_{2n}$.
Let $(\xi,\kappa,a)\in\SSC(\Sp_{2n})$ and we take the simple supercuspidal representation $\pi^{\Sp_{2n}}_{\xi,\kappa,a}$ defined in Section \ref{subsec:ssc-Sp}.

\begin{lem}\label{lem:sumSp}
Let $h \in I_{\Sp_{2n}}^{+}$ be an affine generic element.
Then we have a bijection
\[
\big\{y \in (\pm I_{\Sp_{2n}}^{+}) \backslash \Sp_{2n}(F) \mid yhy^{-1} \in \pm I_{\Sp_{2n}}^{+}\big\}
\cong \{\pm1\}\backslash T_{\Sp_{2n}}(q).
\]
\end{lem}

\begin{proof}
Since the normalizer of the Iwahori subgroup $I_{\Sp_{2n}}$ in $\Sp_{2n}(F)$ is given by $I_{\Sp_{2n}}$, we can identify the left-hand side with $\pm I_{\Sp_{2n}}^{+} \backslash I_{\Sp_{2n}}$ by Lemma \ref{lem:key-lem}.
Then the claim follows from an isomorphism $I_{\Sp_{2n}}^{+} \backslash I_{\Sp_{2n}}\cong T_{\Sp_{2n}}(q)$.
\end{proof}

\begin{prop}\label{prop:charSp}
Let $h \in I_{\Sp_{2n}}^{+}\cap \Sp_{2n}^{\rs}(F)$ be an affine generic element.
Let $(h_1, \ldots, h_{n}, h_{0})$ be the simple affine components of $h$.
Then we have 
\[
\Theta^{\Sp_{2n}}_{\xi,\kappa, a}(h)
=
\frac{1}{2}\Kl_{\alpha}^{n+1}(\psi; w_{\Sp_{2n}})+ \frac{\omega_{0}(ah_{0})}{2}\cdot\Kl_{\alpha}^{n+1}(\psi; w_{\Sp_{2n}}, \chi_{\Sp_{2n}}),
\]
where $\Kl_{\alpha}^{n+1}(\psi; w_{\Sp_{2n}})$ and $\Kl_{\alpha}^{n+1}(\psi; w_{\Sp_{2n}}, \chi_{\Sp_{2n}})$ on the right-hand side are the Kloosterman sums in Definition \ref{defn:Kl} with
\begin{align*}
w_{\Sp_{2n}}&:=(2,\ldots,2,1,1),\quad
\chi_{\Sp_{2n}}:=(\mathbbm{1},\ldots,\mathbbm{1},\omega_{0}), \\
\alpha&:=a\epsilon^{\kappa}h_{1}^{2}\cdots h_{n-1}^{2}h_{n}h_{0}.
\end{align*}
\end{prop}

\begin{proof}
Note that the affine genericity of $h$ implies that $\alpha$ is not zero.

By combining Theorem \ref{thm:CF} with Lemma \ref{lem:sumSp}, we have
\begin{align*}
&\phantom{{}={}}\Theta^{\Sp_{2n}}_{\xi,\kappa, a}(h) \\
&= \sum_{y \in \{\pm1\}\backslash T_{\Sp_{2n}}(q)} \chi^{\Sp_{2n}}_{\xi,\kappa, a}(yhy^{-1}) \\
&= \sum_{(t_{1},\ldots,t_{n})\in\{\pm1\}\backslash (k^{\times})^{n}} 
\psi\biggl(\frac{t_{1}}{t_{2}} h_{1} + \cdots + \frac{t_{n-1}}{t_{n}} h_{n-1} + \epsilon^{\kappa}t_{n}^{2}h_{n} + \frac{a}{t_{1}^2} h_{0}\biggr) \\
&= \sum_{\begin{subarray}{c}t_{1},\ldots,t_{n-1}\in k^{\times} \\t_{n}\in\{\pm1\}\backslash k^{\times}\end{subarray}} 
\psi\biggl(\frac{t_{1}}{t_{2}} h_{1} + \cdots + \frac{t_{n-1}}{1} h_{n-1} + \epsilon^{\kappa}t_{n}^{2} h_{n} + \frac{a}{t_{1}^2 t_{n}^{2}} h_{0}\biggr) \\
&= \sum_{\begin{subarray}{c}s_{1},\ldots,s_{n}\in k^{\times} \\s_{n+1}\in ah_{0}k^{\times2}\\ s_{1}^{2}\cdots s_{n-1}^{2}s_{n}s_{n+1}=\alpha\end{subarray}} 
\psi(s_{1} + \cdots + s_{n-1} + s_{n} + s_{n+1}) \\
&= \sum_{\begin{subarray}{c}s_{1},\ldots, s_{n+1}\in k^{\times}\\ s_{1}^{2}\cdots s_{n-1}^{2}s_{n}s_{n+1}=\alpha\end{subarray}} 
\frac{1+\omega_{0}(s_{n+1}a^{-1}h_{0}^{-1})}{2}\cdot\psi(s_{1} + \cdots + s_{n-1} + s_{n} + s_{n+1}) \\
&=\frac{1}{2}\Kl_{\alpha}^{n+1}\bigl(\psi; (2, \ldots, 2, 1, 1) \bigr)+ \frac{\omega_{0}(ah_{0})}{2}\cdot\Kl_{\alpha}^{n+1}\bigl(\psi; (2, \ldots, 2, 1, 1), (\mathbbm{1},\ldots,\mathbbm{1},\omega_{0})\bigr). 
\end{align*}
Here, in the 3rd equality, we replaced $(t_{1}, \ldots, t_{n-1}, t_{n})$ with $(t_{1}t_{n}, \ldots, t_{n-1}t_{n}, t_{n})$.
\end{proof}

\begin{cor}\label{cor:pmSp}
Let $\xi\in\{\pm1\}$ and $a\in k^{\times}$.
For an affine generic element $h \in I_{\Sp_{2n}}^{+}\cap \Sp_{2n}^{\rs}(F)$ with the simple affine components $(h_1, \ldots, h_{n}, h_{0})$, we have
\begin{align*}
\Bigl(\Theta^{\Sp_{2n}}_{\xi,0,a}+\Theta^{\Sp_{2n}}_{\xi,1,a\epsilon^{-1}}\Bigr)(h)
&=\Kl_{\alpha}^{n+1}(\psi; w_{\Sp_{2n}}),\\
\Bigl(\Theta^{\Sp_{2n}}_{\xi,0,a}-\Theta^{\Sp_{2n}}_{\xi,1,a\epsilon^{-1}}\Bigr)(h)
&=\omega_{0}(ah_{0})\cdot\Kl_{\alpha}^{n+1}(\psi; w_{\Sp_{2n}}, \chi_{\Sp_{2n}}),
\end{align*}
where
\[
w_{\Sp_{2n}}:=(2,\ldots,2,1,1),\quad
\chi_{\Sp_{2n}}:=(\mathbbm{1},\ldots,\mathbbm{1},\omega_{0}),\text{ and}\quad
\alpha:=ah_{1}^{2}\cdots h_{n-1}^{2}h_{n}h_{0}.
\]
\end{cor}

\begin{rem}\label{rem:regss}
In fact, every affine generic element of $\Sp_{2n}(F)$ is regular semisimple.
This can be checked by combining the following:
\begin{itemize}
\item
For an affine generic element of $\Sp_{2n}(F)$ with simple affine components $(h_{1},\ldots,h_{n},h_{0})$, its simple affine components as an element of $I_{\GL_{2n}}^{+}$ are given by $(h_{1},\ldots,h_{n-1},h_{n},h_{n-1},\ldots,h_{1},h_{0})$.
In particular, every affine generic element of $\Sp_{2n}(F)$ is affine generic also as an element of $\GL_{2n}(F)$.
\item
The characteristic polynomial of an affine generic element of $\GL_{2n}(F)$ is Eisenstein (see Lemma \ref{lem:charpoly1}), hence every affine generic element of $\GL_{2n}$ is regular semisimple.
\item
If an element of $\Sp_{2n}$ is regular semisimple as an element of $\GL_{2n}$, then it is regular semisimple as an element of $\Sp_{2n}$.
\end{itemize}
\end{rem}

\subsection{The case of $\SO_{2n}^{\mu}$}\label{subsec:char-SO-ram}
Let $\mu$ be a ramified quadratic character of $F^{\times}$ and $\SO_{2n}^{\mu}$ the quasi-split special even orthogonal group corresponding to $\mu$, where $n\geq2$.

Let $(\xi,a)\in\SSC(\SO_{2n}^{\mu})$ and we take the simple supercuspidal representation $\pi^{\SO_{2n}^{\mu}}_{\xi, a}$ defined in Section \ref{subsec:ssc-SO-ram}.

\begin{lem}\label{lem:sumSOram}
Let $h \in I_{\SO_{2n}^{\mu}}^{+}$ be an affine generic element.
Then we have a bijection
\[
\big\{y \in (\pm I^{+}_{\SO_{2n}^{\mu}}) \backslash {\SO_{2n}^{\mu}}(F) \mid yhy^{-1} \in \pm I_{\SO_{2n}^{\mu}}^{+}\big\}
\cong S_{\SO_{2n}^{\mu}}(q).
\]
\end{lem}

\begin{proof}
By the same argument as in the case of $\Sp_{2n}$ (Lemma \ref{lem:sumSp}), we can identify the left-hand side with the set $(\pm I^{+}_{\SO_{2n}^{\mu}}) \backslash (\pm I_{\SO_{2n}^{\mu}})$.
Since we have an isomorphism $I^{+}_{\SO_{2n}^{\mu}} \backslash I_{\SO_{2n}^{\mu}} \cong S_{\SO_{2n}^{\mu}}(q)$ (recall the description of these groups in Section \ref{subsec:ssc-SO-ram}), we get the claim.
\end{proof}

\begin{prop}\label{prop:charSOram}
Let $h \in I_{\SO_{2n}^{\mu}}^{+}\cap \SO_{2n}^{\mu,\rs}(F)$ be an affine generic element.
Let $(h_1, \ldots, h_{n-1}, h_{0})$ be the simple affine components of $h$.
Then we have 
\[
\Theta^{\SO_{2n}^{\mu}}_{\xi, a}(h) = \Kl_{\alpha}^{n}(\psi),
\]
where $\alpha:=ah_{1}\cdots h_{n-1}h_{0}$.
\end{prop}

\begin{proof}
Note that the affine genericity of $h$ implies that $\alpha$ is not zero.

By combining Theorem \ref{thm:CF} with Lemma \ref{lem:sumSOram}, we have
\begin{align*}
\Theta^{\SO_{2n}^{\mu}}_{\xi,a}(h) 
&= \sum_{y \in S_{\SO_{2n}^{\mu}}(q)} \chi^{\SO_{2n}^{\mu}}_{\xi,a}(yhy^{-1}) \\
&= \sum_{t_{1},\ldots,t_{n-1}\in k^{\times}}
\psi\biggl(\frac{t_{1}}{t_{2}} h_{1} + \cdots + \frac{t_{n-2}}{t_{n-1}}h_{n-2} + t_{n-1}h_{n-1} + \frac{a}{t_{1}} h_{0}\biggr) \\
&= \sum_{\begin{subarray}{c}s_{1},\ldots,s_{n}\in k^{\times} \\s_{1}\cdots s_{n}=\alpha\end{subarray}} 
\psi(s_{1}+\cdots+s_{n}) \\
&=\Kl_{\alpha}^{n}(\psi).
\end{align*}
\end{proof}

\subsection{The case of $\SO_{2n+2}$}\label{subsec:char-SO}
We consider the split even special orthogonal group $\SO_{2n+2}$, where $n\geq2$.

Let $(\xi,\kappa,a,\zeta)\in\SSC(\SO_{2n+2})$ and we take the simple supercuspidal representation $\pi^{\SO_{2n+2}}_{\xi,\kappa,a,\zeta}$ defined in Section \ref{subsec:ssc-SO}.

By the same idea as the proof of \cite[Lemma 3.14]{MR3904769}, we can show the following lemma:
\begin{lem}\label{lem:sumSOspl}
Let $h$ be an affine generic element of $I_{\SO_{2n+2}}^{+}$.
Then the set
\[
\{
y\in \pm I_{\SO_{2n+2}}^{+}\lan\varphi_{\epsilon^{\kappa},-a^{-1}}^{\SO_{2n+2}}\ran\backslash\SO_{2n+2}(F)
\mid
yhy^{-1}\in\pm I_{\SO_{2n+2}}^{+}\lan\varphi_{\epsilon^{\kappa},-a^{-1}}^{\SO_{2n+2}}\ran
\}
\]
is represented by $\{\pm1\}\backslash T_{\SO_{2n+2}}(q)$.
\end{lem}

\begin{proof}
Let $y$ be an element of $\SO_{2n+2}(F)$ satisfying $yhy^{-1}\in\pm I_{\SO_{2n+2}}^{+}\lan\varphi_{\epsilon^{\kappa},-a^{-1}}^{\SO_{2n+2}}\ran$.
By noting that the order of $\varphi_{\epsilon^{\kappa},-a^{-1}}^{\SO_{2n+2}}$ is given by two and $\varphi_{\epsilon^{\kappa},-a^{-1}}^{\SO_{2n+2}}$ normalizes $I_{\SO_{2n+2}}^{+}$, we have $yh^{2}y^{-1}\in I_{\SO_{2n+2}}^{+}$.
As we assume that $p$ is not equal to $2$, also the element $h^{2}$ is affine generic.
Thus, by Lemma \ref{lem:key-lem}, $y$ belongs to the normalizer of the Iwahori subgroup $I_{\SO_{2n+2}}$, which equals $I_{\SO_{2n+2}}\lan\varphi_{\epsilon^{\kappa},-a^{-1}}^{\SO_{2n+2}}\ran$.
Thus we get the claim.
\end{proof}

\begin{prop}\label{prop:charSOspl}
Let $h \in I_{\SO_{2n+2}}^{+}\cap \SO_{2n+2}^{\rs}(F)$ be an affine generic element.
Let $(h_1, \ldots, h_{n+1}, h_{0})$ be the simple affine components of $h$.
Then we have 
\[
\Theta^{\SO_{2n+2}}_{\xi,\kappa,a,\zeta}(h) 
= 
\frac{1}{2}\Kl_{\gamma}^{n+2}(\psi;w_{\SO_{2n+2}})+ \frac{\omega_{0}(\epsilon^{\kappa}h_{n}h_{n+1})}{2}\cdot\Kl_{\gamma}^{n+2}(\psi; w_{\SO_{2n+2}}, \chi_{\SO_{2n+2}}),
\]
where $\Kl_{\gamma}^{n+2}(\psi;w_{\SO_{2n+2}})$ and $\Kl_{\gamma}^{n+2}(\psi; w_{\SO_{2n+2}}, \chi_{\SO_{2n+2}})$ on the right-hand side are the Kloosterman sums in Definition \ref{defn:Kl} with
\begin{align*}
w_{\SO_{2n+2}}&:=(1,\underbrace{2,\ldots, 2}_{n-2},1,1,1),\quad
\chi_{\SO_{2n+2}}:=(\underbrace{\mathbbm{1}, \ldots,\mathbbm{1}}_{n-1}, \omega_{0},\omega_{0},\mathbbm{1}), \\
\gamma&:=a\epsilon^{\kappa}h_{1} h_{2}^{2}\cdots h_{n-1}^{2}h_{n}h_{n+1}h_{0}.
\end{align*}
\end{prop}

\begin{proof}
By Lemma \ref{lem:sumSOspl}, the index set of the character formula (Theorem \ref{thm:CF}) is represented by a set
\[
\{\pm1\}
\backslash
\{ \diag(t_{1}, \ldots, t_{n+1}, t_{n+1}^{-1},\ldots, t_{1}^{-1} \mid t_{i}\in k^{\times} \}, 
\]
we have
\begin{align*}
&\phantom{{}={}}\Theta^{\SO_{2n+2}}_{\xi,\kappa,a,\zeta}(h) \\
&=
\sum_{x\in \{\pm1\}\backslash T_{\SO_{2n+2}}(q)} \tilde{\chi}_{\xi,\kappa,a,\zeta}^{\SO_{2n+2}}(xhx^{-1})\\
&=
\sum_{(t_{1}\ldots,t_{n+1})\in\{\pm1\}\backslash (k^{\times})^{n+1}}
\psi\biggl(\frac{t_{1}}{t_{2}}h_{1}+\cdots+\frac{t_{n}}{t_{n+1}}h_{n}+\epsilon^{\kappa}t_{n}t_{n+1}h_{n+1}+\frac{a}{t_{1}t_{2}}h_{0}\biggr).
\end{align*}
Here, by replacing $t_{i}$ with $t_{i}t_{n+1}$ for each $1\leq i \leq n$, the right-hand side of this equality is equal to
\begin{align*}
&\phantom{{}={}}\sum_{\begin{subarray}{c}t_{1},\ldots, t_{n}\in k^{\times}\\ t_{n+1}\in \{\pm1\}\backslash k^{\times}\end{subarray}}
\psi\biggl(\frac{t_{1}}{t_{2}}h_{1}+\cdots+t_{n}h_{n}+\epsilon^{\kappa}t_{n}t_{n+1}^{2}h_{n+1}+\frac{a}{t_{1}t_{2}t_{n+1}^{2}}h_{0}\biggr)\\
&=
\sum_{\begin{subarray}{c}t_{1},\ldots, t_{n}\in k^{\times}\\ s\in k^{\times}\end{subarray}}
\frac{1+\omega_{0}(s)}{2}
\psi\biggl(\frac{t_{1}}{t_{2}}h_{1}+\cdots+t_{n}h_{n}+\epsilon^{\kappa}t_{n}sh_{n+1}+\frac{a}{t_{1}t_{2}s}h_{0}\biggr)\\
&=
\sum_{\begin{subarray}{c}s_{1},\ldots, s_{n+2}\in k^{\times}\\ s_{1}s_{2}^{2}\cdots s_{n-1}^{2}s_{n}s_{n+1}s_{n+2}\\=a\epsilon^{\kappa}h_{1}h_{2}^{2}\cdots h_{n-1}^{2}h_{n}h_{n+1}h_{0}\end{subarray}}
\frac{1+\omega_{0}(s_{n}s_{n+1}\epsilon^{\kappa}h_{n}h_{n+1})}{2}
\psi(s_{1}+\cdots+s_{n+2})\\
&=
\frac{1}{2}\Kl_{\gamma}^{n+2}(\psi;w_{\SO_{2n+2}})+ \frac{\omega_{0}(\epsilon^{\kappa}h_{n}h_{n+1})}{2}\cdot\Kl_{\gamma}^{n+2}(\psi; w_{\SO_{2n+2}}, \chi_{\SO_{2n+2}}).
\end{align*}
\end{proof}

\begin{cor}\label{cor:pmSOspl}
Let $h \in I_{\SO_{2n+2}}^{+}\cap \SO_{2n+2}^{\rs}(F)$ be an affine generic element.
Let $(h_1, \ldots, h_{n+1}, h_{0})$ be the simple affine components of $h$.
Then we have 
\begin{align*}
\Bigl(\Theta^{\SO_{2n+2}}_{\xi,0,a,\zeta}+\Theta^{\SO_{2n+2}}_{\xi,1,a\epsilon^{-1},\zeta}\Bigr)(h)
&=\Kl_{\gamma}^{n+2}(\psi; w_{\SO_{2n+2}}),\\
\Bigl(\Theta^{\SO_{2n+2}}_{\xi,0,a,\zeta}-\Theta^{\SO_{2n+2}}_{\xi,1,a\epsilon^{-1},\zeta}\Bigr)(h)
&=\omega_{0}(h_{n}h_{n+1})\cdot\Kl_{\gamma}^{n+2}(\psi;w_{\SO_{2n+2}},\chi_{\SO_{2n+2}}),
\end{align*}
where 
\begin{align*}
w_{\SO_{2n+2}}&:=(1,\underbrace{2,\ldots, 2}_{n-2},1,1,1),\quad
\chi_{\SO_{2n+2}}:=(\underbrace{\mathbbm{1}, \ldots,\mathbbm{1}}_{n-1}, \omega_{0},\omega_{0},\mathbbm{1}), \\
\gamma&:=ah_{1} h_{2}^{2}\cdots h_{n-1}^{2}h_{n}h_{n+1}h_{0}.
\end{align*}
\end{cor}

\begin{lem}\label{lem:phiSOspl}
Let $(\alpha,\beta)$ and $(\alpha',\beta')$ be elements of $k^{\times}\times k^{\times}$.
Then $\varphi^{\SO_{2n+2}}_{\alpha,\beta}$ belongs to $\pm I_{\SO_{2n+2}}^{+}\lan\varphi_{\alpha',\beta'}^{\SO_{2n+2}}\ran$ if and only if $(\alpha,\beta)=(\alpha',\beta')$.
\end{lem}

\begin{proof}
Since the ``if'' part is trivial, we show the ``only if'' part.
We suppose that 
\[
\varphi^{\SO_{2n+2}}_{\alpha,\beta}
\in
\pm I_{\SO_{2n+2}}^{+}\lan\varphi_{\alpha',\beta'}^{\SO_{2n+2}}\ran.
\]
Then, since $\varphi_{\alpha,\beta}^{\SO_{2n+2}}$ does not belong to $\pm I_{\SO_{2n+2}}^{+}$, we have
\[
\varphi^{\SO_{2n+2}}_{\alpha,\beta}
\in
\pm I_{\SO_{2n+2}}^{+}\varphi_{\alpha',\beta'}^{\SO_{2n+2}}.
\]
In particular we have
\[
\varphi^{\SO_{2n+2}}_{\alpha,\beta}(\varphi_{\alpha',\beta'}^{\SO_{2n+2}})^{-1}
\in
\pm I_{\SO_{2n+2}}^{+}.
\]
Here the left-hand side is given by
\[
\varphi^{\SO_{2n+2}}_{\alpha,\beta}(\varphi_{\alpha',\beta'}^{\SO_{2n+2}})^{-1}
=
\diag\biggl(\frac{\beta'}{\beta},\underbrace{1,\ldots,1}_{n-1},\frac{\alpha'}{\alpha}, \frac{\alpha}{\alpha'},\underbrace{1,\ldots,1}_{n-1}, \frac{\beta}{\beta'}\biggr).
\]
Thus, this element belongs to the right-hand side only if we have $(\alpha,\beta)=(\alpha',\beta')$.
\end{proof}

\begin{lem}\label{lem:sumSOspl2}
Let $\alpha,\beta\in k^{\times}$, and $h \in I_{\SO_{2n+2}}^{+}\cap \SO_{2n+2}^{\rs}(F)$ be an element such that $(h\varphi_{\alpha,\beta}^{\SO_{2n+2}})^{2}$ is affine generic.
Then the set
\[
\pm I_{\SO_{2n+2}}^{+}\lan\varphi_{\epsilon^{\kappa}, -a^{-1}}^{\SO_{2n+2}}\ran
\big\backslash
\bigl\{x\in \SO_{2n+2}(F) \,\big\vert\, xh\varphi_{\alpha,\beta}^{\SO_{2n+2}}x^{-1} \in \pm I_{\SO_{2n+2}}^{+}\lan\varphi_{\epsilon^{\kappa}, -a^{-1}}^{\SO_{2n+2}}\ran\bigr\}
\]
is represented by
\[
\{\pm1\}
\big\backslash
\bigl\{\diag(x_{1}, \ldots, x_{n+1}, x_{n+1}^{-1}, \ldots, x_{1}^{-1}) 
\,\big\vert\, 
x_{i}\in k^{\times}, x_{n+1}^{-2}\alpha=\epsilon^{\kappa}, x_{1}^{-2}\beta=-a^{-1}
\bigr\}.
\]
\end{lem}

\begin{proof}
Let $x$ be an element of $\SO_{2n+2}(F)$.
If $x$ satisfies that
\[
xh\varphi_{\alpha,\beta}^{\SO_{2n+2}}x^{-1} \in \pm I_{\SO_{2n+2}}^{+}\lan\varphi_{\epsilon^{\kappa},-a^{-1}}^{\SO_{2n+2}}\ran,
\]
then we have 
\[
\bigl(xh\varphi_{\alpha,\beta}^{\SO_{2n+2}}x^{-1}\bigr)^{2}=x\bigl(h\varphi_{\alpha,\beta}^{\SO_{2n+2}}\bigr)^{2}x^{-1} \in  I_{\SO_{2n+2}}.
\]
As $(h\varphi_{\alpha,\beta}^{\SO_{2n+2}})^{2}$ is affine generic, we have
\[
x\in I_{\SO_{2n+2}}\lan\varphi_{\epsilon^{\kappa},-a^{-1}}^{\SO_{2n+2}}\ran
\]
by Lemma \ref{lem:key-lem}.
Thus our task is to determine when an element $x$ of $T_{\SO_{2n+2}}(q)$ satisfies 
\[
xh\varphi_{\alpha,\beta}^{\SO_{2n+2}}x^{-1}
\in 
\pm I_{\SO_{2n+2}}^{+}\lan\varphi_{\epsilon^{\kappa}, -a^{-1}}^{\SO_{2n+2}}\ran.
\tag{$\ast$}
\]
Since we have 
\[
xh\varphi_{\alpha,\beta}^{\SO_{2n+2}}x^{-1}
=
xhx^{-1} \cdot x\varphi_{\alpha,\beta}^{\SO_{2n+2}}x^{-1}
\]
and $xhx^{-1}$ belongs to $I_{\SO_{2n+2}}^{+}$, $(\ast)$ happens if and only if
\[
x\varphi_{\alpha,\beta}^{\SO_{2n+2}}x^{-1}
\in
\pm I_{\SO_{2n+2}}^{+}\lan\varphi_{\epsilon^{\kappa}, -a^{-1}}^{\SO_{2n+2}}\ran.
\]

On the other hand, if we put
\[
x:=\diag(x_{1}, \ldots, x_{n+1}, x_{n+1}^{-1}, \ldots, x_{1}^{-1}),
\]
then the element $x\varphi_{\alpha,\beta}^{\SO_{2n+2}}x^{-1}$ is given by
\[
x\varphi_{\alpha,\beta}^{\SO_{2n+2}}x^{-1}=
\begin{pmatrix}
&&&&&x_{1}^{2}(\beta\varpi)^{-1}\\
&I_{n-1}&&&&\\
&&&x_{n+1}^{2}\alpha^{-1}&&\\
&&x_{n+1}^{-2}\alpha&&&\\
&&&&I_{n-1}&\\
x_{1}^{-2}\beta\varpi&&&&&
\end{pmatrix}.
\]
Hence, by Lemma \ref{lem:phiSOspl}, we have
\[
x\varphi_{\alpha,\beta}^{\SO_{2n+2}}x^{-1}
=
\varphi_{\epsilon^{\kappa}, -a^{-1}}
\]
and
\[
\begin{cases}
x_{n+1}^{-2}\alpha=\epsilon^{\kappa}, \text{ and}\\
x_{1}^{-2}\beta=-a^{-1}.
\end{cases}
\]
\end{proof}

\begin{rem}\label{rem:affgen-square}
Let $h$ be an element of $I_{\SO_{2n+2}}^{+}$ whose simple affine components are given by 
\[
(h_1, \ldots, h_{n-1}, h_{n}, h_{n+1}, h_{0})\in V_{\SO_{2n+2}} \cong k^{\oplus n+2}.
\]
Then, for $(\alpha,\beta)\in k^{\times}\times k^{\times}$, the $\varphi_{\alpha,\beta}^{\SO_{2n+2}}$-conjugated element $\varphi_{\alpha,\beta}^{\SO_{2n+2}}h\varphi_{\alpha,\beta}^{\SO_{2n+2},-1}$ again belongs to $I_{\SO_{2n+2}}^{+}$ and its simple affine components are given by 
\[
(-\beta^{-1}h_{0}, h_{2},\ldots, h_{n-1}, \alpha h_{n+1}, \alpha^{-1}h_{n}, -\beta h_{1}).
\]
Thus the squared element
\[
(h\varphi_{\alpha,\beta}^{\SO_{2n+2}})^{2}
=h\cdot(\varphi_{\alpha,\beta}^{\SO_{2n+2}}h\varphi_{\alpha,\beta}^{\SO_{2n+2},-1})
\]
belongs to $I_{\SO_{2n+2}}^{+}$ and has
\[
(h_{1}-\beta^{-1}h_{0}, 2h_{2},\ldots, 2h_{n-1}, h_{n}+\alpha h_{n+1}, h_{n+1}+\alpha^{-1}h_{n}, h_{0}-\beta h_{1})
\]
as its simple affine components.
Hence $(h\varphi_{\alpha,\beta}^{\SO_{2n+2}})^{2}$ is affine generic if and only if none of these components is zero.
\end{rem}

\begin{prop}\label{prop:charSOspl2}
Let $\alpha\in k^{\times}$, $\beta\in k^{\times}$, and $h \in I_{\SO_{2n+2}}^{+}\cap \SO_{2n+2}^{\rs}(F)$ an element such that $(h\varphi_{\alpha,\beta}^{\SO_{2n+2}})^{2}$ is affine generic.
Let $(h_1, \ldots, h_{n+1}, h_{0})$ be the simple affine components of $h$.
Then we have 
\[
\Theta^{\SO_{2n+2}}_{\xi,\kappa,a,\zeta}\bigl(h\varphi_{\alpha,\beta}^{\SO_{2n+2}}\bigr) 
=
\begin{cases}
\zeta\cdot\Kl_{\pm\delta}^{n}(\psi) & \text{if }\omega_{0}(\alpha)=(-1)^{\kappa} \text{ and }\omega_{0}(\beta)=\omega_{0}(-a^{-1}), \\
0 & \text{ otherwise},
\end{cases}
\]
\[
\text{where}\quad
\delta:={\sqrt{-\beta a}\,}^{-1}{\sqrt{\alpha\epsilon^{-\kappa}}\,}^{-1}a(-\beta h_{1}+h_{0})h_{2}\cdots h_{n-1}(h_{n}+\alpha h_{n+1}).
\]
Here, we put $\Kl_{\pm\delta}^{n}(\psi):= \Kl_{\delta}^{n}(\psi)+\Kl_{-\delta}^{n}(\psi)$ and, in the former case, we fix square roots $\sqrt{-\beta a}$ and $\sqrt{\alpha\epsilon^{-\kappa}}$ of $-\beta a$ and $\alpha\epsilon^{-\kappa}$, respectively.
\end{prop}

\begin{proof}
By Lemma \ref{lem:sumSOspl2}, the index set of the character formula (Theorem \ref{thm:CF}) is given by
\[
\{\pm1\}
\big\backslash
\bigl\{\diag(x_{1}, \ldots, x_{n+1}, x_{n+1}^{-1}, \ldots, x_{1}^{-1}) 
\,\big\vert\, 
x_{i}\in k^{\times}, x_{n+1}^{-2}\alpha=\epsilon^{\kappa}, x_{1}^{-2}\beta=-a^{-1}
\bigr\}.
\]
In particular, the character is not zero only if 
\[
\begin{cases}
\omega_{0}(\alpha)=(-1)^{\kappa}, \text{ and}\\
\omega_{0}(\beta)=\omega_{0}(-a^{-1}).
\end{cases}
\]
From now on, we assume these two conditions.

Then we have
\begin{align*}
&\phantom{{}={}}\Theta^{\SO_{2n+2}}_{\xi,\kappa,a,\zeta}\bigl(h\varphi_{\alpha,\beta}^{\SO_{2n+2}}\bigr)\\
&=
\sum_{x\text{: as above}}
\tilde{\chi}_{\xi,\kappa,a,\zeta}^{\SO_{2n+2}}\bigl(xhx^{-1} \cdot x\varphi_{\alpha,\beta}^{\SO_{2n+2}}x^{-1}\bigr)\\
&=
\zeta
\sum_{\begin{subarray}{c}x_{1},\ldots,x_{n+1}\in k^{\times}\\ x_{n+1}^{2}=\alpha\epsilon^{-\kappa}\\ x_{1}=\sqrt{-\beta a}\end{subarray}}
\psi\biggl(
\frac{x_{1}}{x_{2}}h_{1}+\cdots+\frac{x_{n}}{x_{n+1}}h_{n}+\epsilon^{\kappa}x_{n}x_{n+1}h_{n+1}+a\frac{h_{0}}{x_{1}x_{2}}
\biggr)\\
&=
\zeta
\sum_{\begin{subarray}{c}x_{2},\ldots,x_{n}\in k^{\times}\\ x_{n+1}^{2}=\alpha\epsilon^{-\kappa}\\ x_{1}=\sqrt{-\beta a}\end{subarray}}
\psi\biggl(
\frac{x_{1}h_{1}+ax_{1}^{-1}h_{0}}{x_{2}}+\frac{x_{2}}{x_{3}}h_{2}+\cdots+\frac{x_{n-1}}{x_{n}}h_{n-1}+x_{n}\Bigl(\frac{h_{n}}{x_{n+1}}+x_{n+1}h_{n+1}\epsilon^{\kappa}\Bigr)
\biggr)\\
&=
\zeta
\sum_{\begin{subarray}{c}x_{n+1}=\pm\sqrt{\alpha\epsilon^{-\kappa}}\\ x_{1}=\sqrt{-\beta a}\end{subarray}}
\Kl^{n}_{(x_{1}h_{1}+ax_{1}^{-1}h_{0})h_{2}\cdots h_{n-1}(x_{n+1}^{-1}h_{n}+x_{n+1}h_{n+1}\epsilon^{\kappa})}(\psi)\\
&=
\zeta\cdot
\Kl^{n}_{\pm(\sqrt{-\beta a}h_{1}+a{\sqrt{-\beta a}\,}^{-1}h_{0})h_{2}\cdots h_{n-1}({\sqrt{\alpha\epsilon^{-\kappa}}\,}^{-1}h_{n}+\sqrt{\alpha\epsilon^{-\kappa}}h_{n+1}\epsilon^{\kappa})}(\psi)\\
&=
\zeta\cdot
\Kl^{n}_{\pm{\sqrt{-\beta a}\,}^{-1}{\sqrt{\alpha\epsilon^{-\kappa}}\,}^{-1}a(-\beta h_{1}+h_{0})h_{2}\cdots h_{n-1}(h_{n}+\alpha h_{n+1})}(\psi).
\end{align*}
\end{proof}

\subsection{The case of $\SO_{2n+2}^{\ur}$}\label{subsec:char-SO-ur}
We consider the quasi-split unramified even special orthogonal group $\SO_{2n+2}^{\ur}$, where $n\geq2$.

Let $(\xi,\kappa,a,\zeta)\in\SSC(\SO_{2n+2}^{\ur})$ and we take the simple supercuspidal representation $\pi^{\SO_{2n+2}^{\ur}}_{\xi,\kappa,a,\zeta}$ defined in Section \ref{subsec:ssc-SO}.

As in the proof of Lemma \ref{lem:sumSOspl}, we can show the following lemma:
\begin{lem}\label{lem:sumSOur}
Let $h$ be an affine generic element of $I_{\SO_{2n+2}^{\ur}}^{+}$.
Then the set
\[
\{
y\in \pm I_{\SO_{2n+2}^{\ur}}^{+}\lan\varphi_{(\tilde{\epsilon}/c(\tilde{\epsilon}))^{\kappa},-a^{-1}}^{\SO_{2n+2}^{\ur}}\ran\backslash\SO_{2n+2}^{\ur}(F)
\mid
yhy^{-1}\in\pm I_{\SO_{2n+2}^{\ur}}^{+}\lan\varphi_{(\tilde{\epsilon}/c(\tilde{\epsilon}))^{\kappa},-a^{-1}}^{\SO_{2n+2}^{\ur}}\ran
\}
\]
is represented by $\{\pm1\}\backslash T_{\SO_{2n+2}^{\ur}}(q)$.
\end{lem}

\begin{prop}\label{prop:charSOur}
Let $h \in I_{\SO_{2n+2}^{\ur}}^{+}\cap \SO_{2n+2}^{\ur,\rs}(F)$ be an affine generic element.
Let $(h_1, \ldots, h_{n}, h_{0})$ be the simple affine components of $h$.
Then we have 
\[
\Theta^{\SO_{2n+2}^{\ur}}_{\xi,\kappa,a,\zeta}(h) 
= \frac{1}{2}\Kl_{\gamma}^{n;1}(\psi;w_{\SO_{2n+2}^{\ur}}) + \frac{\omega_{0}\bigl(\Nr(h_{n})\epsilon^{\kappa}\bigr)}{2}\Kl_{\gamma}^{n;1}(\psi;w_{\SO_{2n+2}^{\ur}},\chi_{\SO_{2n+2}^{\ur}}),
\]
where $\Kl_{\gamma}^{n;1}(\psi;w_{\SO_{2n+2}^{\ur}})$ and $\Kl_{\gamma}^{n;1}(\psi;w_{\SO_{2n+2}^{\ur}},\chi_{\SO_{2n+2}^{\ur}})$ are the Kloosterman sums in Definition \ref{defn:Kl} with
\begin{align*}
w_{\SO_{2n+2}^{\ur}}&= (1,2,\ldots,2,1;1),\quad
\chi_{\SO_{2n+2}^{\ur}}=(\mathbbm{1}, \ldots, \mathbbm{1};\omega_{0}),\\
\gamma&=a\epsilon^{\kappa}h_{1}h_{2}^{2}\cdots h_{n-1}^{2}\Nr(h_{n})h_{0}.
\end{align*}
\end{prop}

\begin{proof}
By Lemma \ref{lem:sumSOur}, the index set of the character formula is represented by a set
\[
\{\pm1\}
\backslash
\{ \diag(t_{1}, \ldots, t_{n}, \tilde{t}_{n+1}, t_{n}^{-1},\ldots, t_{1}^{-1}) \mid t_{i}\in k^{\times},\, \tilde{t}_{n+1} \in \Nr^{1} \}.
\]
Thus we have
\begin{align*}
&\phantom{{}={}}\Theta^{\SO_{2n+2}^{\ur}}_{\xi,\kappa,a,\zeta}(h) \\
&=
\sum_{x\in \{\pm1\}\backslash T_{\SO_{2n+2}^{\ur}}(q)} \tilde{\chi}_{\xi,\kappa,a,\zeta}^{\SO_{2n+2}^{\ur}}(xhx^{-1})\\
&=
\sum_{(t_{1},\ldots,\tilde{t}_{n+1})\in\{\pm1\}\backslash ((k^{\times})^{n}\times\Nr^{1})}
\psi\biggl(\frac{t_{1}}{t_{2}}h_{1}+\cdots+\frac{t_{n-1}}{t_{n}}h_{n-1}+\Tr\Bigl(\tilde{\epsilon}^{\kappa}\frac{t_{n}}{\tilde{t}_{n+1}}h_{n}\Bigr)+\frac{a}{t_{1}t_{2}}h_{0}\biggr).
\end{align*}
By noting that 
\begin{align*}
\{\pm1\}\backslash \bigl((k^{\times})^{n}\times\Nr^{1}\bigr)&\xrightarrow{\cong} (k^{\times})^{n-1}\times (k^{\times}\Nr^{1})\\
(t_{1}, \ldots, t_{n}, \tilde{t}_{n+1})&\mapsto \biggl(\frac{t_{1}}{t_{2}}, \ldots, \frac{t_{n-1}}{t_{n}}, \frac{t_{n}}{\tilde{t}_{n+1}}\biggr)
\end{align*}
and that $k^{\times}\Nr^{1}$ consists of elements of $\tilde{k}^{\times}$ whose norms lie in $k^{\times2}$, 
the right-hand side of the above equality is equal to
\[
\sum_{\begin{subarray}{c}s_{1},\ldots, s_{n}\in k^{\times}\\ s_{n+1}\in \tilde{k}^{\times}\\ s_{1}s_{2}^{2}\cdots s_{n-1}^{2}\Nr(s_{n+1})s_{n}\\=a\epsilon^{\kappa}h_{1}h_{2}^{2}\cdots h_{n-1}^{2}\Nr(h_{n})h_{0}\end{subarray}}
\frac{1+\omega_{0}\bigl(\Nr(s_{n+1})\Nr(h_{n})\epsilon^{\kappa}\bigr)}{2}
\psi(s_{1}+\cdots+s_{n})\tilde\psi(s_{n+1})
\]
\[
=\frac{1}{2}\Kl_{\gamma}^{n;1}(\psi;w_{\SO_{2n+2}^{\ur}}) + \frac{\omega_{0}\bigl(\Nr(h_{n})\epsilon^{\kappa}\bigr)}{2}\Kl_{\gamma}^{n;1}(\psi;w_{\SO_{2n+2}^{\ur}},\chi_{\SO_{2n+2}^{\ur}}).
\]
\end{proof}

\begin{cor}\label{cor:pmSOur}
Let $h \in I_{\SO_{2n+2}^{\ur}}^{+}\cap \SO_{2n+2}^{\ur,\rs}(F)$ be an affine generic element.
Let $(h_1, \ldots, h_{n}, h_{0})$ be the simple affine components of $h$.
Then we have 
\begin{align*}
\Bigl(\Theta^{\SO_{2n+2}^{\ur}}_{\xi,0,a,\zeta}+\Theta^{\SO_{2n+2}^{\ur}}_{\xi,1,a\epsilon^{-1},\zeta}\Bigr)(h)
&=\Kl_{\gamma}^{n;1}(\psi;w_{\SO_{2n+2}^{\ur}}),\\
\Bigl(\Theta^{\SO_{2n+2}^{\ur}}_{\xi,0,a,\zeta}-\Theta^{\SO_{2n+2}^{\ur}}_{\xi,1,a\epsilon^{-1},\zeta}\Bigr)(h)
&=\omega_{0}\bigl(\Nr(h_{n})\bigr)\Kl_{\gamma}^{n;1}(\psi; w_{\SO_{2n+2}^{\ur}}, \chi_{\SO_{2n+2}^{\ur}}),
\end{align*}
where
\begin{align*}
w_{\SO_{2n+2}^{\ur}}&= (1,2,\ldots,2,1;1),\quad
\chi_{\SO_{2n+2}^{\ur}}=(\mathbbm{1}, \ldots, \mathbbm{1};\omega_{0}),\\
\gamma&=ah_{1}h_{2}^{2}\cdots h_{n-1}^{2}\Nr(h_{n})h_{0}.
\end{align*}
\end{cor}

In exactly the same manner as the split case, we get the following two lemmas:

\begin{lem}\label{lem:phiSOur}
Let $(\alpha,\beta)$ and $(\alpha',\beta')$ be elements of $\Nr^{1}\times k^{\times}$.
Then $\varphi^{\SO_{2n+2}^{\ur}}_{\alpha,\beta}$ belongs to $\pm I_{\SO_{2n+2}^{\ur}}^{+}\lan\varphi_{\alpha',\beta'}^{\SO_{2n+2}^{\ur}}\ran$ if and only if $(\alpha,\beta)=(\alpha',\beta')$.
\end{lem}

\begin{lem}\label{lem:sumSOur2}
Let $\alpha\in\Nr^{1}$, $\beta\in k^{\times}$, and $h \in I_{\SO_{2n+2}^{\ur}}^{+}\cap \SO_{2n+2}^{\ur,\rs}(F)$ an element such that $(h\varphi_{\alpha,\beta}^{\SO_{2n+2}^{\ur}})^{2}$ is affine generic.
Then the set 
\[
\pm I_{\SO_{2n+2}^{\ur}}^{+}\lan\varphi_{(\tilde{\epsilon}/c(\tilde{\epsilon}))^{\kappa},-a^{-1}}^{\SO_{2n+2}^{\ur}}\ran
\big\backslash
\bigl\{x\in \SO_{2n+2}^{\ur}(F) \,\big\vert\, xh\varphi_{\alpha,\beta}^{\SO_{2n+2}^{\ur}}x^{-1} \in \pm I_{\SO_{2n+2}^{\ur}}^{+}\lan\varphi_{(\tilde{\epsilon}/c(\tilde{\epsilon}))^{\kappa},-a^{-1}}^{\SO_{2n+2}^{\ur}}\ran\bigr\}
\]
is represented by
\[
\{\pm1\}
\big\backslash
\bigl\{\diag(x_{1},\ldots, x_{1}^{-1})\in T_{\SO_{2n+2}^{\ur}}(q) 
\,\big\vert\, 
\alpha\tilde{x}_{n+1}/c(\tilde{x}_{n+1})=(\tilde{\epsilon}/c(\tilde{\epsilon}))^{\kappa},
x_{1}^{-2}\beta=-a^{-1}
\bigr\}.
\]
\end{lem}

\begin{rem}\label{rem:affgen-square-ur}
Similarly to the split case (Remark \ref{rem:affgen-square}), the assumption of Lemma \ref{lem:sumSOur2} can be rephrased as follows:
Let $h$ be an element of $I_{\SO_{2n+2}^{\ur}}^{+}$ whose simple affine components are given by 
\[
(h_1, \ldots, h_{n-1}, h_{n}, h_{0})\in V_{\SO_{2n+2}^{\ur}} \cong k^{\oplus n-1}\oplus \tilde{k}\oplus k.
\]
Then, for $(\alpha,\beta)\in \Nr^{1}\times k^{\times}$, the $\varphi_{\alpha,\beta}^{\SO_{2n+2}^{\ur}}$-conjugated element $\varphi_{\alpha,\beta}^{\SO_{2n+2}^{\ur}}h\varphi_{\alpha,\beta}^{\SO_{2n+2}^{\ur},-1}$ again belongs to $I_{\SO_{2n+2}^{\ur}}^{+}$ and its simple affine components are given by 
\[
(-\beta^{-1}h_{0}, h_{2},\ldots, h_{n-1}, \alpha^{-1}c(h_{n}), -\beta h_{1}),
\]
where $c$ denotes the nontrivial element of $\Gal(\tilde{k}/k)$.
Thus the squared element
\[
(h\varphi_{\alpha,\beta}^{\SO_{2n+2}^{\ur}})^{2}
=h\cdot(\varphi_{\alpha,\beta}^{\SO_{2n+2}^{\ur}}h\varphi_{\alpha,\beta}^{\SO_{2n+2}^{\ur},-1})
\]
belongs to $I_{\SO_{2n+2}^{\ur}}^{+}$ and has
\[
(h_{1}-\beta^{-1}h_{0}, 2h_{2},\ldots, 2h_{n-1}, h_{n}+\alpha^{-1}c(h_{n}), h_{0}-\beta h_{1})
\]
as its simple affine components.
Hence $(h\varphi_{\alpha,\beta}^{\SO_{2n+2}^{\ur}})^{2}$ is affine generic if and only if none of these components is zero.
\end{rem}

\begin{prop}\label{prop:charSOur2}
Let $\beta\in k^{\times}$, and $h \in I_{\SO_{2n+2}^{\ur}}^{+}\cap \SO_{2n+2}^{\ur,\rs}(F)$ an element such that $(h\varphi_{1,\beta}^{\SO_{2n+2}^{\ur}})^{2}$ is affine generic.
Let $(h_1, \ldots, h_{n}, h_{0})$ be the simple affine components of $h$.

Then we have 
\[
\Theta^{\SO_{2n+2}^{\ur}}_{\xi,\kappa,a,\zeta}\bigl(h\varphi_{1,\beta}^{\SO_{2n+2}^{\ur}}\bigr) 
=
\begin{cases}
\zeta\cdot
\Kl^{n}_{\pm\delta}(\psi)
 & \text{if } \kappa=0 \text{ and }
\omega_{0}(\beta)=\omega_{0}(-a^{-1}), \\
0 & \text{otherwise},
\end{cases}
\]
\[
\text{where}\quad
\delta={\sqrt{-\beta a}\,}^{-1}a(-\beta h_{1}+h_{0})h_{2}\cdots h_{n-1}\Tr(h_{n}).
\]
Here, we put $\Kl_{\pm\delta}^{n}(\psi):= \Kl_{\delta}^{n}(\psi)+\Kl_{-\delta}^{n}(\psi)$ and, in the former case, we fix a square root $\sqrt{-\beta a}$ of $-\beta a$.
\end{prop}

\begin{proof}
By the character formula (Theorem \ref{thm:CF}) and the lemma on the index set of the character formula (Lemma \ref{lem:sumSOur2}), the character is not zero only if there exists $x_{1}\in k^{\times}$ and $\tilde{x}_{n+1}\in\Nr^{1}$ satisfying
\[
\begin{cases}
\tilde{x}_{n+1}\in\tilde{\epsilon}^{\kappa}k^{\times}, \text{ and}\\
x_{1}^{-2}\beta=-a^{-1}.
\end{cases}
\tag{$\ast$}
\]
Here note that the first condition of $(\ast)$ is equivalent to $\kappa=0$.
Indeed, if we have $\tilde{x}_{n+1}\in\tilde{\epsilon}^{\kappa}k^{\times}$, then we have
\[
\Nr(\tilde{x}_{n+1})=1 \in \Nr(\tilde{\epsilon}^{\kappa}k^{\times})=\epsilon^{\kappa} k^{\times2}.
\]
Thus we have $\kappa=0$.
Moreover, in this case, the possibility of $\tilde{x}_{n+1}$ is $\pm1$.
On the other hand, the second condition of $(\ast)$ is equivalent to $\omega_{0}(\beta)=\omega_{0}(-a^{-1})$.
From now on, we assume these two conditions.

Then the index set of the character formula is represented by
\[
\bigl\{\diag(x_{1}, \ldots, x_{n},1, x_{n}^{-1}, \ldots, x_{1}^{-1}) 
\,\big\vert\, 
x_{i}\in k^{\times}, x_{1}^{-2}\beta=-a^{-1}
\bigr\}.
\]
Thus we have
\begin{align*}
&\phantom{{}={}}\Theta^{\SO_{2n+2}^{\ur}}_{\xi,0,a,\zeta}\bigl(h\varphi_{1,\beta}^{\SO_{2n+2}^{\ur}}\bigr)\\
&=
\sum_{x\text{: as above}}
\tilde{\chi}_{\xi,0,a,\zeta}^{\SO_{2n+2}^{\ur}}\bigl(xhx^{-1} \cdot x\varphi_{1,\beta}^{\SO_{2n+2}^{\ur}}x^{-1}\bigr)\\
&=
\zeta
\sum_{\begin{subarray}{c}x_{1},\ldots,x_{n}\in k^{\times}\\ x_{1}^{2}=-a\beta\end{subarray}}
\psi\biggl(
\frac{x_{1}}{x_{2}}h_{1}+\cdots+\frac{x_{n-1}}{x_{n}}h_{n-1}+\Tr(x_{n}h_{n})+a\frac{h_{0}}{x_{1}x_{2}}
\biggr)\\
&=
\zeta
\sum_{\begin{subarray}{c}x_{1},\ldots,x_{n}\in k^{\times}\\ x_{1}^{2}=-a\beta\end{subarray}}
\psi\biggl(
\frac{x_{1}h_{1}+ax_{1}^{-1}h_{0}}{x_{2}}+\frac{x_{2}}{x_{3}}h_{2}+\cdots+\frac{x_{n-1}}{x_{n}}h_{n-1}+x_{n}\Tr(h_{n})
\biggr)\\
&=
\zeta\cdot
\Kl^{n}_{\pm(\sqrt{-\beta a}h_{1}+a{\sqrt{-\beta a}\,}^{-1}h_{0})h_{2}\cdots h_{n-1}\Tr(h_{n})}(\psi)\\
&=
\zeta\cdot
\Kl^{n}_{\pm{\sqrt{-\beta a}\,}^{-1}a(-\beta h_{1}+h_{0})h_{2}\cdots h_{n-1}\Tr(h_{n})}(\psi).
\end{align*}
\end{proof}

\section{Arthur's local classification theorem}\label{sec:Arthur}
In this section, we recall Arthur's theory of the endoscopic classification of representations of classical groups, especially in the cases of quasi-split symplectic or even special orthogonal groups over $F$.

To state Arthur's theorem on the classification of representations, we define some notations.
Let $\G$ be either a quasi-split symplectic or even special orthogonal group over $F$.
When $\G$ is an even special orthogonal group $\SO_{2n}^{\mu}$ for a quadratic character $\mu$ of $F^{\times}$, any element $w\in\mathrm{O}_{2n}^{\mu}(F)\smallsetminus\SO_{2n}^{\mu}(F)$ defines an automorphism $\Int(w)$ of $\G$.
Then its image in the group $\Out(\G)$ of outer automorphisms of $\G$ (i.e., the quotient of the group $\mathrm{Aut}(\G)$ of automorphisms of $\G$ by the group $\mathrm{Inn}(\G)$ of inner automorphisms of $\G$) is determined independently of the choice of $w\in\mathrm{O}_{2n}^{\mu}(F)\smallsetminus\SO_{2n}^{\mu}(F)$.
We put $\Sigma(\G)$ to be the subgroup of $\Out(\G)$ of order $2$ generated by the element $\Int(w)$.
We also define $\Sigma(\widehat{\G})$ in a similar way for the dual group $\widehat{\G}$ when $\G$ is, hence also $\widehat{\G}$ is, an even special orthogonal group.
For convenience, we also put $\Sigma(\G)=\{\mathrm{id}\}$ (or $\Sigma(\widehat{\G})=\{\mathrm{id}\}$) when $\G$ is not even special orthogonal.

We denote the set of equivalence classes of irreducible smooth (resp.\ tempered) representations of $G$ by $\Pi(\G)$ (resp.\ $\Pi_{\mathrm{temp}}(\G)$).
Then the group $\Sigma(\mathbf{G})$ acts on these sets.
We write $\widetilde{\Pi}(\G)$ (resp.\ $\widetilde{\Pi}_{\mathrm{temp}}(\G)$) for the set of $\Sigma(\mathbf{G})$-orbits in $\Pi(\G)$ (resp.\ $\Pi_{\mathrm{temp}}(\G)$).
In the case where $\G=\SO_{2n}^{\mu}$ for a quadratic character $\mu$ of $F^{\times}$, for each irreducible representation $\pi$ of $G$, its $\Sigma(\G)$-orbit consists of at most two elements, namely $\pi$ and $\pi^{w}$.
Here $\pi^{w}$ is the $w$-twist of $\pi$ defined by
\[
\pi^{w}(g):=\pi(wgw^{-1}).
\]
For the $\Sigma(\G)$-orbit $\tilde{\pi}$ of $\pi\in\Pi(\G)$, we put
\[
\Theta_{\tilde{\pi}}
:=
\frac{1}{|\tilde{\pi}|}\cdot\sum_{\pi\in\tilde{\pi}}\Theta_{\pi},
\]
where $\Theta_{\pi}$ is the character of $\pi$.

We denote the set of $\widehat{\G}$-conjugacy classes of $L$-parameters (resp.\ tempered $L$-parameters) of $\G$ by $\Phi(\G)$ (resp.\ $\Phi_{\mathrm{temp}}(\G)$), and their $\Sigma(\widehat{\mathbf{G}})$-orbits by $\tPhi(\G)$ (resp.\ $\tPhi_{\mathrm{temp}}(\G)$).
See \cite[Section 1.3]{MR3135650} for the definition of a (tempered) $L$-parameter.
For $\phi\in\tPhi_{\mathrm{temp}}(\G)$, we set
\begin{align*}
S^{\G}_{\phi}&:= \Cent_{\widehat{\G}}\bigl(\mathrm{Im}(\phi)\bigr)\mathrm{, and}\\
\mathcal{S}^{\G}_{\phi}&:= S^{\G}_{\phi}\big/(S^{\G}_{\phi})^{0}Z_{\widehat{\G}}.
\end{align*}
Here we remark that, in order to define these groups, we implicitly fix a representative of $\phi$, namely a homomorphism from $W_{F}\times\SL_{2}(\C)$ to ${}^{L}\G$.

We write $\G^{\rs}$ (resp.\ $\G^{\srs}$) for the set of regular semisimple (resp.\ strongly regular semisimple) elements of $\G$.

\begin{thm}[{\cite[Theorems 1.5.1 and 2.2.1]{MR3135650}}]\label{thm:Arthur}
We have a partition of $\widetilde{\Pi}_{\mathrm{temp}}(\G)$ into the finite subsets (which are called the $L$-packets) which  are parametrized by $L$-parameters:
\[
\widetilde{\Pi}_{\mathrm{temp}}(\G) = \bigsqcup_{\phi\in\tPhi_{\mathrm{temp}}(\G)} \widetilde{\Pi}^{\G}_{\phi}.
\]
Here each $\widetilde{\Pi}_{\phi}^{\G}$ satisfies the following properties:
\begin{description}
\item[(1) pairing with the group $\mathcal{S}^{\G}_{\phi}$] 
We fix a Whittaker datum $\mathfrak{w}_{\G}$ of $\G$.
Let $\phi\in\tPhi_{\mathrm{temp}}(\G)$.
Then there exists a bijection from $\widetilde{\Pi}^{\G}_{\phi}$ to the set of irreducible characters of $\mathcal{S}^{\G}_{\phi}$:
 \[
 \widetilde{\Pi}^{\G}_{\phi} \cong \widehat{\mathcal{S}^{\G}_{\phi}},\quad \tilde{\pi} \mapsto \lan-,\tilde{\pi}\ran
 \]
 depending on the Whittaker datum $\mathfrak{w}_{\G}$ of $\G$.
\item[(2) stability] 
Let $\phi\in\tPhi_{\mathrm{temp}}(\G)$.
Then the sum of the characters of $\Sigma(\mathbf{G})$-orbits belonging to $\widetilde{\Pi}^{\G}_{\phi}$ is stable.
Namely, for every $g, g'\in G^{\srs}$ which are stably conjugate, we have
 \[
 \sum_{\tilde{\pi}\in\widetilde{\Pi}^{\G}_{\phi}} \Theta_{\tilde{\pi}}(g)=\sum_{\tilde{\pi}\in\widetilde{\Pi}^{\G}_{\phi}}  \Theta_{\tilde{\pi}}(g').
 \]
\end{description}
\end{thm}

We note that we have to take care the difference between $\widetilde{\Pi}^{\G}_{\phi}$ and $\Pi^{\G}_{\phi}$ only in the case of even special orthogonal groups.
Thus, in the case of symplectic groups, we abbreviate the tilde from the notations.

Let $\G_{\ad}$ be the adjoint group of $\G$, namely the quotient of $\G$ by its center $\bfZ_{\G}$:
\[
\G_{\ad}:=\G/\bfZ_{\G}.
\]
Then, as a consequence of the stability of $L$-packets (Theorem \ref{thm:Arthur} (2)), each $L$-packet satisfies the following property:
\begin{cor}\label{cor:adj}
Let $\phi$ be a tempered $L$-parameter of $\G$ and $\widetilde{\Pi}^{\G}_{\phi}$ its $L$-packet.
Then $\widetilde{\Pi}^{\G}_{\phi}$ consists of $G_{\ad}$-orbits.
\end{cor}

\begin{proof}
Let $x\in G_{\ad}$.
We have to show that $(\widetilde{\Pi}^{\G}_{\phi})^{x}:=\{\tilde{\pi}^{x}\mid \tilde{\pi}\in\widetilde{\Pi}^{\G}_{\phi}\}$ is equal to $\widetilde{\Pi}^{\G}_{\phi}$.
Here $\tilde{\pi}^{x}$ is the $x$-twist of $\tilde{\pi}$, that is, the set $\{\pi^{x}\mid\pi\in\tilde{\pi}\}$.
We remark that we can think of $\tilde{\pi}^{x}$ as the $\Sigma(\G)$-orbit of $\pi^{x}$ for any representative $\pi$ of $\tilde{\pi}$, since the $x$-twist commutes with the $\Sigma(\G)$-twist.

Since the natural map $\mathbf{G}(\overline{F})\ra \mathbf{G}_{\ad}(\overline{F})$ is surjective, we can take a lift $\ol{x}\in \mathbf{G}(\ol{F})$ of $x$.
Therefore, for any $g\in G^{\srs}$, $g$ and $xgx^{-1}=\ol{x}g\ol{x}^{-1}$ are stably conjugate (note that the stable conjugacy is equivalent to the conjugacy in $\mathbf{G}(\ol{F})$ for strongly regular semisimple elements).
By the stability of the $L$-packets (Theorem \ref{thm:Arthur} (1)), we have
\[
\sum_{\tilde{\pi}^{x}\in(\widetilde{\Pi}^{\G}_{\phi})^{x}} \Theta_{\tilde{\pi}^{x}}(g)
= \sum_{\tilde{\pi}\in\widetilde{\Pi}^{\G}_{\phi}} \Theta_{\tilde{\pi}}(xgx^{-1})
= \sum_{\tilde{\pi}\in\widetilde{\Pi}^{\G}_{\phi}} \Theta_{\tilde{\pi}}(g)
\]
for any $g\in G^{\srs}$.
Since the set $\mathbf{G}^{\srs}$ of strongly regular semisimple elements is Zariski-open in the set $\mathbf{G}^{\rs}$ of regular semisimple elements, the measure of the complement of $G^{\srs}$ in $G^{\rs}$ is zero.
Thus the above equality holds on $G^{\rs}$.
Therefore we have $(\widetilde{\Pi}^{\G}_{\phi})^{x}=\widetilde{\Pi}^{\G}_{\phi}$ by linear independence of the characters of representations.
\end{proof}

The $L$-packets satisfy a lot of expected properties other than the above properties.
We next recall them.
First one is the following property known as the \textit{generic packet conjecture}:

\begin{thm}[{\cite[Proposition 8.3.2]{MR3135650}}, \cite{MR3592596}, and \cite{Atobe:2015aa}]\label{thm:GPC}
Let $\phi\in\tPhi_{\mathrm{temp}}(\G)$.
Then any representative of the $\Sigma(\mathbf{G})$-orbit in the packet $\widetilde{\Pi}^{\G}_{\phi}$ corresponding to the trivial character of $\mathcal{S}^{\G}_{\phi}$ under the bijection in Theorem \ref{thm:Arthur} $(1)$ is $\mathfrak{w}_{\G}$-generic, and any representative of any other $\Sigma(\mathbf{G})$-orbit is not $\mathfrak{w}_{\G}$-generic.
\end{thm}

\begin{rem}
We emphasize that the bijection in Theorem \ref{thm:Arthur} depends on the choice of a Whittaker datum of $\mathbf{G}$.
See Section \ref{subsec:Whittaker} for what kind of Whittaker data are taken in this paper.
\end{rem}

We next recall the endoscopic character relation.
We first consider the case of twisted endoscopy.
Let $\mathbf{G}$ be either $\Sp_{2n}$ or $\SO_{2n}^{\mu}$ over $F$, where $\mu$ is a quadratic character of $F^{\times}$.
Then the dual group of $\mathbf{G}$ is a special orthogonal group $\SO_{N}(\C)$, where
\[
N=
\begin{cases}
2n+1 & \text{if } \mathbf{G}=\Sp_{2n},\\
2n & \text{if } \mathbf{G}=\SO_{2n}^{\mu}.
\end{cases}
\]
We note that, in the case of $\mathbf{G}=\SO_{2n}^{\mu}$, the action of $W_{F}$ on the dual group $\widehat{\G}=\SO_{2n}(\C)$ is described as follows:
\begin{itemize}
\item
when $\sigma \in W_{E_{\mu}}$, then $\sigma$ acts on $\widehat{\G}$ trivially, and
\item
when $\sigma \in W_{F}\smallsetminus W_{E_{\mu}}$, then $\sigma(g)=wgw^{-1}$ for $g\in \SO_{2n}(\C)$, 
where $w$ is a fixed element of $\rmO_{2n}(\C)\smallsetminus\SO_{2n}(\C)$ given by
\[
w:=
\begin{pmatrix}
 I_{n-1}&&&\\
 &0&1&\\
 &1&0&\\
 &&&I_{n-1}
\end{pmatrix}.
\]
\end{itemize}
We consider the following $L$-embedding from the $L$-group ${}^{L}\G$ of $\G$ to that of a general linear group:
\begin{description}
\item[The case where $\mathbf{G}=\Sp_{2n}$]
\begin{align*}
{}^{L}\G=\SO_{2n+1}(\C)\times W_{F}
&\hookrightarrow \GL_{2n+1}(\C)\times W_{F}={}^{L}\!\GL_{2n+1}\\
(h,\sigma)&\mapsto (h,\sigma).
\end{align*}
\item[The case where $\mathbf{G}=\SO_{2n}^{\mu}$]
\begin{align*}
{}^{L}\G=\SO_{2n}(\C)\rtimes W_{F}
&\hookrightarrow \GL_{2n}(\C)\times W_{F}={}^{L}\!\GL_{2n}\\
(h,\sigma)&\mapsto
\begin{cases}
 (h,\sigma)& \text{if } \sigma\in W_{E_{\mu}},\\
 (hw,\sigma)& \text{if } \sigma\notin W_{E_{\mu}}.
\end{cases}
\end{align*}
\end{description}
Then $\mathbf{G}$ is an endoscopic group for $(\GL_{N}, \theta)$ with respect to these $L$-embeddings.
Here recall that $\theta$ is the outer automorphism of $\GL_{N}$ over $F$ defined in Section \ref{subsec:ssc-GL}.

For an $L$-parameter of $\G$, by composing it with the above $L$-embedding from ${}^{L}\G$ to ${}^{L}\!\GL_{N}$, we get an $L$-parameter of $\GL_{N}$.
Note that this operation gives an injection 
\[
\tPhi(\G)\hookrightarrow\Phi(\GL_{N}).
\]
For $\phi\in\tPhi_{\mathrm{temp}}(\G)$, we denote the irreducible tempered representation of $\GL_{N}(F)$ corresponding to the lifted $L$-parameter by $\pi^{\GL_{N}}_{\phi}$.
Note that this representation is self-dual, hence $\theta$-stable, since the corresponding $L$-parameter comes from an $L$-parameter of $\G$.
Now we can state the twisted endoscopic character relation for $\widetilde{\Pi}^{\G}_{\phi}$ and $\pi^{\GL_{N}}_{\phi}$:

\begin{thm}[{Twisted endoscopic character relation for $\GL_{N}$ and $\mathbf{G}$, \cite[Theorem 2.2.1]{MR3135650}}]\label{thm:TECR}
We fix a $\theta$-stable Whittaker datum $\mathfrak{w}_{\GL_{N}}$ of $\GL_{N}$.
For $\phi\in\tPhi_{\mathrm{temp}}(\G)$, by Theorem \ref{thm:Arthur}, we have the finite subset $\widetilde{\Pi}^{\G}_{\phi}\subset\widetilde{\Pi}_{\mathrm{temp}}(\G)$.
Then the following equality of characters holds for every strongly $\theta$-regular $\theta$-semisimple element $g\in \GL_{N}(F)$:
\[
\Theta_{\phi,\theta}^{\GL_{N}}(g)
= \sum_{h\leftrightarrow g/\sim}\frac{D_{\G}(h)^{2}}{D_{\GL_{N},\theta}(g)^{2}} \Delta_{\G,\GL_{N}}(h,g)
\sum_{\tilde{\pi}\in\widetilde{\Pi}_{\phi}}\Theta_{\tilde{\pi}}(h).
\]
Here, 
\begin{itemize}
\item
$h$ in the sum runs over the stable conjugacy classes of strongly regular semisimple elements of $G$ which are norms of $g$ in the sense of Kottwitz--Shelstad (see \cite[Section 3.3]{MR1687096}), 
\item $D_{\G}(h)$ (resp.\ $D_{\GL_{N},\theta}(g)$) is the Weyl discriminant for $h$ (resp.\ the $\theta$-twisted Weyl discriminant for $g$, see \cite[Section 4.5]{MR1687096}), 
\item
$\Delta_{\G,\GL_{N}}$ is Kottwitz--Shelstad's transfer factor for $\G$ and $\GL_{N}$ (see \cite[Section 5.3]{MR1687096} for the definition, note that this depends on the choice of the $\theta$-stable Whittaker datum $\mathfrak{w}_{\GL_{N}}$ of $\GL_{N}$), and
\item
$\Theta_{\phi,\theta}^{\GL_{N}}$ is the $\theta$-twisted character of $\pi_{\phi}^{\GL_{N}}$ normalized via the Whittaker datum $\mathfrak{w}_{\GL_{N}}$.
\end{itemize}\end{thm}

Next we consider the case of standard endoscopy.
We take $\phi\in\tPhi_{\mathrm{temp}}(\G)$.
Then, as in the paragraph before the previous theorem, we may regard an element of $\tPhi(\G)$ as an $N$-dimensional representation of $W_{F}\times\SL_{2}(\C)$.
We decompose $\phi$ into a direct sum of irreducible representations of $W_{F}\times\SL_{2}(\C)$, and write
\[
\phi=\bigoplus_{i=1}^{r} l_{i} \phi_{i} =\bigoplus_{i=1}^{r} l_{i}(\phi_{\mathrm{cusp}, i}\boxtimes\nu_{a_{i}}),
\]
where $\phi_{\mathrm{cusp},i}$ is an irreducible representation of $W_{F}$, $\nu_{a_{i}}$ the $(a_{i}-1)$-th symmetric power of the standard representation of $\SL_{2}(\C)$, and $l_{i}\in\Z_{>0}$ the multiplicity.
We say that $\phi$ is discrete if $\phi_{i}$ is self-dual and $l_{i}=1$ for every $i$.
We can capture the discrete series representations of $G$ by discrete $L$-parameters.
Namely, we have the following decomposition:
\[
\widetilde{\Pi}_{2}(\G)=\bigsqcup_{\phi\in\tPhi_{2}(\G)}\widetilde{\Pi}^{\G}_{\phi},
\]
where $\widetilde{\Pi}_{2}(\G)$ is the set of $\mathrm{Out}(\G)$-orbits of equivalence classes of irreducible discrete series representations of $G$, $\tPhi_{2}(\G)$ the set of $\mathrm{Out}(\G)$-orbits of discrete $L$-parameters of $\G$, and $\widetilde{\Pi}^{\G}_{\phi}$ the $L$-packet for $\phi$ in the sense of Theorem \ref{thm:Arthur}.

Let $\phi\in\tPhi_{2}(\G)$ and we take a decomposition $\phi=\bigoplus_{i=1}^{r}\phi_{i}$ as above.
Then the group $\mathcal{S}^{\G}_{\phi}$ is a finite abelian $2$-group given by
\[
\mathcal{S}^{\G}_{\phi} \cong \Big\{(s_{i})_{i}\in\prod_{i=1}^{r}\mu_{2} \, \Big\vert \, \prod_{i=1}^{r} s_{i}^{\dim \phi_{i}}=1\Big\}
\Big/
Z_{\widehat{\G}}.
\]
Here, each $s_{i}\in\mu_{2}=\{\pm1\}$ corresponds to the $s_{i}$-multiplication on $\phi_{i}$ and the product relation is equivalent to the condition that the element of $\GL_{N}(\C)$ corresponding to $(s_{i})_{i}$ belongs to $\SO_{N}(\C)$.
Note that the center $Z_{\widehat{\G}}$ is given by
\[
\begin{cases}
\langle(-1,\ldots,-1)\rangle&\text{if $N$ is even,}\\
1&\text{if $N$ is odd.}
\end{cases}
\]

Let $s\in \mathcal{S}^{\G}_{\phi}$.
From this, we can define an endoscopic group $\mathbf{H}$ of $\mathbf{G}$ and regard $\phi$ as an $L$-parameter of $\mathbf{H}$ as follows.
First, by conjugation, we may assume 
\[
s=\diag(\underbrace{-1,\ldots,-1}_{d^{-}/2},\underbrace{1,\ldots,1}_{d^{+}},\underbrace{-1,\ldots,-1}_{d^{-}/2}) \in \widehat{\G}=\SO_{N}(\C) \text{ and}
\] 
\[
\mathrm{Im}(\phi) \subset \Cent_{\widehat{\G}}(s)\rtimes W_{F}
\]
 for a non-negative even integer $d^{-}$ and a positive integer $d^{+}$.
Note that the group $\Cent_{\widehat{\G}}(s)$ is given by
\[
\bigl(\rmO_{d^{-}}(\C)\times\rmO_{d^{+}}(\C)\bigr)^{\det=1}
:=\{(g,h)\in\rmO_{d^{-}}(\C)\times\rmO_{d^{+}}(\C) \mid \det(g)\det(h)=1\}.
\]
Here, we regard $(g,h)\in\rmO_{d^{-}}(\C)\times\rmO_{d^{+}}(\C) $ as a matrix
\[
\begin{pmatrix}
 g_{11}&0&g_{12}\\
 0&h&0\\
 g_{21}&0&g_{22}
\end{pmatrix}
\in\mathrm{SO}_{N}(\C)=\widehat{\G},
\]
where
\[
g=
\begin{pmatrix}
 g_{11}&g_{12}\\
 g_{21}&g_{22}
\end{pmatrix}
\in\rmO_{d^{-}}(\C).
\]
Then the group $\Cent_{\widehat{\G}}(s)^{0}$ can be regarded as the image of the dual group of $\mathbf{H}$ under an $L$-embedding from the $L$-group ${}^{L}\mathbf{H}$ for the following group $\mathbf{H}$:
\[
\mathbf{H}:=
\begin{cases}
\SO_{d^{-}}^{\mu'}\times\Sp_{d^{+}-1} & \text{if } \mathbf{G}=\Sp_{2n},\\
\SO_{d^{-}}^{\mu'}\times\SO_{d^{+}}^{\mu\mu'} & \text{if } \mathbf{G}=\SO_{2n}^{\mu}.
\end{cases}
\]
Here $\mu'$ is a quadratic character of $F^{\times}$ determined by $\phi$ and $s$ as follows:
\begin{align*}
W_{F}\xrightarrow{\phi|_{W_{F}}} \bigl(\rmO_{d^{-}}(\C)\times\rmO_{d^{+}}(\C)\bigr)^{\det=1}\rtimes W_{F} &\rightarrow\{\pm1\}\\
(g,h)\rtimes\sigma&\mapsto\det(g),
\end{align*}
and the $L$-embedding from ${}^{L}\mathbf{H}$ to ${}^{L}\G$ is given by the following:
\begin{description}
\item[symplectic case]
\begin{align*}
{}^{L}\mathbf{H}=\bigl(\SO_{d^{-}}(\C)\times\SO_{d^{+}}(\C)\bigr)\rtimes W_{F}
&\hookrightarrow \SO_{2n+1}(\C)\times W_{F}={}^{L}\G\\
(g,h)\rtimes\sigma&\mapsto
\begin{cases}
(g,h)\rtimes\sigma& \text{if } \sigma\in W_{E_{\mu'}}\\
(g,h)S\rtimes\sigma& \text{if } \sigma\notin W_{E_{\mu'}},
\end{cases}
\end{align*}
where
\[
S:=
\begin{pmatrix}
 I_{d^{-}/2-1}&&&&\\
 &&&1&\\
 &&-I_{d^{+}}&&\\
 &1&&&\\
 &&&&I_{d^{-}/2-1}
\end{pmatrix}.
\]

\item[even orthogonal case]
\begin{align*}
{}^{L}\mathbf{H}=\bigl(\SO_{d^{-}}(\C)\times\SO_{d^{+}}(\C)\bigr)\rtimes W_{F}
&\hookrightarrow \SO_{2n}(\C)\rtimes W_{F}={}^{L}\G\\
(g,h)\rtimes\sigma&\mapsto
\begin{cases}
(g,h)\rtimes\sigma& \text{if } \sigma\in W_{E_{\mu'}}\\
(g,h)S\rtimes\sigma& \text{if } \sigma\notin W_{E_{\mu'}},
\end{cases}
\end{align*}
where $S:=I_{2n+1}$ if $\mu'=\mathbbm{1}$ and 
\[
S:=
\begin{pmatrix}
 I_{d^{-}/2-1}&&&&&&&\\
 &&&&&&1&\\
 &&I_{d^{+}/2-1}&&&&&\\
 &&&&1&&&\\ 
 &&&1&&&&\\
 &&&&&I_{d^{+}/2-1}&&\\
 &1&&&&&&\\
 &&&&&&&&I_{d^{-}/2-1}
\end{pmatrix}
\]
if $\mu'\neq\mathbbm{1}$.
\end{description}

%

Since $\phi$ factors through this $L$-embedding, we can regard $\phi$ as an $L$-parameter of $\mathbf{H}$.
Here we remark that in both cases this $L$-embedding induces an injection 
\[
\tPhi_{2}(\mathbf{H})\hookrightarrow\tPhi_{2}(\G).
\]

Now we can state the endoscopic character relation for $\G$ and $\mathbf{H}$.

\begin{thm}[{Endoscopic character relation for $\G$ and $\mathbf{H}$, \cite[Theorem 2.2.1]{MR3135650}}]\label{thm:SECR}
Let $\phi\in\tPhi_{2}(\G)$ and $s\in \mathcal{S}^{\G}_{\phi}$.
Let $\mathbf{H}$ be the endoscopic group of $\mathbf{G}$ defined in the above manner, and we regard $\phi$ as an $L$-parameter of $\mathbf{H}$.
Then, by Theorem \ref{thm:Arthur}, we have the finite subsets $\widetilde{\Pi}^{\G}_{\phi}\subset\widetilde{\Pi}_{2}(\G)$ and $\widetilde{\Pi}^{\mathbf{H}}_{\phi}\subset\widetilde{\Pi}_{2}(\mathbf{H})$.
In this situation, we have the following equality of characters at every strongly regular semisimple element $g\in G$:
\[
\sum_{\tilde{\pi}\in\widetilde{\Pi}^{\G}_{\phi}} \lan s,\tilde{\pi}\ran \Theta_{\tilde{\pi}}(g)
= \sum_{h\leftrightarrow g/\sim}\frac{D_{\mathbf{H}}(h)^{2}}{D_{\G}(g)^{2}} \Delta_{\mathbf{H},\G}(h,g) 
\sum_{\tilde{\pi}_{H}\in\widetilde{\Pi}^{\mathbf{H}}_{\phi}}\Theta_{\tilde{\pi}_{H}}(h).
\]
Here, 
\begin{itemize}
\item
$h$ in the sum runs over the stable conjugacy classes of strongly regular semisimple elements of $H$ which are norms of $g$, 
\item
$\lan-,-\ran$ is the pairing in Theorem \ref{thm:Arthur} $(2)$, and 
\item
$\Delta_{\mathbf{H},\G}$ is Kottwitz--Shelstad's transfer factor for $\mathbf{H}$ and $\G$ (which depends on the Whittaker datum $\mathfrak{w}_{\G}$ of $\mathbf{\G}$).
\end{itemize}
\end{thm}

\begin{rem}
\begin{enumerate}
\item
In \cite[Theorem 2.2.1]{MR3135650}, the above relations (Theorems \ref{thm:TECR} and \ref{thm:SECR}) are stated in terms of the distribution characters.
However, we can rewrite them as equalities of characters (i.e., functions on sets of strongly regular semisimple elements of $H$ and $G$) by using Weyl's integration formula (e.g., see Section 5 in \cite{MR3067291}).
\item
In the above explanation of the endoscopic character relation for standard endoscopy, we focus on the discrete $L$-parameters in order to make a description of the group $\mathcal{S}^{\G}_{\phi}$ and endoscopic groups simple.
However, more generally, we can attach an endoscopic group $\mathbf{H}$ (more precisely, a set of endoscopic data) to each element of the group $\mathcal{S}^{\G}_{\phi}$ for a tempered $L$-parameter $\phi$, and the endoscopic character relation for them holds.
See \cite[36 page]{MR3135650} for details.
\item
Later (in Sections \ref{sec:SO-ram} and \ref{sec:SO-ur}), we compute the above endoscopic character relations (Theorems \ref{thm:TECR} and \ref{thm:SECR}) for simple supercuspidal representations.
In those cases, the Whittaker normalization of the $\theta$-twisted characters of $\GL_{N}$ coincides with the normalization used in Section \ref{sec:char} (see Remark \ref{rem:normalization}).
Therefore we can combine the results in Section \ref{sec:char} with Theorems \ref{thm:TECR} and \ref{thm:SECR} without any scalar multiple.
\end{enumerate}
\end{rem}

We finally refer to a general fact, which was established by M{\oe}glin and Xu, on a parametrization of supercuspidal representations in $L$-packets.
For simplicity, in the rest of this section, we consider only the case where $\G:=\Sp_{2n}$.
See \cite[Section 3]{MR3713922} for the full explanation in the cases of orthogonal groups (we will use M{\oe}glin--Xu's result only for the symplectic groups in this paper).
Let $\phi$ be an element of $\Phi_{2}(\G)$.
We decompose $\phi$ into a multiplicity-free direct sum of irreducible representations of $W_{F}\times\SL_{2}(\C)$:
\[
\phi=\bigoplus_{i=1}^{r}\phi_{i}=\bigoplus_{i=1}^{r}\phi_{\mathrm{cusp}, i}\boxtimes\nu_{a_{i}}.
\]
Via this decomposition, we define a finite set $\Jord(\phi)$ for $\phi$ as follows:
\[
\Jord(\phi)
:=
\{(\rho_{i},a_{i})\mid1\leq i\leq r\}
\]
where $\rho_{i}$ is the irreducible supercuspidal representation of $\GL_{\dim\phi_{\mathrm{cusp},i}}(F)$ which corresponds to $\phi_{\mathrm{cusp},i}$ under the local Langlands correspondence for general linear groups.

Recall that we have an identification 
\[
\mathcal{S}^{\G}_{\phi} \cong \Bigl\{(s_{i})_{i}\in \prod_{i=1}^{r}\mu_{2} \, \Big\vert \, \prod_{i=1}^{r} s_{i}^{\dim \phi_{i}}=1\Bigr\}
\]
(the center $Z_{\widehat{\G}}$ is trivial since now we are assuming that $\G=\Sp_{2n}$).
The order of this group is given by $2^{r-1}$ (since $\dim\phi_{i}$ is odd for at least one $i$, the condition $\prod_{i=1}^{r} s_{i}^{\dim \phi_{i}}=1$ is non-trivial).
Thus, by noting that $\prod_{i=1}^{r} (-1)^{\dim \phi_{i}}=(-1)^{2n+1}=-1$, we have a canonical isomorphism
\[
\Bigl\{(s_{i})_{i}\in \prod_{i=1}^{r}\mu_{2} \, \Big\vert \, \prod_{i=1}^{r} s_{i}^{\dim \phi_{i}}=1\Bigr\}
\xrightarrow{\cong}
\Bigl(\prod_{i=1}^{r}\mu_{2}\Bigr)\Big/\lan(-1,\ldots,-1)\ran.
\]
Hence we have an isomorphism between the character groups of these groups:
\[
\widehat{\mathcal{S}^{\G}_{\phi}}
\cong
\Bigl\{(\varepsilon_{i})_{i}\in \prod_{i=1}^{r}\widehat{\mu_{2}} \, \Big\vert \, \prod_{i=1}^{r} \varepsilon_{i}=\mathbbm{1}\Bigr\}.
\]
Finally, by using an identification between $\widehat{\mu_{2}}$ and $\mu_{2}$ given by
\[
\widehat{\mu_{2}}\cong\mu_{2}\colon
\begin{cases}
\mathbbm{1}\mapsto 1,\\
\text{(nontrivial character)} \mapsto -1,
\end{cases}
\]
we regard $\varepsilon=(\varepsilon_{i})_{i}\in\widehat{\mathcal{S}^{\G}_{\phi}}$ as a $\mu_{2}$-valued function on $\Jord(\phi)$ by setting
\[
\varepsilon(\rho_{i},a_{i}):=\varepsilon_{i} \quad\text{ for } (\rho_{i},a_{i})\in\Jord(\phi). 
\]

\begin{thm}[{\cite[Theorem 3.3]{MR3713922} and \cite[Theorem 1.5.1]{MR2767522}}]\label{thm:Jord}
Let $\phi\in\Phi_{2}(\G)$ and $\Pi^{\G}_{\phi}$ the $L$-packet of $\phi$.
We assume that the $L$-packet $\Pi^{\G}_{\phi}$ contains a supercuspidal representation of $G$.
\begin{enumerate}
 \item 
 The set $\mathrm{Jord}(\phi)$ satisfies the following condition:
 \[
 (\rho, a)\in\Jord(\phi) \text{ and } a-2>0 \implies (\rho, a-2)\in\Jord(\phi).
 \]
 \item
 Let $\pi\in\Pi^{\G}_{\phi}$ and $\varepsilon\in\widehat{\mathcal{S}^{\G}_{\phi}}$ the element which corresponds to $\pi$ via the isomorphism $\Pi_{\phi}\cong\widehat{\mathcal{S}^{\G}_{\phi}}$ in Theorem \ref{thm:Arthur}.
 Then $\pi$ is supercuspidal if and only if the following two conditions hold:
 \begin{align*}
  \text{$(\mathrm{i})$}\quad &(\rho, a), (\rho,a-2)\in\Jord(\phi) \implies \varepsilon(\rho,a)\varepsilon(\rho,a-2)=-1, \text{ and}\\
  \text{$(\mathrm{ii})$}\quad &(\rho, 2)\in\Jord(\phi) \implies \varepsilon(\rho, 2)=-1.
 \end{align*}
 \end{enumerate}
\end{thm}

The following is a key consequence of this fact which will be needed later.

\begin{cor}\label{cor:gensc}
Let $\phi\in\Phi_{2}(\G)$ and $\Pi^{\G}_{\phi}$ the $L$-packet of $\phi$.
We assume that the $L$-packet $\Pi^{\G}_{\phi}$ contains a $\mathfrak{w}_{\G}$-generic supercuspidal representation of $G$.
Then every element of $\Pi^{\G}_{\phi}$ is a supercuspidal representation.
\end{cor}

\begin{proof}
Let $\pi_{0}\in\Pi^{\G}_{\phi}$ be a $\mathfrak{w}_{\mathbf{G}}$-generic supercuspidal representation of $G$.
Then, by Theorem \ref{thm:GPC}, $\pi_{0}$ corresponds to the trivial character $\varepsilon_{0}\in\widehat{\mathcal{S}^{\G}_{\phi}}$.
On the other hand, since $\pi_{0}$ is supercuspidal, $\varepsilon_{0}$ should satisfy the conditions (i) and (ii) in Theorem \ref{thm:Jord} (2).
Therefore $\Jord(\phi)$ does not contain any pair $(\rho,a)$ whose $a$ is greater than $1$.
Then the assumptions of (i) and (ii) are always satisfied for every $\varepsilon\in\widehat{\mathcal{S}^{\G}_{\phi}}$.
Again by Theorem \ref{thm:Jord} (2), every representation in $\Pi^{\G}_{\phi}$ is supercuspidal.
\end{proof}

\section{Simple supercuspidal $L$-packet of $\Sp_{2n}$}\label{sec:Sp}
In this section, we let $\mathbf{G}$ be the symplectic group $\Sp_{2n}$ over $F$.

Recall that, by Arthur's local classification theorem (Theorem \ref{thm:Arthur}), the set $\Pi(\Sp_{2n})$ of equivalence classes of irreducible smooth representations of $\Sp_{2n}(F)$ is partitioned into a disjoint union of the $L$-packets.
For $(\xi,0,a)\in\SSC(\Sp_{2n})$, we consider the corresponding simple supercuspidal representation $\pi_{\xi,0,a}^{\G}$ of $\Sp_{2n}(F)$ defined in Section \ref{subsec:ssc-Sp}.
Let $\phi\in\Phi_{\mathrm{temp}}(\G)$ be the $L$-parameter of $\pi_{\xi,0,a}^{\G}$.

The goal of this section is to prove the following theorem:

\begin{thm}\label{thm:packetSp}
The $L$-packet $\Pi_{\phi}^{\G}$ of $\phi$ consists of two elements $\pi^{\G}_{\xi,0,a}$ and $\pi^{\G}_{\xi,1,a\epsilon^{-1}}$.
\end{thm}

\subsection{Adjoint orbits of simple supercuspidal representations of $\Sp_{2n}$}\label{subsec:adjoint-Sp}
In this subsection, we show the following properties:
\begin{itemize}
\item
every member of $\Pi^{\G}_{\phi}$ is supercuspidal, and
\item
$\Pi^{\G}_{\phi}$ contains no more simple supercuspidal representations other than $\pi^{\G}_{\xi,1,a\epsilon^{-1}}$.
\end{itemize}
To show this, we start from determining the $G_{\ad}$-orbits of simple supercuspidal representations.

\begin{prop}\label{prop:sscadj}
The $G_{\ad}$-orbit of the simple supercuspidal representation $\pi^{\G}_{\xi,0,a}$ is given by
\[
\{\pi^{\G}_{\xi,0,a}, \pi^{\G}_{\xi,1,a\epsilon^{-1}}\}.
\]
\end{prop}

\begin{rem}
This proposition is stated in \cite[Section 9.5 Remark]{MR2730575} without proof.
For the sake of completeness, we explain a proof.
\end{rem}

\begin{lem}\label{lem:sscadj1}
We put $t:=\diag(\sqrt{\epsilon}, \ldots, \sqrt{\epsilon}, {\sqrt{\epsilon}\,}^{-1},\ldots, {\sqrt{\epsilon}\,}^{-1}) \in \mathbf{G}(\ol{F})$.
Here $\sqrt{\epsilon}\in\ol{F}^{\times}$ is a square root of $\epsilon$.
Then the image of this element in $\mathbf{G}_{\ad}(\ol{F})$ under the natural projection $\G(\ol{F})\twoheadrightarrow\G_{\ad}(\ol{F})$ belongs to $G_{\ad}$ and we have
\[
(\pi^{\G}_{\xi,0,a})^{t}\cong\pi^{\G}_{\xi,1,a\epsilon^{-1}}
\quad
\text{and}
\quad
(\pi^{\G}_{\xi,1,a\epsilon^{-1}})^{t}\cong\pi^{\G}_{\xi,0,a}.
\]
\end{lem}

\begin{proof}
For every $\sigma\in\Gal(\ol{F}/F)$, we have $t^{-1}\sigma(t)\in\mathbf{Z}_{\mathbf{G}}(\ol{F})$.
Thus the image of $t$ in $\G_{\ad}(\ol{F})$ belongs to $G_{\ad}$.

We next show the second assertion.
Since the $t$-conjugation preserves $\pm I_{\G}^{+}$, we have
\[
(\pi^{\G}_{\xi,0,a})^{t}\cong\cInd_{\pm I_{\G}^{+}}^{G} (\chi^{\G}_{\xi,0,a})^{t},
\]
and the character $(\chi^{\G}_{\xi,0,a})^{t}$ is equal to $\chi^{\G}_{\xi,1,a\epsilon^{-1}}$.
Therefore $(\pi^{\G}_{\xi,0,a})^{t}$ is equivalent to $\pi^{\G}_{\xi,1,a\epsilon^{-1}}$.
On the other hand, since $t^{2}$ belongs to $G$, we have $\pi^{\G}_{\xi,0,a}\cong(\pi^{\G}_{\xi,0,a})^{t^{2}}$.
Since we have $(\pi^{\G}_{\xi,0,a})^{t^{2}}\cong(\pi^{\G}_{\xi,1,a\epsilon^{-1}})^{t}$ as we already showed, we get the assertion.
\end{proof}

\begin{lem}\label{lem:sscadj2}
We put $s:=\diag(\sqrt{\varpi}, \ldots, \sqrt{\varpi}, {\sqrt{\varpi}\,}^{-1}, \ldots,{\sqrt{\varpi}\,}^{-1})\in \mathbf{G}(\ol{F})$.
Here $\sqrt{\varpi}\in\ol{F}^{\times}$ is a square root of $\varpi$.
Then the image of this element in $\mathbf{G}_{\ad}(\ol{F})$ belongs to $G_{\ad}$ and we have 
\[
\{\pi^{\G}_{\xi,0,a}, \pi^{\G}_{\xi,1,a\epsilon^{-1}}\}^{s}
=
\{\pi^{\G}_{\xi,0,a}, \pi^{\G}_{\xi,1,a\epsilon^{-1}}\}.
\]
\end{lem}

\begin{proof}
For the same reason as in the previous lemma, $s$ is an element of $G_{\ad}$.

We show the second assertion.
If we put 
\[
u:=
\begin{pmatrix}
0&I_{n}\\
-I_{n}&0
\end{pmatrix} \in G,
\]
then the $s^{-1}u$-conjugation preserves $\pm I_{\G}^{+}$.
Thus, for each $\kappa\in\{0,1\}$, we have
\[
(\pi^{\G}_{\xi,\kappa,a\epsilon^{-\kappa}})^{s^{-1}}
\cong
(\pi^{\G}_{\xi,\kappa,a\epsilon^{-\kappa}})^{s^{-1}u}
\cong 
\cInd_{\pm I_{\G}^{+}}^{G} (\chi^{\G}_{\xi,\kappa,a\epsilon^{-\kappa}})^{s^{-1}u}.
\]
The character $(\chi^{\G}_{\xi,\kappa,a\epsilon^{-\kappa}})^{s^{-1}u}$ is an affine generic character of $\pm I_{\G}^{+}$ which is defined by
\begin{align*}
g=(g_{ij})_{ij} &\mapsto \psi(g_{12}+\cdots+g_{n-1,n}-a\epsilon^{-\kappa}g_{n,n+1}-\epsilon^{\kappa}g_{2n,1}\varpi^{-1}) \text{ for $g\in I_{\G}^{+}$, and}\\
-1&\mapsto\xi.
\end{align*}
This affine generic character is conjugate under $G$ to the affine generic character
\[
\begin{cases}
\chi^{\G}_{\xi,0,a} & \text{if $-a\epsilon^{-\kappa}\in k^{\times2}$,}\\
\chi^{\G}_{\xi,1,a\epsilon^{-1}} & \text{if $-a\epsilon^{-\kappa}\notin k^{\times2}$.}
\end{cases}
\]
Therefore $s^{-1}$-conjugation preserves the set $\{\pi^{\G}_{\xi,0,a}, \pi^{\G}_{\xi,1,a\epsilon^{-1}}\}$.
This completes the proof.
\end{proof}

\begin{proof}[Proof of Proposition \ref{prop:sscadj}]
We have the following two short exact sequences:
\[
1 \rightarrow \mu_{2} \rightarrow \overline{F}{}^{\times} \xrightarrow{x\mapsto x^{2}} \overline{F}{}^{\times} \rightarrow 1,
\]
\[
1 \rightarrow \bfZ_{\G}(\overline{F})(\cong\mu_{2}) \rightarrow \mathbf{G}(\overline{F}) \rightarrow \mathbf{G}_{\ad}(\overline{F}) \rightarrow 1.
\]
We put $\lambda$ and $\lambda_{\ad}$ to be
\begin{align*}
\lambda&\colon \overline{F}{}^{\times} \rightarrow \mathbf{G}(\overline{F});\quad x\mapsto \diag(x,\ldots,x,x^{-1},\ldots,x^{-1}),\\
\lambda_{\ad}&\colon \overline{F}{}^{\times} \rightarrow \mathbf{G}_{\ad}(\overline{F});\quad x\mapsto \diag(\sqrt{x},\ldots,\sqrt{x},{\sqrt{x}\,}^{-1},\ldots,{\sqrt{x}\,}^{-1}),
\end{align*}
respectively.
Here we remark that $\lambda_{\ad}(x)$ is independent of the choice of $\sqrt{x}$ since $\lambda_{\ad}(x)$ belongs to the adjoint group $\G_{\ad}(\overline{F})$.
Then the above two exact sequences fit into the following commutative diagram:
\[
\xymatrix{
1\ar[r] & \mu_{2}\ar[r]\ar@{=}[d] &\overline{F}{}^{\times}\ar[r]\ar^-{\lambda}[d] & \overline{F}{}^{\times}\ar[r]\ar^-{\lambda_{\ad}}[d] & 1\\
1\ar[r] & \mu_{2}\ar[r] & \mathbf{G}(\overline{F})\ar[r] & \mathbf{G}_{\ad}(\overline{F})\ar[r] & 1\lefteqn{.}
}
\]
By taking long exact sequences of the Galois cohomology, we get the following:
\[
\xymatrix{
F^{\times}\ar[r]\ar^-{\lambda}[d] & F^{\times}\ar[r]\ar^-{\lambda_{\ad}}[d] & H^{1}(F,\mu_{2})\ar[r]\ar@{=}[d] & 1\\
G\ar[r] & G_{\ad}\ar[r] & H^{1}(F,\mu_{2})\ar[r] & 1\lefteqn{.}
}
\]
Here we used the vanishings of $H^{1}(F,\overline{F}{}^{\times})$ (Hilbert's 90th theorem) and $H^{1}(F,\G(\overline{F}))$ (this follows from the simply-connectedness of $\G$ by Kneser's theorem).
In particular, the quotient $G_{\ad}/G$ of $G_{\ad}$ by the image of $G$ in $G_{\ad}$ is isomorphic to $F^{\times}/F^{\times2}$.

Since we have
\[
F^{\times}/F^{\times2}
\cong
\bigl(\lan\varpi\ran/\lan\varpi^{2}\ran\bigr)
\times
\bigl(\lan\epsilon\ran/\lan\epsilon^{2}\ran\bigr),
\]
and the elements $\varpi$ and $\epsilon$ map to $s$ and $t$ under the map $\lambda_{\ad}$, respectively, the quotient $G_{\ad}/G$ is generated by $s$ and $t$.
Therefore the claim follows by Lemmas \ref{lem:sscadj1} and \ref{lem:sscadj2}.
\end{proof}


We next recall a fact about a generic representation in an adjoint orbit of simple supercuspidal representations.
Before we state the fact, we explain our choice of a Whittaker datum of $\G$.
In this paper, we take a Whittaker datum $\mathfrak{w}_{\mathbf{G}}$ of $\mathbf{G}=\Sp_{2n}$ to be $(\mathbf{B}_{\Sp_{2n}}, \lambda_{\Sp_{2n}})$, where
\begin{itemize}
 \item
 $\mathbf{B}_{\Sp_{2n}}$ is the upper-triangular Borel subgroup  of $\Sp_{2n}$, and
 \item
 $\lambda_{\Sp_{2n}}$ is the character  of the unipotent radical $U_{\Sp_{2n}}$ of $B_{\Sp_{2n}}$ defined by
 \[
 \lambda_{\Sp_{2n}}(y)=\psi(y_{1,2}+\cdots+y_{n,n+1}) \text{ for } y=(y_{ij}) \in U_{\Sp_{2n}}.
 \]
\end{itemize}

\begin{prop}\label{prop:genssc}
Each $G_{ad}$-orbit of simple supercuspidal representations of $G$ contains an exactly one $\mathfrak{w}_{\G}$-generic representation.
In particular, exactly one of $\pi^{\G}_{\xi,0,a}$ and $\pi^{\G}_{\xi,1,a\epsilon^{-1}}$ is $\mathfrak{w}_{\G}$-generic.
\end{prop}

\begin{proof}
The first assertion is \cite[Proposition 5.1]{MR3001735}.
The second assertion immediately follows from the first one and Proposition \ref{prop:sscadj}.
\end{proof}

\begin{cor}\label{cor:genssc}
Every member of the $L$-packet $\Pi_{\phi}^{\G}$ is supercuspidal.
Furthermore, $\Pi^{\G}_{\phi}$ has no more simple supercuspidal representation other than $\pi^{\G}_{\xi,0,a}$ and $\pi^{\G}_{\xi,1,a\epsilon^{-1}}$.
\end{cor}

\begin{proof}
By Proposition \ref{prop:genssc}, $\Pi^{\G}_{\phi}$ contains a $\mathfrak{w}_{\mathbf{G}}$-generic supercuspidal representation.
Thus, by Corollary \ref{cor:gensc}, $\Pi_{\phi}^{\G}$ consists only of supercuspidal representations.

On the other hand, by Corollary \ref{cor:adj}, $\Pi_{\phi}^{\G}$ consists of $G_{\ad}$-orbits.
If $\Pi_{\phi}^{\G}$ contains a simple supercuspidal representation other than $\pi^{\G}_{\xi,0,a}$ and $\pi^{\G}_{\xi,1,a\epsilon^{-1}}$, then $\Pi_{\phi}^{\G}$ also contains its $G_{\ad}$-orbits.
In particular, by Proposition \ref{prop:genssc}, $\Pi_{\phi}^{\G}$ contains another $\mathfrak{w}_{\G}$-generic representation.
This contradicts the uniqueness of a $\mathfrak{w}_{\mathbf{G}}$-generic representation in each $L$-packet (Theorem \ref{thm:GPC}). 
\end{proof}

\subsection{Formal degrees of simple supercuspidal representations}\label{subsec:deg-Sp}
We next show that our $L$-packet $\Pi_{\phi}^{\G}$ does not contain a depth-zero supercuspidal representation.

We fix a Haar measure $dg$ of $G$.
We take the Haar measure $dz$ of $Z_{\G}$ such that $dz(Z_{\G})=2$ (recall that $Z_{\G}\cong\{\pm1\}$).
Let $d\dot{g}$ denote the quotient measure of $G/Z_{\G}$ determined by $dg$ and $dz$.
Then, for an irreducible discrete series representation $\pi$ of $G$, the \textit{formal degree} of $\pi$ is defined as follows.
We first take a $G$-invariant non-degenerate Hermitian form 
\[
(-, -) \colon \pi \times \pi \rightarrow \C.
\]
Here we note that such a pairing always exists by the assumption that $\pi$ is a discrete series representation, and that it is unique up to a scalar-multiple by Schur's lemma.
Then, by \textit{Schur's orthogonality relation}, there exists a positive real number $\deg(\pi)\in\R_{>0}$ satisfying
\[
\int_{G/Z_{\mathbf{G}}} (\pi(g)v_{1}, v_{2})\cdot\overline{(\pi(g)v_{3}, v_{4})} \, d\dot{g}
=
\deg(\pi)^{-1} (v_{1}, v_{3})\cdot\overline{(v_{2}, v_{4})}
\]
for every $v_{1},v_{2},v_{3},v_{4}\in\pi$.
This number does not depend on the choice of $(-,-)$ (but depends on the Haar measure $d\dot{g}$).
We call $\deg(\pi)$ the formal degree of $\pi$ with respect to $d\dot{g}$.

Now we consider the formal degrees of representations which belong to the $L$-packet $\Pi_{\phi}^{\G}$.
Since every member of $\Pi_{\phi}^{\G}$ is supercuspidal (Corollary \ref{cor:genssc}), hence a discrete series representation, we can consider its formal degree.

The key fact is the following:
\begin{thm}[{\cite[Corollary 9.10]{MR1070599}}]\label{thm:Shahidi}
For every discrete $L$-parameter $\phi'$ of $\G$, all representations in the $L$-packet $\Pi^{\G}_{\phi'}$ have the same formal degree.
\end{thm}

By this theorem, in order to show that $\Pi_{\phi}^{\G}$ does not contain a depth-zero supercuspidal representation, it is enough to check that the formal degrees of depth-zero supercuspidal representations are distinct from those of simple supercuspidal representations.

The following lemma is well-known (e.g., see \cite[Theorem A.14]{MR1423022}):

\begin{lem}\label{lem:deg}
Let $K$ be a compact open subgroup of $G$ which contains $Z_{\mathbf{G}}$.
Let $\rho$ be an irreducible smooth representation of $K$, and we assume that $\pi:=\cInd_{K}^{G}\rho$ is irreducible, hence supercuspidal.
Then we have
\[
\deg(\pi) = \dim(\rho)\cdot d\dot{g}(K/Z_{\mathbf{G}})^{-1}.
\] 
\end{lem}

By using this lemma, we compute the formal degrees of depth-zero and simple supercuspidal representations.

\begin{lem}\label{lem:deg_{ssc}}
Let $\pi$ be a simple supercuspidal representation of $G$.
Then the formal degree of $\pi$ is given by $dg(I_{\G}^{+})^{-1}$.
\end{lem}

\begin{proof}
This immediately follows from the definition of the simple supercuspidal representations of $G=\Sp_{2n}(F)$ (Sections \ref{subsec:ssc} and \ref{subsec:ssc-Sp}) and Lemma \ref{lem:deg}.
Note that, as $I_{\G}^{+}$ does not contain $Z_{\G}$, we have $d\dot{g}(Z_{\G}I_{\G}^{+}/Z_{\G})=dg(I_{\G}^{+})$.
\end{proof}

\begin{prop}\label{prop:deg}
The formal degrees of irreducible depth-zero supercuspidal representations of $G$ are strictly smaller than those of simple supercuspidal representations of $G$.
\end{prop}

\begin{proof}
By \cite[Proposition 6.8]{MR1371680}, every irreducible depth-zero supercuspidal representation is given by $\cInd_{N_{G}(P)}^{G}\rho$.
Here, $N_{G}(P)$ is the normalizer of some maximal parahoric subgroup $P$ of $G$, and $\rho$ is an irreducible representation of $N_{G}(P)$ which contains the inflation of an irreducible cuspidal representation of the quotient $\mathbb{P}=P/P^{+}$ of $P$ by its pro-unipotent radical $P^{+}$.
We note that $\mathbb{P}$ is the set of $k$-valued points of a reductive group over $k$ and that in our case $N_{G}(P)$ always equals $P$ by the simply-connectedness of $\mathbf{G}=\Sp_{2n}$.
In particular, $\rho$ is the inflation of an irreducible cuspidal representation of $\mathbb{P}$.

We show that the formal degree of $\cInd_{P}^{G}\rho$ is smaller than $dg(I_{\G}^{+})^{-1}$, which equals the formal degree of simple supercuspidal representations by Lemma \ref{lem:deg_{ssc}}.
By Lemma \ref{lem:deg}, we have
\[
\deg(\cInd_{P}^{G}\rho)= \dim(\rho)\cdot d\dot{g}(P/Z_{\mathbf{G}})^{-1}.
\]
Thus we have to show that the following inequality holds:
\[
\dim(\rho) < d\dot{g}(P/Z_{\mathbf{G}})\cdot dg(I_{\G}^{+})^{-1}.
\tag{$\star$}
\]

By noting that the square sum of the dimensions of the all irreducible representations of the finite group $\mathbb{P}$ is equal to $|\mathbb{P}|=[P:P^{+}]$ and that $\mathbb{P}$ is not a trivial group, we have $\dim(\rho)^{2} < [P:P^{+}]$.
Thus, to show the inequality $(\star)$, it suffices to show that the following inequality holds:
\[
[P:P^{+}] \leq d\dot{g}(P/Z_{\G})^{2}\cdot dg(I_{\G}^{+})^{-2}.
\]
By conjugation, we may assume that the alcove corresponding to $I_{\G}$ contains the point corresponding to the parahoric subgroup $P$, hence that 
\[
P^{+}\subset I_{\G}^{+}\subset I_{\G}\subset P.
\]
Then the above inequality is equivalent to
\[
[I_{\G}^{+}:P^{+}]^{2}\cdot2^{2} \leq [P:P^{+}].
\]
We note that the image of $I_{\G}$ in $\mathbb{P}$ is a Borel subgroup of $\mathbb{P}$ and the image of $I_{\G}^{+}$ in $\mathbb{P}$ is its unipotent radical.
We denote them by $\mathbb{B}$ and $\mathbb{U}$.
Then we can rewrite the above inequality as
\[
|\mathbb{U}|^{2}\cdot2^{2} \leq |\mathbb{P}|.
\]
However, since the map
\[
\mathbb{B}\times\mathbb{U}^{\mathrm{op}}\ra \mathbb{P}, \, (x,y)\mapsto xy
\]
is injective, we have
\[
|\mathbb{U}|^{2} \cdot|\mathbb{B}/\mathbb{U}| \leq |\mathbb{P}|.
\]
As $\mathbf{G}=\Sp_{2n}$ is split of rank $n$, $\mathbb{B}/\mathbb{U}$ is a split torus of rank $n$ (see \cite[Section 3.5]{MR546588}).
Hence we have $|\mathbb{B}/\mathbb{U}|=(q-1)^{n}$.
Therefore this inequality holds unless $n=1$ and $k=\F_{3}$.

We finally consider the case that $n=1$ and $k=\F_{3}$.
In this case, we show the inequality $(\star)$ directly.
When $n=1$ and $k=\F_{3}$, the reduction $\mathbb{P}$ of a maximal parahoric subgroup $P$ is given by $\SL_{2}(\F_{3})$.
Thus the right-hand side of the inequality $(\star)$ is equal to
\[
|\SL_{2}(\F_{3})/\mathbb{U}|\cdot 2^{-1}=24\cdot3^{-1}\cdot2^{-1}=4.
\]
On the other hand, the dimension of irreducible cuspidal representations of $\SL_{2}(\F_{3})$ is given by $2$ (e.g., see \cite[Section 5.2]{MR1153249}).
Thus we get $(\star)$.
\end{proof}

By combining Theorem \ref{thm:Shahidi} with Proposition \ref{prop:deg}, we get the following:

\begin{prop}\label{prop:deg_{packet}}
The $L$-packet $\Pi_{\phi}^{\G}$ does not contain any irreducible depth-zero supercuspidal representation.
\end{prop}

\subsection{Simple supercuspidal $L$-packets of $\Sp_{2n}$}\label{subsec:packet-Sp}

Now we show Theorem \ref{thm:packetSp}.
\begin{proof}[Proof of Theorem \ref{thm:packetSp}]
We put 
\[
\Pi_{\phi}^{\G}=\bigl\{\pi_{0}:=\pi^{\G}_{\xi,0,a},\pi_{1}:=\pi^{\G}_{\xi,1,a\epsilon^{-1}},\ldots,\pi_{r-1}\bigr\},
\]
where $r\geq2$ is the order of $\Pi_{\phi}^{\G}$.
Then, by Corollary \ref{cor:genssc} and Proposition \ref{prop:deg_{packet}},
\begin{itemize}
\item
every $\pi_{k}$ is supercuspidal,
\item
$\pi_{k}$ is not simple supercuspidal for $2\leq k\leq r-1$, and
\item
$\pi_{k}$ is not depth-zero for $2\leq k\leq r-1$.
\end{itemize}

We take $\varepsilon_{k}\in\{\pm1\}$ for each $0\leq k \leq r-1$ and we consider the following linear combination of characters
\[
\sum_{k=0}^{r-1} \varepsilon_{k}\Theta_{\pi_{k}}.
\]
We first show the following claim.

\begin{claim*}
There exists an affine generic element $g\in I_{\G}^{+}\subset G$ such that the above function $\sum_{k=0}^{r-1} \varepsilon_{k}\Theta_{\pi_{k}}$ does not vanish at $g$.
\end{claim*}

Here we note that every affine generic element of $I_{\G}^{+}\subset G$ is regular semisimple (see Remark \ref{rem:regss}).
Hence, for any affine generic element of $I_{\G}^{+}$, we can always consider the value of the character $\Theta_{\pi_{k}}$ at it.

\begin{proof}[Proof of Claim]
By Corollary \ref{cor:pmSp} and Proposition \ref{prop:Kl} (2), we can take an affine generic element $g\in I_{\G}^{+}$ such that $\varepsilon_{0}\Theta_{\pi_{0}}+\varepsilon_{1}\Theta_{\pi_{1}}$ takes a non-zero value $c$ at $g$.
If $r$ is equal to $2$, then there is nothing to prove.
We consider the case where $r$ is greater than $2$.
In this case, since the characters of simple supercuspidal representations at affine generic elements depend only on their simple affine components (Proposition \ref{prop:charSp}), $\varepsilon_{0}\Theta_{\pi_{0}}+\varepsilon_{1}\Theta_{\pi_{1}}$ is constant on the coset $g I_{\G}^{++}$.
On the other hand, if $\sum_{k=0}^{r-1} \varepsilon_{k}\Theta_{\pi_{k}}$ is identically zero on affine generic elements, then we have
\begin{align*}
\sum_{k=0}^{r-1} \varepsilon_{k}\Theta_{\pi_{k}}(\mathbbm{1}_{gI_{\G}^{++}})
&=\sum_{k=0}^{r-1} \varepsilon_{k}\int_{gI_{\G}^{++}}\Theta_{\pi_{k}}(y)\,dy\\
&=c\cdot dy(gI_{\G}^{++})+\sum_{k=2}^{r-1} \varepsilon_{k}\int_{gI_{\G}^{++}}\Theta_{\pi_{k}}(y)\,dy =0.
\end{align*}
Therefore, for at least one $k$ such that $2\leq k \leq r-1$, we should have 
\[
\int_{gI_{\G}^{++}}\Theta_{\pi_{k}}(y)\,dy \neq 0.
\]
However, by the property of the character, we also have
\[
\int_{gI_{\G}^{++}}\Theta_{\pi_{k}}(y)\,dy
=
\tr\Big(\pi_{k}\bigl(\mathbbm{1}_{gI_{\G}^{++}}\bigr)\,\Big\vert\, \pi_{k}^{I_{\G}^{++}}\Bigr)
\]
(note that $\mathbbm{1}_{gI_{\G}^{++}}$ is bi-$I_{\G}^{++}$-invariant).
Thus non-vanishing of this integral implies that in particular $\pi_{k}^{I_{\G}^{++}}$ is not zero.
By noting the fact that $I_{\G}^{++}$ is the $(\frac{1}{2n}+)$-th Moy--Prasad filtration attached to the barycenter of an alcove, the depth of $\pi_{k}$ is not greater than $\frac{1}{2n}$ by the definition of the depth of representations (see Section \ref{subsec:depth} for details).
Therefore, by Proposition \ref{prop:depth-ssc}, $\pi_{k}$ is either simple supercuspidal or depth-zero supercuspidal.
This is a contradiction.
\end{proof}

Now we back to the proof of Theorem \ref{thm:packetSp}.
If we take $s\in\mathcal{S}^{\G}_{\phi}$, then an endoscopic group $\mathbf{H}$ is defined and we can regard $\phi$ as an $L$-parameter of $\mathbf{H}$ as in Section \ref{sec:Arthur}, and the following endoscopic character relation (Theorem \ref{thm:SECR}) holds for any $g\in G^{\srs}$:
\[
\sum_{k=0}^{r-1} \lan s,\pi_{k}\ran \Theta_{\pi_{k}}(g)
= \sum_{h\leftrightarrow g/\sim} \frac{D_{\mathbf{H}}(h)^{2}}{D_{\G}(g)^{2}} \Delta_{\mathbf{H},\G}(h,g) \sum_{\tilde{\pi}\in\widetilde{\Pi}_{\phi}^{\mathbf{H}}}\Theta_{\tilde{\pi}}(h),
\]
where $h$ in the sum runs through the set of stable conjugacy classes of elements of ${H}^{\srs}$ which is a norm of $g$.
Here we note that, as we saw in Section \ref{sec:Arthur}, $\mathbf{H}$ is of the form $\mathbf{H}_{-}\times\mathbf{H}_{+}$, where 
\begin{itemize}
\item
$\mathbf{H}_{-}=\SO_{d^{-}}^{\mu}$ for a quadratic character $\mu$ of $F^{\times}$,
\item
$\mathbf{H}_{+}=\Sp_{d^{+}}$, and
\item
$d^{-}$ and $d^{+}$ are non-negative even integers satisfying $d^{-}+d^{+}=2n$.
\end{itemize}

By the above claim, the left-hand side of this equality is not zero for an affine generic element $g$.
Therefore, in particular there exists an element $h$ of ${H}^{\srs}$ which is a norm of $g$.
We put $h=(h_{-},h_{+})\in H$.
Then, by the definition of the norm correspondence, the characteristic polynomial $p_{g}\in F[T]$ of $g$ is equal to the product of those $p_{h_{-}}, p_{h_{+}} \in F[T]$ of $h_{-},h_{+}$.

By noting that affine generic elements of $\Sp_{2n}(F)$ are affine generic also as elements of $\GL_{2n}(F)$ and that characteristic polynomials of affine generic elements are Eisenstein (see Lemma \ref{lem:charpoly1}), we know that $p_{g}$ is irreducible of degree $2n$ over $F$.
On the other hand, the degrees of $p_{h_{+}}$ and $p_{h_{-}}$ are given by $d^{-}$ and $d^{+}$, respectively.
Hence the possibilities of pairs $(d^{-},d^{+})$ are only $(0,2n)$ and $(2n,0)$.
In other words, $\mathbf{H}$ must be either $\Sp_{2n}$ or $\SO_{2n}^{\mu}$.
This implies that the order of $\mathcal{S}^{\G}_{\phi}$ is not greater than two.
Since we have a bijection between $\Pi_{\phi}^{\G}$ and the character group of $\mathcal{S}_{\phi}^{\G}$ (see Theorem \ref{thm:Arthur} (1)), also the order of $\Pi_{\phi}^{\G}$ is not greater than two.

On the other hand, our $L$-packet $\Pi_{\phi}^{\G}$ contains at least two elements $\pi_{0}=\pi_{\xi,0,a}^{\G}$ and $\pi_{1}=\pi_{\xi,1,a\epsilon^{-1}}^{\G}$.
Thus we can conclude that its order is given by two.
\end{proof}

As a consequence of the above proof, we get the following corollary.

\begin{cor}\label{cor:decomp}
The $L$-parameter $\phi$ is trivial on $\SL_{2}(\C)$ and decomposed into the following direct sum as a $(2n+1)$-dimensional representation of $W_{F}$:
\[
\phi=\phi_{0}\oplus \det\circ\phi_{0},
\]
where $\phi_{0}$ is an orthogonal irreducible representation of dimension $2n$.
In particular, $\phi$ factors through an endoscopic group $\SO_{2n}^{\mu}$ of $\G$, where $\mu$ is the quadratic character of $F^{\times}$ corresponding to $\det\circ\phi_{0}$ under the local class field theory.
\end{cor}

\begin{proof}
By the proof of Corollary \ref{cor:gensc}, the set $\Jord(\phi)$ consists only of elements of the form $(\rho, 1)$.
In particular, $\phi$ is trivial on $\SL_{2}(\C)$.

Since the order of $\mathcal{S}^{\G}_{\phi}$ is $2$, $\phi$ can be decomposed into a direct sum of two irreducible representations of $W_{F}$.
On the other hand, by the proof of Theorem \ref{thm:packetSp}, the endoscopic group of $\G$ which is determined by the unique nontrivial element of $\mathcal{S}^{\G}_{\phi}$ is given by $\SO_{2n}^{\mu}$ for a quadratic character $\mu$ of $F^{\times}$.
Thus the one of two irreducible constituents of $\phi$ is $2n$-dimensional.
We denote it by $\phi_{0}$.
Since the image of $\phi$ is contained in 
\[
\bigl(\rmO_{2n}(\C)\times\rmO_{1}(\C)\bigr)^{\det=1}\times W_{F},
\]
the other one is given by $\det\circ\phi_{0}$.
Moreover, $\mu$ corresponds to $\det\circ\phi_{0}$.
\end{proof}

\section{Transfer factors at affine generic elements}\label{sec:tran}
The purpose of this section is to compute the transfer factors at some special elements such as affine generic elements of Iwahori subgroups by using Waldspurger's formula in \cite{MR2672539}.
%
%
In this section, we let $\mathbf{G}$ be one of the following groups:
\[
\text{twisted }\GL_{N},\quad
\Sp_{2n},\quad
\SO_{2n+2},\quad
\SO_{2n+2}^{\ur}
\]
(we assume $n\geq2$ when $\G=\SO_{2n+2}^{(\ur)}$), and put
\[
\mathbf{H}:=
\begin{cases}
\Sp_{2n} & \text{if }\mathbf{G}= \text{twisted }\GL_{2n+1},\\
\SO_{2n}^{\mu} & \text{if }\mathbf{G}= \text{twisted }\GL_{2n},\\
\SO_{2n}^{\mu} & \text{if }\mathbf{G}= \Sp_{2n},\\
\SO_{2n}^{\mu}\times\SO_{2}^{\mu} & \text{if }\mathbf{G}= \SO_{2n+2},\\
\SO_{2n}^{\mu}\times\SO_{2}^{\bar{\mu}} & \text{if }\mathbf{G}= \SO_{2n+2}^{\ur},
\end{cases}
\]
where $\mu$ is a ramified quadratic character of $F^{\times}$, and $\bar{\mu}$ is the ramified quadratic character of $F^{\times}$ such that $\mu\neq\bar{\mu}$.
Furthermore, we define an involution $\theta$ of $\G$ to be
\[
\theta(g):=
\begin{cases}
J_{N}\cdot{}^{t}\!g^{-1}\cdot J_{N}^{-1}& \text{if } \G=\text{twisted }\GL_{N},\\
g & \text{otherwise}. 
\end{cases}
\]
Then, as we described in Section \ref{sec:Arthur}, $\mathbf{H}$ is an endoscopic group for $(\G,\theta)$.
In the following, we compute the transfer factors for these pairs.

\subsection{Parametrization of conjugacy classes}\label{subsec:param-conj}
We first recall a parametrization of strongly regular semisimple conjugacy classes of classical groups, which will be used in the statement of Waldspurger's explicit formula of transfer factors.
The basic reference of this section is \cite{MR2672539} (see also \cite{MR3067291}).

The ($\theta$-)conjugacy classes of strongly ($\theta$-)regular ($\theta$-)semisimple elements of $G$ can be classified in terms of linear algebraic data as follows:

\begin{description}
\item[The case of twisted $\GL_{2n}$]
The $\theta$-conjugacy classes of strongly $\theta$-regular $\theta$-semisimple elements are parametrized by the following data 
\[
(I, \{F_{\pm i}\}_{i\in I}, \{F_{i}\}_{i\in I}, \{x_{i}\}_{i\in I}),
\]
where
 \begin{itemize}
 \item $I$ is a finite set, 
 \item $\{F_{\pm i}\}_{i\in I}$ is a set of finite extensions of $F$ satisfying $\sum_{i\in I} [F_{\pm i} : F]=n$, 
 \item each $F_{i}$ is an {\'e}tale $F_{\pm i}$-algebra of $F_{\pm i}$-dimension $2$, and
 \item $x_{i}\in F_{i}^{\times}$ satisfying $F[x_{i}]=F_{i}$.
 \end{itemize}

We recall the correspondence between the above data and the $\theta$-conjugacy classes briefly.
For a data $(I, \{F_{\pm i}\}_{i\in I}, \{F_{i}\}_{i\in I}, \{x_{i}\}_{i\in I})$ satisfying the above conditions, we consider the following $F$-bilinear form on $\bigoplus_{i\in I}F_{i}$:
\begin{align*}
\bigoplus_{i\in I}F_{i}\times \bigoplus_{i\in I}F_{i} &\rightarrow F \\
\Bigl(\sum_{i\in I} w_{i}, \sum_{i\in I} w'_{i}\Bigr)&\mapsto \sum_{i\in I}\Tr_{F_{i}/F}\bigl(\tau_{i}(w_{i})\cdot w'_{i}\cdot x_{i}\bigr),
\end{align*}
where $\tau_{i}$ is the unique nontrivial $F_{\pm i}$-automorphism of $F_{i}$.
Note that this bilinear form is non-degenerate by the assumption that each $x_{i}$ is invertible in $F_{i}$.
On the other hand, we can identify the twisted $\GL_{2n}$ with the space of non-degenerate $F$-bilinear forms on a $2n$-dimensional $F$-vector space as follows:
\begin{align*}
\GL_{2n}(F) &\cong \bigl\{F^{\oplus2n}\times F^{\oplus2n}\rightarrow F \text{: non-degenerate bilinear form}\bigr\}\\
g &\mapsto {}^{t}\!g^{-1}J_{2n}=J_{2n}\theta(g), \\
\theta(J_{2n}^{-1}h) &\mapsfrom h.
\end{align*}
Here, in the right-hand side, we regard ${}^{t}\!g^{-1}J_{2n}$ as a bilinear form on $F^{\oplus2n}$ whose representation matrix with respect to the canonical basis is given by the matrix ${}^{t}\!g^{-1}J_{2n}$.
Note that the twisted $\GL_{2n}$ and the space of non-degenerate bilinear forms have structures of bi-$\GL_{2n}(F)$-torsors via
\begin{align*}
x\cdot g \cdot y &:= xg\theta(y), \text{ for }g\in\GL_{2n}(F), \text{ and}\\
x\cdot g \cdot y &:= {}^{t}\!x^{-1}gy, \text{ for a bilinear form } g, 
\end{align*}
and the above bijection preserves these bi-$\GL_{2n}(F)$-torsor structures.

Thus, by this identification, we get a $\theta$-conjugacy class of $\GL_{2n}(F)$ from the above bilinear form
(note that this $\theta$-conjugacy class is independent of the choice of an identification $\bigoplus_{i\in I}F_{i}\cong F^{\oplus2n}$).


\item[The case of symplectic group $\Sp_{2n}$]
The conjugacy classes of strongly regular semisimple elements of $\Sp_{2n}(F)$ are parametrized by the following data 
\[
(I, \{F_{\pm i}\}_{i\in I}, \{F_{i}\}_{i\in I}, \{c_{i}\}_{i\in I}, \{y_{i}\}_{i\in I}),
\]
where
 \begin{itemize}
 \item $I$, $\{F_{\pm i}\}_{i\in I}$, $\{F_{i}\}_{i\in I}$, are as in the case of twisted $\GL_{2n}$, 
 \item $c_{i}\in F_{i}^{\times}/\Nr_{F_{i}/F_{\pm i}}(F_{i}^{\times})$ such that a(ny) representative $\dot{c}_{i}$ of $c_{i}$ satisfies $\tau_{i}(\dot{c}_{i})=-\dot{c}_{i}$ for the unique nontrivial element $\tau_{i}\in \mathrm{Aut}_{F_{\pm i}}(F_{i})$, and
 \item $y_{i}\in F_{i}^{\times}$ satisfying $F[y_{i}]=F_{i}$ and $y_{i}\cdot \tau_{i}(y_{i})=1$.
 \end{itemize}

For such data, we first define an $F$-bilinear form $q_{V}$ on $V:=\bigoplus_{i\in I}F_{i}$ as follows:
\begin{align*}
q_{V} \colon V\times V &\rightarrow F \\
\Bigl(\sum_{i\in I} w_{i}, \sum_{i\in I} w'_{i}\Bigr)&\mapsto \sum_{i\in I}\Tr_{F_{i}/F}\bigl(\tau_{i}(w_{i})\cdot w'_{i}\cdot \dot{c}_{i}\bigr).
\end{align*}
Then this is a symplectic form on $V$ by the assumption on each $c_{i}$.
Moreover, by the assumption on each $y_{i}$, the automorphism $y$ of $V$ defined by
\[
y\cdot \Bigl(\sum_{i\in I}w_{i} \Bigr) = \Bigl(\sum_{i\in I} y_{i} \cdot w_{i} \Bigr)
\]
preserves this symplectic form $q_{V}$.
Thus we get a conjugacy class of $\Sp_{2n}(F)$ by fixing an identification $\Sp_{2n} \cong \Sp(V,q_{V})$.
Furthermore, the resulting conjugacy class is independent of the choices of a set of representatives $\{\dot{c}_{i}\}_{i\in I}$ and an identification $\Sp_{2n} \cong \Sp(V,q_{V})$.

\item[The case of even special orthogonal group $\SO_{2n}^{\mu}$]
Let $q_{\mu}$ be a symmetric bilinear form on $F^{\oplus2n}$ whose representation matrix with respect the canonical basis is given by $J_{\mu}$ for a quadratic character $\mu$.
Then the $\rmO_{2n}^{\mu}(F)$-conjugacy classes of strongly regular semisimple elements of $\SO_{2n}^{\mu}(F)$ are parametrized by the following data 
\[
(I, \{F_{\pm i}\}_{i\in I}, \{F_{i}\}_{i\in I}, \{c_{i}\}_{i\in I}, \{y_{i}\}_{i\in I}),
\]
where
 \begin{itemize}
 \item $I$, $\{F_{\pm i}\}_{i\in I}$, $\{F_{i}\}_{i\in I}$, are as in the case of twisted $\GL_{2n}$, 
 \item $c_{i}\in F_{i}^{\times}/\Nr_{F_{i}/F_{\pm i}}(F_{i}^{\times})$ such that a(ny) representative $\dot{c}_{i}$ of $c_{i}$ satisfies $\tau_{i}(\dot{c}_{i})=\dot{c}_{i}$ and $(F^{\oplus2n}, q_{\mu})\cong (W, q_{W})$, and
 \item $y_{i}\in F_{i}^{\times}$ satisfying $F[y_{i}]=F_{i}$ and $y_{i}\cdot \tau_{i}(y_{i})=1$.
 \end{itemize}
Here, $W:=\bigoplus_{i\in I}F_{i}$ and $q_{W}$ is a symmetric bilinear form on $W$ defined as follows:
\begin{align*}
q_{W} \colon W\times W &\rightarrow F \\
\Bigl(\sum_{i\in I}w_{i}, \sum_{i\in I}w'_{i}\Bigr)&\mapsto \sum_{i\in I}\Tr_{F_{i}/F}\bigl(\tau_{i}(w_{i})\cdot w'_{i}\cdot \dot{c}_{i}\bigr).
\end{align*}
This is a symmetric form on $W$ by the assumption on each $c_{i}$.
Moreover, by the assumption on each $y_{i}$, the automorphism $y$ of $V$ defined by
\[
y\cdot \Bigl(\sum_{i\in I}w_{i} \Bigr) = \Bigl(\sum_{i\in I} y_{i} \cdot w_{i} \Bigr)
\]
preserves this symmetric form $q_{W}$.
Thus we get an $\rmO_{2n}^{\mu}(F)$-conjugacy class of $\SO_{2n}^{\mu}(F)$ by fixing an identification $\SO_{2n}^{\mu}=\SO(F^{\oplus2n}, q_{\mu}) \cong \SO(W,q_{W})$.
Furthermore, the resulting conjugacy class is independent of the choices of a set of representatives $\{\dot{c}_{i}\}_{i\in I}$ and an identification $\SO_{2n}^{\mu}\cong\SO(W,q_{W})$.
\end{description}

We remark that, in order to get a one-to-one correspondence between the linear algebraic data and the conjugacy classes in the above procedure, we have to introduce an equivalence relation on the linear algebraic data in an obvious way; $(I, \{F_{\pm i}\}_{i\in I}, \{F_{i}\}_{i\in I}, \{c_{i}\}_{i\in I}, \{y_{i}\}_{i\in I})$ and $(J, \{F_{\pm j}\}_{j\in J}, \{F_{j}\}_{j\in J}, \{c_{j}\}_{j\in J}, \{y_{j}\}_{j\in J})$ are equivalent if and only if we have a bijection $\sigma\colon I\rightarrow J$ and an isomorphism $F_{i}\cong F_{\sigma(i)}$ for each $i\in I$ such that $F_{\pm i}$, $c_{i}$, and $y_{i}$ are mapped to $F_{\pm \sigma(i)}$, $c_{\sigma(i)}$, and $y_{\sigma(i)}$.
We write
\[
(I, \{F_{\pm i}\}_{i\in I}, \{F_{i}\}_{i\in I}, \{c_{i}\}_{i\in I}, \{y_{i}\}_{i\in I})
\cong
(J, \{F_{\pm j}\}_{j\in J}, \{F_{j}\}_{j\in J}, \{c_{j}\}_{j\in J}, \{y_{j}\}_{j\in J})
\]
when two data are equivalent in this sense.

\begin{rem}
Note that, in the case of even special orthogonal group $\SO_{2n}^{\mu}$, the above linear algebraic data parametrizes $\rmO_{2n}^{\mu}(F)$-conjugacy classes of regular semisimple elements.
We can easily check that if a regular semisimple element is ``very regular'' (i.e., does not have $\pm1$ as its eigenvalue), then its $\rmO_{2n}^{\mu}(F)$-conjugacy class splits into exactly two $\SO_{2n}^{\mu}(F)$-conjugacy classes.
Moreover, those two $\SO_{2n}^{\mu}(F)$-conjugacy classes are not even stably $\SO_{2n}^{\mu}$-conjugate.
(See \cite[Section 1.3, Le cas sp\'ecial orthogonal pair; Section 1.4]{MR2672539}.)
\end{rem}

\begin{rem}
In the case of symplectic group, the stable conjugacy classes can be obtained by forgetting the datum $\{c_{i}\}_{i\in I}$.
More precisely, two $\Sp_{2n}(F)$-conjugacy classes of regular semisimple elements corresponding to 
\[
(I, \{F_{\pm i}\}_{i\in I}, \{F_{i}\}_{i\in I}, \{c_{i}\}_{i\in I}, \{y_{i}\}_{i\in I})
\]
and
\[
(J, \{F_{\pm j}\}_{j\in J}, \{F_{j}\}_{j\in J}, \{c_{j}\}_{j\in J}, \{y_{j}\}_{j\in J})
\]
are stably conjugate if and only if
\[
(I, \{F_{\pm i}\}_{i\in I}, \{F_{i}\}_{i\in I}, \{y_{i}\}_{i\in I})
\cong
(J, \{F_{\pm j}\}_{j\in J}, \{F_{j}\}_{j\in J}, \{y_{j}\}_{j\in J}).
\]

Similarly, in the case of even special orthogonal group, the stable $\rmO_{2n}$-conjugacy classes are obtained by forgetting the datum $\{c_{i}\}_{i\in I}$.
Again note that, for very regular semisimple elements, any stable $\rmO_{2n}$-conjugacy class splits into exactly two stable $\SO_{2n}$-conjugacy classes.
\end{rem}



\subsection{Norm correspondence}\label{subsec:norm}
We can interpret the norm correspondence for strongly ($\theta$-) regular ($\theta$-) semisimple elements in terms of the above parametrizations as follows (see Section 1.9 in \cite{MR2672539}):
\begin{description}

\item[(1) The case of twisted $\GL_{2n}$ and $\SO_{2n}^{\mu}$]
Let $x\in \GL^{\strs}_{2n}(F)$ and $y\in\SO_{2n}^{\mu,\srs}(F)$, and we assume that they correspond to 
\begin{align*}
&(I, \{F_{\pm i}\}_{i\in I}, \{F_{i}\}_{i\in I}, \{x_{i}\}_{i\in I}) \text{ and}\\
&\bigl(J, \{F_{\pm j}\}_{j\in J}, \{F_{j}\}_{j\in J}, \{c_{j}\}_{j\in J}, \{y_{j}\}_{j\in J}\bigr),
\end{align*}
via the parametrizations in the previous section, respectively.
Moreover, we assume that $y$ is \textit{very regular}, namely does not have $\pm1$ as an eigenvalue.
Then $y$ is a norm of $x$ if and only if the following hold:
\begin{itemize}
 \item
 $(I, \{F_{\pm i}\}_{i\in I}, \{F_{i}\}_{i\in I})$ is equal to $(J, \{F_{\pm j}\}_{j\in J}, \{F_{j}\}_{j\in J})$, and 
 \item
 $x_{i}\cdot \tau_{i}(x_{i})^{-1}=-y_{i}$ for every $i\in I$.
\end{itemize}

\item[(2) The case of $\Sp_{2n}$ and $\SO_{2n}^{\mu}$]
Let $x\in \Sp^{\srs}_{2n}(F)$ and $y\in\SO_{2n}^{\mu,\srs}(F)$, and we assume that they correspond to 
\begin{align*}
&\bigl(I, \{F_{\pm i}\}_{i\in I}, \{F_{i}\}_{i\in I}, \{c_{i}\}_{i\in I}, \{x_{i}\}_{i\in I}\bigr) \text{ and}\\
&\bigl(J, \{F_{\pm j}\}_{j\in J}, \{F_{j}\}_{j\in J}, \{c_{j}\}_{j\in J}, \{y_{j}\}_{j\in J}\bigr),
\end{align*}
via the parametrizations in the previous section, respectively.
Moreover, we assume that $y$ is very regular.
Then $y$ is a norm of $x$ if and only if the following hold:
\begin{itemize}
 \item
 $(I, \{F_{\pm i}\}_{i\in I}, \{F_{i}\}_{i\in I})$ is equal to $(J, \{F_{\pm j}\}_{j\in J}, \{F_{j}\}_{j\in J})$, and 
 \item
 $x_{i}=y_{i}$ for every $i\in I$.
\end{itemize}

\item[(3) The case of $\SO_{2n+2}^{\mu\mu'}$ and $\SO_{2n}^{\mu}\times\SO_{2}^{\mu'}$]
Let $x\in \SO_{2n+2}^{\mu\mu',\srs}(F)$ and $(y, y')\in\SO_{2n}^{\mu,\srs}(F)\times\SO_{2}^{\mu',\srs}(F)$, and we assume that they correspond to 
\begin{align*}
&\bigl(I, \{F_{\pm i}\}_{i\in I}, \{F_{i}\}_{i\in I}, \{c_{i}\}_{i\in I}, \{x_{i}\}_{i\in I}\bigr),\\
&\bigl(J, \{F_{\pm j}\}_{j\in J}, \{F_{j}\}_{j\in J}, \{c_{j}\}_{j\in J}, \{y_{j}\}_{j\in J}\bigr) \text{ and}\\
&\bigl(J', \{F_{\pm j}\}_{j\in J'}, \{F_{j}\}_{j\in J'}, \{c_{j}\}_{j\in J'}, \{y_{j}\}_{j\in J'}\bigr)
\end{align*}
via the parametrizations in the previous section, respectively.
Moreover, we assume that these elements are very regular.
If $(y, y')$ is a norm of $x$, then the following holds:
\begin{itemize}
 \item
 $(I, \{F_{\pm i}\}_{i\in I}, \{F_{i}\}_{i\in I})$ is equal to the combined data
 \[
 (J, \{F_{\pm j}\}_{j\in J}, \{F_{j}\}_{j\in J})\sqcup(J', \{F_{\pm j}\}_{j\in J'}, \{F_{j}\}_{j\in J'}),
 \] and 
 \item
 $x_{i}=y_{i}$ for every $i\in I=J\sqcup J'$.
\end{itemize}
Conversely, if these conditions hold, then $(y, y')$ is a norm of exactly one of $x$ and $x'$, where $x'$ is an element of $\SO_{2n+2}^{\mu\mu',\srs}(F)$ which is conjugate to $x$ via an element of $\rmO_{2n+2}^{\mu\mu',\srs}(F)\smallsetminus\SO_{2n+2}^{\mu\mu',\srs}(F)$.
\end{description}

\subsection{Choice of Whittaker data and the invariant $\eta_{\mathbf{G}}$}\label{subsec:Whittaker}
To state Waldspurger's result on the transfer factor, we next recall some invariant attached to a Whittaker datum (note that the transfer factors depend on the choice of a Whittaker datum).
We first recall that a Whittaker datum of $\mathbf{G}$ is a pair of
\begin{itemize}
\item
a Borel subgroup $\mathbf{B}_{\mathbf{G}}$ defined over $F$, and
\item
a generic character of $\mathbf{U}_{\mathbf{G}}(F)$, where $\mathbf{U}_{\mathbf{G}}$ is the unipotent radical of $\mathbf{B}_{\mathbf{G}}$,
\end{itemize}
and that taking such data is equivalent to taking an $F$-splitting of $\mathbf{G}$ (see \cite[Section 5.3]{MR1687096}).

For a ($\theta$-stable) Whittaker datum of $\mathbf{G}$, the invariant $\eta_{\mathbf{G}}$ is defined as follows:

\begin{description}
\item[(1) The case of twisted $\GL_{N}$]
Let $(\mathbf{B}_{\GL_{N}},\lambda_{\GL_{N}})$ be a $\theta$-stable Whittaker datum of $\GL_{N}$.
Then this defines a $\theta$-stable $F$-splitting $(\mathbf{B}_{\GL_{N}}, \mathbf{T}_{\GL_{N}}, \{u_{a}\}_{a})$ of $\GL_{N}$, where $a$ is over the set of simple roots of the maximal torus $\mathbf{T}_{\GL_{N}}$ in $\GL_{N}$ with respect to the Borel subgroup $\mathbf{B}_{\GL_{N}}$.
We set $N_{\GL_{N}}:=\sum_{a}u_{a}$ and define a bilinear form on $F^{\oplus N}$ by
\[
F^{\oplus N}\times F^{\oplus N} \ra F \colon
(v,v')\mapsto J_{N}\bigl(v, N_{\GL_{N}}^{N-1}v'\bigr).
\]
Then this form is symmetric and equivalent to the sum of the $(N-1)$-dimensional null form and a $1$-dimensional form $\eta_{\GL_{N}}$.
We can regard $\eta_{\GL_{N}}$ as an element of $F^{\times}/F^{\times2}$, and denote it by $\eta_{\GL_{N}}$ again.

\item[(2) The case of symplectic group $\Sp_{2n}$]
Let $(\mathbf{B}_{\Sp_{2n}},\lambda_{\Sp_{2n}})$ be a Whittaker datum of $\Sp_{2n}$.
Then this defines an $F$-splitting $(\mathbf{B}_{\Sp_{2n}}, \mathbf{T}_{\Sp_{2n}}, \{u_{b}\}_{b})$ of $\Sp_{2n}$, and we get an element $\eta_{\Sp_{2n}}$ of $F^{\times}/F^{\times2}$ by using $N_{\Sp_{2n}}:=\sum_{b}u_{b}$ and $J_{2n}$ in the same manner as above.

\item[(3) The case of unramified special orthogonal group $\SO_{2n+2}^{(\ur)}$]
Let $(\mathbf{B}_{\SO_{2n+2}^{(\ur)}},\lambda_{\SO_{2n+2}^{(\ur)}})$ be a Whittaker datum of $\SO_{2n+2}^{(\ur)}$, and $(\mathbf{B}_{\SO_{2n+2}^{(\ur)}}, \mathbf{T}_{\SO_{2n+2}^{(\ur)}}, \{u_{b}\}_{b})$ the corresponding $F$-splitting of $\SO_{2n+2}^{(\ur)}$.
We set $N_{\SO_{2n+2}^{(\ur)}}:=\sum_{b}u_{b}$ and define a bilinear form on $F^{\oplus2n+2}$ by
\[
F^{\oplus2n+2}\times F^{\oplus2n+2} \ra F \colon
(v,v')\mapsto q_{\mu}\bigl(v, N_{\SO_{2n+2}^{(\ur)}}^{2n}v'\bigr).
\]
Then we get an element $\eta_{\SO_{2n+2}^{(\ur)}}$ of $F^{\times}/F^{\times2}$ in the same manner as above.
\end{description}

In this paper, we take the following Whittaker data for classical groups:
\begin{description}
\item[(1) The case of twisted $\GL_{N}$]
$\mathfrak{w}_{\GL_{N}}=(\mathbf{B}_{\GL_{N}}, \lambda_{\GL_{N}})$ consisting of
\begin{itemize}
 \item
 the upper-triangular Borel subgroup $\mathbf{B}_{\GL_{N}}$ of $\GL_{N}$, and
 \item
 the character $\lambda_{\GL_{N}}$ of the unipotent radical $U_{\GL_{N}}$ of $B_{\GL_{N}}$ defined by
 \[
 \lambda_{\GL_{N}}(x)=\psi(x_{1,2}+\cdots+x_{N-1,N}) \text{ for } x=(x_{ij}) \in U_{\GL_{N}},
 \]
 where $\psi$ is the fixed nontrivial additive character of $F$.
\end{itemize}
Note that $\mathfrak{w}_{\GL_{N}}$ is $\theta$-stable.

\item[(2) The case of symplectic group $\Sp_{2n}$]
$\mathfrak{w}_{\Sp_{2n}}=(\mathbf{B}_{\Sp_{2n}}, \lambda_{\Sp_{2n}})$ consisting of
\begin{itemize}
 \item
 the upper-triangular Borel subgroup $\mathbf{B}_{\Sp_{2n}}$ of $\Sp_{2n}$, and
 \item
 the character $\lambda_{\Sp_{2n}}$ of the unipotent radical $U_{\Sp_{2n}}$ of $B_{\Sp_{2n}}$ defined by
 \[
 \lambda_{\Sp_{2n}}(y)=\psi(y_{1,2}+\cdots+y_{n,n+1}) \text{ for } y=(y_{ij}) \in U_{\Sp_{2n}}.
 \]
\end{itemize}
Note that this choice is the same as in Section \ref{subsec:adjoint-Sp}.

\item[(3) The case of unramified special orthogonal group $\SO_{2n+2}^{(\ur)}$]
$\mathfrak{w}_{\SO_{2n+2}^{(\ur)}}=(\mathbf{B}_{\SO_{2n+2}^{(\ur)}}, \lambda_{\SO_{2n+2}^{(\ur)}})$ consisting of
\begin{itemize}
 \item
 the following ``upper-triangular'' Borel subgroup $\mathbf{B}_{\SO_{2n+2}^{(\ur)}}$ of $\SO_{2n+2}^{(\ur)}$:
\[
\mathbf{B}_{\SO_{2n+2}^{(\ur)}}:=\begin{pmatrix}
\ast&\cdots&\multicolumn{1}{c:}{\ast}&\ast&\multicolumn{1}{c:}{\ast}&&&\\
 &\ddots&\multicolumn{1}{c:}{\vdots}&\vdots&\multicolumn{1}{c:}{\vdots}&&\ast&\\
 &&\multicolumn{1}{c:}{\ast}&\ast&\multicolumn{1}{c:}{\ast}&&&\\
 \cdashline{1-8}
 &&\multicolumn{1}{c:}{}&\ast&\multicolumn{1}{c:}{\ast}&\ast&\cdots&\ast\\
 &&\multicolumn{1}{c:}{}&\ast&\multicolumn{1}{c:}{\ast}&\ast&\cdots&\ast\\
 \cdashline{1-8}
 &&\multicolumn{1}{c:}{}&&\multicolumn{1}{c:}{}&\ast&\cdots&\ast\\
 &&\multicolumn{1}{c:}{}&&\multicolumn{1}{c:}{}&&\ddots&\vdots\\
 &&\multicolumn{1}{c:}{}&&\multicolumn{1}{c:}{}&&&\ast
\end{pmatrix}
\]
  \item
 the character $\lambda_{\SO_{2n+2}^{(\ur)}}$ of the unipotent radical $U_{\SO_{2n+2}^{(\ur)}}$ of $B_{\SO_{2n+2}^{(\ur)}}$ defined by, for $y=(y_{ij}) \in U_{\SO_{2n+2}^{(\ur)}}$,
 \[
 \lambda_{\SO_{2n+2}^{(\ur)}}(y)=
 \begin{cases}
\psi(y_{1,2}+\cdots+y_{n-1,n}+y_{n-1,n+1}) & \text{if split,}\\
\psi(y_{1,2}+\cdots+y_{n-2,n-1}+2y_{n-1,n}) & \text{if non-split.}
 \end{cases}
 \]
\end{itemize}
\end{description}

\begin{prop}\label{prop:eta}
For the above Whittaker datum, the invariant $\eta_{\G}$ is given by
\[
\eta_{\mathbf{G}}
=
\begin{cases}
(-1)^{N-1} & \text{if }\mathbf{G}= \text{twisted }\GL_{N},\\
-1 & \text{if }\mathbf{G}= \Sp_{2n},\\
(-1)^{n}\cdot2 & \text{if }\mathbf{G}= \SO_{2n+2},\\
(-1)^{n}\cdot4 & \text{if }\mathbf{G}= \SO_{2n+2}^{\ur}.
\end{cases}
\]
\end{prop}

\begin{proof}
In the case of the twisted $\GL_{N}$, $N_{\GL_{N}}$ is given by
\[
\begin{pmatrix}
0&I_{N-1}\\
0&0
\end{pmatrix}.
\]
Thus we have
\[
N_{\GL_{N}}^{N-1}
=
\begin{pmatrix}
0&1\\
0\cdot I_{N-1}&0
\end{pmatrix},
\]
and the matrix representation of the bilinear form in the definition of $\eta_{\GL_{N}}$ with respect to the canonical basis of $F^{\oplus N}$ is given by
\[
J_{N}N_{\GL_{N}}^{N-1}=
\begin{pmatrix}
0\cdot I_{N-1}&0\\
0&(-1)^{N-1}
\end{pmatrix}.
\]
Hence we have $\eta_{\GL_{N}}=(-1)^{N-1}$. 

In the case of $\Sp_{2n}$, $N_{\Sp_{2n}}$ is given by the same matrix as $\GL_{2n}$.
Thus, by exactly the same computation, we get $\eta_{\Sp_{2n}}=-1$.

We next consider the case of $\SO_{2n+2}$.
In this case, we have
\[
N_{\SO_{2n+2}}
=
\begin{pmatrix}
0&I_{n-1}&&&&\\
&&1&1&&\\
&&&&-1&\\
&&&&-1&\\
&&&&&-I_{n-1}\\
&&&&&0
\end{pmatrix}.
\]
Thus we have
\[
N_{\SO_{2n+2}}^{2n}
=
\begin{pmatrix}
0&(-1)^{n}\cdot2\\
0\cdot I_{2n+1}&0
\end{pmatrix},
\]
and we get 
\[
J_{\mathbbm{1}}N_{\SO_{2n+2}}^{2n}
=
\begin{pmatrix}
0\cdot I_{2n+1}&0\\
0&(-1)^{n}\cdot2
\end{pmatrix}.
\]
Thus $\eta_{\SO_{2n+2}}$ is given by $(-1)^{n}\cdot 2$.

We finally compute the case of $\SO_{2n+2}^{\ur}$.
In this case we have
\[
N_{\SO_{2n+2}^{\ur}}
=
\begin{pmatrix}
0&I_{n-1}&&&&\\
&&2&0&&\\
&&&&-2&\\
&&&&0&\\
&&&&&-I_{n-1}\\
&&&&&0
\end{pmatrix}.
\]
Thus we have
\[
N_{\SO_{2n+2}^{\ur}}^{2n}
=
\begin{pmatrix}
0&(-1)^{n}\cdot4\\
0\cdot I_{2n+1}&0
\end{pmatrix},
\]
and we get 
\[
J_{\ur}N_{\SO_{2n+2}^{\ur}}^{2n}
=
\begin{pmatrix}
0\cdot I_{2n+1}&0\\
0&(-1)^{n}\cdot4
\end{pmatrix}.
\]
Thus $\eta_{\SO_{2n+2}^{\ur}}$ is given by $(-1)^{n}\cdot 4$.
\end{proof}

\subsection{Normalized transfer factors}
Recall that the normalized (with respect to the fixed Whittaker datum) transfer factors for $\gamma\in H^{\srs}$ and $\delta\in G^{\strs}$ are defined as the following product:
\[
\Delta_{\mathbf{H},\G}(\gamma, \delta):=\Delta_{\mathbf{H},\G}^{\mathrm{IV}}(\gamma,\delta)\cdot \Delta_{\mathbf{H},\G,\mathrm{IV}}(\gamma,\delta)\cdot \varepsilon_{\mathbf{H},\G},
\]
where, 
\begin{itemize}
\item
$\Delta_{\mathbf{H},\G}^{\mathrm{IV}}$ is the product of first, second, and third factors (see Section 4 of \cite{MR1687096} for details), 
\item
$\Delta_{\mathbf{H},\G,\mathrm{IV}}$ is the ratio of the $\theta$-twisted Weyl discriminants of $\delta$ to the standard Weyl discriminant of $\gamma$:
\[
\Delta_{\mathbf{H},\G,\mathrm{IV}}(\gamma, \delta):=
\frac{D_{\mathbf{G},\theta}(\delta)}{D_{\mathbf{H}}(\gamma)},
\]
and
\item
$\varepsilon_{\mathbf{H},\G}$ is the ratio of the root number of the root data of $\mathbf{G}$ to that of $\mathbf{H}$:
\[
\varepsilon_{\mathbf{H},\G}:=\frac{\varepsilon\bigl(\frac{1}{2},X^{\ast}(\mathbf{T}_{\mathbf{G}})^{\theta}\otimes_{\Z}\C,\psi\bigr)}{\varepsilon\bigl(\frac{1}{2},X^{\ast}(\mathbf{T}_{\mathbf{H}})\otimes_{\Z}\C,\psi\bigr)}.
\]
Here we took a $\theta$-stable Borel pair $(\mathbf{B}_{\mathbf{G}},\mathbf{T}_{\mathbf{G}})$ of $\mathbf{G}$ and a Borel pair $(\mathbf{B}_{\mathbf{H}},\mathbf{T}_{\mathbf{H}})$ of $\mathbf{H}$ which are defined over $F$.
Note that the resulting factor $\varepsilon_{\mathbf{H},\G}$ does not depend on such a choice.
\end{itemize}

\begin{prop}\label{prop:epsilon}
Let $\mu$ be a ramified quadratic character of $F^{\times}$.
The factor $\varepsilon_{\mathbf{H},\G}$ is given by
\[
\varepsilon_{\mathbf{H},\G}
=
\begin{cases}
1& \text{if } (\mathbf{G}, \mathbf{H})=(\mathrm{twisted }\GL_{2n+1}, \Sp_{2n}),\\
G(\omega_{0}, \psi)^{-1}\cdot q^{\frac{1}{2}}& \text{if } (\mathbf{G}, \mathbf{H})=(\mathrm{twisted }\GL_{2n}, \SO_{2n}^{\mu}),\\
G(\omega_{0}, \psi)^{-1}\cdot q^{\frac{1}{2}}& \text{if } (\mathbf{G}, \mathbf{H})=(\Sp_{2n}, \SO_{2n}^{\mu}),\\
\omega_{0}(-1)& \text{if } (\mathbf{G}, \mathbf{H})=(\SO_{2n+2}, \SO_{2n}^{\mu}\times\SO_{2}^{\mu}),\\
-\omega_{0}(-1)& \text{if } (\mathbf{G}, \mathbf{H})=(\SO_{2n+2}^{\ur}, \SO_{2n}^{\mu}\times\SO_{2}^{\mu\cdot\mu_{\ur}}).
\end{cases}
\]
Here $G(\omega_{0},\psi)$ is the Gauss sum attached to the nontrivial quadratic character $\omega_{0}$ of $ k^{\times}$ and the fixed nontrivial additive character $\psi$ of $k$ (see Proposition \ref{prop:Kl}).
\end{prop}
			
\begin{proof}
Let $\mathbf{G}_{0}$ be either $\GL_{N}$, $\Sp_{2n}$, or $\SO_{2n}^{\mu}$, where $\mu$ is any quadratic character of $F^{\times}$.
Then we have
\[
X^{\ast}(\mathbf{T}_{\mathbf{G}_{0}})^{\theta}\otimes_{\Z}\C
\cong
\begin{cases}
\mathbbm{1}^{\oplus \lfloor\frac{N}{2}\rfloor} & \text{if } \mathbf{G}_{0}=\GL_{N}, \\
\mathbbm{1}^{\oplus n} & \text{if } \mathbf{G}_{0}=\Sp_{2n}, \\
\mathbbm{1}^{\oplus n-1}\oplus \mu & \text{if } \mathbf{G}_{0}=\SO_{2n}^{\mu}
\end{cases}
\]
as representations of $\Gal(\overline{F}/F)$.
On the other hand, for a quadratic character $\mu$ of $F$, we have
\[
\varepsilon\biggl(\frac{1}{2},\, \mu,\, \psi\biggr)=
\begin{cases}
1 & \text{if $\mu=\mathbbm{1}$},\\
-1 & \text{if $\mu$ is unramified},\\
G(\omega_{0}, \psi)\cdot q^{-\frac{1}{2}}& \text{if $\mu$ is ramified}\\
\end{cases}
\]
(see, for example, \cite[23.5]{MR2234120}).
Here recall that $\omega_{0}$ is the unique nontrivial quadratic character of $k^{\times}$, and that we took $\psi$ to be of level one.

Then, by using the following well-known relation
\[
G(\omega_{0},\psi)^{2}=q\cdot \omega_{0}(-1),
\]
we get the results.
\end{proof}

\subsection{Waldspurger's formula for transfer factors}\label{subsec:Waldspurger}
Now we can state Waldspurger's formula:
\begin{prop}[{\cite[1.10 Proposition]{MR2672539}}]\label{prop:Walds}
Let $y$ be an element of ${H}^{\srs}$ and 
\[
(I, \{F_{\pm i}\}_{i\in I}, \{F_{i}\}_{i\in I}, \{c'_{i}\}_{i\in I}, \{y_{i}\}_{i\in I})
\]
its corresponding data.
We put $I^{\ast}:=\{i\in I \mid \text{$F_{i}$ is a field}\}$.
When $\mathbf{H}$ is an even special orthogonal group, we assume that $y$ is very regular.
We define a polynomial $P(T)\in \ol{F}[T]$ by
\[
P(T):=\prod_{i\in I}\prod_{\phi\in\Hom_{F}(F_{i}, \ol{F})} \bigl(T-\phi(y_{i})\bigr),
\]
where $\Hom_{F}(F_{i}, \ol{F})$ is the set of homomorphisms from $F_{i}$ to $\ol{F}$ as $F$-algebras.
\begin{description}
\item[(0) The case where $\mathbf{G}= \text{twisted }\GL_{2n+1}$]
 Let $x$ be an element of $\GL_{2n+1}^{\strs}(F)$ such that $y$ is a norm of $x$.
 Then we have
 \[
  \Delta_{\mathbf{H},\G}^{\mathrm{IV}}(y,x)=1.
 \]
 
\item[(1) The case where $\mathbf{G}= \text{twisted }\GL_{2n}$]
 Let $x$ be an element of $\GL_{2n}^{\strs}(F)$ such that $y$ is a norm of $x$, 
 and 
 \[
(I, \{F_{\pm i}\}_{i\in I}, \{F_{i}\}_{i\in I}, \{x_{i}\}_{i\in I})
 \]
 its corresponding data.
 For each $i\in I^{\ast}$, we put 
 \[
 C_{i}:= \eta_{\G}x_{i}^{-1}P'(y_{i})P(-1)y_{i}^{1-n}(1+y_{i}).
 \]
 Then $C_{i}$ belongs to $F_{\pm i}^{\times}$ and we have 
 \[
 \Delta_{\mathbf{H},\G}^{\mathrm{IV}}(y,x)=\prod_{i\in I^{\ast}} \mathrm{sgn}_{F_{i}/F_{\pm i}} (C_{i}),
 \]
where $\mathrm{sgn}_{F_{i}/F_{\pm i}}$ is the character of $F_{\pm i}^{\times}$ corresponding to the quadratic extension $F_{i}/F_{\pm i}$ via the local class field theory, namely
 \[
 \mathrm{sgn}_{F_{i}/F_{\pm i}}\colon 
 F_{\pm i}^{\times} \twoheadrightarrow F_{\pm i}^{\times}\big/ \Nr_{F_{i}/F_{\pm i}}(F_{i}^{\times})\cong\{\pm1\}\hookrightarrow\C^{\times}.
 \]

\item[(2) The case where $\mathbf{G}=\Sp_{2n}$]
 Let $x$ be an element of $\Sp_{2n}^{\srs}(F)$ such that $y$ is a norm of $x$, 
 and 
 \[
(I, \{F_{\pm i}\}_{i\in I}, \{F_{i}\}_{i\in I}, \{c_{i}\}_{i\in I}, \{x_{i}\}_{i\in I})
 \]
 its corresponding data.
 For each $i\in I^{\ast}$, we put 
 \[
 C_{i}:= -\eta_{\G}c_{i}P'(y_{i})P(-1)y_{i}^{1-n}.
 \]
 Then $C_{i}$ belongs to $F_{\pm i}^{\times}$ and we have 
 \[
 \Delta_{\mathbf{H},\G}^{\mathrm{IV}}(y,x)=\prod_{i\in I^{\ast}} \mathrm{sgn}_{F_{i}/F_{\pm i}}(C_{i}).
 \]

\item[(3) The case where $\mathbf{G}=\SO_{2n+2}^{(\ur)}$]
  Let $x$ be an element of $\SO_{2n+2}^{(\ur),\srs}(F)$ such that $y$ is a norm of $x$, 
 and 
 \[
(I, \{F_{\pm i}\}_{i\in I}, \{F_{i}\}_{i\in I}, \{c_{i}\}_{i\in I}, \{x_{i}\}_{i\in I})
 \]
 its corresponding data.
 For each $i\in I^{\ast}$, we put 
 \[
 C_{i}:= 2\eta_{\G}c_{i}P'(y_{i})P(-1)y_{i}^{-n}(1+y_{i})(y_{i}-1)^{-1}.
 \]
 Then $C_{i}$ belongs to $F_{\pm i}^{\times}$ and we have 
 \[
 \Delta_{\mathbf{H},\G}^{\mathrm{IV}}(y,x)=\prod_{i\in I_{2n}^{\mu}\cap I^{\ast}} \mathrm{sgn}_{F_{i}/F_{\pm i}}(C_{i}),
 \] 
where $I_{2n}^{\mu}$ is the set of indices belonging to the $\SO_{2n}^{\mu}$-part of the data corresponding to $y$.
\end{description}
\end{prop}

\begin{rem}
We note that $c_{i}$ in the equalities of (2) and (3) is the data of $x$, and that $c'_{i}$ of $y$ does not appear.
In general, the transfer factor is invariant under the stable conjugacy of the endoscopic group (see \cite[Lemma 5.1.B]{MR1687096}).
\end{rem}

Note that, in the case of $\G=\GL_{2n+1}$ and $\mathbf{H}=\Sp_{2n}$, the normalized transfer factor is trivial:
\begin{prop}\label{prop:triviality}
For $\gamma\in \Sp_{2n}^{\srs}(F)$ and $\delta\in \GL_{2n+1}^{\strs}(F)$ such that $\gamma$ is a norm of $\delta$, we have
\[
\Delta_{\Sp_{2n}, \GL_{2n+1}}(\gamma, \delta)=1.
\]
\end{prop}

\begin{proof}
By Propositions \ref{prop:epsilon} and \ref{prop:Walds} (0), it suffices to check the triviality of $\Delta_{\mathrm{IV}}$.
This can be done by comparing the definitions of the norm correspondence and the Weyl discriminants.
See \cite[Lemma 4.10]{MR3904769} for example.
Note that, in \cite{MR3904769}, the triviality of $\Delta_{\mathrm{IV}}$ is proved for the case where $\G=\GL_{2n}$ and $\mathbf{H}=\SO_{2n+1}$.
However, the same strategy works in our case.
\end{proof}

\subsection{Affine generic data of $\SO_{2n}^{\mu}$}\label{subsec:tran-SO-ram}
Let $u\in k^{\times}$ and $\pm\xi_{u}$ the following data:
\[
\pm\xi_{u}:=\Biggl(F[\varphi_{u}^{2}],\, F[\varphi_{u}],\, c'_{u}:=\frac{(-1)^{n}}{2n\varpi u},\, \pm y_{u}:=\pm\frac{1+\varphi_{u}}{1-\varphi_{u}}\Biggr).
\]
Here, an index set $I$ is a singleton and we omitted it from the above data, and $\varphi_{u}$ is a fixed root of the polynomial $T^{2n}-\varpi u \in F[T]$.
We write $\tau$ for the unique nontrivial element of $\Gal(F[\varphi_{u}]/F[\varphi_{u}^{2}])$.
Note that we have $\tau(\varphi_{u})=-\varphi_{u}$.

\begin{lem}\label{lem:type}
The data $\pm\xi_{u}$ correspond to elements of $\SO_{2n}^{\mu}(F)$ for an exactly one quadratic character $\mu$ of $F^{\times}$.
Moreover, the character $\mu$ is ramified and corresponds to the following quadratic extension of $F$:
\[
E_{\mu}=F\Bigl(\sqrt{(-1)^{n-1}\varpi u}\Bigr).
\]
\end{lem}

\begin{proof}
We first recall that, from the data $\pm\xi_{u}$, the symmetric bilinear form $q_{F[\varphi_{u}]}$ is defined by
\begin{align*}
q_{F[\varphi_{u}]} \colon F[\varphi_{u}]\times F[\varphi_{u}] &\rightarrow F \\
(w, w')&\mapsto \Tr_{F[\varphi_{u}]/F}\bigl(\tau(w)\cdot w'\cdot c'_{u}\bigr).
\end{align*}
We compute this.
Since elements of $\Hom_{F}(F[\varphi_{u}], \overline{F})$ are given by
\[
\varphi_{u} \mapsto \varphi_{u}\zeta_{2n}^{j}
\]
for $1\leq j\leq2n$, we have
\[
\Tr_{F[\varphi_{u}]/F}(\varphi_{u}^{i})
=\sum_{j=0}^{2n-1} (\varphi_{u}\zeta_{2n}^{j})^{i}
=\varphi_{u}^{i}\sum_{j=0}^{2n-1} \zeta_{2n}^{ij}
=\begin{cases}
2n\varpi u&\text{if }i=2n,\\
0&\text{if }1\leq i<2n.
\end{cases}
\]
Here $\zeta_{2n}$ is a primitive $2n$-th root of unity.
Therefore, if we choose a basis of $F[\varphi_{u}]$ over $F$ to be 
\[
\{\varphi_{u}^{2n-1}, \ldots, \varphi_{u}^{n+1}, \varphi_{u}^{n}, \varphi_{u}^{2n}, -\varphi_{u}^{n-1}, \ldots, (-1)^{n-1}\varphi_{u}\},
\]
then the representation matrix of $q_{F[\varphi_{u}]}$ is given by
\[
\begin{pmatrix}
 &&&&&&&1\\
 &&&&&&\adots&\\
 &&&&&1&&\\
 &&&1&&&&\\
 &&&&(-1)^{n}\varpi u&&&\\
 &&1&&&&&\\
 &\adots&&&&&&\\
 1&&&&&&&
 \end{pmatrix}.
\]
\end{proof}

We compute a matrix representation of the $\Sigma(\SO_{2n}^{\mu})$-orbit of the elements of $\SO_{2n}^{\mu}(F)$ corresponding to the data $\xi_{u}$.
If we choose a basis of $F[\varphi_{u}]$ over $F$ to be 
\[
\{\varphi_{u}^{2n-1}, \ldots, \varphi_{u}^{n+1}, \varphi_{u}^{n}, \varphi_{u}^{2n}, \varphi_{u}^{n-1}, \ldots, \varphi_{u}\},
\]
then the matrix representation of $\xi_{u}$ with respect to this basis is given by
\[
\frac{1}{1-\varpi u}
\begin{pmatrix}
 1+\varpi u&&\multicolumn{1}{c:}{2}&2&\multicolumn{1}{c:}{2\varpi u}&&&\\
 &\ddots&\multicolumn{1}{c:}{}&\vdots&\multicolumn{1}{c:}{\vdots}&&2&\\
 2\varpi u&&\multicolumn{1}{c:}{1+\varpi u}&2&\multicolumn{1}{c:}{2\varpi u}&&&\\
 \cdashline{1-10}
 2\varpi u&\cdots&\multicolumn{1}{c:}{2\varpi u}&1+\varpi u&\multicolumn{1}{c:}{2\varpi u}&2&\cdots&2\\
 2&\cdots&\multicolumn{1}{c:}{2}&2&\multicolumn{1}{c:}{1+\varpi u}&2&\cdots&2\\
 \cdashline{1-10}
 &&\multicolumn{1}{c:}{}&2\varpi u&\multicolumn{1}{c:}{2\varpi u}&1+\varpi u&&2\\
 &2\varpi u&\multicolumn{1}{c:}{}&\vdots&\multicolumn{1}{c:}{\vdots}&&\ddots&\\
 &&\multicolumn{1}{c:}{}&2\varpi u&\multicolumn{1}{c:}{2\varpi u}&2\varpi u&&1+\varpi u\\
\end{pmatrix}.
\]
Here, in order to simplify this expression, we introduce the following matrices:
\[
A_{n}=
\begin{pmatrix}
1+\varpi u&&2\\
&\ddots&\\
2\varpi u&&1+\varpi u
\end{pmatrix} \in M_{n-1,n-1}(F),
\]
\[
B_{n}=
\begin{pmatrix}
2&\cdots&2\\
\vdots&\ddots&\vdots\\
2&\cdots&2
\end{pmatrix} \in M_{n-1,n-1}(F),
\]
\[
X_{n}:=
\begin{pmatrix}
1\\
\vdots\\
1
\end{pmatrix} \in M_{n-1,1}(F),\quad\text{and}\quad
Y_{n}:=
\begin{pmatrix}
1&\cdots&1
\end{pmatrix} \in M_{1,n-1}(F).
\]
Then the above matrix is expressed as follows:
\[
\frac{1}{1-\varpi u}
\begin{pmatrix}
A_{n}&2X_{n}&2\varpi uX_{n}&B_{n}\\
2\varpi uY_{n}&1+\varpi u&2\varpi u&2Y_{n}\\
2Y_{n}&2&1+\varpi u&2Y_{n}\\
\varpi uB_{n}&2\varpi uX_{n}&2\varpi uX_{n}&A_{n}
\end{pmatrix}.
\]

By the computation in the proof of Lemma \ref{lem:type}, 
if we take a basis of $F[\varphi_{u}]$ to be
\[
\{\varphi_{u}^{2n-1}, \ldots, \varphi_{u}^{n+1}, \varphi_{u}^{n}, v\varphi_{u}^{2n}, -\varphi_{u}^{n-1}, \ldots, (-1)^{n-1}\varphi_{u}\},
\]
then the representation matrix of the quadratic form $q_{F[\varphi_{u}]}$ is given by $J_{\mu}$.
Here $v$ is an element of $k^{\times}$ satisfying 
\[
\begin{cases}
v^{2} (-1)^{n-1}u =1 & \text{if } (-1)^{n-1}u\in k^{\times2},\\
v^{2} (-1)^{n-1}u =\epsilon & \text{if } (-1)^{n-1}u\notin k^{\times2}.
\end{cases}
\]
Thus, if we put
\[
Q:=\diag\bigl(\underbrace{1, \ldots, 1}_{n}, v^{-1}, \underbrace{-1, 1, \ldots, (-1)^{n-1}}_{n-1}\bigr),
\]
then 
\[
g_{u}^{\SO_{2n}^{\mu}}
:=
\frac{1}{1-\varpi u}\cdot
Q
\begin{pmatrix}
A_{n}&2X_{n}&2\varpi uX_{n}&B_{n}\\
2\varpi uY_{n}&1+\varpi u&2\varpi u&2Y_{n}\\
2Y_{n}&2&1+\varpi u&2Y_{n}\\
\varpi uB_{n}&2\varpi uX_{n}&2\varpi uX_{n}&A_{n}
\end{pmatrix}
Q^{-1}
\]
is an element of $\SO_{2n}^{\mu}$ which corresponds to the data $\xi_{u}$. 
In particular, we get the following:
\begin{prop}\label{prop:affgen-SOram}
The element $g_{u}^{\SO_{2n}^{\mu}}$ belongs to the pro-unipotent radical $I_{\SO_{2n}^{\mu}}^{+}$ of the Iwahori subgroup and its simple affine components are given by
\[
(2, \ldots, 2, 2v^{-1}).
\]
\end{prop}
Here note that, for each $u$, there exists two candidates of $v\in k^{\times}$ satisfying the above condition.
In fact, there exists two stable $\SO_{2n}^{\mu}$-conjugacy classes corresponding to the data $\xi_{u}$, and the other one is represented by
\[
w_{\mu}g_{u}^{\SO_{2n}^{\mu}}w_{\mu}^{-1}
\]
(recall that 
\[
w_{\mu}
=
\begin{pmatrix}
I_{n-1}&&&\\
&1&&\\
&&-1&\\
&&&I_{n-1}
\end{pmatrix}
\in
\rmO_{2n}^{\mu}(F)\smallsetminus\SO_{2n}^{\mu}(F)).
\]

\begin{lem}\label{lem:pol+}
Let $P_{\xi_{u}}(T)$ be the polynomial corresponding to $\xi_{u}$ in Proposition \ref{prop:Walds}.
Namely, we set
\[
P_{\xi_{u}}(T)=\prod_{j=1}^{2n} \Biggl(T-\frac{1+\varphi_{u}\zeta_{2n}^{j}}{1-\varphi_{u}\zeta_{2n}^{j}}\Biggr).
\]
Then we have
\[
P'_{\xi_{u}}(y_{u}) \cdot P_{\xi_{u}}(-1)\equiv
\frac{2\varphi_{u}^{2n-1}}{(1-\varphi_{u})^{2n-1}} \cdot\frac{2n}{1+\varphi_{u}}
\mod \Nr_{F[\varphi_{u}]/F[\varphi_{u}^2]}\bigl(F[\varphi_{u}]^{\times}\bigr).
\]
\end{lem}

\begin{proof}
We have
\begin{align*}
P_{\xi_{u}}'(y_{u})&=\prod_{j=1}^{2n-1} \Biggl(\frac{1+\varphi_{u}}{1-\varphi_{u}}-\frac{1+\varphi_{u}\zeta_{2n}^{j}}{1-\varphi_{u}\zeta_{2n}^{j}}\Biggr)\\
&=\prod_{j=1}^{2n-1} \frac{2\varphi_{u}\bigl(1-\zeta_{2n}^{j}\bigr)}{(1-\varphi_{u})\bigl(1-\varphi_{u}\zeta_{2n}^{j}\bigr)}=\frac{(2\varphi_{u})^{2n-1}}{(1-\varphi_{u})^{2n-1}} \cdot\frac{2n}{\prod_{j=1}^{2n-1}\bigl(1-\varphi_{u}\zeta_{2n}^{j}\bigr)}.
\end{align*}
On the other hand, we have
\[
P_{\xi_{u}}(-1)
=\prod_{j=1}^{2n} \Biggl(-1-\frac{1+\varphi_{u}\zeta_{2n}^{j}}{1-\varphi_{u}\zeta_{2n}^{j}}\Biggr)
=\frac{(-2)^{2n}}{\prod_{j=1}^{2n}\bigl(1-\varphi_{u}\zeta_{2n}^{j}\bigr)}.
\]
Thus we get
\begin{align*}
P'_{\xi_{u}}(y_{u}) \cdot P_{\xi_{u}}(-1)
&=\frac{(2\varphi_{u})^{2n-1}}{(1-\varphi_{u})^{2n-1}} \cdot\frac{2n}{\prod_{j=1}^{2n-1}\bigl(1-\varphi_{u}\zeta_{2n}^{j}\bigr)}\cdot \frac{(-2)^{2n}}{\prod_{j=1}^{2n}\bigl(1-\varphi_{u}\zeta_{2n}^{j}\bigr)}\\
&\equiv \frac{2\varphi_{u}^{2n-1}}{(1-\varphi_{u})^{2n-1}} \cdot\frac{2n}{1+\varphi_{u}}.
\end{align*}
\end{proof}

\begin{lem}\label{lem:pol-}
Let $P_{-\xi_{u}}(T)$ be the polynomial corresponding to $-\xi_{u}$ in Proposition \ref{prop:Walds}.
Namely, we set
\[
P_{-\xi_{u}}(T)=\prod_{j=1}^{2n} \Biggl(T+\frac{1+\varphi_{u}\zeta_{2n}^{j}}{1-\varphi_{u}\zeta_{2n}^{j}}\Biggr).
\]
Then we have
\[
P'_{-\xi_{u}}(-y_{u}) \cdot P_{-\xi_{u}}(-1)\equiv
\frac{2\varphi_{u}^{-1}}{(1-\varphi_{u})^{2n-1}} \cdot\frac{2n}{1+\varphi_{u}}
\mod \Nr_{F[\varphi_{u}]/F[\varphi_{u}^2]}\bigl(F[\varphi_{u}]^{\times}\bigr).
\]
\end{lem}

\begin{proof}
We have
\[
P_{-\xi_{u}}'(-y_{u})
= (-1)^{2n-1} P'_{\xi_{u}}(y_{u})
= -P'_{\xi_{u}}(y_{u}).
\]
On the other hand, we have
\[
P_{-\xi_{u}}(-1)
=\prod_{j=1}^{2n} \frac{2\varphi_{u}\zeta_{2n}^{j}}{1-\varphi_{u}\zeta_{2n}^{j}}
=\frac{(2\varphi_{u})^{2n}\cdot\zeta_{2n}^{\frac{2n(2n-1)}{2}}}{\prod_{j=1}^{2n}\bigl(1-\varphi_{u}\zeta_{2n}^{j}\bigr)}
=-\varphi_{u}^{2n}P_{\xi_{u}}(-1).
\]
Thus we get
\[
P'_{-\xi_{u}}(-y_{u}) \cdot P_{-\xi_{u}}(-1)
=\varphi_{u}^{2n}P'_{\xi_{u}}(y_{u}) \cdot P_{\xi_{u}}(-1)
\equiv\frac{2\varphi_{u}^{-1}}{(1-\varphi_{u})^{2n-1}} \cdot\frac{2n}{1+\varphi_{u}}.
\]
\end{proof}

\begin{lem}\label{lem:discSOram}
We have
\[
D_{\SO_{2n}^{\mu}}\bigl(\pm g_{u}^{\SO_{2n}^{\mu}}\bigr)
=
\prod_{\begin{subarray}{c}k\in\{\pm1\}\\ l\in\{\pm1\}\end{subarray}}
\prod_{1\leq i<j \leq n}
\biggl|\biggl(\frac{1+\varphi_{u}\zeta_{2n}^{i}}{1-\varphi_{u}\zeta_{2n}^{i}}\biggr)^{k} \biggl(\frac{1+\varphi_{u}\zeta_{2n}^{j}}{1-\varphi_{u}\zeta_{2n}^{j}}\biggr)^{l}-1 \biggr|^{\frac{1}{2}}.
\]
\end{lem}

\begin{proof}
Since the ratio of $g_{u}^{\SO_{2n}^{\mu}}$ to $-g_{u}^{\SO_{2n}^{\mu}}$ belongs to the center of $\SO_{2n}^{\mu}$, we have
\[
D_{\SO_{2n}^{\mu}}\bigl(g_{u}^{\SO_{2n}^{\mu}}\bigr)=D_{\SO_{2n}^{\mu}}\bigl(-g_{u}^{\SO_{2n}^{\mu}}\bigr).
\]

To compute $D_{\SO_{2n}^{\mu}}(g_{u}^{\SO_{2n}^{\mu}})$, we fix an isomorphism $\SO_{2n,\ol{F}}^{\mu}\cong\SO_{2n,\ol{F}}$ and take an element 
\[
t=\diag(t_{1},\ldots,t_{n},t_{n}^{-1},\ldots,t_{1}^{-1}) \in\SO_{2n}(\ol{F})
\]
which is conjugate to $g_{u}^{\SO_{2n}^{\mu}}$ and lies in the diagonal maximal torus of $\SO_{2n}(\ol{F})$.
Then, since the set of absolute roots of $\SO_{2n}$ is given by
\[
\{\pm e_{i}\pm e_{j} \mid 1\leq i< j\leq n\},
\]
we have
\[
D_{\SO_{2n}^{\mu}}(g_{u}^{\SO_{2n}^{\mu}})
= \prod_{\alpha\in\{\pm e_{i}\pm e_{j}\}} \bigl|\alpha(t)-1\bigr|^{\frac{1}{2}}.
\]
On the other hand, by the definition of $g_{u}^{\SO_{2n}^{\mu}}$, we have
\[
\bigl\{t_{i}^{\pm1}\mid i=1,\ldots,n\bigr\}
=
\Biggl\{\frac{1+\varphi_{u}\zeta_{2n}^{j}}{1-\varphi_{u}\zeta_{2n}^{j}} \,\Bigg\vert\, j=1,\ldots,2n \Biggr\}.
\]
Thus we get the assertion.
\end{proof}

In the following, we write shortly $\mathbb{D}$ for this value:
\[
\mathbb{D}
:=
\prod_{\begin{subarray}{c}k\in\{\pm1\}\\ l\in\{\pm1\}\end{subarray}}
\prod_{1\leq i<j \leq n}
\biggl|\biggl(\frac{1+\varphi_{u}\zeta_{2n}^{i}}{1-\varphi_{u}\zeta_{2n}^{i}}\biggr)^{k} \biggl(\frac{1+\varphi_{u}\zeta_{2n}^{j}}{1-\varphi_{u}\zeta_{2n}^{j}}\biggr)^{l}-1 \biggr|^{\frac{1}{2}}.
\]

\subsection{The case of $(\text{twisted }\GL_{2n}, \SO_{2n}^{\mu})$}\label{subsec:tran-GL}
For $u\in k^{\times}$, we set
\begin{align*}
\xi^{\GL_{2n}}_{u} &:=\Biggl(F[\varphi_{u}^{2}],\, F[\varphi_{u}],\, x_{u}=\frac{\varphi_{u}^{-1}}{2n\varpi u}\cdot(1+\varphi_{u})\Biggr), \text{ and}\\
\tilde{\xi}^{\GL_{2n}}_{u} &:=\Biggl(F[\varphi_{u}^{2}],\, F[\varphi_{u}],\, \varphi_{u}x_{u}=\frac{1}{2n\varpi u}\cdot(1+\varphi_{u})\Biggr).
\end{align*}
Then $\xi_{u}$ (resp.\ $-\xi_{u}$) corresponds to $\xi^{\GL_{2n}}_{u}$ (resp.\ $\tilde{\xi}^{\GL_{2n}}_{u}$) in the sense of Section \ref{subsec:norm}.

We compute the $\theta$-conjugacy classes corresponding to $\xi_{u}^{\GL_{2n}}$ and $\tilde{\xi}^{\GL_{2n}}_{u}$.
Let us compute the matrix representations of the bilinear forms
\begin{align*}
F[\varphi_{u}]\times F[\varphi_{u}]&\rightarrow F\\
(v,v')&\mapsto \Tr_{F[\varphi_{u}]/F} \bigl(\tau(v)v'x_{u}\bigr), \text{ and}\\
(v,v')&\mapsto \Tr_{F[\varphi_{u}]/F} \bigl(\tau(v)v'\varphi_{u}x_{u}\bigr).
\end{align*}
We take a basis of $F[\varphi_{u}]$ to be
\[
\bigl\{\varphi_{u}^{2n}, \ldots, \varphi_{u}\bigr\}.
\]
Then the matrix representations of above two bilinear forms are given by 
\[
J_{2n}\begin{pmatrix}
1&1&&\\
&\ddots&\ddots&\\
&&\ddots&1\\
\varpi u&&&1
\end{pmatrix}
=J_{2n}(1+\varphi_{u}^{\GL_{2n}}),\quad\text{and}
\]
\[
J_{2n}
\begin{pmatrix}
0&1&1&&\\
&\ddots&\ddots&\ddots&\\
&&\ddots&\ddots&1\\
\varpi u&&&\ddots&1\\
\varpi u&\varpi u&&&0
\end{pmatrix}
=J_{2n}\varphi_{u}^{\GL_{2n}}(1+\varphi_{u}^{\GL_{2n}}).
\]
Therefore the corresponding $\theta$-conjugacy classes are represented by
\[
\theta(1+\varphi_{u}^{\GL_{2n}}),\quad\text{and}\quad
\theta\bigl(\varphi_{u}^{\GL_{2n}}(1+\varphi_{u}^{\GL_{2n}})\bigr)
\]
(recall the correspondence between elements of $\GL_{2n}(F)$ and bilinear forms on $F^{\oplus2n}$ in Section \ref{subsec:param-conj}).
Since $g\in\GL_{2n}(F)$ is $\theta$-conjugate to $\theta(g)\in\GL_{2n}(F)$ (note that $\theta(g)=g^{-1}\cdot g \cdot \theta(g)$), if we put
\begin{align*}
g_{u}^{\GL_{2n}}&:=1+\varphi_{u}^{\GL_{2n}},\quad\text{and}\\
\tilde{g}_{u}^{\GL_{2n}}&:=\varphi_{u}^{\GL_{2n}}(1+\varphi_{u}^{\GL_{2n}}),
\end{align*}
then these are strongly $\theta$-regular $\theta$-semisimple elements of $\GL_{2n}(F)$ corresponding to the data $\xi_{u}$ and $\tilde{\xi}_{u}$.
Note that we have the following:
\begin{prop}\label{prop:affgen-GL}
The element $g^{\GL_{2n}}_{u}$ is an affine generic element of $\GL_{2n}(F)$ with simple affine components 
\[
(\underbrace{1,\ldots,1}_{2n-1},u).
\]
\end{prop}

\begin{prop}\label{prop:123GL}
Let $(\gamma, \delta)$ be either $(g_{u}^{\SO_{2n}^{\mu}}, g_{u}^{\GL_{2n}})$ or $(-g_{u}^{\SO_{2n}^{\mu}}, \tilde{g}_{u}^{\GL_{2n}})$.
Then we have
\[
\Delta_{\SO_{2n}^{\mu},\GL_{2n}}^{\mathrm{IV}}(\gamma,\delta)=
\omega_{0}(-1).
\]
\end{prop}

\begin{proof}
By Propositions \ref{prop:Walds} and \ref{prop:eta}, we have
\[
\Delta_{\SO_{2n}^{\mu},\GL_{2n}}^{\mathrm{IV}}\bigl(g_{u}^{\SO_{2n}^{\mu}}, g_{u}^{\GL_{2n}}\bigr)
=\sgn_{F[\varphi_{u}]/F[\varphi_{u}^2]}
\Bigl(
-x_{u}^{-1}\cdot P_{\xi_{u}}'(y_{u})\cdot P_{\xi_{u}}(-1) \cdot y_{u}^{1-n}(1+y_{u})
\Bigr).
\]
By Lemma \ref{lem:pol+}, we have
\begin{align*}
&\phantom{{}\equiv{}} {-x_{u}^{-1}}\cdot P'_{\xi_{u}}(y_{u})\cdot P_{\xi_{u}}(-1) \cdot y_{u}^{1-n}(1+y_{u})\\
&\equiv
-2n\varpi u\cdot\frac{\varphi_{u}}{1+\varphi_{u}}\cdot \frac{2\varphi_{u}^{2n-1}}{(1-\varphi_{u})^{2n-1}} \cdot\frac{2n}{1+\varphi_{u}}
 \cdot \frac{(1+\varphi_{u})^{1-n}}{(1-\varphi_{u})^{1-n}}\cdot\frac{2}{1-\varphi_{u}}\\
&\equiv
-1.
\end{align*}
Thus we get
\[
\Delta_{\SO_{2n}^{\mu},\GL_{2n}}^{\mathrm{IV}}\bigl(g_{u}^{\SO_{2n}^{\mu}}, g_{u}^{\GL_{2n}}\bigr)
=\sgn_{F[\varphi_{u}]/F[\varphi_{u}^2]}(-1).
\]
By noting that the quadratic extension $F[\varphi_{u}]/F[\varphi_{u}^2]$ is ramified, the restriction of the character $\sgn_{F[\varphi_{u}]/F[\varphi_{u}^2]}$ to $\mcO_{F[\varphi_{u}^2]}^{\times}$ is given by the lift of the nontrivial quadratic character of the residue field of $F[\varphi_{u}^2]$.
By also noting that the extension $F[\varphi_{u}^2]/F$ is totally ramified, the residue field of $F[\varphi_{u}^2]$ is nothing but the residue field $k$ of $F$.
Thus we get $\sgn_{F[\varphi_{u}]/F[\varphi_{u}^2]}(-1)=\omega_{0}(-1)$.

We next compute $\Delta_{\SO_{2n}^{\mu},\GL_{2n}}^{\mathrm{IV}}\bigl(-g_{u}^{\SO_{2n}^{\mu}}, \tilde{g}_{u}^{\GL_{2n}}\bigr)$.
Similarly to the above case, by Lemma \ref{lem:pol-}, we have
\begin{align*}
&\phantom{{}\equiv{}} {-(\varphi_{u}x_{u})^{-1}}\cdot P'_{-\xi_{u}}(-y_{u})\cdot P_{-\xi_{u}}(-1) \cdot (-y_{u})^{1-n}(1-y_{u})\\
&\equiv
-2n\varpi u\cdot\frac{1}{1+\varphi_{u}}
\cdot\frac{2\varphi_{u}^{-1}}{(1-\varphi_{u})^{2n-1}} \cdot\frac{2n}{1+\varphi_{u}}
\cdot(-1)^{1-n} \cdot \frac{(1+\varphi_{u})^{1-n}}{(1-\varphi_{u})^{1-n}}\cdot\frac{-2\varphi_{u}}{1-\varphi_{u}}\\
&\equiv
\varpi u\cdot(-1)^{1-n}\\
&\equiv
-1.
\end{align*}
Hence we have
\[
\Delta_{\SO_{2n}^{\mu},\GL_{2n}}^{\mathrm{IV}}\bigl(-g_{u}^{\SO_{2n}^{\mu}}, \tilde{g}_{u}^{\GL_{2n}}\bigr)
=\omega_{0}(-1).
\]
\end{proof}

\begin{lem}\label{lem:discGL}
Let $\delta$ be either $g_{u}^{\GL_{2n}}$ or $\tilde{g}_{u}^{\GL_{2n}}$.
Then we have
\[
D_{\GL_{2n},\theta}(\delta)=
\begin{cases}
q^{-\frac{1}{2}}\cdot\mathbb{D} & \text{if }\delta=g_{u}^{\GL_{2n}},\\
\mathbb{D} & \text{if }\delta=\tilde{g}_{u}^{\GL_{2n}}.
\end{cases}
\]
\end{lem}

\begin{proof}
Let $s\in \GL_{2n}(\ol{F})$ be a diagonal matrix $\diag(s_{1},\ldots,s_{2n})$ which is $\theta$-conjugate to $\delta$ in $\GL_{2n}(\ol{F})$.
Let $t_{1},\ldots,t_{n}\in\overline{F}{}^{\times}$ satisfying
\[
\{t_{i}^{\pm1}\mid i=1,\ldots,n\}
=
\biggl\{
\frac{s_{i}}{s_{2n+1-i}}
\,\bigg\vert\,
i=1,\ldots,2n
\biggr\}.
\]
Then, by combining \cite[Lemma 4.10]{MR3904769} with \cite[Lemma 4.1]{MR3904769}, we have
\[
D_{\GL_{2n},\theta}(\delta)
=
\prod_{\begin{subarray}{c}k\in\{\pm1\}\\ l\in\{\pm1\}\end{subarray}}
\prod_{1\leq i<j \leq n}
|t_{i}^{k}t_{j}^{l}-1|^{\frac{1}{2}}
\prod_{k\in\{\pm1\}}
\prod_{1\leq i \leq n}
|t_{i}^{k}-1|^{\frac{1}{2}}.
\]
More precisely, in \cite[Lemma 4.10]{MR3904769}, for every strongly $\theta$-regular $\theta$-semisimple element $\delta$ of $\GL_{2n}$, we proved the coincidence of the twisted Weyl discriminant of $\delta$ with the Weyl discriminant of an element $\gamma$ of $\SO_{2n+1}$ which is a norm of $\delta$.
On the other hand, when $\gamma\in\SO_{2n+1}(\overline{F})$ is a norm of $\delta$, the set of eigenvalues of $\gamma$ is given by $\{t_{i}^{\pm1}\mid i=1,\ldots,n\}\sqcup\{1\}$ (see \cite[Lemma 4.1]{MR3904769}).
Then, by definition, the Weyl discriminant of $\gamma$ is given by the right-hand side of the above equality.

Now we first compute $D_{\GL_{2n},\theta}(g_{u}^{\GL_{2n}})$.
Since $g_{u}^{\SO_{2n}^{\mu}}$ is a norm of $g_{u}^{\GL_{2n}}$, by the same argument as \cite[Lemma 4.1]{MR3904769}, the set $\{s_{i}/s_{2n+1-i}\mid i=1,\ldots,2n\}$ for $g_{u}^{\GL_{2n}}$ is equal to the set of eigenvalues of $g_{u}^{\SO_{2n}^{\mu}}$.
On the other hand, by the proof of Lemma \ref{lem:discSOram}, the set of eigenvalues of $g_{u}^{\SO_{2n}^{\mu}}$ is given by
\[
\Biggl\{\frac{1+\varphi_{u}\zeta_{2n}^{j}}{1-\varphi_{u}\zeta_{2n}^{j}} \,\Bigg\vert\, j=1,\ldots,2n \Biggr\}.
\]
Thus we can take $\{t_{1},\ldots,t_{n}\}$ to be
\[
\Biggl\{\frac{1+\varphi_{u}\zeta_{2n}^{j}}{1-\varphi_{u}\zeta_{2n}^{j}} \,\Bigg\vert\, j=1,\ldots,n \Biggr\}.
\]
Hence $D_{\GL_{2n},\theta}(g_{u}^{\GL_{2n}})$ is given by
\[
\prod_{\begin{subarray}{c}k\in\{\pm1\}\\ l\in\{\pm1\}\end{subarray}}
\prod_{1\leq i<j \leq n}
\biggl|\biggl(\frac{1+\varphi_{u}\zeta_{2n}^{i}}{1-\varphi_{u}\zeta_{2n}^{i}}\biggr)^{k} \biggl(\frac{1+\varphi_{u}\zeta_{2n}^{j}}{1-\varphi_{u}\zeta_{2n}^{j}}\biggr)^{l}-1 \biggr|^{\frac{1}{2}}
\cdot
\prod_{k\in\{\pm1\}}
\prod_{1\leq i \leq n}
\biggl|\biggl(\frac{1+\varphi_{u}\zeta_{2n}^{i}}{1-\varphi_{u}\zeta_{2n}^{i}}\biggr)^{k}-1 \biggr|^{\frac{1}{2}}.
\]
The first product is equal to $\mathbb{D}$.
We compute the second product.
Since we have
\[
\prod_{k\in\{\pm1\}}
\prod_{1\leq i \leq n}
\biggl|\biggl(\frac{1+\varphi_{u}\zeta_{2n}^{i}}{1-\varphi_{u}\zeta_{2n}^{i}}\biggr)^{k}-1 \biggr|^{\frac{1}{2}}
=
\prod_{1\leq i \leq 2n}
\biggl|\frac{1+\varphi_{u}\zeta_{2n}^{i}}{1-\varphi_{u}\zeta_{2n}^{i}}-1 \biggr|^{\frac{1}{2}}
=
\prod_{1\leq i \leq 2n}
\biggl|\frac{2\varphi_{u}\zeta_{2n}^{i}}{1-\varphi_{u}\zeta_{2n}^{i}}\biggr|^{\frac{1}{2}}
\]
and
\[
\biggl|\frac{2\varphi_{u}\zeta_{2n}^{i}}{1-\varphi_{u}\zeta_{2n}^{i}}\biggr|
=
q^{-\frac{1}{2n}},
\]
the second product equals $q^{-\frac{1}{2}}$.
Thus we get the desired equality.

We next compute $D_{\GL_{2n},\theta}(\tilde{g}_{u}^{\GL_{2n}})$.
In this case, $-g_{u}^{\SO_{2n}^{\mu}}$ is a norm of $\tilde{g}_{u}^{\GL_{2n}}$.
Since the set of eigenvalues of $-g_{u}^{\SO_{2n}^{\mu}}$ is given by
\[
\Biggl\{-\frac{1+\varphi_{u}\zeta_{2n}^{j}}{1-\varphi_{u}\zeta_{2n}^{j}} \,\Bigg\vert\, j=1,\ldots,2n \Biggr\},
\]
by the same argument as above, $D_{\GL_{2n},\theta}(\tilde{g}_{u}^{\GL_{2n}})$ is equal to
\[
\prod_{\begin{subarray}{c}k\in\{\pm1\}\\ l\in\{\pm1\}\end{subarray}}
\prod_{1\leq i<j \leq n}
\biggl|\biggl(\frac{1+\varphi_{u}\zeta_{2n}^{i}}{1-\varphi_{u}\zeta_{2n}^{i}}\biggr)^{k} \biggl(\frac{1+\varphi_{u}\zeta_{2n}^{j}}{1-\varphi_{u}\zeta_{2n}^{j}}\biggr)^{l}-1 \biggr|^{\frac{1}{2}}
\cdot
\prod_{k\in\{\pm1\}}
\prod_{1\leq i \leq n}
\biggl|-\biggl(\frac{1+\varphi_{u}\zeta_{2n}^{i}}{1-\varphi_{u}\zeta_{2n}^{i}}\biggr)^{k}-1 \biggr|^{\frac{1}{2}}.
\]
As we have
\[
\biggl|-\frac{1+\varphi_{u}\zeta_{2n}^{i}}{1-\varphi_{u}\zeta_{2n}^{i}}-1 \biggr|
=
\biggl|\frac{-2}{1-\varphi_{u}\zeta_{2n}^{i}}\biggr|
=
1,
\]
the second product is equal to $1$.
Thus we get the desired equality.
\end{proof}

\begin{prop}\label{prop:4GL}
Let $(\gamma, \delta)$ be either $(g_{u}^{\SO_{2n}^{\mu}}, g_{u}^{\GL_{2n}})$ or $(-g_{u}^{\SO_{2n}^{\mu}}, \tilde{g}_{u}^{\GL_{2n}})$.
Then we have
\[
\Delta_{\SO_{2n}^{\mu},\GL_{2n},\mathrm{IV}}(\gamma,\delta)=
\begin{cases}
q^{-\frac{1}{2}} & \text{if }(\gamma,\delta)=(g_{u}^{\SO_{2n}^{\mu}}, g_{u}^{\GL_{2n}}),\\
1 &  \text{if }(\gamma,\delta)=(-g_{u}^{\SO_{2n}^{\mu}}, \tilde{g}_{u}^{\GL_{2n}}).
\end{cases}
\]
\end{prop}

\begin{proof}
This follows from the definition of the fourth factor and Lemmas \ref{lem:discSOram} and \ref{lem:discGL}.
\end{proof}

In summary, we get the following:
\begin{prop}\label{prop:tranGL}
Let $(\gamma, \delta)$ be either $(g_{u}^{\SO_{2n}^{\mu}}, g_{u}^{\GL_{2n}})$ or $(-g_{u}^{\SO_{2n}^{\mu}}, \tilde{g}_{u}^{\GL_{2n}})$.
Then we have
\[
\Delta_{\SO_{2n}^{\mu},\GL_{2n}}(\gamma,\delta)=
\begin{cases}
\omega_{0}(-1)G(\omega_{0},\psi)^{-1} & \text{if }(\gamma,\delta)=(g_{u}^{\SO_{2n}^{\mu}}, g_{u}^{\GL_{2n}}),\\
\omega_{0}(-1)G(\omega_{0},\psi)^{-1}q^{\frac{1}{2}} &  \text{if }(\gamma,\delta)=(-g_{u}^{\SO_{2n}^{\mu}}, \tilde{g}_{u}^{\GL_{2n}}).
\end{cases}
\]
\end{prop}

\begin{proof}
We combine Propositions \ref{prop:123GL}, \ref{prop:4GL}, and \ref{prop:epsilon}.
\end{proof}

\subsection{The case of $(\Sp_{2n}, \SO_{2n}^{\mu})$}\label{subsec:tran-Sp}
For $u\in k^{\times}$, we put
\[
\pm\xi^{\Sp_{2n}}_{u}:=\Biggl(F[\varphi_{u}^{2}],\, F[\varphi_{u}],\, c_{u}=\frac{-\varphi_{u}}{2n\varpi u},\, \pm y_{u}=\pm\frac{1+\varphi_{u}}{1-\varphi_{u}}\Biggr).
\]
Then $\xi_{u}$ (resp.\ $-\xi_{u}$) corresponds to $\xi^{\Sp_{2n}}_{u}$ (resp.\ $-\xi^{\Sp_{2n}}_{u}$) in the sense of Section \ref{subsec:norm}.

If we take a basis of $F[\varphi_{u}]$ to be
\[
\bigl\{\varphi_{u}^{2n}, \ldots, \varphi_{u}\bigr\},
\]
then the matrices corresponding to $\pm\xi_{u}^{\Sp_{2n}}$ are given by
\[
\pm
g_{u}^{\Sp_{2n}}:=
\pm\frac{1}{1-\varpi u}
\begin{pmatrix}
1+\varpi u&&2\\
&\ddots&\\
2\varpi u&&1+\varpi u
\end{pmatrix}.
\]
In particular, we have the following:
\begin{prop}\label{prop:affgen-Sp}
The element $g_{u}^{\Sp_{2n}}$ is an affine generic element with simple affine components 
\[
(\underbrace{2, \ldots, 2}_{n}, 2u).
\]
\end{prop}

\begin{prop}\label{prop:123Sp}
Let $(\gamma, \delta)$ be either $(g_{u}^{\SO_{2n}^{\mu}}, g_{u}^{\Sp_{2n}})$ or $(-g_{u}^{\SO_{2n}^{\mu}}, -g_{u}^{\Sp_{2n}})$.
Then we have
\[
\Delta_{\SO_{2n}^{\mu},\Sp_{2n}}^{\mathrm{IV}}(\gamma,\delta)=
\begin{cases}
\omega_{0}(-2) & \text{if }(\gamma,\delta)=(g_{u}^{\SO_{2n}^{\mu}}, g_{u}^{\Sp_{2n}}),\\
\omega_{0}(2) &  \text{if }(\gamma,\delta)=(-g_{u}^{\SO_{2n}^{\mu}}, -g_{u}^{\Sp_{2n}}).
\end{cases}
\]
\end{prop}

\begin{proof}
We first consider the case of $(g_{u}^{\SO_{2n}^{\mu}}, g_{u}^{\Sp_{2n}})$.
By Propositions \ref{prop:Walds} and \ref{prop:eta}, we have
\[
\Delta_{\SO_{2n}^{\mu},\Sp_{2n}}^{\mathrm{IV}}\bigl(g_{u}^{\SO_{2n}^{\mu}}, g_{u}^{\Sp_{2n}}\bigr)
=\sgn_{F[\varphi_{u}]/F[\varphi_{u}^2]}
\Bigl(
\frac{-\varphi_{u}}{2n\varpi u}\cdot P_{\xi_{u}}'(y_{u})\cdot P_{\xi_{u}}(-1) \cdot y_{u}^{1-n}
\Bigr).
\]
By Lemma \ref{lem:pol+}, we have
\[
\quad\frac{-\varphi_{u}}{2n\varpi u}\cdot P'_{\xi_{u}}(y_{u})\cdot P_{\xi_{u}}(-1) \cdot y_{u}^{1-n}
\equiv
\frac{-\varphi_{u}}{2n\varpi u}
\cdot \frac{2\varphi_{u}^{2n-1}}{(1-\varphi_{u})^{2n-1}} \cdot\frac{2n}{1+\varphi_{u}}
\cdot \frac{(1+\varphi_{u})^{1-n}}{(1-\varphi_{u})^{1-n}}
\equiv-2 
\]
(note that $\varphi_{u}^{2n}=\varpi u$ and $\tau(\varphi_{u})=-\varphi_{u}$).
Thus we get the desired equality.

We next consider the case of $(-g_{u}^{\SO_{2n}^{\mu}}, -g_{u}^{\Sp_{2n}})$.
Similarly to the above case, by using Lemma \ref{lem:pol-}, we have
\begin{align*}
&\phantom{{}\equiv{}}\frac{-\varphi_{u}}{2n\varpi u}\cdot P'_{-\xi_{u}}(-y_{u})\cdot P_{-\xi_{u}}(-1) \cdot (-y_{u})^{1-n}\\
&=
\frac{(-1)^{n}\varphi_{u}}{2n\varpi u}
\cdot \frac{2\varphi_{u}^{-1}}{(1-\varphi_{u})^{2n-1}} \cdot\frac{2n}{1+\varphi_{u}}
\cdot \frac{(1+\varphi_{u})^{1-n}}{(1-\varphi_{u})^{1-n}}\\
&\equiv
(-1)^{n}2(\varpi u)^{-1}
\equiv2. 
\end{align*}
Then, by using Propositions \ref{prop:Walds} and \ref{prop:eta}, we get the desired equality.
\end{proof}

\begin{lem}\label{lem:discSp}
We have
\[
D_{\Sp_{2n}}(\pm g_{u}^{\Sp_{2n}})
=
\mathbb{D}\cdot q^{-\frac{1}{2}}.
\]
\end{lem}

\begin{proof}
Since the ratio of $g_{u}^{\Sp_{2n}}$ to $-g_{u}^{\Sp_{2n}}$ belongs to the center of $\Sp_{2n}$, they have the same Weyl discriminant.
Thus it is enough to compute $D_{\Sp_{2n}}(g_{u}^{\Sp_{2n}})$.
The set of eigenvalues of $g_{u}^{\Sp_{2n}}$ is given by
\[
\Biggl\{\frac{1+\varphi_{u}\zeta_{2n}^{j}}{1-\varphi_{u}\zeta_{2n}^{j}} \,\Bigg\vert\, j=1,\ldots,2n \Biggr\}.
\]
Since the absolute roots of $\Sp_{2n}$ are given by
\[
\{\pm e_{i}\pm e_{j}\mid 1\leq i<j\leq n\}\sqcup \{\pm2e_{i}\mid1\leq i\leq n\},
\]
$D_{\Sp_{2n}}(g_{u}^{\Sp_{2n}})$ equals
\[
\prod_{\begin{subarray}{c}k\in\{\pm1\}\\ l\in\{\pm1\}\end{subarray}}
\prod_{1\leq i<j \leq n}
\biggl|\biggl(\frac{1+\varphi_{u}\zeta_{2n}^{i}}{1-\varphi_{u}\zeta_{2n}^{i}}\biggr)^{k} \biggl(\frac{1+\varphi_{u}\zeta_{2n}^{j}}{1-\varphi_{u}\zeta_{2n}^{j}}\biggr)^{l}-1 \biggr|^{\frac{1}{2}}
\cdot
\prod_{k\in\{\pm1\}}
\prod_{1\leq i \leq n}
\biggl|\biggl(\frac{1+\varphi_{u}\zeta_{2n}^{i}}{1-\varphi_{u}\zeta_{2n}^{i}}\biggr)^{2k}-1 \biggr|^{\frac{1}{2}}.
\]
The first product is equal to $\mathbb{D}$.
On the other hand, as we have
\[
\biggl|\biggl(\frac{1+\varphi_{u}\zeta_{2n}^{i}}{1-\varphi_{u}\zeta_{2n}^{i}}\biggr)^{2}-1 \biggr|
=
\biggl|\frac{4\varphi_{u}\zeta_{2n}^{i}}{(1-\varphi_{u}\zeta_{2n}^{i})^{2}}\biggr|
=
q^{-\frac{1}{2n}},
\]
we get the result.
\end{proof}

\begin{prop}\label{prop:4Sp}
Let $(\gamma, \delta)$ be either $(g_{u}^{\SO_{2n}^{\mu}}, g_{u}^{\Sp_{2n}})$ or $(-g_{u}^{\SO_{2n}^{\mu}}, -g_{u}^{\Sp_{2n}})$.
Then we have
\[
\Delta_{\SO_{2n}^{\mu},\Sp_{2n},\mathrm{IV}}(\gamma,\delta)=q^{-\frac{1}{2}}.
\]
\end{prop}

\begin{proof}
This follows from the definition of the fourth factor and Lemmas \ref{lem:discSOram} and \ref{lem:discSp}.
\end{proof}

In summary, we get the following:
\begin{prop}\label{prop:tranSp}
Let $(\gamma, \delta)$ be either $(g_{u}^{\SO_{2n}^{\mu}}, g_{u}^{\Sp_{2n}})$ or $(-g_{u}^{\SO_{2n}^{\mu}}, -g_{u}^{\Sp_{2n}})$.
Then we have
\[
\Delta_{\SO_{2n}^{\mu},\Sp_{2n}}(\gamma,\delta)=
\begin{cases}
\omega_{0}(-2)G(\omega_{0},\psi)^{-1} & \text{if }(\gamma,\delta)=(g_{u}^{\SO_{2n}^{\mu}}, g_{u}^{\Sp_{2n}}),\\
\omega_{0}(2)G(\omega_{0},\psi)^{-1} &  \text{if }(\gamma,\delta)=(-g_{u}^{\SO_{2n}^{\mu}}, -g_{u}^{\Sp_{2n}}).
\end{cases}
\]
\end{prop}

\begin{proof}
We combine Propositions \ref{prop:123Sp}, \ref{prop:4Sp}, and \ref{prop:epsilon}.
\end{proof}

\subsection{The case of $(\SO_{2n+2}, \SO_{2n}^{\mu}\times\SO_{2}^{\mu})$}\label{subsec:tran-SO}
We put 
\[
\mathbf{H}_{-}:=\SO_{2n}^{\mu},\quad
\mathbf{H}_{+}:=\SO_{2}^{\mu},\quad\text{and}\quad
\mathbf{H}:=\mathbf{H}_{-}\times\mathbf{H}_{+}.
\]
Here we assume that $\mu$ is a ramified quadratic character of $F^{\times}$.

Let $u$ be an element of $k^{\times}$ satisfying 
\[
E_{\mu}=F\biggl(\sqrt{(-1)^{n-1}\varpi u}\biggr).
\]
Note that this condition is equivalent to the condition that $\Nr_{E_{\mu}/F}(\sqrt{(-1)^{n-1}\varpi u})=(-1)^{n}\varpi u$ is mapped to $1$ via $\mu$.
Since the restriction of $\mu$ to $\mcO^{\times}$ is given by the lift of the nontrivial quadratic character $\omega_{0}$ of $k^{\times}$, the above condition is furthermore equivalent to the condition that 
\[
\omega_{0}\bigl((-1)^{n-1}u\bigr)=\mu(-\varpi).
\]

For such a $u\in k^{\times}$, we set
\begin{align*}
\xi_{u}^{-} &:=\Biggl(F[\varphi_{u}^{2}],\, F[\varphi_{u}],\, c^{-}_{u}=\frac{(-1)^{n+1}}{2n\varpi u},\, y^{-}_{u}=\frac{1+\varphi_{u}}{1-\varphi_{u}}\Biggr), \text{ and}\\
\xi_{u}^{+} &:=\Biggl(F,\, F[\sqrt{\varpi \tilde{u}}],\, c^{+}_{u}=\frac{-1}{2\varpi \tilde{u}},\, y^{+}_{u}=\frac{1+\sqrt{\varpi \tilde{u}}}{1-\sqrt{\varpi \tilde{u}}}\Biggr),
\end{align*}
where
\[
\tilde{u}:=(-1)^{n+1}u.
\]
Then the data $\xi_{u}^{-}\sqcup\xi_{u}^{+}$ corresponds to $(\xi_{u}^{-}, \xi_{u}^{+})$ in the sense of Section \ref{subsec:norm}.
Here we remark that $\xi_{u}^{-}$ is slightly different from $\xi_{u}$ defined in Section \ref{subsec:tran-SO-ram} ($c^{-}_{u}$ differs from $c'_{u}$).
We compute a matrix representation of the $\Sigma(\SO_{2n+2})$-orbit of the conjugacy classes of $\SO_{2n+2}$ corresponding to this data $\xi_{u}^{-}\sqcup\xi_{u}^{+}$.

Let us first consider the following basis of $F[\varphi_{u}]\oplus F[\sqrt{\varpi \tilde{u}}]$:
\[
\bigl\{\varphi_{u}^{2n}, \ldots, \varphi_{u}, \varpi \tilde{u}, \sqrt{\varpi \tilde{u}} \bigr\}.
\]
Then the matrix representation of the bilinear form corresponding to the data $\xi_{u}^{-}\sqcup\xi_{u}^{+}$ is given by
\[
\begin{pmatrix}
 J_{u}^{-}& 0 \\
 0 & J_{u}^{+}
\end{pmatrix}, 
\]
where 
\[
J_{u}^{-}:=
\begin{pmatrix}
(-1)^{n+1}\varpi u&0\\
0&(-1)^{n}J_{2n-1}
\end{pmatrix}, 
\text{ and }
J_{u}^{+}:=
\begin{pmatrix}
-\varpi \tilde{u}&\\
&1
\end{pmatrix}.
\]
We write $\diag(\xi_{u}^{-},\xi_{u}^{+})$ for the element of $\SO(\diag(J_{u}^{-},J_{u}^{+}))$ obtained by considering the matrix representation of the action of $(y_{u}^{-},y_{u}^{+})$ with respect to the above basis.
If we put
\begin{align*}
Q&:=
\begin{pmatrix}
-1&0&0&0&0&\frac{(-1)^{n}}{2\varpi u}\\
0&I_{n-1}&0&0&0&0\\
0&0&\frac{-1}{2}&1&0&0\\
0&0&0&0&I_{n-1}&0\\
-1&0&0&0&0&\frac{(-1)^{n+1}}{2\varpi u}\\
0&0&\frac{1}{2}&1&0&0
\end{pmatrix}, \text{ and}\\
D&:=
\diag(\underbrace{1,\ldots,1}_{n+2}, \underbrace{1,-1,1, \ldots, (-1)^{n-2}}_{n-1},1),
\end{align*}
then we have
\[
{}^{t}\!(QD)\cdot
\begin{pmatrix}
J_{u}^{-}&0\\
0&J_{u}^{+}
\end{pmatrix}\cdot
QD
=
J_{\mathbbm{1}}.
\]
In other words, the matrix representation of the bilinear form corresponding to the data $\xi_{u}^{-}\sqcup\xi_{u}^{+}$ is given by $J_{\mathbbm{1}}$ with respect to the basis $\{\mathbf{e}_{1},\ldots,\mathbf{e}_{2n+2}\}$;
\[
\mathbf{e}_{i}
=
\begin{cases}
-\varphi_{u}^{2n}-\varpi\tilde{u} & i=1,\\
\varphi_{u}^{2n+1-i} & 2\leq i\leq n,\\
-\frac{1}{2}\varphi_{u}^{n}+\frac{1}{2}\sqrt{\varpi\tilde{u}} & i=n+1,\\
\varphi_{u}^{n}+\sqrt{\varpi\tilde{u}} & i=n+2,\\
(-1)^{n+3-i}\varphi_{u}^{2n+2-i} & n+3\leq i\leq 2n+1,\\
\frac{(-1)^{n}\varphi_{u}^{2n}}{2\varpi u}-\frac{(-1)^{n}\varpi \tilde{u}}{2\varpi u} & i=2n+2.
\end{cases}
\]
Thus, if we put
\[
g_{u}^{\SO_{2n+2}}
:=
(QD)^{-1}\cdot
\diag(\xi_{u}^{-},\xi_{u}^{+})\cdot QD,
\]
then $g_{u}^{\SO_{2n+2}}$ belongs to $\SO_{2n+2}(F)$.

When $n+1$ is odd, the matrix $Q^{-1}\diag(\xi_{u}^{-}, \xi_{u}^{+})Q$ is given by
\[
\frac{1}{1-\varpi u}
\begin{pmatrix}
\frac{1+(\varpi u)^{2}}{1+\varpi u}&-Y_{n}&\frac{\varpi u}{1+\varpi u}&\frac{-2}{1+\varpi u}&-Y_{n}&\frac{-1}{1+\varpi u}\\
-2\varpi uX_{n}&A_{n}&-X_{n}&2X_{n}&B_{n}&X_{n}\\
\frac{4\varpi u}{1+\varpi u}&-2\varpi u Y_{n}&\frac{1+(\varpi u)^{2}}{1+\varpi u}&\frac{-4\varpi u}{1+\varpi u}&-2Y_{n}&\frac{-2\varpi u}{1+\varpi u}\\
-\frac{2(\varpi u)^{2}}{1+\varpi u}&\varpi uY_{n}&\frac{-\varpi u}{1+\varpi u}&\frac{1+(\varpi u)^{2}}{1+\varpi u}&Y_{n}&\frac{1}{1+\varpi u}\\
-2\varpi uX_{n}&\varpi uB_{n}&-\varpi u X_{n}&2\varpi uX_{n}&A_{n}&X_{n}\\
\frac{-4(\varpi u)^{2}}{1+\varpi u}&2\varpi uY_{n}&\frac{-2\varpi u}{1+\varpi u}&\frac{4(\varpi u)^{2}}{1+\varpi u}&2\varpi uY_{n}&\frac{1+(\varpi u)^{2}}{1+\varpi u}
\end{pmatrix}.
\]

When $n+1$ is even, the matrix $Q^{-1}\diag(\xi_{u}^{-}, \xi_{u}^{+})Q$ is given by
\[
\frac{1}{1-\varpi u}
\begin{pmatrix}
1+\varpi u&-Y_{n}&0&-2&-Y_{n}&0\\
-2\varpi uX_{n}&A_{n}&-X_{n}&2X_{n}&B_{n}&-X_{n}\\
0&-2\varpi u Y_{n}&1+\varpi u&0&-2Y_{n}&2\\
-2\varpi u&\varpi uY_{n}&0&1+\varpi u&Y_{n}&0\\
-2\varpi uX_{n}&\varpi uB_{n}&-\varpi u X_{n}&2\varpi uX_{n}&A_{n}&-X_{n}\\
0&-2\varpi uY_{n}&2\varpi u&0&-2\varpi uY_{n}&1+\varpi u
\end{pmatrix}.
\]

In particular we have the following:
\begin{prop}\label{prop:affgen-SO}
The element $g_{u}^{\SO_{2n+2}}$ belongs to the standard Iwahori subgroup of $\SO_{2n+2}(F)$ and its simple affine components are given by
\[
\bigl(-1, \underbrace{2, \ldots, 2}_{n-2}, -1, 2, (-1)^{n-1}2u \bigr).
\]
\end{prop}

Now we recall that $g_{u}^{\mathbf{H}_{-}}\in H_{-}$ defined in Section \ref{subsec:tran-SO-ram} has the data $\xi_{u}$.
We take an element $h_{+}$ of $H_{+}$ having the data $\xi_{u}^{+}$.
Then, by a description of the norm correspondence in Section \ref{subsec:norm}, the element $(g_{u}^{\mathbf{H}_{-}},h_{+})\in H$ is a norm of exactly one of $g_{u}^{\SO_{2n+2}}$ and $wg_{u}^{\SO_{2n+2}}w^{-1}$, where $w$ is the element of $\rmO_{2n+2}(F)\smallsetminus\SO_{2n+2}(F)$ defined in Section \ref{subsec:ssc-SO}.
We put $\tilde{g}_{u}^{\SO_{2n+2}}$ to be the one of them having $(g_{u}^{\mathbf{H}_{-}},h_{+})$ as its norm.

We next consider the matrix representation of $\xi_{u}^{-}\sqcup-\xi_{u}^{+}$.
We can check easily that 
\[
(QD)^{-1}\cdot
\begin{pmatrix}
I_{2n}&\\
&-I_{2}
\end{pmatrix}
\cdot QD
= \varphi_{-2^{-1},2u(-1)^{n+1}}^{\SO_{2n+2}}
\]
\[
=
\begin{pmatrix}
0&0&0&0&0&\frac{(-1)^{n+1}}{2\varpi u}\\
0&I_{n-1}&0&0&0&0\\
0&0&0&-2&0&0\\
0&0&-\frac{1}{2}&0&0&0\\
0&0&0&0&I_{n-1}&0\\
2\varpi u (-1)^{n+1}&0&0&0&0&0
\end{pmatrix}.
\]
Thus the element
\[
g_{u}^{\SO_{2n+2}}\cdot \varphi_{-2^{-1},2u(-1)^{n+1}}^{\SO_{2n+2}}
\]
has the data $\xi_{u}^{-}\sqcup-\xi_{u}^{+}$.
Similarly to the above case, the element $(g_{u}^{\mathbf{H}_{-}},-h_{+})\in H$ is a norm of exactly one of $g_{u}^{\SO_{2n+2}}\varphi_{-2^{-1},2u(-1)^{n+1}}^{\SO_{2n+2}}$ and $wg_{u}^{\SO_{2n+2}}\varphi_{-2^{-1},2u(-1)^{n+1}}^{\SO_{2n+2}}w^{-1}$.
We put $(\tilde{g}_{u}^{\SO_{2n+2}})'$ to be the one of them having $(g_{u}^{\mathbf{H}_{-}},-h_{+})$ as its norm.

\begin{prop}\label{prop:123SOspl}
Let $\gamma\in H^{\srs}$ be a norm of $\delta\in\SO_{2n+2}^{\srs}(F)$.
Then we have
\[
\Delta_{\mathbf{H},\SO_{2n+2}}^{\mathrm{IV}}(\gamma,\delta)=
\begin{cases}
\omega_{0}(-2) & \text{if }(\gamma,\delta) \text{ corresponds to }\bigl((\xi_{u}^{-},\xi_{u}^{+}),\, \xi_{u}^{-}\sqcup\xi_{u}^{+}\bigr),\\
\omega_{0}(2) & \text{if }(\gamma,\delta)  \text{ corresponds to } \bigl((-\xi_{u}^{-},-\xi_{u}^{+}),\, -\xi_{u}^{-}\sqcup-\xi_{u}^{+}\bigr),\\
\omega_{0}(2) & \text{if }(\gamma,\delta)  \text{ corresponds to } \bigl((\xi_{u}^{-},-\xi_{u}^{+}),\, \xi_{u}^{-}\sqcup-\xi_{u}^{+}\bigr).
\end{cases}
\]
\end{prop}

\begin{proof}
We consider the first case.
By Propositions \ref{prop:Walds} and \ref{prop:eta}, the transfer factor is given by
\[
\sgn_{F[\varphi_{u}]/F[\varphi_{u}^2]}
\Bigl(
(-1)^{n}4 c^{-}_{u} P'_{(\xi_{u}^{-},\xi_{u}^{+})}(y^{-}_{u}) P_{(\xi_{u}^{-},\xi_{u}^{+})}(-1) (y^{-}_{u})^{-n} \frac{y^{-}_{u}+1}{y^{-}_{u}-1}
\Bigr),
\]
where $P_{(\xi_{u}^{-}, \xi_{u}^{+})}$ is the polynomial corresponding to $(\xi_{u}^{-}, \xi_{u}^{+})$ which is given by
\[
P_{(\xi_{u}^{-},\xi_{u}^{+})}(T)
=P_{\xi_{u}^{-}}(T)
\cdot \biggl(T-\frac{1+\sqrt{\varpi \tilde{u}}}{1-\sqrt{\varpi \tilde{u}}}\biggr)
\cdot \biggl(T-\frac{1-\sqrt{\varpi \tilde{u}}}{1+\sqrt{\varpi \tilde{u}}}\biggr).
\]
Thus we have
\begin{align*}
P_{(\xi_{u}^{-},\xi_{u}^{+})}'(y^{-}_{u})
&= P_{\xi_{u}}'(y^{-}_{u})
\cdot \frac{2(\varphi_{u}-\sqrt{\varpi \tilde{u}})}{(1-\varphi_{u})(1-\sqrt{\varpi \tilde{u}})}
\cdot \frac{2(\varphi_{u}+\sqrt{\varpi \tilde{u}})}{(1-\varphi_{u})(1+\sqrt{\varpi \tilde{u}})}\\
&=P_{\xi_{u}}'(y^{-}_{u})
\cdot \frac{4(\varphi_{u}^{2}-\varpi \tilde{u})}{(1-\varphi_{u})^{2}(1-\varpi \tilde{u})}, \text{ and}
\end{align*}
\[
P_{(\xi_{u}^{-},\xi_{u}^{+})}(-1)
= P_{\xi_{u}^{-}}(-1)
\cdot \frac{-2}{1-\sqrt{\varpi \tilde{u}}}
\cdot \frac{-2}{1+\sqrt{\varpi \tilde{u}}}\\
= P_{\xi_{u}^{-}}(-1)
\cdot \frac{4}{1-\varpi \tilde{u}}.
\]

Hence, by noting that $P_{\xi_{u}^{-}}=P_{\xi_{u}}$ and $y^{-}_{u}=y_{u}$, we get
\begin{align*}
&\phantom{{}\equiv{}}
(-1)^{n}4 c^{-}_{u} P'_{(\xi_{u}^{-},\xi_{u}^{+})}(y^{-}_{u}) P_{(\xi_{u}^{-},\xi_{u}^{+})}(-1) (y^{-}_{u})^{-n} \frac{y^{-}_{u}+1}{y^{-}_{u}-1}\\
&\equiv
-\frac{1}{2n\varpi u}
\frac{\varphi_{u}^{2}-\varpi \tilde{u}}{(1-\varphi_{u})^{2}} 
P_{\xi_{u}}'(y^{-}_{u})P_{\xi_{u}}(-1)
\frac{(1+\varphi_{u})^{-n}}{(1-\varphi_{u})^{-n}}
\frac{2}{2\varphi_{u}}
\\
&\equiv
-\frac{1}{2n\varpi u}
\frac{\varphi_{u}^{2}-\varpi \tilde{u}}{(1-\varphi_{u})^{2}} 
\frac{2\varphi_{u}^{2n-1}}{(1-\varphi_{u})^{2n-1}} \cdot\frac{2n}{1+\varphi_{u}}
\frac{(1+\varphi_{u})^{-n}}{(1-\varphi_{u})^{-n}}
\frac{2}{2\varphi_{u}}\\
&\equiv
-2(1-\varphi_{u}^{-2}\varpi \tilde{u})
\end{align*}
by Lemma \ref{lem:pol+}.
%
Finally, since $F[\varphi_{u}]/F[\varphi_{u}^{2}]$ is a tamely ramified quadratic extension, $1-\varphi_{u}^{-2}\varpi\tilde{u}$ lies in the image of $\Nr_{F[\varphi_{u}]/F[\varphi_{u}^{2}]}$.
Thus we get
\[
-2(1-\varphi_{u}^{-2}\varpi \tilde{u})
\equiv
-2
\mod \Nr_{F[\varphi_{u}]/F[\varphi_{u}^{2}]}(F[\varphi_{u}]^{\times}).
\]

Next we show the second case.
The polynomial $P_{(-\xi_{u}^{-}, -\xi_{u}^{+})}$ corresponding to $(-\xi_{u}^{-}, -\xi_{u}^{+})$ is given by
\[
P_{(-\xi_{u}^{-},-\xi_{u}^{+})}(T)
=P_{-\xi_{u}^{-}}(T)
\cdot \biggl(T+\frac{1+\sqrt{\varpi \tilde{u}}}{1-\sqrt{\varpi \tilde{u}}}\biggr)
\cdot \biggl(T+\frac{1-\sqrt{\varpi \tilde{u}}}{1+\sqrt{\varpi \tilde{u}}}\biggr)
\]
Thus we have
\begin{align*}
P_{(-\xi_{u}^{-},-\xi_{u}^{+})}'(-y^{-}_{u})
&= P_{-\xi_{u}^{-}}'(-y^{-}_{u})
\cdot \frac{-2(\varphi_{u}-\sqrt{\varpi \tilde{u}})}{(1-\varphi_{u})(1-\sqrt{\varpi \tilde{u}})}
\cdot \frac{-2(\varphi_{u}+\sqrt{\varpi \tilde{u}})}{(1-\varphi_{u})(1+\sqrt{\varpi \tilde{u}})}\\
&=P_{-\xi_{u}^{-}}'(-y^{-}_{u})
\cdot \frac{4(\varphi_{u}^{2}-\varpi \tilde{u})}{(1-\varphi_{u})^{2}(1-\varpi \tilde{u})}, \text{ and}
\end{align*}
\[
P_{(-\xi_{u}^{-},-\xi_{u}^{+})}(-1)
= P_{-\xi_{u}^{-}}(-1)
\cdot \frac{2\sqrt{\varpi \tilde{u}}}{1-\sqrt{\varpi \tilde{u}}}
\cdot \frac{-2\sqrt{\varpi \tilde{u}}}{1+\sqrt{\varpi \tilde{u}}}
= P_{-\xi_{u}^{-}}(-1)
\cdot \frac{-4\varpi \tilde{u}}{1-\varpi \tilde{u}}.
\]
Then, by Lemmas \ref{lem:pol+} and \ref{lem:pol-}, the ratio of the interior of the sign character of Waldspurger's formula for $((\xi_{u}^{-},\xi_{u}^{+}),\, \xi_{u}^{-}\sqcup\xi_{u}^{+})$ to that of $((-\xi_{u}^{-},-\xi_{u}^{+}),\, -\xi_{u}^{-}\sqcup-\xi_{u}^{+})$ is given by
\begin{align*}
&\phantom{{}\equiv{}} (-1)^{n}\frac{(y^{-}_{u}+1)^{2}}{(y^{-}_{u}-1)^{2}}
P'_{(\xi_{u}^{-},\xi_{u}^{+})}(y^{-}_{u}) P_{(\xi_{u}^{-},\xi_{u}^{+})}(-1) 
P'_{(-\xi_{u}^{-},-\xi_{u}^{+})}(-y^{-}_{u})^{-1} P_{(-\xi_{u}^{-},-\xi_{u}^{+})}(-1)^{-1}\\
&=(-1)^{n}\frac{(y^{-}_{u}+1)^{2}}{(y^{-}_{u}-1)^{2}}
P'_{\xi_{u}}(y^{-}_{u}) P_{\xi_{u}}(-1) 
P'_{-\xi_{u}}(-y^{-}_{u})^{-1} P_{-\xi_{u}}(-1)^{-1}
\cdot(-\varpi \tilde{u})^{-1}\\
&\equiv
(-1)^{n}\frac{2^{2}}{(2\varphi_{u})^{2}}\varphi_{u}^{2n}
\cdot\bigl((-1)^{n}\varpi u\bigr)^{-1}\\
&\equiv
\varphi_{u}^{-2} 
\equiv-1.
\end{align*}

Finally we show the third case.
The polynomial $P_{(\xi_{u}^{-}, -\xi_{u}^{+})}$ corresponding to $(\xi_{u}^{-}, -\xi_{u}^{+})$ is given by
\[
P_{(\xi_{u}^{-},-\xi_{u}^{+})}(T)
=P_{\xi_{u}^{-}}(T)
\cdot \biggl(T+\frac{1+\sqrt{\varpi \tilde{u}}}{1-\sqrt{\varpi \tilde{u}}}\biggr)
\cdot \biggl(T+\frac{1-\sqrt{\varpi \tilde{u}}}{1+\sqrt{\varpi \tilde{u}}}\biggr).
\]
Thus we have
\begin{align*}
P_{(\xi_{u}^{-},-\xi_{u}^{+})}'(y^{-}_{u})
&= P_{\xi_{u}^{-}}'(y^{-}_{u})
\cdot \frac{2(1-\varphi_{u}\sqrt{\varpi \tilde{u}})}{(1-\varphi_{u})(1-\sqrt{\varpi \tilde{u}})}
\cdot \frac{2(1+\varphi_{u}\sqrt{\varpi \tilde{u}})}{(1-\varphi_{u})(1+\sqrt{\varpi \tilde{u}})}\\
&=P_{\xi_{u}^{-}}'(y^{-}_{u})
\cdot \frac{4(1-\varphi_{u}^{2}\varpi \tilde{u})}{(1-\varphi_{u})^{2}(1-\varpi \tilde{u})}, \text{ and}
\end{align*}
\[
P_{(\xi_{u}^{-},-\xi_{u}^{+})}(-1)
= P_{\xi_{u}^{-}}(-1)
\cdot \frac{2\sqrt{\varpi \tilde{u}}}{1-\sqrt{\varpi \tilde{u}}}
\cdot \frac{-2\sqrt{\varpi \tilde{u}}}{1+\sqrt{\varpi \tilde{u}}}
= P_{\xi_{u}^{-}}(-1)
\cdot \frac{-4\varpi \tilde{u}}{1-\varpi \tilde{u}}.
\]
Then the ratio of the interior of the sign character of Waldspurger's formula for $((\xi_{u}^{-},\xi_{u}^{+}),\, \xi_{u}^{-}\sqcup\xi_{u}^{+})$ to that of $((\xi_{u}^{-},-\xi_{u}^{+}),\, \xi_{u}^{-}\sqcup-\xi_{u}^{+})$ is given by
\begin{align*}
&\phantom{{}\equiv{}}
P'_{(\xi_{u}^{-},\xi_{u}^{+})}(y^{-}_{u}) P_{(\xi_{u}^{-},\xi_{u}^{+})}(-1) 
P'_{(\xi_{u}^{-},-\xi_{u}^{+})}(y^{-}_{u})^{-1} P_{(\xi_{u}^{-},-\xi_{u}^{+})}(-1)^{-1}\\
&=
\frac{\varphi_{u}^{2}-\varpi \tilde{u}}{1-\varphi_{u}^{2}\varpi \tilde{u}}\cdot(-\varpi \tilde{u})^{-1}\\
&\equiv-1.
\end{align*}
\end{proof}

\begin{lem}\label{lem:discSOspl}
We have
\[
D_{\SO_{2n+2}}(\gamma)=
\begin{cases}
 q^{-1}\cdot \mathbb{D} & \text{if }\gamma \text{ corresponds to }(\xi_{u}^{-},\xi_{u}^{+}),\\
 q^{-1}\cdot \mathbb{D} & \text{if }\gamma \text{ corresponds to } (-\xi_{u}^{-},-\xi_{u}^{+}),\\
\mathbb{D} & \text{if }\gamma \text{ corresponds to } (\xi_{u}^{-},-\xi_{u}^{+}).
\end{cases}
\]
\end{lem}

\begin{proof}
Since the ratio of the element in the first case to that in the second case belongs to the center of $\SO_{2n+2}$, it is enough to compute the first and third cases.

Let $\{x_{1}, \ldots, x_{2n}\}=\{x_{1}^{\pm1}, \ldots, x_{n}^{\pm1}\}$ (resp.\ $\{y_{1}, y_{2}\}=\{y^{\pm1}\}$) be the set of roots of $P_{\xi_{u}^{-}}$ (resp.\ $P_{\xi_{u}^{+}}$).
Then, since the set of absolute roots of $\SO_{2n+2}$ is given by
\[
\{\pm e_{i} \pm e_{j} \mid 1\leq i<j\leq n+1\},
\]
we have
\[
D_{\SO_{2n+2}}(\xi_{u}^{-}, \pm\xi_{u}^{+})
=D_{\SO_{2n}}(\xi_{u}^{-}) \cdot
\prod_{i=1}^{2n}\prod_{j=1}^{2} |x_{i}\mp y_{j}|^{\frac{1}{2}}.
\]
Here, we write simply $D_{\SO_{2n+2}}(\xi_{u}^{-}, \pm\xi_{u}^{+})$ for the Weyl discriminant of $\gamma$ which corresponds to $(\xi_{u}^{-}, \pm\xi_{u}^{+})$.
Similarly, we write simply $D_{\SO_{2n}}(\xi_{u}^{-})$ for the Weyl discriminant of an element of $\SO_{2n}$ which corresponds to $\xi_{u}^{-}$.
By Lemma \ref{lem:discSOram}, the first factor $D_{\SO_{2n}}(\xi_{u}^{-})$ is given by $\mathbb{D}$.

Since we have
\begin{align*}
\prod_{i=1}^{2n}\prod_{j=1}^{2} (x_{i}- y_{j})
&=\prod_{i=1}^{2n} 
\biggl(\frac{1+\varphi_{u}\zeta_{2n}^{i}}{1-\varphi_{u}\zeta_{2n}^{i}}-\frac{1+\sqrt{\varpi \tilde{u}}}{1-\sqrt{\varpi \tilde{u}}}\biggr)
\biggl(\frac{1+\varphi_{u}\zeta_{2n}^{i}}{1-\varphi_{u}\zeta_{2n}^{i}}-\frac{1-\sqrt{\varpi \tilde{u}}}{1+\sqrt{\varpi \tilde{u}}}\biggr)\\
&=\prod_{i=1}^{2n} 
\frac{2(\varphi_{u}\zeta_{2n}^{i}-\sqrt{\varpi \tilde{u}})}{(1-\varphi_{u}\zeta_{2n}^{i})(1-\sqrt{\varpi \tilde{u}})}
\cdot
\frac{2(\varphi_{u}\zeta_{2n}^{i}+\sqrt{\varpi \tilde{u}})}{(1-\varphi_{u}\zeta_{2n}^{i})(1+\sqrt{\varpi \tilde{u}})}
\end{align*}
and 
\begin{align*}
\prod_{i=1}^{2n}\prod_{j=1}^{2} (x_{i}+y_{j})
&=\prod_{i=1}^{2n} 
\biggl(\frac{1+\varphi_{u}\zeta_{2n}^{i}}{1-\varphi_{u}\zeta_{2n}^{i}}+\frac{1+\sqrt{\varpi \tilde{u}}}{1-\sqrt{\varpi \tilde{u}}}\biggr)
\biggl(\frac{1+\varphi_{u}\zeta_{2n}^{i}}{1-\varphi_{u}\zeta_{2n}^{i}}+\frac{1-\sqrt{\varpi \tilde{u}}}{1+\sqrt{\varpi \tilde{u}}}\biggr)\\
&=\prod_{i=1}^{2n} 
\frac{2(1-\varphi_{u}\zeta_{2n}^{i}\sqrt{\varpi \tilde{u}})}{(1-\varphi_{u}\zeta_{2n}^{i})(1-\sqrt{\varpi \tilde{u}})}
\cdot
\frac{2(1+\varphi_{u}\zeta_{2n}^{i}\sqrt{\varpi \tilde{u}})}{(1-\varphi_{u}\zeta_{2n}^{i})(1+\sqrt{\varpi \tilde{u}})},
\end{align*}
we get
\[
\prod_{i=1}^{2n}\prod_{j=1}^{2} |x_{i}- y_{j}|^{\frac{1}{2}}
=q^{-2n\times(\frac{1}{2n}+\frac{1}{2n})\times\frac{1}{2}}=q^{-1}, \text{ and }\quad
\prod_{i=1}^{2n}\prod_{j=1}^{2} |x_{i}+ y_{j}|^{\frac{1}{2}}
=1.
\]
Thus we get the desired equalities. 
\end{proof}

\begin{prop}\label{prop:4SOspl}
We have
\[
\Delta_{\mathbf{H},\SO_{2n+2},\mathrm{IV}}(\gamma,\delta)=
\begin{cases}
q^{-1} & \text{if }(\gamma,\delta) \text{ corresponds to } \bigl((\xi_{u}^{-},\xi_{u}^{+}),\, \xi_{u}^{-}\sqcup\xi_{u}^{+}\bigr),\\
q^{-1} & \text{if }(\gamma,\delta) \text{ corresponds to } \bigl((-\xi_{u}^{-},-\xi_{u}^{+}),\, -\xi_{u}^{-}\sqcup-\xi_{u}^{+}\bigr),\\
1 & \text{if }(\gamma,\delta) \text{ corresponds to } \bigl((\xi_{u}^{-},-\xi_{u}^{+}),\, \xi_{u}^{-}\sqcup-\xi_{u}^{+}\bigr).
\end{cases}
\]
\end{prop}

\begin{proof}
This follows from the definition of the fourth factor and Lemmas \ref{lem:discSOram} and \ref{lem:discSOspl}.
\end{proof}

In summary, we get the following:
\begin{prop}\label{prop:tranSOspl}
Let $\gamma\in H^{\srs}$ be a norm of $\delta\in\SO_{2n+2}^{\srs}(F)$.
Then we have
\[
\Delta_{\mathbf{H},\SO_{2n+2}}(\gamma,\delta)=
\begin{cases}
\omega_{0}(2)q^{-1} & \text{if }(\gamma,\delta) \text{ corresponds to } \bigl((\xi_{u}^{-},\xi_{u}^{+}),\, \xi_{u}^{-}\sqcup\xi_{u}^{+}\bigr),\\
\omega_{0}(-2)q^{-1} &  \text{if }(\gamma,\delta) \text{ corresponds to } \bigl((-\xi_{u}^{-},-\xi_{u}^{+}),\, -\xi_{u}^{-}\sqcup-\xi_{u}^{+}\bigr),\\
\omega_{0}(-2) &  \text{if }(\gamma,\delta) \text{ corresponds to } \bigl((\xi_{u}^{-},-\xi_{u}^{+}),\, \xi_{u}^{-}\sqcup-\xi_{u}^{+}\bigr).
\end{cases}
\]
\end{prop}

\begin{proof}
We combine Propositions \ref{prop:123SOspl}, \ref{prop:4SOspl}, and \ref{prop:epsilon}.
\end{proof}

\subsection{The case of $(\SO_{2n+2}^{\ur}, \SO_{2n}^{\mu}\times\SO_{2}^{\bar{\mu}})$}\label{subsec:tran-SO-ur}
We put 
\[
\mathbf{H}_{-}:=\SO_{2n}^{\mu},\quad
\mathbf{H}_{+}:=\SO_{2}^{\bar{\mu}},\text{ and}\quad
\mathbf{H}:=\mathbf{H}_{-}\times\mathbf{H}_{+}.
\]
Here we assume that $\mu$ and $\bar{\mu}$ are distinct ramified quadratic characters of $F^{\times}$ (note that we have $\mu_{\ur}=\mu\bar{\mu}$).

Let $u$ be an element of $k^{\times}$ satisfying 
\[
E_{\mu}=F\biggl(\sqrt{(-1)^{n-1}\varpi u}\biggr).
\]
Recall that this condition is equivalent to the condition that 
\[
\omega_{0}\bigl((-1)^{n-1}u\bigr)=\mu(-\varpi)
\]
(see Section \ref{subsec:tran-SO}).

For such a $u\in k^{\times}$, we set
\begin{align*}
\xi_{u}^{-} &:=\Biggl(F[\varphi_{u}^{2}],\, F[\varphi_{u}],\, c^{-}_{u}=\frac{(-1)^{n}}{2n\varpi u},\, y^{-}_{u}=\frac{1+\varphi_{u}}{1-\varphi_{u}}\Biggr), \text{ and}\\
\xi_{u}^{+} &:=\Biggl(F,\, F[\sqrt{\varpi \tilde{u}}],\, c^{+}_{u}=\frac{-\epsilon}{2\varpi^{2}\tilde{u}^{2}},\, y^{+}_{u}=\frac{1+\sqrt{\varpi \tilde{u}}}{1-\sqrt{\varpi \tilde{u}}}\Biggr),
\end{align*}
where
\[
\tilde{u}:=(-1)^{n+1}u\epsilon.
\]
Then the data $\xi_{u}^{-}\sqcup\xi_{u}^{+}$ corresponds to $(\xi_{u}^{-}, \xi_{u}^{+})$ in the sense of Section \ref{subsec:norm}.
We compute a matrix representation of the $\Sigma(\SO_{2n+2}^{\ur})$-orbits of the conjugacy classes of $\SO_{2n+2}^{\ur}$ corresponding to this data $\xi_{u}^{-}\sqcup\xi_{u}^{+}$.

Let us consider the following basis of $F[\varphi_{u}]\oplus F[\sqrt{\varpi \tilde{u}}]$:
\[
\bigl\{\varphi_{u}^{2n}, \ldots, \varphi_{u}, \varpi \tilde{u}, \sqrt{\varpi \tilde{u}} \bigr\}.
\]
Then the matrix representation of the bilinear form corresponding to the data $\xi_{u}^{-}\sqcup\xi_{u}^{+}$ is given by
\[
\begin{pmatrix}
 J_{u}^{-}& 0 \\
 0 & J_{u}^{+}
\end{pmatrix}, 
\]
where 
\[
J_{u}^{-}:=
\begin{pmatrix}
(-1)^{n}\varpi u&0\\
0&(-1)^{n-1}J_{2n-1}
\end{pmatrix}
\text{ and}\quad
J_{u}^{+}:=
\begin{pmatrix}
-\epsilon&0\\
0&\frac{1}{(-1)^{n+1}\varpi u}
\end{pmatrix}.
\]
We write $\diag(\xi_{u}^{-},\xi_{u}^{+})$ for the element of $\SO(\diag(J_{u}^{-},J_{u}^{+}))$ obtained by considering the matrix representation of the action of $(y_{u}^{-},y_{u}^{+})$ with respect to the above basis.

If we put
\begin{align*}
Q&:=
\begin{pmatrix}
1&0&0&0&\frac{(-1)^{n}}{2\varpi u}\\
0&I_{n}&0&0&0\\
0&0&0&I_{n-1}&0\\
0&0&1&0&0\\
(-1)^{n+1}\varpi u&0&0&0&\frac{1}{2}
\end{pmatrix}, \text{ and}\\
D&:=
\diag\bigl(\underbrace{1, \ldots, 1}_{n+2}, \underbrace{-1, 1, \ldots, (-1)^{n-1}}_{n-1}, 1 \bigr),
\end{align*}
then we have
\[
{}^{t}\!(QD)\cdot
\begin{pmatrix}
J_{u}^{-}&0\\
0&J_{u}^{+}
\end{pmatrix}\cdot
QD
=
J_{\ur}.
\]
In other words, the matrix representation of the bilinear form corresponding to the data $\xi_{u}^{-}\sqcup\xi_{u}^{+}$ is given by $J_{\ur}$ with respect to the basis $\{\mathbf{e}_{1},\ldots,\mathbf{e}_{2n+2}\}$;
\[
\mathbf{e}_{i}
=
\begin{cases}
\varphi_{u}^{2n}+(-1)^{n+1}\varpi u\sqrt{\varpi\tilde{u}} & i=1,\\
\varphi_{u}^{2n+1-i} & 2\leq i\leq n+1,\\
\varpi \tilde{u} & i=n+2,\\
(-1)^{n+2-i}\varphi_{u}^{2n+2-i} & n+3\leq i\leq 2n+1,\\
\frac{(-1)^{n}\varphi_{u}^{2n}}{2\varpi u}+\frac{\sqrt{\varpi \tilde{u}}}{2} & i=2n+2.
\end{cases}
\]
Thus, if we put
\[
g_{u}^{\SO_{2n+2}^{\ur}}
:=
(QD)^{-1}\cdot
\diag(\xi_{u}^{-},\xi_{u}^{+})\cdot QD,
\]
then $g_{u}^{\SO_{2n+2}^{\ur}}$ belongs to $\SO_{2n+2}^{\ur}(F)$.

When $n+1$ is odd, the matrix $Q^{-1}\diag(\xi_{u}^{-}, \xi_{u}^{+})Q$ is given by
\[
\frac{1}{1-\varpi u}
\begin{pmatrix}
\frac{1+\epsilon(\varpi u)^{2}}{1+\epsilon\varpi u}&Y_{n}&1&\epsilon\frac{1-\varpi u}{1+\epsilon\varpi u}&Y_{n}&\frac{1+\epsilon}{2(1+\epsilon\varpi u)}\\
2\varpi uX_{n}&A_{n}&2X_{n}&0&B_{n}&X_{n}\\
2\varpi u&2\varpi uY_{n}&1+\varpi u&0&2Y_{n}&1\\
-2\varpi u\frac{1-\varpi u}{1+\epsilon\varpi u}&0&0&(1-\varpi u)\frac{1-\epsilon\varpi u}{1+\epsilon\varpi u}&0&\frac{1-\varpi u}{1+\epsilon\varpi u}\\
2\varpi uX_{n}&\varpi uB_{n}&2\varpi uX_{n}&0&A_{n}&X_{n}\\
\frac{2(1+\epsilon)(\varpi u)^{2}}{1+\epsilon\varpi u}&2\varpi uY_{n}&2\varpi u&-2\epsilon\varpi u\frac{1-\varpi u}{1+\epsilon\varpi u}&2\varpi uY_{n}&\frac{1+\epsilon(\varpi u)^{2}}{1+\epsilon\varpi u}\\
\end{pmatrix}.
\]

When $n+1$ is even, this matrix $Q^{-1}\diag(\xi_{u}^{-}, \xi_{u}^{+})Q$ is given by
\[
\frac{1}{1-\varpi u}
\begin{pmatrix}
\frac{1-\epsilon(\varpi u)^{2}}{1-\epsilon\varpi u}&Y_{n}&1&\epsilon\frac{1-\varpi u}{1-\epsilon\varpi u}&Y_{n}&\frac{-1+\epsilon}{2(1-\epsilon\varpi u)}\\
2\varpi uX_{n}&A_{n}&2X_{n}&0&B_{n}&-X_{n}\\
2\varpi u&2\varpi uY_{n}&1+\varpi u&0&2Y_{n}&-1\\
2\varpi u\frac{1-\varpi u}{1-\epsilon\varpi u}&0&0&(1-\varpi u)\frac{1+\epsilon\varpi u}{1-\epsilon\varpi u}&0&\frac{1-\varpi u}{1-\epsilon\varpi u}\\
2\varpi uX_{n}&\varpi uB_{n}&2\varpi uX_{n}&0&A_{n}&-X_{n}\\
\frac{2(-1+\epsilon)(\varpi u)^{2}}{1-\epsilon\varpi u}&-2\varpi uY_{n}&-2\varpi u&2\epsilon\varpi u\frac{1-\varpi u}{1-\epsilon\varpi u}&-2\varpi uY_{n}&\frac{1-\epsilon(\varpi u)^{2}}{1-\epsilon\varpi u}\\
\end{pmatrix}.
\]

In particular we have the following:
\begin{prop}\label{prop:affgen-SOur}
The element $g_{u}^{\SO_{2n+2}^{\ur}}$ belongs to the standard Iwahori subgroup of $\SO_{2n+2}^{\ur}(F)$ and its simple affine components are given by
\[
\bigl(1, \underbrace{2, \ldots, 2}_{n-2}, 2, (-1)^{n-1}2u \bigr).
\]
\end{prop}

We take an element $h_{+}$ of $H_{+}$ having the data $\xi_{u}^{+}$.
Then, by the same argument as in the split case (Section \ref{subsec:tran-SO}), the element $(g_{u}^{\mathbf{H}_{-}},h_{+})\in H$ is a norm of exactly one of $g_{u}^{\SO_{2n+2}^{\ur}}$ and $w_{\ur}g_{u}^{\SO_{2n+2}^{\ur}}w_{\ur}^{-1}$, where $w_{\ur}$ is the element of $\rmO_{2n+2}^{\ur}(F)\smallsetminus\SO_{2n+2}^{\ur}(F)$ defined in Section \ref{subsec:ssc-SO-ur}.
We put $\tilde{g}_{u}^{\SO_{2n+2}^{\ur}}$ to be the one of them having $(g_{u}^{\mathbf{H}_{-}},h_{+})$ as its norm.

We next consider the matrix representation of $\xi_{u}^{-}\sqcup-\xi_{u}^{+}$.
We can check easily that 
\[
(QD)^{-1}\cdot
\begin{pmatrix}
I_{2n}&\\
&-I_{2}
\end{pmatrix}
\cdot QD
= \varphi_{1,2u(-1)^{n}}^{\SO_{2n+2}^{\ur}}
\]
\[
=
\begin{pmatrix}
0&0&0&0&\frac{(-1)^{n}}{2\varpi u}\\
0&I_{n}&0&0&0\\
0&0&-1&0&0\\
0&0&0&I_{n-1}&0\\
2\varpi u (-1)^{n}&0&0&0&0
\end{pmatrix}.
\]
Thus the element
\[
g_{u}^{\SO_{2n+2}^{\ur}}\cdot \varphi_{1,2u(-1)^{n}}^{\SO_{2n+2}^{\ur}}.
\]
has the data $\xi_{u}^{-}\sqcup-\xi_{u}^{+}$ and the element $(g_{u}^{\mathbf{H}_{-}},-h_{+})\in H$ is a norm of exactly one of $g_{u}^{\SO_{2n+2}^{\ur}}\varphi_{1,2u(-1)^{n}}^{\SO_{2n+2}^{\ur}}$ and $w_{\ur}g_{u}^{\SO_{2n+2}^{\ur}}\varphi_{1,2u(-1)^{n}}^{\SO_{2n+2}^{\ur}}w_{\ur}^{-1}$.
We put $(\tilde{g}_{u}^{\SO_{2n+2}^{\ur}})'$ to be the one of them having $(g_{u}^{\mathbf{H}_{-}},-h_{+})$ as its norm.

\begin{prop}\label{prop:123SOur}
Let $\gamma\in H^{\srs}$ be a norm of $\delta\in\SO_{2n+2}^{\ur,\srs}(F)$.
Then we have
\[
\Delta_{\mathbf{H},\SO_{2n+2}^{\ur}}^{\mathrm{IV}}(\gamma,\delta)=
\begin{cases}
1 & \text{if }(\gamma,\delta) \text{ corresponds to }\bigl((\xi_{u}^{-},\xi_{u}^{+}),\, \xi_{u}^{-}\sqcup\xi_{u}^{+}\bigr),\\
-\omega_{0}(-1) & \text{if }(\gamma,\delta) \text{ corresponds to } \bigl((-\xi_{u}^{-},-\xi_{u}^{+}),\, -\xi_{u}^{-}\sqcup-\xi_{u}^{+}\bigr),\\
-\omega_{0}(-1) & \text{if }(\gamma,\delta) \text{ corresponds to } \bigl((\xi_{u}^{-},-\xi_{u}^{+}),\, \xi_{u}^{-}\sqcup-\xi_{u}^{+}\bigr).
\end{cases}
\]
\end{prop}

\begin{proof}
Since the differences between this case and the split case are $\eta_{\G}$, $c^{-}_{u}$, and $\tilde{u}$.
Thus, by the same computation as in the proof of Proposition \ref{prop:123SOspl}, we get the results (note that $\omega_{0}(-\epsilon)=-\omega_{0}(-1)$).
\end{proof}

By the exactly same computation as in Lemma \ref{lem:discSOspl} and Proposition \ref{prop:4SOspl}, we get the following.

\begin{prop}\label{prop:4SOur}
We have
\[
\Delta_{\mathbf{H},\SO_{2n+2}^{\ur},\mathrm{IV}}(\gamma,\delta)=
\begin{cases}
q^{-1} & \text{if }(\gamma,\delta) \text{ corresponds to } \bigl((\xi_{u}^{-},\xi_{u}^{+}),\, \xi_{u}^{-}\sqcup\xi_{u}^{+}\bigr),\\
q^{-1} & \text{if }(\gamma,\delta) \text{ corresponds to } \bigl((-\xi_{u}^{-},-\xi_{u}^{+}),\, -\xi_{u}^{-}\sqcup-\xi_{u}^{+}\bigr),\\
1 & \text{if }(\gamma,\delta) \text{ corresponds to } \bigl((\xi_{u}^{-},-\xi_{u}^{+}),\, \xi_{u}^{-}\sqcup-\xi_{u}^{+}\bigr).
\end{cases}
\]
\end{prop}

In summary, we get the following:
\begin{prop}\label{prop:tranSOur}
Let $\gamma\in H^{\srs}$ be a norm of $\delta\in\SO_{2n+2}^{\ur,\srs}(F)$.
Then we have
\[
\Delta_{\mathbf{H},\SO_{2n+2}^{\ur}}(\gamma,\delta)=
\begin{cases}
-\omega_{0}(-1)q^{-1} & \text{if }(\gamma,\delta) \text{ corresponds to } \bigl((\xi_{u}^{-},\xi_{u}^{+}),\, \xi_{u}^{-}\sqcup\xi_{u}^{+}\bigr),\\
q^{-1} &  \text{if }(\gamma,\delta) \text{ corresponds to } \bigl((-\xi_{u}^{-},-\xi_{u}^{+}),\, -\xi_{u}^{-}\sqcup-\xi_{u}^{+}\bigr),\\
1 &  \text{if }(\gamma,\delta) \text{ corresponds to } \bigl((\xi_{u}^{-},-\xi_{u}^{+}),\, \xi_{u}^{-}\sqcup-\xi_{u}^{+}\bigr).
\end{cases}
\]
\end{prop}

\section{Endoscopic lifting of simple supercuspidal $L$-packets of $\SO_{2n}^{\mu}$}\label{sec:SO-ram}
In Section \ref{sec:Sp}, we investigated the structure of the $L$-packets of $\Sp_{2n}$ consisting of simple supercuspidal representations.
The purpose in this section is to determine their endoscopic lifts to $\GL_{2n+1}$.
To do this, we once ``descend'' them to a ramified even special orthogonal group and again lift them to $\GL_{2n}$.

In this section, we put $\mathbf{G}:=\Sp_{2n}$.
By Theorem \ref{thm:packetSp}, for any $(\xi,0,a)\in\SSC(\Sp_{2n})$,
\[
\Pi_{\phi}^{\G}=\{\pi^{\G}_{\xi,0,a}, \pi^{\G}_{\xi,1,a\epsilon^{-1}}\}
\]
is an $L$-packet of $\G$.
Here we write $\phi$ for the $L$-parameter corresponding to these two representations.
Then we can regard $\phi$ as an $L$-parameter of $\GL_{2n+1}$ by composing with the $L$-embedding from ${}^{L}\!\Sp_{2n}$ to ${}^{L}\!\GL_{2n+1}$.
We denote the irreducible smooth representation of $\GL_{2n+1}(F)$ corresponding to this $L$-parameter by $\pi_{\phi}^{\GL_{2n+1}}$.

Recall that, by Corollary \ref{cor:decomp}, the $L$-parameter $\phi$ factors through an endoscopic group $\mathbf{H}$ of $\G$.
Furthermore, this group $\mathbf{H}$ is given by $\SO_{2n}^{\mu}$ with the quadratic character $\mu$ of $F^{\times}$ corresponding to $\det\circ\phi_{0}$, where $\phi_{0}$ is the $2n$-dimensional irreducible constituent of $\phi$ as a representation of $W_{F}$ (note that $\phi$ is trivial on $\SL_{2}(\C)$).
Then, by regarding $\phi_{0}$ as an $L$-parameter of $\mathbf{H}$, we get an $L$-packet $\widetilde{\Pi}_{\phi_{0}}^{\mathbf{H}}$ of $\mathbf{H}$.
Moreover, by composing $\phi_{0}$ with the $L$-embedding from the $L$-group of $\mathbf{H}$ to that of $\GL_{2n}$, we get the $L$-parameter $\phi_{0}$ of $\GL_{2n}$.

In summary, we have the following diagram of the $L$-groups of $\GL_{2n+1}$, $\G=\Sp_{2n}$, $\mathbf{H}=\SO_{2n}^{\mu}$, and $\GL_{2n}$ (see Section \ref{sec:Arthur} for the definitions of the $L$-embeddings):
\[
\xymatrix{
&{}^{L}\! \GL_{2n+1}=\GL_{2n+1}(\C)\times W_{F}\\
W_{F}\ar@{->}[r]^-{\phi}\ar@{->}[rd]_-{\phi_{0}}&{}^{L}\G=\SO_{2n+1}(\C)\times W_{F}\ar@{^{(}->}[u]& {}^{L}\! \GL_{2n}=\GL_{2n}(\C)\times W_{F}.\\
&{}^{L}\mathbf{H}=\SO_{2n}(\C)\rtimes W_{F}\ar@{^{(}->}[u] \ar@{^{(}->}[ru]&
}
\]

If we denote the corresponding representation of $\GL_{2n}(F)$ by $\pi_{\phi_{0}}^{\GL_{2n}}$, then we have the following:

\begin{prop}\label{prop:lift}
\begin{enumerate}
\item
The representation $\pi_{\phi_{0}}^{\GL_{2n}}$ is supercuspidal.
\item
We have 
\[
\pi_{\phi}^{\GL_{2n+1}}\cong\pi_{\phi_{0}}^{\GL_{2n}}\boxplus\omega_{\det\circ\phi_{0}}:=\nInd_{\mathbf{P}_{2n,1}(F)}^{\GL_{2n+1}(F)}\pi_{\phi_{0}}^{\GL_{2n}}\boxtimes \omega_{\det\circ\phi_{0}},
\]
where 
\begin{itemize}
 \item
 $\mathbf{P}_{2n,1}$ is the standard parabolic subgroup of $\GL_{2n+1}$ with the Levi subgroup $\GL_{2n}\times\GL_{1}$,
 \item
 $\nInd$ is the normalized parabolic induction from $\mathbf{P}_{2n,1}(F)$ to $\GL_{2n+1}(F)$, and
 \item
 $\omega_{\det\circ\phi_{0}}$ is the character of $F^{\times}$ corresponding to $\det\circ\phi_{0}$ under the local class field theory (note that this equals $\mu$).
 \end{itemize}
\end{enumerate}
\end{prop}

\begin{proof}
The first assertion follows from Corollary \ref{cor:decomp} and a property of the local Langlands correspondence of $\GL_{N}$ that supercuspidal representations corresponds to irreducible representations of $W_{F}\times\SL_{2}(\C)$ whose $\SL_{2}(\C)$-part are trivial.
The second assertion follows from Corollary \ref{cor:decomp} and the compatibility of the local Langlands correspondence and the parabolic induction.
\end{proof}

By Proposition \ref{prop:lift}, to determine the endoscopic lift $\pi_{\phi}^{\GL_{2n+1}}$ of $\Pi_{\phi}^{\G}$ from $\G$ to $\GL_{2n+1}$ is equivalent to determine the endoscopic lift $\pi_{\phi_{0}}^{\GL_{2n}}$ of $\widetilde{\Pi}_{\phi_{0}}^{\mathbf{H}}$ from $\mathbf{H}$ to $\GL_{2n}$.

%
%

\subsection{Depth bound for the lifted representations of $\GL_{2n+1}$}\label{subsec:depth-bound}
We write simply $\Theta_{\phi,\theta}^{\GL_{2n+1}}$ for the $\theta$-twisted character $\Theta_{\pi_{\phi}^{\GL_{2n+1}},\theta}$ of $\pi_{\phi}^{\GL_{2n+1}}$.
Here we normalize the $\theta$-twisted character via the fixed $\theta$-stable Whittaker datum $\mathfrak{w}_{\GL_{2n+1}}$ of $\GL_{2n+1}$ in Section \ref{subsec:Whittaker} (see Remark \ref{rem:normalization} for the way to normalize the twisted character).

\begin{prop}\label{prop:TECRGL_{2n+1}}
For every pair $(g, h)\in \GL_{2n+1}^{\strs}(F)\times G^{\srs}$ such that $h$ is a norm of $g$, we have
\[
\Theta_{\phi,\theta}^{\GL_{2n+1}}(g)
  = \bigl(\Theta^{\G}_{\xi,0,a}+\Theta^{\G}_{\xi,1,a\epsilon^{-1}}\bigr)(h).
\]
\end{prop}

\begin{proof}
By Theorem \ref{thm:TECR}, for every $g\in \GL_{2n+1}^{\strs}(F)$, we have
\[
\Theta_{\phi,\theta}^{\GL_{2n+1}}(g)=\sum_{h\leftrightarrow g / \sim} \frac{D_{\G}(h)^{2}}{D_{\GL_{2n+1},\theta}(g)^{2}}\Delta_{\G,\GL_{2n+1}}(h,g)  \bigl(\Theta^{\G}_{\xi,0,a}+\Theta^{\G}_{\xi,1,a\epsilon^{-1}}\bigr)(h),
\]
where $h\in G^{\srs}$ runs over the stable conjugacy classes of norms of $g$.
Thus our task is to show:
\begin{enumerate}
\item
a norm of $g$ is at most unique up to stable conjugacy, 
\item
$\Delta_{\G,\GL_{2n+1}}(h,g)$ is equal to $1$ for a norm $h$ of $g$, and
\item
the ratio of the Weyl discriminants $D_{\G}(h)/D_{\GL_{2n+1},\theta}(g)$ is trivial.
\end{enumerate}
The first one follows from the description of the norm correspondence for $\GL_{2n+1}$ and $\Sp_{2n}$ in terms of Waldspurger's data.
See Section 1.9 in \cite{MR2672539} for the details.
The second one was done in Proposition \ref{prop:triviality}.
Also the third one is explained in the proof of Proposition \ref{prop:triviality}; this can be checked by going back to the definition of the Weyl discriminant and the norm correspondence as in the proof of \cite[Lemma 4.10]{MR3904769}.
\end{proof}

\begin{rem}
We can also show the claim (1) in the above proof by considering the definition of the norm correspondence directly.
See, for example, \cite[Lemma 4.4]{MR3904769}.
\end{rem}

\begin{prop}\label{prop:depth}
The depth of the representation $\pi_{\phi_{0}}^{\GL_{2n}}$ is not greater than $\frac{1}{2n}$.
In particular, $\pi_{\phi_{0}}^{\GL_{2n}}$ is either depth-zero or simple supercuspidal.
\end{prop}

\begin{proof}
By Proposition \ref{prop:lift}, we have
\[
\pi_{\phi}^{\GL_{2n+1}}\cong\pi_{\phi_{0}}^{\GL_{2n}}\boxplus\omega_{\det\circ\phi_{0}}.
\]
Then, since the parabolic induction preserves the depth of representations (see \cite[Theorem 5.2]{MR1371680}), we have
\[
\depth\bigl(\pi_{\phi}^{\GL_{2n+1}}\bigr)=\max\bigl\{\depth\bigl(\pi_{\phi_{0}}^{\GL_{2n}}\bigr), \depth(\omega_{\det\circ\phi_{0}})\bigr\}.
\]
Thus, if we can show the depth of $\pi_{\phi}^{\GL_{2n+1}}$ is not greater than $\frac{1}{2n}$, then so is the depth of $\pi_{\phi_{0}}^{\GL_{2n}}$.
Moreover, by Proposition \ref{prop:depth-ssc}, $\pi_{\phi_{0}}^{\GL_{2n}}$ is either depth-zero supercuspidal or simple supercuspidal.

Therefore our task is to show that 
\[
\depth\bigl(\pi_{\phi}^{\GL_{2n+1}}\bigr)\leq\frac{1}{2n}.
\]
By the definition of the depth of representations (see Section \ref{subsec:depth}), it suffices to find a point $\bold{x}$ of the apartment $\mathcal{A}_{\red}(\GL_{2n+1},\T_{\GL_{2n+1}})$ of the diagonal maximal torus $\T_{\GL_{2n+1}}$ in $\GL_{2n+1}$ such that $\pi_{\phi}^{\GL_{2n+1}}$ has a non-zero $\GL_{2n+1,\bold{x},\frac{1}{2n}+}$-fixed vector.
Here, $\GL_{2n+1,\bold{x},\frac{1}{2n}+}$ is the next step of the $\frac{1}{2n}$-th Moy--Prasad filtration of the parahoric subgroup $\GL_{2n+1,\bold{x}}$ of $\GL_{2n+1}(F)$ associated to $\bold{x}$.

Let $\T_{\SL_{2n+1}}$ be the diagonal maximal torus in $\SL_{2n+1}$, then the reduced apartment $\mathcal{A}_{\red}(\GL_{2n+1},\T_{\GL_{2n+1}})$ is nothing but $\mathcal{A}(\SL_{2n+1},\T_{\SL_{2n+1}})$ by definition (see Section \ref{subsubsec:apartment}, note that $\GL_{2n+1,\der}=\SL_{2n+1}$ and $\T_{\GL_{2n+1},\der}=\T_{\SL_{2n+1}}$).
The real vector space $X_{\ast}(\T_{\SL_{2n+1}})\otimes_{\Z}\R$ is given by
\[
\Bigl\{\sum_{i=1}^{2n+1}a_{i}e_{i}^{\vee}\in\bigoplus_{i=1}^{2n+1}\R e_{i}^{\vee} \,\Big\vert\, \sum_{i=1}^{2n+1}a_{i}=0\Bigr\},
\]
where $\{e^{\vee}_{i}\}_{i=1}^{2n+1}$ is the standard basis of the cocharacter group $X_{\ast}(\T_{\GL_{2n+1}})$.
We may identify $\mathcal{A}_{\red}(\GL_{2n+1},\T_{\GL_{2n+1}})$ with $X_{\ast}(\T_{\SL_{2n+1}})\otimes_{\Z}\R$ so that the affine transformations given by $N_{\SL_{2n+1}}/\T_{\SL_{2n+1},c}$ (see Section \ref{subsec:BT}) are described as follows.
First note that $\T_{\SL_{2n+1},c}=\T_{\SL_{2n+1}}(\mcO)$ and thus any element of $N_{\SL_{2n+1}}/\T_{\SL_{2n+1},c}$ is represented by the product of the following two types of elements:
\begin{itemize}
\item
a permutation matrix $n_{\sigma}$ representing a permutation $\sigma$ of $(1,2,\ldots,n)$ (we consider $-n_{\sigma}$ instead of $n_{\sigma}$ when the determinant of $n_{\sigma}$ equals $-1$),
\item
a diagonal matrix of the form $\diag(\varpi^{r_{1}},\ldots,\varpi^{r_{2n+1}})$ with integers $r_{1}, r_{2},\cdots, r_{2n+1}\in\Z$ satisfying $\sum_{i=1}^{2n+1} r_{i}=0$.
\end{itemize}
Then the affine transformation (written by $\tilde{\nu}$ under the notation in Section \ref{subsec:BT}) with respect to these elements are given as follows through the identification of $\mathcal{A}_{\red}(\GL_{2n+1},\T_{\GL_{2n+1}})$ with $X_{\ast}(\T_{\SL_{2n+1}})\otimes_{\Z}\R$):
\begin{itemize}
\item
for any $\sum_{i=1}^{2n+1}a_{i}e^{\vee}_{i} \in X_{\ast}(\T_{\SL_{2n+1}})\otimes_{\Z}\R$, we have
\[
\tilde{\nu}(\pm n_{\sigma})\biggl(\sum_{i=1}^{2n+1}a_{i}e^{\vee}_{i}\biggr)
=
\sum_{i=1}^{2n+1}a_{i}e^{\vee}_{\sigma(i)},
\]
\item
for any $\sum_{i=1}^{2n+1}a_{i}e^{\vee}_{i} \in X_{\ast}(\T_{\SL_{2n+1}})\otimes_{\Z}\R$, we have
\[
\tilde{\nu}\bigl(\diag(\varpi^{r_{1}},\ldots,\varpi^{r_{2n+1}})\bigr)\biggl(\sum_{i=1}^{2n+1}a_{i}e^{\vee}_{i}\biggr)
=
\sum_{i=1}^{2n+1}(a_{i}+r_{i})e^{\vee}_{i}.
\]
\end{itemize}

We put
\[
\bold{x}:=\frac{n}{2n}e^{\vee}_{1}+\frac{n-1}{2n}e^{\vee}_{2}+\cdots-\frac{(n-1)}{2n}e^{\vee}_{2n}-\frac{n}{2n}e^{\vee}_{2n+1} \in \mathcal{A}_{\red}(\GL_{2n+1},\T_{\GL_{2n+1}})
\]
and show that this point satisfies the above condition, that is, $\pi_{\phi}^{\GL_{2n+1}}$ has a non-zero $\GL_{2n+1,\bold{x},\frac{1}{2n}+}$-fixed vector.
It can be checked that the Moy--Prasad filtration of $\GL_{2n+1,\bold{x}}$ is $\theta$-stable and its first three steps are described as follows:
\hdashlinegap=1pt
\[
P_{\bold{x}}:=\GL_{2n+1,\bold{x},0}=
\begin{pmatrix}
 \mcO^{\times}&& \multicolumn{1}{c:}{}&\mfp^{-1}\\
 \cdashline{4-4}
 &\ddots&\mcO&\\
 &\mfp&\ddots&\\
 &&&\mcO^{\times}
\end{pmatrix}, 
\]
\[
P_{\bold{x}}^{+}:=\GL_{2n+1,\bold{x},0+}=\GL_{2n+1,\bold{x},\frac{1}{2n}}=
\begin{pmatrix}
 1+\mfp&&&\mcO\\
 &\ddots&\mcO&\\
 &\mfp&\ddots&\\
 \cdashline{1-1}
 \multicolumn{1}{c:}{\mfp^{2}}&&&1+\mfp
\end{pmatrix}, \text{ and}
\]
\[
P_{\bold{x}}^{++}:=\GL_{2n+1,\bold{x},\frac{1}{2n}+}=
\begin{pmatrix}
 1+\mfp&\multicolumn{1}{c:}{\mfp}&&&\mcO\\
 \cdashline{3-3}
 &\ddots&\multicolumn{1}{c:}{\ddots}&&\\
 &\mfp&\ddots&&\\
 \cdashline{1-1} \cdashline{5-5}
 \multicolumn{1}{c:}{\mfp^{2}}&&&\ddots&\mfp\\
 \cdashline{2-2}
 \mfp^{2}&\multicolumn{1}{c:}{\mfp^{2}}&&&1+\mfp
\end{pmatrix}.
\]

We take $u\in k^{\times}$ such that $\Kl_{2^{2n}au}^{n+1}(\psi; w_{\Sp_{2n}})\neq0$ (such $u$ exists by Proposition \ref{prop:Kl} (2)), and put
\[
g_{u}:=
\begin{pmatrix}
 1&1&&0&\\
 &\ddots&\ddots&&\\
 &0&\ddots&\ddots&\\
 \frac{\varpi u}{2}&&&\ddots&1\\
0&\frac{\varpi u}{2}&&&1
\end{pmatrix}
\in P_{\bold{x}}^{+},
\]
and consider the following claim:
\begin{claim*}
\begin{enumerate}
 \item
 The element $g_{u}$ is strongly $\theta$-regular $\theta$-semisimple and we have 
 \[
 \Theta_{\phi,\theta}^{\GL_{2n+1}}(g_{u})=\Kl_{2^{2n}au}^{n+1}(\psi; w_{\Sp_{2n}}).
 \]
 \item
 Let $x\in \GL_{2n+1}^{\strs}(F)$ be a strongly $\theta$-regular $\theta$-semisimple element which belongs to $g_{u}P_{\bold{x}}^{++}$.
 Then the twisted character $\Theta_{\phi,\theta}^{\GL_{2n+1}}(x)$ is equal to either $\Kl_{2^{2n}au}^{n+1}\bigl(\psi; w_{\Sp_{2n}}\bigr)$ or $0$.
\end{enumerate}
\end{claim*}

We first admit this claim and show the proposition.
Let $f$ be the characteristic function of $g_{u}P_{\bold{x}}^{++}$.
By the property of the twisted character, we have 
\[
\tr\pi_{\phi,\theta}^{\GL_{2n+1}}(f)
=\int_{\GL_{2n+1}(F)}f(x)\Theta_{\phi,\theta}^{\GL_{2n+1}}(x)\,dx
=\int_{g_{u}P_{\bold{x}}^{++}}\Theta_{\phi,\theta}^{\GL_{2n+1}}(x)\,dx.
\]
Then, by the assumption on $u$, the above claim, and the local constancy of the twisted character, this integral is not zero.
On the other hand, since $P_{\bold{x}}^{++}$ is normal in $P_{\bold{x}}^{+}$, $f$ is bi-$P_{\bold{x}}^{++}$-invariant.
Thus we have
\[
\tr\pi_{\phi,\theta}^{\GL_{2n+1}}(f)=
\tr\bigg(\pi_{\phi}^{\GL_{2n+1}}(f)\circ I_{\theta}\,\bigg\vert\, \Bigl(\pi_{\phi}^{\GL_{2n+1}}\Bigr)^{P_{\bold{x}}^{++}}\biggr).
\]
Here $I_{\theta}$ is the intertwiner $I_{\theta}\colon\pi_{\phi}^{\GL_{2n+1}}\cong(\pi_{\phi}^{\GL_{2n+1}})^{\theta}$ normalized with respect to the fixed Whittaker datum of $\GL_{2n+1}$.
Hence the non-vanishing of the trace $\tr\pi_{\phi,\theta}^{\GL_{2n+1}}(f)$ implies that $(\pi_{\phi}^{\GL_{2n+1}})^{P_{\bold{x}}^{++}}$ is nonzero.
\end{proof}

To show the claim in the above proof, we first show the following lemma:
\begin{lem}\label{lem:charpoly1}
Let $x:=(x_{ij})_{ij}$ be an element of the pro-unipotent radical of the standard Iwahori subgroup of $\GL_{N}(F)$, that is 
\[
x\in
\begin{pmatrix}
1+\mfp&&\mcO\\
&\ddots&\\
\mfp&&1+\mfp
\end{pmatrix}.
\]
Let $p_{x}(T)\in F[T]$ be the characteristic polynomial of $x$, and we write 
\[
p_{x}(T)=(T-1)^{N}+a_{N-1}(T-1)^{N-1}+\cdots+a_{1}(T-1)+a_{0}.
\]
Then we have
\begin{enumerate}
 \item
 $a_{i}\in\mfp$ for each $0\leq i\leq N-1$, and
 \item
 $a_{0}\equiv -x_{12}\cdots x_{N-1,N} x_{N1} \quad \mod \mfp^{2}$. 
\end{enumerate}
In particular, if $x$ is affine generic, then $p_{x}(T)$ is an Eisenstein polynomial with respect to $T-1$.
\end{lem}

\begin{proof}
Since $x$ belongs to the pro-unipotent radical of the Iwahori subgroup, 
we have 
\[
\det(T-x)\equiv (T-1)^{N} \quad\mod \mfp.
\]

We show the second assertion.
We consider the minor expansion of 
\[
p_{x}(T)=\det(T-x)
=\det
\begin{pmatrix}
 T-x_{11}&-x_{12}&\cdots&-x_{1N}\\
 -x_{21}&T-x_{22}&\cdots&-x_{2N}\\
 \vdots&\vdots&\ddots&\vdots\\
 -x_{N1}&-x_{N2}&\cdots&T-x_{NN}
\end{pmatrix}
\]
 with respect to the first column.
Then we have:
\[
\det(T-x)=(T-x_{11})\cdot p_{x,1}(T)+x_{21}\cdot p_{x,2}(T)+\cdots+(-1)^{N}x_{N1}\cdot p_{x,N}(T),
\]
where $p_{x,i}(T)$ is the determinant of the $(i,1)$-th minor matrix of $T-x$.

To determine the constant term of $p_{x}(T)$ modulo $\mfp^{2}$ with respect to $(T-1)$, it suffices to compute $p_{x}(T)$ modulo $(T-1, \mfp^{2})$.
Since $x$ lies in the pro-unipotent radical of the Iwahori subgroup, we have 
\begin{align*}
T-x_{11}&\equiv0 \quad \mod (T-1, \mfp),\text{ and}\\
x_{i1}&\equiv0 \quad \mod \mfp\quad\text{ for }1<i\leq N.
\end{align*}
On the other hand, for $1\leq i\leq N$, the $(i,1)$-th minor matrix of $x$ modulo $\mfp$ is given by
\[
\begin{pmatrix}
x_{12}&&&\multicolumn{1}{c:}{}&&&\\
1&&\ast&\multicolumn{1}{c:}{}&&&\\
0&\ddots&&\multicolumn{1}{c:}{}&&\ast&\\
&&1&\multicolumn{1}{c:}{x_{i-1,i}}&&&\\
\cdashline{1-7}
&&&\multicolumn{1}{c:}{}&1&\cdots&x_{i+1,N}\\
&&0&\multicolumn{1}{c:}{}&&\ddots&\vdots\\
&&&\multicolumn{1}{c:}{}&0&&1
\end{pmatrix},
\]
hence we have
\[
p_{x,i}(T) \equiv 
\begin{cases}
0 & \text{if } 1\leq i<N-1\\
(-1)^{N-1}x_{12}\cdots x_{N-1, N} & \text{if } i=N
\end{cases}
\mod (T-1, \mfp).
\]
Finally, $x_{N1}$ belongs to $\mfp$ and we have
\[
p_{x}(T)
\equiv a_{0}
\equiv -x_{12}\cdots x_{N-1, N}x_{N1} \quad \mod (T-1,\mfp^{2}).
\]
\end{proof}

\begin{lem}\label{lem:charpoly2}
Let $x\in P_{\bold{x}}^{+}$ and we put $x\theta(x)=(z_{ij})_{ij}$.
Let $p_{x,\theta}(T)\in F[T]$ be the characteristic polynomial of $x\theta(x)$, and we write
\[
p_{x,\theta}(T)=(T-1)^{2n+1}+a_{2n}(T-1)^{2n}+\cdots+a_{1}(T-1)+a_{0}.
\]
Then we have
\begin{enumerate}
 \item
 $a_{i}\in\mfp$ for each $0\leq i\leq 2n$,
 \item
 $a_{0}=0$, and
 \item
 $a_{1}\equiv -z_{2,3}\cdots z_{2n-1,2n}(z_{2n, 2n+1}z_{2n+1,2}+z_{2n,1}z_{1,2}) \quad\mod\mfp^{2}$.
\end{enumerate}
\end{lem}

\begin{proof}
Since $x\in P_{\bold{x}}^{+}$ and $\theta$ preserves $P_{\bold{x}}^{+}$, also $x\theta(x)$ is in $P_{\bold{x}}^{+}$.
Thus we have
\[
p_{x,\theta}(T)\equiv (T-1)^{2n+1} \quad\mod\mfp.
\]

We next prove (2).
We have to show that $1$ is an eigenvalue of $x\theta(x)$.
Let 
\[
\{\alpha_{1}, \ldots, \alpha_{2n+1}\}
\]
 be the set of eigenvalues of $x\theta(x)$.
Since $\theta$ is the composition of the $J_{2n+1}$-conjugation and the transpose-inverse, the set of eigenvalues of $\theta(x\theta(x))$ is given by 
\[
\{\alpha_{1}^{-1}, \ldots, \alpha_{2n+1}^{-1}\}.
\]
However, since we have
\[
\theta\bigl(x\theta(x)\bigr)=\theta(x)x=\theta(x)\cdot x\theta(x)\cdot\theta(x)^{-1},
\]
two elements $x\theta(x)$ and $\theta(x\theta(x))$ are conjugate.
Hence we have
\[
\{\alpha_{1}, \ldots, \alpha_{2n+1}\}=\{\alpha_{1}^{-1}, \ldots, \alpha_{2n+1}^{-1}\}.
\]
By the oddness of $2n+1$, we have $\alpha_{i}^{2}=1$ for at least one $i$.
Thus it suffices to show that $-1$ is not an eigenvalue of $x\theta(x)$.
However this is obvious since we have $\alpha_{i}\equiv1 \pmod{\mathfrak{p}}$ for every $i$ by (1) (recall that we assume that $p$ is not equal to $2$).

Finally we show (3).
We again consider the minor expansion of $p_{x,\theta}(T)$ with respect the first column:
\[
\det\bigl(T-x\theta(x)\bigr)=(T-z_{11})\cdot p_{x,\theta,1}(T)+z_{21}\cdot p_{x,\theta,2}(T)-\cdots-z_{2n+1,1}\cdot p_{x,\theta,2n+1}(T),
\]
where $p_{x,\theta,i}(T)$ is the determinant of the $(i,1)$-th minor matrix of $T-x\theta(x)$.

By the same argument as in the proof of Lemma \ref{lem:charpoly1}, we have
\begin{align*}
\det\bigl(T-x\theta(x)\bigr)
&\equiv(T-z_{11})\cdot p_{x,\theta,1}(T)+z_{2n,1}\cdot p_{x,\theta,2n}(T)
\quad\mod\bigl((T-1)^{2},\mfp^{2}\bigr)
\end{align*}
(note that $z_{2n+1,1}$ belongs to $\mfp^{2}$).
By Lemma \ref{lem:charpoly1}, the constant term of $p_{x,\theta,1}(T)$ (with respect to $(T-1)$) is given by $-z_{2,3}\cdots z_{2n,2n+1}z_{2n+1,2}$ modulo $\mfp^{2}$.
On the other hand, we have
\begin{align*}
p_{x,\theta,2n}(T)
&\equiv\det
\begin{pmatrix}
-z_{12}&&&\multicolumn{1}{c:}{}&\\
T-1&&\ast&\multicolumn{1}{c:}{}&\\
0&\ddots&&\multicolumn{1}{c:}{}&\ast\\
&&T-1&\multicolumn{1}{c:}{-z_{2n-1,2n}}&\\
\cdashline{1-5}
&&0&\multicolumn{1}{c:}{}&T-1
\end{pmatrix}\quad \mod\mfp \\
&\equiv(T-1)\cdot\det
\begin{pmatrix}
-z_{12}&&&\\
0&&\ast&\\
&\ddots&&\\
0&&0&-z_{2n-1,2n}
\end{pmatrix}\quad\mod\bigl((T-1)^{2},\mfp\bigr)\\
&\equiv -z_{12}\cdots z_{2n-1,2n}(T-1).
\end{align*}

Thus we get
\begin{multline*}
\det\bigl(T-x\theta(x)\bigr)\equiv
-z_{2,3}\cdots z_{2n-1,2n}(z_{2n, 2n+1}z_{2n+1,2}+z_{2n,1}z_{1,2})(T-1)\\
\quad\mod\bigl((T-1)^{2}, \mfp^{2}\bigr).
\end{multline*}

\end{proof}

\begin{proof}[Proof of Claim]
We show (1).
First, by Lemma \ref{lem:charpoly2} and irreducibility of Eisenstein polynomials, $g_{u}$ is a semisimple element.
We set
\[
\varphi'_{u}:=g_{u}-1
=
\begin{pmatrix}
 0&1&0&\hdots&0\\
 \vdots&\ddots&\ddots&&\vdots\\
 0&&\ddots&\ddots&0\\
 \frac{\varpi u}{2}&\ddots&&\ddots&1\\
0&\frac{\varpi u}{2}&0&\hdots&0
\end{pmatrix}.
\]
Then we have
\[
{\varphi'}_{u}^{i}=
\begin{cases}
\begin{pmatrix}
0&\hdots&\hdots&0&1&0&\hdots&0\\
\vdots&\ddots&&&\ddots&\ddots&\ddots&\vdots\\
0&&\ddots&&&\ddots&\ddots&0\\
\frac{\varpi u}{2}&&&0&&&\ddots&1\\
0&\varpi u&&&&&&0\\
\vdots&&\ddots&&&\ddots&&\vdots\\
\vdots&0&&\varpi u&&&\ddots&\vdots\\
0&\hdots&\hdots&0&\frac{\varpi u}{2}&0&\hdots&0
\end{pmatrix}&\text{if $2\leq i\leq 2n-1$,}\vspace{1mm}\\
\begin{pmatrix}
\frac{\varpi u}{2}&0&\hdots&0&1\\
0&\varpi u&\ddots&&0\\
\vdots&\ddots&\ddots&\ddots&\vdots\\
0&&\ddots&\varpi u&0\\
\frac{\varpi^{2}u^{2}}{4}&0&\hdots&0&\frac{\varpi u}{2}
\end{pmatrix} &\text{if $i=2n$,}\\
\varpi u \varphi'_{u}& \text{if $i=2n+1$,}
\end{cases}
\]
and
\[
J_{2n+1}{}^{t}\varphi'_{u}J_{2n+1}^{-1}
=
-\varphi'_{u}.
\]
Thus we have
\begin{align*}
g_{u}\theta(g_{u})
&= (1+\varphi'_{u})\cdot J_{2n+1}{}^{t}(1+\varphi'_{u})^{-1}J_{2n+1}^{-1}\\
&= (1+\varphi'_{u})\cdot J_{2n+1}(1-{}^{t}{\varphi'}_{u}+{}^{t}{\varphi'}_{u}^{2}-\cdots)J_{2n+1}^{-1}\\
&= (1+\varphi'_{u})\cdot (1+\varphi'_{u}+{\varphi'}_{u}^{2}+\cdots)\\
&= 1+\frac{2}{1-\varpi u}({\varphi'}_{u}+{\varphi'}_{u}^{2}+\cdots+{\varphi'}_{u}^{2n})\\
&=(1-\varpi u)^{-1} 
\begin{pmatrix}
 1&&&&\\
 \varpi u&1+\varpi u&&\mbox{\Large 2}&\\
 \vdots&&\ddots&&\\
 \varpi u&\mbox{\Large $2\varpi u$}&&1+\varpi u&\\
 \frac{\varpi^{2}u^2}{2}&\varpi u&\cdots&\varpi u&1
\end{pmatrix}
\end{align*}
and this element coincides with $\mathfrak{N}(1+\varphi_{u}^{\GL_{2n}})$ in \cite[Section 4.4]{MR3904769} (note that $\varphi_{u}^{\GL_{2n}}$ is denoted by $\varphi_{u}$ in \cite{MR3904769}).
In particular, $g_{u}\theta(g_{u})$ is strongly regular semisimple as an element of $\SO_{2n+1}(F)$ (\cite[Proposition 4.12]{MR3904769}).
Furthermore, since the modulo $\mathfrak{p}$ reduction of $g_{u}\theta(g_{u})$ is an upper-triangular matrix whose every diagonal entry is given by $1$, $-1$ is not an eigenvalue of $g_{u}\theta(g_{u})$ by the assumption that $p$ is not equal to $2$.
Thus $g_{u}\theta(g_{u})$ is strongly regular also as an element of $\GL_{2n+1}(F)$.
Therefore, by noting that $g_{u}\theta(g_{u})=\theta(g_{u})g_{u}$, \cite[Lemmas 4.3 and 4.11]{MR3904769} imply that that $g_{u}$ is a strongly $\theta$-regular $\theta$-semisimple element.

Next, if we put 
\[
h_{u}:=(1+\varphi_{u}^{\GL_{2n}})\theta(1+\varphi_{u}^{\GL_{2n}})=
\frac{1}{1-\varpi u}\begin{pmatrix}
 1+\varpi u&&2\\
 &\ddots&\\
 2\varpi u&&1+\varpi u
\end{pmatrix}
\in\Sp_{2n}(F),
\]
then $h_{u}$ is a norm of $g_{u}$.
Indeed, by the definition of the norm correspondence, $h_{u}$ is a norm of $g_{u}$ if and only if we have 
\[
\mathrm{Eig}\bigl(g_{u}\theta(g_{u})\bigr)=\mathrm{Eig}(h_{u})\sqcup\{1\},
\]
where $\mathrm{Eig}$ is the set of eigenvalues (for example, see \cite[Lemma 4.1]{MR3904769}).
However, since we have $g_{u}\theta(g_{u})=\mathfrak{N}(1+\varphi_{u}^{\GL_{2n}})$, this follows by Proposition 4.12 (1) in \cite{MR3904769}.
Therefore, by the twisted endoscopic character relation for $\GL_{2n+1}$ and $\Sp_{2n}$ (Proposition \ref{prop:TECRGL_{2n+1}}), we have
\[
\Theta_{\phi,\theta}^{\GL_{2n+1}}(g_{u})=\Theta^{\G}_{\xi,0,a}(h_{u})+\Theta^{\G}_{\xi,1,a\epsilon^{-1}}(h_{u}).
\]
By Corollary \ref{cor:pmSp}, the right-hand side of this equality is given by 
\[
\Kl_{2^{2n}au}^{n+1}(\psi; w_{\Sp_{2n}}), 
\]
and this completes the proof of (1).

We next show (2).
We take $x\in g_{u}P_{\bold{x}}^{++}$.
If $x$ does not have a norm in $G$ which belongs to $\pm I_{\mathbf{G}}^{+}$, then the character $\Theta_{\phi,\theta}^{\GL_{2n+1}}(x)$ is zero by Proposition \ref{prop:TECRGL_{2n+1}} and the character formula (Theorem \ref{thm:CF}).
Thus we may assume that $x$ has a norm $y\in \pm I_{\mathbf{G}}^{+}\subset G$.
We let $p_{x,\theta}(T)\in F[T]$ and $p_{y}(T)\in F[T]$ be the characteristic polynomials of $x\theta(x)$ and $y$, respectively.
Then, again by noting the above interpretation of the norm correspondence in terms of eigenvalues, we have
\[
p_{x,\theta}(T)=p_{y}(T)\cdot(T-1).
\]
We put $y=(y_{ij})_{ij}$ and $z=(z_{ij})_{ij}:=x\theta(x)$.
Then, by Lemmas \ref{lem:charpoly1} and \ref{lem:charpoly2}, we have
\[
-y_{12}\cdots y_{2n-1,2n}y_{2n,1}
\equiv
-z_{2,3}\cdots z_{2n-1,2n}(z_{2n, 2n+1}z_{2n+1,2}+z_{2n,1}z_{1,2}) \quad\mod\mfp^{2}.
\]
On the other hand, we note that $P_{\bold{x}}^{+}/P_{\bold{x}}^{++}$ is isomorphic to $(\mathcal{O}/\mathfrak{p})^{\oplus2n}\oplus(\mathfrak{p}/\mathfrak{p}^{2})^{\oplus2}$ by the map
\[
(x_{ij})_{ij}
\mapsto
(x_{12},\ldots,x_{2n,2n+1},x_{2n,1},x_{2n+1,2})
\]
and the action of $\theta$ on $P_{\bold{x}}^{+}/P_{\bold{x}}^{++}$ is described on $(\mathcal{O}/\mathfrak{p})^{\oplus2n}\oplus(\mathfrak{p}/\mathfrak{p}^{2})^{\oplus2}$ by
\[
(x_{12},\ldots,x_{2n,2n+1},x_{2n,1},x_{2n+1,2})
\mapsto
(x_{2n,2n+1},\ldots,x_{1,2},x_{2n+1,2},x_{2n,1}).
\]
Since $x\in g_{u}P_{\bold{x}}^{++}$ and the image of $g_{u}$ in $P_{\bold{x}}^{+}/P_{\bold{x}}^{++}$ is given by $(1,\ldots,1,\frac{\varpi u}{2},\frac{\varpi u}{2})$, the image of $z=x\theta(x)$ in $P_{\bold{x}}^{+}/P_{\bold{x}}^{++}$ is given by $(2,\ldots,2,\varpi u,\varpi u)$.
In other words, we have
\[
z_{1,2}\equiv\cdots\equiv z_{2n,2n+1}\equiv2 \mod \mfp,
\]
\[
z_{2n,1}\equiv z_{2n+1,2}\equiv \varpi u \mod \mfp^{2}.
\]
Hence we get
\[
z_{2,3}\cdots z_{2n-1,2n}(z_{2n, 2n+1}z_{2n+1,2}+z_{2n,1}z_{1,2})\equiv 2^{2n}\varpi u \quad\mod\mfp^{2}.
\]
By the twisted endoscopic character relation (Proposition \ref{prop:TECRGL_{2n+1}}) and Corollary \ref{cor:pmSp}, we get
\[
\Theta_{\phi,\theta}^{\GL_{2n+1}}(x)
=\Theta^{\G}_{\xi,0,a}(y)+\Theta^{\G}_{\xi,1,a\epsilon^{-1}}(y)
=\Kl_{2^{2n}au}^{n+1}(\psi; w_{\Sp_{2n}}).
\]
\end{proof}

\subsection{Simple supercuspidal $L$-packet of $\SO_{2n}^{\mu}$}\label{subsec:packer-SO-ram}
We next determine the $L$-packet $\widetilde{\Pi}_{\phi_{0}}^{\mathbf{H}}$.
By the irreducibility of $\phi_{0}$, the group $\mathcal{S}^{\mathbf{H}}_{\phi_{0}}$ is trivial.
Therefore, by Theorem \ref{thm:Arthur} (1), $\widetilde{\Pi}_{\phi_{0}}^{\mathbf{H}}$ is a singleton.
Let $\tilde{\pi}^{\mathbf{H}}$ be the unique $\rmO_{2n}^{\mu}(F)$-orbit of irreducible smooth representations of $H$ which belongs to the $L$-packet $\widetilde{\Pi}_{\phi_{0}}^{\mathbf{H}}$ of $\phi_{0}$.
In this section, we determine $\tilde{\pi}^{\mathbf{H}}$.

We start from determining the quadratic character $\mu$ in $\mathbf{H}=\SO_{2n}^{\mu}$:
\begin{prop}\label{prop:mu}
The quadratic character $\mu$ corresponds to the quadratic extension
\[
F\Bigl(\sqrt{(-1)^{n-1}a\varpi}\Bigr)
\]
of $F$.
\end{prop}

\begin{proof}
We take $u\in k^{\times}$ and the element $g_{u}^{\G}\in G^{srs}$ considered in Section \ref{subsec:tran-Sp}.
We consider the endoscopic character relation (Theorem \ref{thm:SECR}) for $\Pi_{\phi}^{\G}$ and $\widetilde{\Pi}_{\phi_{0}}^{\mathbf{H}}$ :
\[
C\cdot
\bigl(\Theta^{\G}_{\xi,0,a}-\Theta^{\G}_{\xi,1,a\epsilon^{-1}}\bigr)(g_{u}^{\G})
  = \sum_{h\leftrightarrow g_{u}^{\G}/\sim} \frac{D_{\mathbf{H}}(h)^{2}}{D_{\G}(g_{u}^{\G})^{2}} \Delta_{\mathbf{H},\G}(h,g_{u}^{\G}) \Theta_{\tilde{\pi}^{\mathbf{H}}}(h).
\]
Here note that $C$ is a sign which is defined by
\[
C=
\begin{cases}
1 & \text{if $\pi_{\xi,0,a}^{\G}$ is $\mathfrak{w}_{\G}$-generic},\\
-1 & \text{if $\pi_{\xi,1,a\epsilon^{-1}}^{\G}$ is $\mathfrak{w}_{\G}$-generic}.
\end{cases}
\]

If the left-hand side of this equality is not zero, then there exists a norm of $g_{u}^{\G}$ in $H=\SO_{2n}^{\mu}(F)$.
Since $g_{u}^{\G}$ corresponds to the data $\xi_{u}^{\G}$, $\mu$ is characterized as the unique quadratic character such that $\SO_{2n}^{\mu}$ has an element which corresponds to the data $\xi_{u}$.
Then, by Lemma \ref{lem:type}, $\mu$ is the quadratic character of $F^{\times}$ corresponding to the ramified quadratic extension $F\bigl(\sqrt{(-1)^{n-1}\varpi u}\bigr)$.

Thus it is enough to find $u\in k^{\times}$ such that the left-hand side of the above equality does not vanish.
Since the simple affine components of $g_{u}^{\G}$ are given by $(2, \ldots, 2, 2u)$ (see Proposition \ref{prop:affgen-Sp}), by Corollary \ref{cor:pmSp}, we have
\[
\Theta^{\G}_{\xi,0,a}(g_{u}^{\G})-\Theta^{\G}_{\xi,1,a\epsilon^{-1}}(g_{u}^{\G})
=\omega_{0}(2au)\cdot \Kl_{2^{2n}au}^{n+1}(\psi;w_{\Sp_{2n}},\chi_{\Sp_{2n}}).
\]
By Proposition \ref{prop:Kl} (2), there exists $u\in k^{\times}$ such that $\Kl_{2^{2n}au}^{n+1}(\psi;w_{\Sp_{2n}},\chi_{\Sp_{2n}})$ is not zero.
Moreover, by Lemma \ref{lem:Klvan}, such $u$ belongs to $a^{-1}k^{\times2}$.
Therefore $\mu$ is the quadratic character of $F^{\times}$ corresponding to the ramified quadratic extension
\[
F\Bigl(\sqrt{(-1)^{n-1}a\varpi}\Bigr)
\]
of $F$.
\end{proof}

We take a representative $\pi^{\mathbf{H}}$ of $\tilde{\pi}^{\mathbf{H}}$.

\begin{prop}\label{prop:ssc-descent}
The representation $\pi^{\mathbf{H}}$ is simple supercuspidal.
\end{prop}

\begin{proof}
We first recall a description of the $L$-packet $\Pi_{\phi}^{\G}$ in terms of the local theta correspondence according to Gan--Ichino (\cite[Section C.2]{MR3166215}).

The groups $\Sp_{2n}$ and $\rmO_{2n}^{\mu}$ of our interest can be realized as a reductive dual pair of type (D) in the sense of \cite[Section 3]{MR3166215}.
More precisely, in the notation of \cite[Section 3]{MR3166215}, we choose data of classical groups as follows:
\begin{itemize}
\item
We let $V$ be a $2n$-dimensional vector space over $F$ equipped with a non-degenerate symplectic form $(-,-)\colon V\times V\rightarrow F$.
Then the associated classical group $H(V)$ is isomorphic to $\Sp_{2n}$.
\item
We let $W$ be a $2n$-dimensional vector space over $F$ equipped with a non-degenerate symmetric form $\langle-,-\rangle\colon W\times W\rightarrow F$ such that the associated classical group $G(W)$ is isomorphic to $\rmO_{2n}^{\mu}$.
\end{itemize}
Then the pair $(\Sp_{2n},\rmO_{2n}^{\mu})$ is of the case (D).

Since the invariant $l$ of this pair is given by $-1$ (see \cite[Section 3]{MR3166215}), we may apply \cite[Theorem C.5]{MR3166215} to the character twist of our $L$-parameter $\phi_{0}\otimes\chi_{W}^{-1}$ of $G(W)=\rmO_{2n}^{\mu}$.
Here note that an equivalence class of $L$-parameters of $\rmO_{2n}^{\mu}$ are nothing but a $\Sigma(\SO_{2n}(\C))$-orbit of equivalence classes of $L$-parameters of $\SO_{2n}^{\mu}$, in other words, elements of $\widetilde{\Phi}(\SO_{2n}^{\mu})$.
According to \cite[Theorem C.5]{MR3166215}, we put an $L$-parameter $\theta(\phi_{0}\otimes\chi_{W}^{-1})$ of $H(V)$ to be
\[
\theta(\phi_{0}\otimes\chi_{W}^{-1})
:=(\phi_{0}\otimes\chi_{W}^{-1}\otimes\chi_{V}^{-1}\chi_{W})\oplus\chi_{W}
=\phi_{0}\oplus\mu
=\phi,
\]
where $\chi_{V}$ and $\chi_{W}$ are given by $\mathbbm{1}$ and $\mu$ in this situation, respectively (see \cite[Section 3]{MR3166215}).
Then, since the $L$-parameter $\phi_{0}\otimes\chi_{W}^{-1}$ does not contain $\chi_{V}$, the assertions of \cite[Theorem C.5 (i)]{MR3166215} hold.
In particular, for any member $\pi$ of the $L$-packet $\Pi_{\phi_{0}\otimes\chi_{W}^{-1}}^{G(W)}$, its small theta lift $\theta_{V,W,\boldsymbol{\chi},\psi}(\pi)$ to $H(V)$ (with respect to fixed auxiliary data $\boldsymbol{\chi}=(\chi_{V},\chi_{W})$ and $\psi$, see \cite[Sections 3 and 5]{MR3166215}) is non-zero irreducible and belongs to the $L$-packet $\Pi_{\phi}^{H(V)}=\Pi_{\phi}^{\G}$, which consists of two simple supercuspidal representations of $\Sp_{2n}(F)$.

Then, since the depth of any simple supercuspidal representation of $\Sp_{2n}(F)$ is given by $\frac{1}{2n}$ (Proposition \ref{prop:depth-ssc}), the depth of any element of $\Pi_{\phi_{0}\otimes\chi_{W}^{-1}}^{G(W)}$ is equal $\frac{1}{2n}$ by Pan's depth-preserving theorem for theta correspondence (\cite{MR1909608}).

Now let us investigate the relationship between the $L$-packets $\Pi_{\phi_{0}\otimes\chi_{W}^{-1}}^{G(W)}$ and $\widetilde{\Pi}_{\phi_{0}}^{\mathbf{H}}$.
Note that $G(W)$ is the full orthogonal group $\rmO_{2n}^{\mu}$, hence the set $\Pi_{\phi_{0}\otimes\chi_{W}^{-1}}^{G(W)}$ consists of irreducible smooth representations of $\rmO_{2n}^{\mu}(F)$.
This set can be interpreted as a coarse $L$-packet $\widetilde{\Pi}_{\phi_{0}\otimes\chi_{W}^{-1}}^{\mathbf{H}}$, that is, a finite set of $\Sigma(\mathbf{H})$-orbits of irreducible smooth representations of $H=\SO_{2n}^{\mu}(F)$, by the following basic fact on the correspondence between irreducible smooth representations of $\rmO_{2n}^{\mu}(F)$ and $\Sigma(\mathbf{H})$-orbits of irreducible smooth representations of $H$ (see \cite[Section 13.3]{MR3166215}): for an irreducible smooth representation $\pi$ of $\rmO_{2n}^{\mu}(F)$, 
\begin{itemize}
\item
if the restriction $\pi|_{H}$ to $H$ is irreducible, then $\{\pi|_{H}\}$ is a $\Sigma(\mathbf{H})$-orbit of irreducible smooth representations of $H$, and
\item
if the restriction $\pi|_{H}$ to $H$ is reducible, then $\pi|_{H}\cong\pi_{1}\oplus\pi_{2}$ for two irreducible constituent $\pi_{1}$ and $\pi_{2}$ and $\{\pi_{1},\pi_{2}\}$ is a $\Sigma(\mathbf{H})$-orbit of irreducible smooth representations of $H$.
\end{itemize}
Furthermore, by Proposition \ref{prop:spin-twist}, we have
\[
\widetilde{\Pi}_{\phi_{0}\otimes\chi_{W}^{-1}}^{\mathbf{H}}
=
\widetilde{\Pi}_{\phi_{0}}^{\mathbf{H}}\otimes(\chi_{W}^{-1}\circ\spin),
\]
where $\spin$ is the spinor norm of $H$, which is a homomorphism from $H$ to $F^{\times}/F^{\times2}$ (see Section \ref{sec:spin}).

Since the character $\chi_{W}^{-1}\circ\spin$ is quadratic, $\chi_{W}^{-1}\circ\spin$ is trivial on the pro-unipotent radical of any parahoric subgroup of $H$.
In particular, twisting by the character $\chi_{W}^{-1}\circ\spin$ does not change the depth of a representation of $H$.
Therefore, by combining this fact with the previous observation that the depth of any member of $\widetilde{\Pi}_{\phi_{0}\otimes\chi_{W}^{-1}}^{\mathbf{H}}$ (or $\Pi_{\phi_{0}\otimes\chi_{W}^{-1}}^{G(W)}$) is given by $\frac{1}{2n}$, we conclude that the depth of any member of $\widetilde{\Pi}_{\phi_{0}}^{\mathbf{H}}$ is given by $\frac{1}{2n}$.
Finally, by Proposition \ref{prop:depth-ssc}, this implies that any member of $\widetilde{\Pi}_{\phi_{0}}^{\mathbf{H}}$ is simple supercuspidal (or, strictly speaking, a representative of a $\Sigma(\mathbf{H})$-orbit of simple supercuspidal representations).
This completes the proof.
\end{proof}

\begin{rem}
In the proof of the above proposition, we have to take care of the difference between the local theta correspondences used in \cite{MR3166215} and \cite{MR1909608}.
More precisely, in \cite{MR3166215}, the local theta correspondence is defined by using Kudla's splitting for metaplectic covers.
On the other hand, in \cite{MR1909608}, the local theta correspondence is defined by using Pan's splitting for metaplectic covers.
Because of this difference of the choices of splittings, the local theta correspondence used in \cite{MR3166215} differs from the one in \cite{MR1909608} by a character twist.
For the groups treated in this paper (symplectic groups and orthogonal groups), the character measuring the difference of splittings values in $\{\pm1\}$, hence does not affect the depth of representations.
Thus we can combine two results in \cite{MR3166215} and \cite{MR1909608} as above.
\end{rem}

Now we put $\pi^{\mathbf{H}}=\pi_{\xi',a'}^{\mathbf{H}}$ for $(\xi',a')\in\SSC(\mathbf{H})$.

\begin{defn}\label{def:nu}
We put $\nu_{\mu}$ to be an element of $\{0,1\}$ satisfying 
\[
\omega_{0}(\epsilon^{\nu_{\mu}})=\omega_{0}\bigl((-1)^{n-1}a\bigr).
\]
\end{defn}
Note that, by Proposition \ref{prop:mu}, we have
\[
\nu_{\mu}=
\begin{cases}
0& \text{if $\mu$ corresponds to $F(\sqrt{\varpi})$},\\
1& \text{if $\mu$ corresponds to $F(\sqrt{\varpi\epsilon})$}.
\end{cases}
\]

\begin{thm}\label{thm:ram-packet}
We have
\[
C=1,\quad
\xi'=\xi\cdot\omega_{0}(-1),
 \text{ and}\quad
{a'}^{2}=(-1)^{n-1}\epsilon^{\nu_{\mu}}4a.
\]
Here $C$ is the sign defined in the proof of Proposition \ref{prop:mu}.
Thus, in other words, the representation $\pi_{\xi,0,a}^{\G}$ is $\mathfrak{w}_{\G}$-generic.
\end{thm}

\begin{proof}
We again consider the endoscopic character relation for $\Pi_{\phi}^{\G}$ and $\widetilde{\Pi}_{\phi_{0}}^{\mathbf{H}}$ at $g_{u}^{\G}$, which is considered in the proof of Proposition \ref{prop:mu}:
\[
C\cdot
\bigl(\Theta^{\G}_{\xi,0,a}-\Theta^{\G}_{\xi,1,a\epsilon^{-1}}\bigr)
(g_{u}^{\G})
  = \sum_{h\leftrightarrow g_{u}^{\G}/\sim} \frac{D_{\mathbf{H}}(h)^{2}}{D_{\G}(g_{u}^{\G})^{2}} \Delta_{\mathbf{H},\G}(h,g_{u}^{\G}) \Theta_{\tilde{\pi}^{\mathbf{H}}}(h).
\]
Here note that $g_{u}^{\G}$ has a norm in $H$ only if we have $u\in a^{-1}k^{\times2}$.
In the case where $u\notin a^{-1}k^{\times2}$, both sides of the above equality are zero.

Let $u\in a^{-1}k^{\times2}$.
Recall that the stable conjugacy classes of the norms of $g_{u}^{\G}$ in $H$ are represented by two elements $g_{u}^{\mathbf{H}}$ and $w_{\mu}g_{u}^{\mathbf{H}}w_{\mu}^{-1}$ (see Sections \ref{subsec:tran-SO-ram} and \ref{subsec:tran-Sp}).
Thus, by the above relation and Propositions \ref{prop:4Sp} and \ref{prop:tranSp}, we get
\begin{align*}
C\cdot
\bigl(\Theta^{\G}_{\xi,0,a}-\Theta^{\G}_{\xi,1,a\epsilon^{-1}}\bigr)
(g_{u}^{\G})
 &= \omega_{0}(-2)\cdot q\cdot G(\omega_{0},\psi)^{-1} \bigl(\Theta_{\tilde{\pi}^{\mathbf{H}}}(g_{u}^{\mathbf{H}})+\Theta_{\tilde{\pi}^{\mathbf{H}}}(w_{\mu}g_{u}^{\mathbf{H}}w_{\mu}^{-1})\bigr)\\ 
 &= \omega_{0}(-2)\cdot q\cdot G(\omega_{0},\psi)^{-1} \bigl(\Theta^{\mathbf{H}}_{\xi',a'}(g_{u}^{\mathbf{H}})+\Theta^{\mathbf{H}}_{\xi',a'}(w_{\mu}g_{u}^{\mathbf{H}}w_{\mu}^{-1})\bigr).
\end{align*}

Now we recall that the simple affine components of $g_{u}^{\mathbf{H}}$ and $w_{\mu}g_{u}^{\mathbf{H}}w_{\mu}^{-1}$ are given by 
\[
(2, \ldots, 2, 2v^{-1})\quad\text{and}\quad(2, \ldots, 2, -2v^{-1}), 
\]
respectively (see Proposition \ref{prop:affgen-SOram}).
Here $\pm v$ are the square roots of 
\[
\begin{cases}
(-1)^{n-1} u^{-1} & \text{if } (-1)^{n-1} u \in {k^{\times}}^{2}, \\
(-1)^{n-1} u^{-1}\epsilon & \text{if } (-1)^{n-1} u \notin {k^{\times}}^{2}.
\end{cases}
\]
Thus, by Corollary \ref{cor:pmSp} and Proposition \ref{prop:charSOram}, we get 
\[
C\cdot
\omega_{0}(2au)\cdot \Kl_{2^{2n}au}^{n+1}(\psi;w_{\Sp_{2n}},\chi_{\Sp_{2n}})
  = \omega_{0}(-2)\cdot q\cdot G(\omega_{0},\psi)^{-1}\cdot\Kl^{n}_{\pm2^{n}a'v^{-1}}(\psi).
\]

We consider the sum of this equality over $u\in a^{-1}k^{\times2}$ twisted by a multiplicative character $\chi$ on $k^{\times}$:
\begin{itemize}
\item
By Lemma \ref{lem:Klvan}, we have
\begin{align*}
&\sum_{u\in a^{-1}k^{\times2}} \chi(u)\cdot\LHS\\
&=C\cdot\omega_{0}(2)\chi(2^{2n}a)^{-1}\sum_{u'\in k^{\times}} \chi\omega_{0}(u')\cdot\Kl_{u'}^{n+1}(\psi;w_{\Sp_{2n}},\chi_{\Sp_{2n}}),
\end{align*}
where we replaced $2^{2n}au$ with $u'$.
By Proposition \ref{prop:Kl} (1), this equals
\[
C\cdot\omega_{0}(2)\chi(2^{2n}a)^{-1}\cdot G(\chi^{2},\psi)^{n-1}\cdot G(\chi\omega_{0},\psi)\cdot G(\chi,\psi).
\]
\item
By Proposition \ref{prop:Kl} (1), we have
\begin{align*}
&\sum_{u\in a^{-1}k^{\times2}} \chi(u)\cdot\RHS\\
&= \omega_{0}(-2)\cdot q\cdot G(\omega_{0},\psi)^{-1}\sum_{v\in k^{\times}}\chi\bigl((-1)^{n-1}v^{-2}\epsilon^{\nu_{\mu}}\bigr)\cdot\Kl^{n}_{2^{n}a'v^{-1}}(\psi)\\
&= \omega_{0}(-2)\cdot q\cdot G(\omega_{0},\psi)^{-1}\cdot\chi\bigl(2^{-2n}{a'}^{-2}(-1)^{n-1}\epsilon^{\nu_{\mu}}\bigr)\sum_{v'\in k^{\times}}\chi(v')^{2}\cdot\Kl^{n}_{v'}(\psi)\\
&= \omega_{0}(-2)\cdot q\cdot G(\omega_{0},\psi)^{-1}\cdot\chi\bigl(2^{-2n}{a'}^{-2}(-1)^{n-1}\epsilon^{\nu_{\mu}}\bigr)G(\chi^{2},\psi)^{n}.
\end{align*}
Here, in the second equality, we replaced $2^{n}a'v^{-1}$ with $v'$.
\end{itemize}
Therefore we have
\[
C\cdot\chi(a)^{-1}\cdot G(\chi\omega_{0},\psi)G(\chi,\psi)
=
q\cdot\omega_{0}(-1)\chi\bigl({a'}^{-2}(-1)^{n-1}\epsilon^{\nu_{\mu}}\bigr)\cdot G(\omega_{0},\psi)^{-1}G(\chi^{2},\psi).
\]
By using the Hasse-Davenport product relation (Lemma \ref{lem:HD} (2))
\[
G(\chi\omega_{0},\psi)G(\chi,\psi)=G(\chi^{2},\psi)G(\omega_{0},\psi)\chi(4)^{-1},
\]
and the relation
\[
G(\omega_{0}, \psi)^{2}=q \cdot \omega_{0}(-1)
\]
in Lemma \ref{lem:HD} (1), we get
\[
C\cdot
\chi(a)^{-1}
=
\chi\bigl(4{a'}^{-2}(-1)^{n-1}\epsilon^{\nu_{\mu}}\bigr).
\]
Since this equality holds for every $\chi$, we have
\[
C=1 \quad\text{and}\quad
{a'}^{2}=(-1)^{n-1}\epsilon^{\nu_{\mu}}4a.
\]

We finally compute $\xi'$.
Let us take $u\in a^{-1}k^{\times2}$ such that 
\[
\Kl_{2^{2n}au}^{n+1}(\psi;w_{\Sp_{2n}},\chi_{\Sp_{2n}})\neq0
\]
(note that there exists such an element by Proposition \ref{prop:Kl} and Lemma \ref{lem:Klvan}).
Then, by the endoscopic character relation (Theorem \ref{thm:SECR}) for $g_{u}^{\G}$ and $-g_{u}^{\G}$ and Propositions \ref{prop:4Sp} and \ref{prop:tranSp}, we have
\[
\bigl(\Theta^{\G}_{\xi,0,a}-\Theta^{\G}_{\xi,1,a\epsilon^{-1}}\bigr)(g_{u}^{\G})
  = \omega_{0}(-2)\cdot q\cdot G(\omega_{0},\psi)^{-1} \bigl(\Theta^{\mathbf{H}}_{\xi',a'}(g_{u}^{\mathbf{H}})+\Theta^{\mathbf{H}}_{\xi',a'}(w_{\mu}g_{u}^{\mathbf{H}}w_{\mu}^{-1})\bigr)
\]
and
\[
\xi\cdot\bigl(\Theta^{\G}_{\xi,0,a}-\Theta^{\G}_{\xi,1,a\epsilon^{-1}}\bigr)(g_{u}^{\G})
  = \xi'\cdot\omega_{0}(2)\cdot q\cdot G(\omega_{0},\psi)^{-1} \bigl(\Theta^{\mathbf{H}}_{\xi',a'}(g_{u}^{\mathbf{H}})+\Theta^{\mathbf{H}}_{\xi',a'}(w_{\mu}g_{u}^{\mathbf{H}}w_{\mu}^{-1})\bigr).
\]
By the assumption on $u$, the left-hand sides of these equalities are not zero.
Hence we get
\[
\xi'=\xi\cdot\omega_{0}(-1).
\]
\end{proof}

\subsection{Endoscopic lifting from $\SO_{2n}^{\mu}$ to $\GL_{2n}$}\label{subsec:lift-SO-ram-GL}
In this subsection, we determine the endoscopic lift $\pi_{\phi_{0}}^{\GL_{2n}}$ of $\pi^{\mathbf{H}}_{\xi',a'}$ from $\mathbf{H}$ to $\GL_{2n}$.

By Proposition \ref{prop:depth}, $\pi_{\phi_{0}}^{\GL_{2n}}$ is depth-zero supercuspidal or simple supercuspidal.
We first show the simple supercuspidality of $\pi_{\phi_{0}}^{\GL_{2n}}$.
We recall a fact about depth-zero supercuspidal representations.

\begin{prop}\label{prop:MP}
Let $\pi$ be an irreducible depth-zero supercuspidal representation of $\GL_{2n}(F)$.
Then there exists an irreducible smooth representation $\tilde{\rho}$ of $Z_{\GL_{2n}}\GL_{2n}(\mcO)$ such that
\begin{itemize}
 \item
 its restriction to $\GL_{2n}(\mcO)$ is an inflation of an irreducible cuspidal representation $\rho$ of $\GL_{2n}(k)$, and 
 \item
 $\pi$ is equivalent to $\cInd_{Z_{\GL_{2n}}\GL_{2n}(\mcO)}^{\GL_{2n}(F)}\tilde{\rho}$
\end{itemize}
(namely, the pair $(\GL_{2n}(\mcO), \tilde{\rho})$ is a depth-zero unrefined minimal K-type of $\pi$ in the sense of \cite[Section 3.4]{MR1371680}).
Moreover, if $\pi$ is $\theta$-stable, then $\tilde{\rho}$ is also $\theta$-stable.
\end{prop}

\begin{proof}
By Theorem 6.1.2 in \cite{MR2480618}, there exists a pair of a hyperspecial subgroup $P$ and an irreducible smooth representation of $Z_{\GL_{2n}}P$ satisfying the required conditions.
Since all hyperspecial subgroups of $\GL_{2n}(F)$ are conjugate, we get the first assertion.

We next assume that $\pi$ is $\theta$-stable.
Then we have
\[
\cInd_{Z_{\GL_{2n}}\GL_{2n}(\mcO)}^{\GL_{2n}(F)}\tilde{\rho}\cong\pi
\cong\pi^{\theta}
\cong\bigl(\cInd_{Z_{\GL_{2n}}\GL_{2n}(\mcO)}^{\GL_{2n}(F)}\tilde{\rho}\bigr)^{\theta}
\cong\cInd_{Z_{\GL_{2n}}\GL_{2n}(\mcO)}^{\GL_{2n}(F)}\tilde{\rho}^{\theta},
\]
namely, $(\GL_{2n}(\mcO), \tilde{\rho}^{\theta})$ is also a depth-zero unrefined minimal $K$-type of $\pi$.
Therefore $(\GL_{2n}(\mcO), \tilde{\rho})$ and $(\GL_{2n}(\mcO), \tilde{\rho}^{\theta})$ are associated by Theorem 3.5 in \cite{MR1371680}.
That is, there exists $g\in \GL_{2n}(F)$ such that
\begin{enumerate}
\item
the intersection $\GL_{2n}(\mcO)\cap \GL_{2n}(\mcO)^{g}$ surjects onto both $\GL_{2n}(\mcO)/\GL_{2n}(\mcO)^{+}$ and $\GL_{2n}(\mcO)^{g}/(\GL_{2n}(\mcO)^{g})^{+}$, and
\item
$\tilde{\rho}$ and $(\tilde{\rho}^{\theta})^{g}$ are isomorphic on $\GL_{2n}(\mcO)\cap \GL_{2n}(\mcO)^{g}$.
\end{enumerate}
Here, $\GL_{2n}(\mcO)^{+}$ (resp.\ $(\GL_{2n}(\mcO)^{g})^{+}$) denotes the pro-unipotent radical of $\GL_{2n}(\mcO)$ (resp.\ $\GL_{2n}(\mcO)^{g}$).
However the first condition implies that $g\in Z_{\GL_{2n}}\GL_{2n}(\mcO)$.
Indeed, by the Cartan decomposition, $g$ lies in a double coset $\GL_{2n}(\mcO)t\GL_{2n}(\mcO)$ for some $t=\diag(\varpi^{r_{1}},\ldots,\varpi^{r_{2n}})$, where $\varpi$ is the fixed uniformizer of $F$ and $r_{1},\ldots,r_{2n}$ are integers satisfying $r_{1}\geq\cdots\geq r_{2n}$.
Then the condition (1) implies that the map
\[
\GL_{2n}(\mcO)\cap \GL_{2n}(\mcO)^{t}
\rightarrow
\GL_{2n}(\mcO)/\GL_{2n}(\mcO)^{+}
\]
is surjective.
Note that $\GL_{2n}(\mcO)^{t}=t^{-1}\GL_{2n}(\mcO)t$ consists of matrices $x=(x_{ij})_{ij}$ whose $(i,j)$-entry $x_{ij}$ satisfies $\val(x_{ij})\geq-r_{i}+r_{j}$.
Hence, the surjectivity of the above map implies that $-r_{i}+r_{j}\geq0$ must hold for any $i$ and $j$.
Thus we must have $r_{1}=\cdots=r_{2n}$.
In other words, $t$ lies in $Z_{\GL_{2n}}$.
 
From this observation, we know that $\tilde{\rho}$ is isomorphic to $\tilde{\rho}^{\theta}$ on $\GL_{2n}(\mcO)$.
Furthermore, the restrictions of $\tilde{\rho}$ and $\tilde{\rho}^{\theta}$ to $Z_{\GL_{2n}}$ are equal to the central character of $\pi$.
Therefore $\tilde{\rho}$ and $\tilde{\rho}^{\theta}$ are isomorphic on $Z_{\GL_{2n}}\GL_{2n}(\mcO)$.
\end{proof}

To show the simple supercuspidality of $\pi_{\phi_{0}}^{\GL_{2n}}$, we prove the following lemma.

\begin{lem}\label{lem:special-Iwahori}
Let $g\in I_{\GL_{2n}}^{+}$ be an affine generic element.
Then the group
\[
\{x\in \GL_{2n}(F) \mid xgx^{-1}\in Z_{\GL_{2n}}\GL_{2n}(\mcO) \}
\]
is equal to $Z_{\GL_{2n}}\GL_{2n}(\mcO)\lan\varphi_{u}^{\GL_{2n}}\ran$, for any $u\in k^{\times}$.
\end{lem}

\begin{proof}
The group $Z_{\GL_{2n}}\GL_{2n}(\mcO)\lan\varphi_{u}^{\GL_{2n}}\ran$ is obviously contained in the group in the assertion (recall that $\varphi_{u}^{\GL_{2n}}$ normalizes the standard Iwahori subgroup).
Thus it is suffice to show the other inclusion.

Let $x$ be an element of $\GL_{2n}(F)$ satisfying $xgx^{-1}\in Z_{\GL_{2n}}\GL_{2n}(\mcO)$.
Then, since the valuation of the determinant of $xgx^{-1}$ is zero, $xgx^{-1}$ belongs to $\GL_{2n}(\mcO)$. 
On the other hand, since $g$ lies in the Iwahori subgroup, the characteristic polynomial of $xgx^{-1}$ is $\mcO$-coefficient and congruent to $(T-1)^{2n}$ modulo $\mfp$.
In particular, the characteristic polynomial of the reduction $\ol{xgx^{-1}}\in \GL_{2n}(k)$ is given by $(T-1)^{2n}$.
Therefore $\ol{xgx^{-1}}$ can be upper-triangulated in $\GL_{2n}(k)$.
This implies that there exists an element $y$ of $\GL_{2n}(\mcO)$ such that $yxgx^{-1}y^{-1}$ belongs to the Iwahori subgroup.
Then we can apply Lemma \ref{lem:key-lem} to the element $yx$, and we have $yx\in Z_{\GL_{2n}}I_{\GL_{2n}}\lan\varphi_{u}^{\GL_{2n}}\ran$.
Hence $x$ belongs to $Z_{\GL_{2n}}\GL_{2n}(\mcO)\lan\varphi_{u}^{\GL_{2n}}\ran$.
\end{proof}

\begin{prop}\label{prop:not0}
The depth of $\pi_{\phi_{0}}^{\GL_{2n}}$ is not zero.
In particular, $\pi_{\phi_{0}}^{\GL_{2n}}$ is simple supercuspidal.
\end{prop}

\begin{proof}
We suppose that the depth of $\pi_{\phi_{0}}^{\GL_{2n}}$ is zero.
By using Proposition \ref{prop:MP}, $\pi_{\phi_{0}}^{\GL_{2n}}\cong\cInd_{Z_{\GL_{2n}}\GL_{2n}(\mcO)}^{\GL_{2n}(F)}\tilde{\rho}$ for a $\theta$-stable irreducible smooth representation of $Z_{\GL_{2n}}\GL_{2n}(\mcO)$.
If we let $I_{\theta}$ be an intertwiner $I_{\theta}\colon\tilde{\rho}\cong\tilde{\rho}^{\theta}$ such that $\cInd I_{\theta}$ coincides with the intertwiner of $\pi_{\phi_{0}}^{\GL_{2n}}$ normalized via the fixed Whittaker datum $\mathfrak{w}_{\GL_{2n}}$, then the twisted character formula (Theorem \ref{thm:TCF}) gives the equality
\[
\Theta_{\phi_{0},\theta}^{\GL_{2n}}(g) = 
\sum_{\begin{subarray}{c}x\in Z_{\GL_{2n}}\GL_{2n}(\mcO)\backslash \GL_{2n}(F)\\ xg\theta(x)^{-1}\in Z_{\GL_{2n}}\GL_{2n}(\mcO) \end{subarray}}
\tr \bigl(\tilde{\rho}(xg\theta(x)^{-1})\circ I_{\theta}\bigr)
\]
for $g\in \GL_{2n}^{\strs}(F)$.
Here we write simply $\Theta_{\phi_{0},\theta}^{\GL_{2n}}$ for the $\theta$-twisted character of $\pi_{\phi_{0}}^{\GL_{2n}}$.

By using this, we show that the $\theta$-twisted character $\Theta_{\phi_{0},\theta}^{\GL_{2n}}$ at $\varphi_{u}^{\GL_{2n}}(1+\varphi_{u}^{\GL_{2n}})$ vanishes for every $u\in k^{\times}$.
To show this, it suffices to check that the index set of the sum in the above formula is empty.
Here note that, in Section \ref{subsec:tran-GL}, we considered a realization of this element $\varphi_{u}^{\GL_{2n}}(1+\varphi_{u}^{\GL_{2n}})$ via Waldspurger's parametrization of strongly $\theta$-regular $\theta$-semisimple elements.
In particular, $\varphi_{u}^{\GL_{2n}}(1+\varphi_{u}^{\GL_{2n}})$ is strongly $\theta$-regular $\theta$-semisimple.
 
If $x\in\GL_{2n}(F)$ satisfies
\[
x\varphi_{u}^{\GL_{2n}}\bigl(1+\varphi_{u}^{\GL_{2n}}\bigr)\theta(x)^{-1}\in Z_{\GL_{2n}}\GL_{2n}(\mcO),
\]
then we have
\begin{align*}
&\phantom{{}={}}x\varphi_{u}^{\GL_{2n}}\bigl(1+\varphi_{u}^{\GL_{2n}}\bigr)\theta(x)^{-1}\cdot\theta\Bigl(x\varphi_{u}^{\GL_{2n}}\bigl(1+\varphi_{u}^{\GL_{2n}}\bigr)\theta(x)^{-1}\Bigr)\\
&=-x\bigl(1+\varphi_{u}^{\GL_{2n}}\bigr)\theta\bigl(1+\varphi_{u}^{\GL_{2n}}\bigr)x^{-1}\\
&\in Z_{\GL_{2n}}\GL_{2n}(\mcO)
\end{align*}
(recall that $\theta(\varphi_{u}^{\GL_{2n}})=-(\varphi_{u}^{\GL_{2n}})^{-1}$).
Thus $x$ belongs to $Z_{\GL_{2n}}\GL_{2n}(\mcO)\lan\varphi_{u}^{\GL_{2n}}\ran$ by the affine genericity of 
\[
\bigl(1+\varphi_{u}^{\GL_{2n}}\bigr)\theta\bigl(1+\varphi_{u}^{\GL_{2n}}\bigr)
\]
and Lemma \ref{lem:special-Iwahori}.
If we put 
\[
x=zg\bigl(\varphi_{u}^{\GL_{2n}}\bigr)^{k}
\]
for $z\in Z_{\GL_{2n}}$, $g\in \GL_{2n}(\mcO)$, and $k\in\Z$, then we have
\[
x\varphi_{u}^{\GL_{2n}}\bigl(1+\varphi_{u}^{\GL_{2n}}\bigr)\theta(x)^{-1}
=(-1)^{k}z^{2}g\bigl(\varphi_{u}^{\GL_{2n}}\bigr)^{2k+1}(1+\varphi_{u}^{\GL_{2n}})\theta(g)^{-1}.
\]
Hence we have
\[
\val\circ\det\Bigl((-1)^{k}z^{2}g\bigl(\varphi_{u}^{\GL_{2n}}\bigr)^{2k+1}(1+\varphi_{u}^{\GL_{2n}})\theta(g)^{-1}\Bigr)
=
2k+1+2\val\circ\det(z).
\]
On the other hand, we have
\[
\val\circ\det\bigl(Z_{\GL_{2n}}\GL_{2n}(\mcO)\bigr)=2n\Z.
\]
Thus $x\varphi_{u}^{\GL_{2n}}\bigl(1+\varphi_{u}^{\GL_{2n}}\bigr)\theta(x)^{-1}
$ does not belong to $Z_{\GL_{2n}}\GL_{2n}(\mcO)$, and this is a contradiction.
Hence the index set of the sum in the character formula is empty, and we have
\[
\Theta_{\phi_{0},\theta}^{\GL_{2n}}\bigl(\varphi_{u}^{\GL_{2n}}(1+\varphi_{u}^{\GL_{2n}})\bigr)=0
\]
for any $u\in k^{\times}$.

On the other hand, by the endoscopic character relation (Theorem \ref{thm:TECR}) for $\GL_{2n}$ and $\mathbf{H}$ and Propositions \ref{prop:4GL} and \ref{prop:tranGL}, we have
\[
\Theta_{\phi_{0},\theta}^{\GL_{2n}}\bigl(\varphi_{u}^{\GL_{2n}}(1+\varphi_{u}^{\GL_{2n}})\bigr)
=\omega_{0}(-1)q^{\frac{1}{2}}G(\omega_{0},\psi)^{-1} \bigl(\Theta^{\mathbf{H}}_{\xi',a'}(-g_{u}^{\mathbf{H}})+\Theta^{\mathbf{H}}_{\xi',a'}(-w_{\mu}g_{u}^{\mathbf{H}}w_{\mu}^{-1})\bigr).
\]
As we see in the proof of Theorem \ref{thm:ram-packet}, the right-hand side of this equality can be written in terms of the Kloosterman sum, and is not identically zero with respect to $u\in k^{\times}$ by Proposition \ref{prop:Kl} (2).
This is a contradiction.
\end{proof}

\begin{thm}\label{thm:endolift}
The representation $\pi_{\phi_{0}}^{\GL_{2n}}$ of $\GL_{2n}(F)$ is equivalent to $\pi^{\GL_{2n}}_{\omega_{0},4a,\zeta}$, where $\zeta=\xi\cdot q^{-\frac{1}{2}}\omega_{0}(-1)G(\omega_{0},\psi)$.
\end{thm}

\begin{proof}
By Proposition \ref{prop:not0}, $\pi_{\phi_{0}}^{\GL_{2n}}$ is a simple supercuspidal representation of $\GL_{2n}(F)$.
We put $(\omega,b,\zeta)$ to be the element of $\SSC(\GL_{2n})$ corresponding to $\pi_{\phi_{0}}^{\GL_{2n}}$.
Here note that the central character of $\pi_{\phi_{0}}^{\GL_{2n}}$ is given by $\mu$, which corresponds to $\det\circ\phi_{0}$ under the local class field theory.
Since $\mu$ is a nontrivial ramified quadratic character as described in Proposition \ref{prop:mu}, the first parameter $\omega$ of $(\omega,b,\zeta)$ is given by the nontrivial quadratic character $\omega_{0}$ of $k^{\times}$ (see Section \ref{subsec:ssc-GL}).

To investigate the relation between $(\xi',a')$ and $(\omega_{0},b,\zeta)$, we consider the endoscopic character relations for the following two cases:
\begin{itemize}
\item Theorem \ref{thm:TECR} for $\GL_{2n}$ and $\mathbf{H}$, and
\item Theorem \ref{thm:SECR} for $\G$ and $\mathbf{H}$.
\end{itemize}
Let $u\in a^{-1}k^{\times2}$.
Then the data $\xi_{u}$ defines an element $g_{u}^{\mathbf{H}}$ of $H$ (see the proof of Proposition \ref{prop:mu}).
By Propositions \ref{prop:4GL} and \ref{prop:tranGL}, the first one at the element $g_{u}^{\GL_{2n}}$ is given by
\[
\Theta_{\omega_{0},b,\zeta,\theta}^{\GL_{2n}}(g_{u}^{\GL_{2n}})
=
\omega_{0}(-1)\cdot q\cdot G(\omega_{0},\psi)^{-1} \bigl(\Theta_{\tilde{\pi}^{\mathbf{H}}}(g_{u}^{\mathbf{H}})+\Theta_{\tilde{\pi}^{\mathbf{H}}}(w_{\mu}g_{u}^{\mathbf{H}}w_{\mu}^{-1})\bigr).
\]
By the proof of Theorem \ref{thm:ram-packet}, the second one at the element $g_{u}^{\G}$ is given by
\[
\bigl(\Theta^{\G}_{\xi,0,a}-\Theta^{\G}_{\xi,1,a\epsilon^{-1}}\bigr)
(g_{u}^{\G})
=
\omega_{0}(-2)\cdot q\cdot G(\omega_{0},\psi)^{-1} \bigl(\Theta_{\tilde{\pi}^{\mathbf{H}}}(g_{u}^{\mathbf{H}})+\Theta_{\tilde{\pi}^{\mathbf{H}}}(w_{\mu}g_{u}^{\mathbf{H}}w_{\mu}^{-1})\bigr)
\]
(recall that the constant $C$ is equal to $1$).
By combining these two equalities, for any $u \in a^{-1}k^{\times2}$, we have
\[
\omega_{0}(2)\cdot\Theta_{\omega_{0},b,\zeta,\theta}^{\GL_{2n}}\bigl(g_{u}^{\GL_{2n}}\bigr)
=
\bigl(\Theta^{\G}_{\xi,0,a}-\Theta^{\G}_{\xi,1,a\epsilon^{-1}}\bigr)\bigl(g_{u}^{\G}\bigr).
\]
Since the simple affine components of $g_{u}^{\GL_{2n}}$ and $g_{u}^{\G}$ are given by $(1,\ldots,1,u)$ and $(2,\ldots,2,2u)$, respectively (Propositions \ref{prop:affgen-GL} and \ref{prop:affgen-Sp}), we have
\begin{align*}
\RHS
&=\omega_{0}(2au)\cdot\Kl_{2^{2n}au}^{n+1}(\psi; w_{\G}, \chi_{\G}), \text{ and}\\
\LHS&= \omega_{0}(2bu)\cdot\Kl_{2^{2(n-1)}bu}^{n+1}(\psi; w_{\GL_{2n},\theta},\chi_{\GL_{2n},\theta})
\end{align*}
by Corollary \ref{cor:pmSp} and Proposition \ref{prop:charTGL}.
As we have $w_{\G}=w_{\GL_{2n},\theta}$ and $\chi_{\G}=\chi_{\GL_{2n},\theta}$, we get
\[
\omega_{0}(au)\cdot\Kl_{2^{2n}au}^{n+1}(\psi; w_{\G}, \chi_{\G})
= \omega_{0}(bu)\cdot\Kl_{2^{2(n-1)}bu}^{n+1}(\psi; w_{\G}, \chi_{\G})
\]
for any $u\in a^{-1}k^{\times2}$.
By Proposition \ref{prop:Kl} (2) and Lemma \ref{lem:Klvan}, the left-hand side is not zero for some $u\in a^{-1}k^{\times2}$.
Then, again by Lemma \ref{lem:Klvan}, we have $b\equiv a$ mod $k^{\times2}$.
Thus we get 
\[
\Kl_{2^{2n}au}^{n+1}(\psi; w_{\G}, \chi_{\G})
= \Kl_{2^{2(n-1)}bu}^{n+1}(\psi; w_{\G}, \chi_{\G})
\]
for any $u\in a^{-1}k^{\times2}$.
Here note that this equality holds also for any element $u$ of $k^{\times}\smallsetminus a^{-1}k^{\times2}$ by Lemma \ref{lem:Klvan}.
Thus we get $b=4a$ by Proposition \ref{prop:FT}.

We next determine $\zeta$.
We combine two endoscopic character relations again.
Let $u\in a^{-1}k^{\times2}$.
By Propositions \ref{prop:4GL} and \ref{prop:tranGL}, the endoscopic character relation for $\GL_{2n}$ and $\mathbf{H}$ (Theorem \ref{thm:TECR}) at $\tilde{g}_{u}^{\GL_{2n}}=\varphi_{u}^{\GL_{2n}}g_{u}^{\GL_{2n}}$ is given by
\[
\Theta_{\omega_{0},4a,\zeta,\theta}^{\GL_{2n}}(\tilde{g}_{u}^{\GL_{2n}})
=
\omega_{0}(-1)\cdot q^{\frac{1}{2}}\cdot G(\omega_{0},\psi)^{-1} \bigl(\Theta_{\tilde{\pi}^{\mathbf{H}}}(-g_{u}^{\mathbf{H}})+\Theta_{\tilde{\pi}^{\mathbf{H}}}(-w_{\mu}g_{u}^{\mathbf{H}}w_{\mu}^{-1})\bigr).
\]
On the other hand, by Propositions \ref{prop:4Sp} and \ref{prop:tranSp}, the endoscopic character relation for $\G$ and $\mathbf{H}$ (Theorem \ref{thm:SECR}) at $-g_{u}^{\G}$ is given by 
\[
\bigl(\Theta^{\G}_{\xi,0,a}-\Theta^{\G}_{\xi,1,a\epsilon^{-1}}\bigr)
(-g_{u}^{\G})
=
\omega_{0}(2)\cdot q\cdot G(\omega_{0},\psi)^{-1} \bigl(\Theta_{\tilde{\pi}^{\mathbf{H}}}(-g_{u}^{\mathbf{H}})+\Theta_{\tilde{\pi}^{\mathbf{H}}}(-w_{\mu}g_{u}^{\mathbf{H}}w_{\mu}^{-1})\bigr).
\]
By combining these two equalities, for any $u\in a^{-1}k^{\times2}$, we have
\[
q^{\frac{1}{2}}\cdot\omega_{0}(-2)\cdot\Theta_{\omega_{0},4a,\zeta,\theta}^{\GL_{2n}}\bigl(\varphi_{u}^{\GL_{2n}}g_{u}^{\GL_{2n}}\bigr)
=
\bigl(\Theta^{\G}_{\xi,0,a}-\Theta^{\G}_{\xi,1,a\epsilon^{-1}}\bigr)\bigl(-g_{u}^{\G}\bigr).
\]
On the other hand, by Corollary \ref{cor:pmSp} and Proposition \ref{prop:charTGL2}, we have
\begin{align*}
\RHS
&=\xi\cdot\omega_{0}(2)\cdot\Kl_{2^{2n}au}^{n+1}(\psi; w_{\G}, \chi_{\G}), \text{ and}\\
\LHS
&=
q^{\frac{1}{2}}\cdot\omega_{0}(-2)\cdot\zeta\cdot \bigl( \Kl_{2^{n}v}^{n}(\psi) + \Kl_{-2^{n}v}^{n}(\psi) \bigr), 
\end{align*}
where we put $u=a^{-1}v^{2}$.
Here note that, by Lemma \ref{lem:Klvan}, $\Kl_{2^{2n}au}^{n+1}(\psi; w_{\G}, \chi_{\G})$ is zero for $u\notin a^{-1}k^{\times2}$.
Thus we have
\[
\sum_{u\in a^{-1}k^{\times2}}\RHS
=
\sum_{u\in k^{\times}}\RHS.
\]
Hence, by Proposition \ref{prop:Kl} (1), we have
\[
\sum_{u\in a^{-1}k^{\times2}}\RHS
=\xi\cdot\omega_{0}(2)G(\mathbbm{1},\psi)^{n}G(\omega_{0},\psi)
=(-1)^{n}\xi\cdot\omega_{0}(2)G(\omega_{0},\psi).
\]
On the other hand, again by Proposition \ref{prop:Kl} (1), we have
\begin{align*}
\sum_{u\in a^{-1}k^{\times2}}\LHS
&=q^{\frac{1}{2}}\cdot\omega_{0}(-2)\cdot\zeta\cdot\sum_{v^{2}\in k^{\times2}} \bigl( \Kl_{2^{n}v}^{n}(\psi) + \Kl_{-2^{n}v}^{n}(\psi) \bigr)\\
&=q^{\frac{1}{2}}\cdot\omega_{0}(-2)\cdot\zeta\cdot\sum_{v\in k^{\times}} \Kl_{2^{n}v}^{n}(\psi)\\
&=(-1)^{n}q^{\frac{1}{2}}\cdot\omega_{0}(-2)\cdot\zeta.
\end{align*}
Thus we get 
\[
(-1)^{n}\xi\cdot\omega_{0}(2)G(\omega_{0},\psi)=(-1)^{n}q^{\frac{1}{2}}\cdot\omega_{0}(-2)\cdot\zeta.
\]
Hence $\zeta$ is given by
\[
\xi\cdot q^{-\frac{1}{2}}\cdot\omega_{0}(-1)G(\omega_{0},\psi).
\]
\end{proof}

In summary, we get the following:
\begin{thm}\label{thm:packetSOram}
Let $\mu$ be a ramified quadratic character of $F^{\times}$.
Let $(\xi',a')\in\SSC(\SO_{2n}^{\mu})$ and
\[
\tilde{\pi}_{\xi',a'}^{\SO_{2n}^{\mu}}:=\Bigl\{\pi_{\xi',a'}^{\SO_{2n}^{\mu}}, \pi_{\xi',-a'}^{\SO_{2n}^{\mu}}\Bigr\}
\]
the $\Sigma(\SO_{2n}^{\mu})$-orbit of $\pi_{\xi',a'}^{\SO_{2n}^{\mu}}$.
Then $\{\tilde{\pi}_{\xi',a'}^{\SO_{2n}^{\mu}}\}$ is an $L$-packet of $\SO_{2n}^{\mu}$.
\end{thm}

\begin{thm}\label{thm:liftSOramtoGL}
Let $\mu$ be a ramified quadratic character of $F^{\times}$ and $(\xi',a')\in\SSC(\SO_{2n}^{\mu})$.
Then the endoscopic lift of a simple supercuspidal $L$-packet $\{\tilde{\pi}_{\xi',a'}^{\SO_{2n}^{\mu}}\}$ of $\SO_{2n}^{\mu}$ to $\GL_{2n}$ is given by $\pi_{\omega_{0},b,\zeta}^{\GL_{2n}}$, where
\[
b:=(-1)^{n-1}a'^{2}\epsilon^{-\nu_{\mu}}, \text{ and}\quad
\zeta:=q^{-\frac{1}{2}}G(\omega_{0},\psi)\xi'.
\]
\end{thm}

\begin{thm}\label{thm:liftSptoGL}
Let $(\xi,0,a)\in\SSC(\Sp_{2n})$.
Then the $L$-packet $\{\pi_{\xi,0,a}^{\Sp_{2n}}, \pi_{\xi,1,a\epsilon^{-1}}^{\Sp_{2n}}\}$ of $\Sp_{2n}$ is given by the endoscopic lift of the $L$-packet $\{\tilde{\pi}_{\xi',a'}^{\SO_{2n}^{\mu}}\}$ of $\SO_{2n}^{\mu}$.
Here $\mu$ is the quadratic character of $F^{\times}$ corresponding to the ramified quadratic extension $F(\sqrt{(-1)^{n-1}a\varpi})$ and $(\xi',a')$ is as described in Theorem \ref{thm:ram-packet}.
Moreover, the endoscopic lift of the $L$-packet $\{\pi_{\xi,0,a}^{\Sp_{2n}}, \pi_{\xi,1,a\epsilon^{-1}}^{\Sp_{2n}}\}$ of $\Sp_{2n}$ to $\GL_{2n+1}$ is given by $\pi_{\omega_{0},4a,\zeta}^{\GL_{2n}}\boxplus\mu$.
Here, 
\[
\zeta=\xi\cdot q^{-\frac{1}{2}}\omega_{0}(-1)G(\omega_{0},\psi)
\]
(note that $\mu$ coincides with the central character of $\pi_{\omega_{0},4a,\zeta}^{\GL_{2n}}$).
\end{thm}

\[
\xymatrix{
\G=\Sp_{2n}& \{\pi^{\G}_{\xi,0,a},\pi^{\G}_{\xi,1,a\epsilon^{-1}}\}&\\
\mathbf{H}=\SO_{2n}^{\mu}\ar@{~>}[u]_-{\text{standard endoscopy}}& \{\pi^{\mathbf{H}}_{\xi',a'}\}\ar@{~>}[u]_-{\text{endoscopic lifting}}&\xi'=\xi\omega_{0}(-1), a'^{2}=(-1)^{n-1}\epsilon^{\nu_{\mu}}4a\\
\GL_{2n}\ar@{<~}[u]_-{\text{twisted endoscopy}}&\pi^{\GL_{2n}}_{\omega_{0},4a,\zeta} \ar@{<~}[u]_-{\text{endoscopic lifting}}&\zeta=\xi\cdot q^{-\frac{1}{2}}\omega_{0}(-1)G(\omega_{0},\psi)
}
\]

\section{Simple supercuspidal $L$-packet of $\SO_{2n+2}$ and $\SO_{2n+2}^{\ur}$}\label{sec:SO-ur}

Our aim in this section is to determine the structures of simple supercuspidal $L$-packets of $\SO_{2n+2}$ and $\SO_{2n+2}^{\ur}$ and their $L$-parameters.
To do this, we consider the endoscopy of the following two types (see Section \ref{sec:Arthur} for a precise description of $L$-embeddings):
\begin{description}
\item[split case]
We put
\begin{align*}
\G&:=\SO_{2n+2}, \text{ and}\\
\mathbf{H}&:=\mathbf{H}_{-}\times\mathbf{H}_{+}, \text{ where }
\mathbf{H}_{-}:=\SO_{2n}^{\mu} \text{ and }
\mathbf{H}_{+}:=\SO_{2}^{\mu}.
\end{align*}
Here $\mu$ is a ramified quadratic character of $F^{\times}$.
Then the group $\mathbf{H}$ is an endoscopic group of $\G$.

\item[unramified case]
\begin{align*}
\G&:=\SO_{2n+2}^{\ur}, \text{ and}\\
\mathbf{H}&:=\mathbf{H}_{-}\times\mathbf{H}_{+}, \text{ where }
\mathbf{H}_{-}:=\SO_{2n}^{\mu} \text{ and }
\mathbf{H}_{+}:=\SO_{2}^{\bar{\mu}}.
\end{align*}
Here $\mu$ and $\bar{\mu}$ are distinct ramified quadratic characters of $F^{\times}$.
Then the group $\mathbf{H}$ is an endoscopic group of $\G$.
\end{description}

We denote the quadratic character corresponding to the group $\G$ by $\mu_{\G}$.
Namely, we put
\[
\mu_{\G}:=
\begin{cases}
\mathbbm{1}&\text{if }\G=\SO_{2n+2},\\
\mu_{\ur}&\text{if }\G=\SO_{2n+2}^{\ur}.
\end{cases}
\]

\subsection{Construction of simple supercuspidal $L$-packets with $\zeta=1$}\label{subsec:zeta=1}
Let $(\xi_{-},a_{-})\in\SSC(\mathbf{H}_{-})$.
We consider the $L$-parameter $\phi_{-}$ (resp.\ $\phi_{+}$) of the simple supercuspidal representation $\pi^{\mathbf{H}_{-}}_{\xi_{-},a_{-}}$ of $H_{-}$ (resp.\ the trivial representation $\mathbbm{1}$ of $H_{+}$):
\[
\Pi(\mathbf{H}_{-})\supset\Pi_{\phi_{-}}^{\mathbf{H}_{-}}\ni\pi^{\mathbf{H}_{-}}_{\xi_{-},a_{-}} \quad\longleftrightarrow\quad \phi_{-}\colon W_{F}\times\SL_{2}(\C)\ra{}^{L}\mathbf{H}_{-}
\]
\[
\Pi(\mathbf{H}_{+})\supset\Pi_{\phi_{+}}^{\mathbf{H}_{+}}=\{\mathbbm{1}\} \quad\longleftrightarrow\quad \phi_{+}\colon W_{F}\times\SL_{2}(\C)\ra{}^{L}\mathbf{H}_{+}
\]
Then we get an $L$-parameter $(\phi_{-},\phi_{+})$ of $\mathbf{H}$.
By composing it with the $L$-embedding from ${}^{L}\mathbf{H}$ to ${}^{L}\G$ (see Section \ref{sec:Arthur}), we obtain an $L$-parameter $\phi$ of $\G$:
\[
\xymatrix{
& \widehat{\G}\times W_{F}={}^{L}\G\\
W_F\times\SL_2(\C) \ar[r]_-{(\phi_{-},\phi_{+})} \ar[ru]^-{\phi}& (\widehat{\mathbf{H}}_{-}\times \widehat{\mathbf{H}}_{+})\rtimes W_{F}={}^{L}\mathbf{H}\ar@{^{(}->}[u]\\
}
\]
In this subsection, we show the simple supercuspidality of the $L$-packet $\w{\Pi}_{\phi}^{\G}$ and determine the simple supercuspidal representations contained in $\w{\Pi}_{\phi}^{\G}$.

Recall that we may regard an $L$-parameter of an even special orthogonal group as an representation of $W_{F}\times\SL_{2}(\C)$ by composing it with an $L$-embedding to the $L$-group of a general linear group (see Section \ref{sec:Arthur}).
Then, as $(2n+2)$-dimensional representations of $W_{F}\times\SL_{2}(\C)$, we have $\phi\cong\phi_{-}\oplus\phi_{+}$, and $\phi_{-}$ and $\phi_{+}$ have the following properties:
\begin{lem}\label{lem:gal}
As representations of $W_{F}\times\SL_{2}(\C)$, 
\begin{enumerate}
\item
$\phi_{-}$ is irreducible, and
\item
$\phi_{+}$ is equivalent to $\mathbbm{1}\oplus\mu\cdot\mu_{\G}=\begin{cases}\mathbbm{1}\oplus\mu & \text{if }\G=\SO_{2n+2},\\ \mathbbm{1}\oplus\bar{\mu} & \text{if }\G=\SO_{2n+2}^{\ur}.\\\end{cases}$
\end{enumerate}
\end{lem}

\begin{proof}
The first assertion is a consequence of the previous section (Theorem \ref{thm:liftSOramtoGL}).
Namely, since the endoscopic lift of a simple supercuspidal representation of $H_{-}$ to $\GL_{2n}(F)$ is again simple supercuspidal, in particular supercuspidal, the corresponding $L$-parameter is irreducible as an representation of $W_{F}\times\SL_{2}(\C)$.

We show the second claim.
Since $\mathbf{H}_{+}$ is a torus, the $L$-parameter $\phi_{+}$ corresponding to the trivial representation $\mathbbm{1}$ is the trivial homomorphism
\begin{align*}
\phi_{+}\colon W_{F}\times\SL_{2}(\C)&\rightarrow\SO_{2}(\C)\rtimes W_{F}\\
(\sigma, x)&\mapsto 1\rtimes\sigma.
\end{align*}
(this is a general property of the local Langlands correspondence for tori).
Thus the composite of $\phi_{+}$ with the natural $L$-embedding from ${}^{L}\mathbf{H}_{+}$ to ${}^{L}\!\GL_{2}$ and the projection ${}^{L}\!\GL_{2}\twoheadrightarrow \GL_{2}(\C)$ is given by
\begin{align*}
\phi_{+}\colon W_{F}\times\SL_{2}(\C)&\rightarrow\GL_{2}(\C)\\
(\sigma, x)&\mapsto 
\begin{cases}
\begin{pmatrix}1&\\&1\end{pmatrix} & \text{if } \sigma\in W_{E_{\mu\cdot\mu_{\G}}},\\
\begin{pmatrix}&1\\1&\end{pmatrix} & \text{ otherwise}.
\end{cases}
\end{align*}
This representation is equivalent to $\mathbbm{1}\oplus\mu\cdot\mu_{\G}$.
\end{proof}

\begin{lem}\label{lem:Lpartwist}
For any quadratic character $\mu'$ of $F^{\times}$, the $2n$-dimensional representation $\phi_{-}\otimes\mu'$ of $W_{F}\times\SL_{2}(\C)$ is again the $L$-parameter of a simple supercuspidal representation of $\mathbf{H}_{-}$.
\end{lem}

\begin{proof}
By Theorem \ref{thm:liftSOramtoGL}, the endoscopic lift of $\pi^{\mathbf{H}_{-}}_{\xi_{-},a_{-}}$ to $\GL_{2n}$ is given by $\pi^{\GL_{2n}}_{\omega_{0},b,\zeta}$ (see Theorem \ref{thm:liftSOramtoGL} for the description of $b$ and $\zeta$).
Conversely, all simple supercuspidal representations of $\GL_{2n}(F)$ corresponding to the data of the form $(\omega_{0},b',\zeta')\in\SSC^{\theta}_{\omega_{0}}(\GL_{2n})$ satisfying $\omega_{0}(b')=\omega_{0}(b)$ are obtained by the endoscopic lift from $\mathbf{H}_{-}$.
Therefore it suffices to show that if we regard $\phi_{-}\otimes\mu$ as an $L$-parameter of $\GL_{2n}$, then it corresponds to such a simple supercuspidal representation of $\GL_{2n}(F)$.

By the compatibility of the local Langlands correspondence for $\GL_{2n}$ and the character twist, the representation corresponding to $\phi_{-}\otimes\mu'$ is given by 
\[
\pi^{\GL_{2n}}_{\omega_{0},b,\zeta}\otimes(\mu'\circ\det)
\cong
\cInd^{\GL_{2n}(F)}_{Z_{\GL_{2n}}I_{\GL_{2n}}^{+}\lan\varphi_{b^{-1}}^{\GL_{2n}}\ran} \bigl(\tilde{\chi}^{\GL_{2n}}_{\omega_{0},b,\zeta}\otimes\mu'\circ\det\bigr).
\]
On the other hand, since $\mu'$ is a quadratic character and we assume that $p\neq2$, $\mu'\circ\det$ is trivial on $Z_{\GL_{2n}}$ and $I_{\GL_{2n}}^{+}$.
Moreover, we have $\mu'\circ\det(\varphi_{b^{-1}}^{\GL_{2n}})\in\{\pm1\}$.
Thus $\tilde{\chi}^{\GL_{2n}}_{\omega_{0},b,\zeta}\otimes\mu'\circ\det$ is equal to either $\tilde{\chi}^{\GL_{2n}}_{\omega_{0},b,\zeta}$ itself or $\tilde{\chi}^{\GL_{2n}}_{\omega_{0},b,-\zeta}$.
This completes the proof.
\end{proof}

\begin{prop}\label{prop:GIP}
The $L$-packet $\widetilde{\Pi}_{\phi}^{\G}$ of $\phi$ consists of two $\Sigma(\G)$-orbits of simple supercuspidal representations of $G$, each of which consists of a single element.
\end{prop}

\begin{proof}
First, by Lemma \ref{lem:gal}, the order of the group $\mathcal{S}^{\G}_{\phi}$ is given by two.
Therefore also the order of the $L$-packet $\widetilde{\Pi}_{\phi}^{\G}$ for $\phi$ is given by two by Theorem \ref{thm:Arthur} (1).
On the other hand, by the observation in Sections \ref{subsec:ssc-SO} and \ref{subsec:ssc-SO-ur}, each $\Sigma(\G)$-orbit of simple supercuspidal representations of $G$ is a singleton.
Therefore our task is to show that $\widetilde{\Pi}_{\phi}^{\G}$ consists only of simple supercuspidal representations.

To show this, we carry out an argument using the theta correspondence, based on the same idea as in the proof of Proposition \ref{prop:ssc-descent}.
We note that, with the notations of \cite{MR3166215}, the pair $(\Sp_{2n}, \rmO_{2n+2}^{(\ur)})$ is realized as a reductive dual pair of type $(\mathrm{C}'')$ with invariant $l=-1$:
\begin{itemize}
\item
We let $V$ be a $2n+2$-dimensional vector space over $F$ equipped with a non-degenerate symmetric form $(-,-)\colon V\times V\rightarrow F$ such that the associated classical group $H(V)$ is isomorphic to $\rmO_{2n+2}^{(\ur)}$.
\item
We let $W$ be a $2n$-dimensional vector space over $F$ equipped with a non-degenerate symplectic form $\langle-,-\rangle\colon W\times W\rightarrow F$.
Then the associated classical group $G(W)$ is isomorphic to $\Sp_{2n}$.
\end{itemize}

Let $\phi'$ be the following $L$-parameter of $\Sp_{2n}$: 
\[
\phi':=\phi_{-}\otimes\mu_{\G}\oplus\mu.
\]
By Lemma \ref{lem:Lpartwist}, the first constituent $\phi_{-}\otimes\mu_{\G}$ is the $L$-parameter of $\mathbf{H}_{-}$ corresponding to a simple supercuspidal representation of $H_{-}$.
Moreover, we have $\det(\phi_{-}\otimes\mu_{\G})=\det(\phi_{-})=\mu$.
Therefore, by Theorem \ref{thm:liftSptoGL}, the $L$-packet of $\Sp_{2n}$ corresponding to the above $L$-parameter $\phi'$ consists of two simple supercuspidal representations.



We consider the theta lift of the $L$-packet $\Pi_{\phi'}^{\Sp_{2n}}$ to $\rmO_{2n+2}^{(\ur)}(F)$.
For $\pi\in\Pi_{\phi'}^{\Sp_{2n}}$, we let $\theta_{V,W,\boldsymbol{\chi},\psi}(\pi)$ denote its small theta lift to $\rmO_{2n+2}^{(\ur)}(F)$.
Since $\phi'$ does not contain the character $\mu_{\G}$, by \cite[Theorem C.5 (i)]{MR3166215}, $\theta_{V,W,\boldsymbol{\chi},\psi}(\pi)$ is not zero for every $\pi\in\Pi_{\phi'}^{\Sp_{2n}}$ and the map 
\[
\pi \mapsto \theta(\pi)|_{G}
\]
gives a bijection
\[
\Pi_{\phi'}^{\Sp_{2n}}\rightarrow\w{\Pi}_{\theta(\phi')}^{\G},
\]
where
\[
\theta(\phi')
=
\phi'\otimes\mu_{\G}^{-1}\oplus\mathbbm{1}
=\phi_{-}\oplus\mu\cdot\mu_{\G}\oplus\mathbbm{1}
=\phi
\]
(we remark that $\Sp_{2n}$ has no nontrivial pure inner form).

Then the same argument as in the proof of Proposition \ref{prop:ssc-descent} (i.e., use Pan's depth preservation theorem \cite{MR1909608} and note that simple supercuspidals of $\Sp_{2n}$ and $\SO_{2n+2}^{\ur}$ are characterized as the representations of depth $\frac{1}{2n}$) implies that any member of $\w{\Pi}_{\phi}^{\G}$ is simple supercuspidal.
\end{proof}

We next determine the members of $\w{\Pi}_{\phi}^{\G}$.
By Proposition \ref{prop:GIP}, the $L$-packet $\w{\Pi}_{\phi}^{\G}$ consists of two $\Sigma(\G)$-orbits of simple supercuspidal representations of $G$.
Note that any such $\Sigma(\G)$-orbit consists of only one simple supercuspidal representation of $G$ (see Sections \ref{subsec:ssc-SO}, \ref{subsec:ssc-SO-ur}).
We denote one of the members of $\w{\Pi}_{\phi}^{\G}$ by
\[
\pi^{\G}_{\xi,\kappa,a\epsilon^{-\kappa},\zeta}
\]
with $(\xi,\kappa,a\epsilon^{-\kappa},\zeta)\in\SSC(\G)$, i.e., $\xi\in\{\pm1\}$, $\kappa\in\{0,1\}$, $a\epsilon^{-\kappa}\in k^{\times}$, $\zeta\in\{\pm1\}$.

\begin{lem}\label{lem:adjointSO}
The other member of $\w{\Pi}_{\phi}^{\G}$ is given by
\[
\pi^{\G}_{\xi,1-\kappa,a\epsilon^{\kappa-1},\zeta}.
\]
\end{lem}

\begin{proof}
By Corollary \ref{cor:adj}, the $L$-packet $\w{\Pi}_{\phi}^{\G}$ is stable under under the action of $G_{\ad}$.
Thus, since we already know that the order of $\w{\Pi}_{\phi}^{\G}$ is given by $2$, it suffices to find an element $t$ of $G_{\ad}$ satisfying
\[
(\pi^{\G}_{\xi,0,a,\zeta})^{t}
\cong
\pi^{\G}_{\xi,1,a\epsilon^{-1},\zeta}.
\]

We first consider the case where $\G=\SO_{2n+2}$.
In this case, we put
\[
t:=\diag(\sqrt{\epsilon},\ldots,\sqrt{\epsilon},\sqrt{\epsilon}^{-1},\ldots,\sqrt{\epsilon}^{-1})\in G_{\ad}.
\]
Since $t^{-1}\sigma(t)$ belongs to $\bfZ_{\G}(\overline{F})$ for any $\sigma\in\Gamma_{F}$, this element indeed lies in $G_{\ad}=\G_{\ad}(F)$.
Moreover, $t$ normalizes $I_{\SO_{2n+2}}^{+}$ and we have
\[
(\chi^{\SO_{2n+2}}_{\xi,0,a})^{t}
=
\chi^{\SO_{2n+2}}_{\xi,1,a\epsilon^{-1}}.
\]
Furthermore, as 
\begin{align*}
t^{-1}\varphi^{\SO_{2n+2}}_{1,-a^{-1}}t
&=
t^{-1}
\begin{pmatrix}
&&&&&(-a^{-1}\varpi)^{-1}\\
&I_{n-1}&&&&\\
&&&1&&\\
&&1&&&\\
&&&&I_{n-1}&\\
-a^{-1}\varpi&&&&&
\end{pmatrix}
t\\
&=
\begin{pmatrix}
&&&&&(-\epsilon a^{-1}\varpi)^{-1}\\
&I_{n-1}&&&&\\
&&&\epsilon^{-1}&&\\
&&\epsilon&&&\\
&&&&I_{n-1}&\\
-\epsilon a^{-1}\varpi&&&&&
\end{pmatrix}
=
\varphi^{\SO_{2n+2}}_{\epsilon,-(a\epsilon^{-1})^{-1}},
\end{align*}
we have
\[
(\pm I_{\SO_{2n+2}}^{+} \langle\varphi^{\SO_{2n+2}}_{1,-a^{-1}}\rangle)^{t}
=
\pm I_{\SO_{2n+2}}^{+} \langle\varphi^{\SO_{2n+2}}_{\epsilon,-(a\epsilon^{-1})^{-1}}\rangle.
\]
Therefore we get
\begin{align*}
(\pi^{\G}_{\xi,0,a,\zeta})^{t}
&\cong
(\cInd_{\pm I_{\SO_{2n+2}}^{+} \langle\varphi^{\SO_{2n+2}}_{1,-a^{-1}}\rangle}^{G} \tilde{\chi}^{\SO_{2n+2}}_{\xi,0,a,\zeta})^{t}\\
&\cong
\cInd_{(\pm I_{\SO_{2n+2}}^{+} \langle\varphi^{\SO_{2n+2}}_{1,-a^{-1}}\rangle)^{t}}^{G} (\tilde{\chi}^{\SO_{2n+2}}_{\xi,0,a,\zeta})^{t}\\
&\cong
\cInd_{\pm I_{\SO_{2n+2}}^{+} \langle\varphi^{\SO_{2n+2}}_{\epsilon,-(a\epsilon^{-1})^{-1}}\rangle}^{G} \tilde{\chi}^{\SO_{2n+2}}_{\xi,1,a\epsilon^{-1},\zeta}
\cong
\pi^{\G}_{\xi,1,a\epsilon^{-1},\zeta}.
\end{align*}

We next consider the case where $\G=\SO_{2n+2}^{\ur}$.
In this case, 
we put
\begin{align*}
t:=
\diag(\sqrt{\epsilon},\ldots,\sqrt{\epsilon},\sqrt{\epsilon}\tilde{\epsilon}^{-1},\sqrt{\epsilon}^{-1},\ldots,\sqrt{\epsilon}^{-1})
\in G_{\ad}.
\end{align*}
Here recall that $\tilde{\epsilon}\in\tilde{k}$ is a fixed element satisfying $\Nr(\tilde{\epsilon})=\epsilon$.
If we let $x,y\in k$ be the elements satisfying $\tilde{\epsilon}^{-1}=x+\sqrt{\epsilon}y$, then $\sqrt{\epsilon}\tilde{\epsilon}^{-1}$ is expressed by
\[
\begin{pmatrix}
\sqrt{\epsilon}x&\sqrt{\epsilon}\epsilon y\\
\sqrt{\epsilon}y&\sqrt{\epsilon}x
\end{pmatrix}
\]
as an element of $\SO_{2}^{\ur}(\overline{F})$.
Since $t^{-1}\sigma(t)$ belongs to $\bfZ_{\G}(\overline{F})$ for any $\sigma\in\Gamma_{F}$, this element indeed lies in $G_{\ad}=\G_{\ad}(F)$.
Moreover, $t$ normalizes $I_{\SO_{2n+2}^{\ur}}^{+}$ and we have
\[
(\chi^{\SO_{2n+2}^{\ur}}_{\xi,0,a})^{t}
=
\chi^{\SO_{2n+2}^{\ur}}_{\xi,1,a\epsilon^{-1}}.
\]
Furthermore, as 
\begin{align*}
t^{-1}\varphi^{\SO_{2n+2}^{\ur}}_{1,-a^{-1}}t
&=
t^{-1}
\begin{pmatrix}
&&&&(-a^{-1}\varpi)^{-1}\\
&I_{n-1}&&&\\
&&I_{2}&&\\
&&&I_{n-1}&\\
-a^{-1}\varpi&&&&
\end{pmatrix}
w_{\ur}t\\
&=
\begin{pmatrix}
&&&&(-\epsilon a^{-1}\varpi)^{-1}\\
&I_{n-1}&&&\\
&&\tilde{\epsilon}/c(\tilde{\epsilon})&&\\
&&&I_{n-1}&\\
-\epsilon a^{-1}\varpi&&&&
\end{pmatrix}
w_{\ur}\\
&=\varphi^{\SO_{2n+2}^{\ur}}_{\tilde{\epsilon}/c(\tilde{\epsilon}),-(a\epsilon^{-1})^{-1}},
\end{align*}
we have
\[
(\pm I_{\SO_{2n+2}^{\ur}}^{+} \langle\varphi^{\SO_{2n+2}^{\ur}}_{1,-a^{-1}}\rangle)^{t}
=
\pm I_{\SO_{2n+2}^{\ur}}^{+} \langle\varphi^{\SO_{2n+2}^{\ur}}_{\tilde{\epsilon}/c(\tilde{\epsilon}),-(a\epsilon^{-1})^{-1}}\rangle.
\]
Therefore, in the same manner as above, we get
\[
(\pi^{\G}_{\xi,0,a,\zeta})^{t}
\cong
\pi^{\G}_{\xi,1,a\epsilon^{-1},\zeta}.
\]

\end{proof}

\begin{thm}\label{thm:packetSOspl}
The $L$-packet $\w{\Pi}_{\phi}^{\G}$ of $\phi$ is  given by
\[
\bigl\{ \pi^{\G}_{\xi, 0,a,1}, \pi^{\G}_{\xi,1,a\epsilon^{-1},1} \bigr\},
\]
where
\[
\xi=(-1)^{t_{\G}}\omega_{0}(-1)\cdot\xi_{-}
\quad\text{and}\quad
a=\frac{a_{-}^{2}}{2^{2+t_{\G}}\epsilon^{\nu_{\mu}}}.
\]
Here we put
\[
t_{\G}=
\begin{cases}
0&\text{if }\G=\SO_{2n+2},\\
1&\text{if }\G=\SO_{2n+2}^{\ur}.
\end{cases}
\]
\end{thm}

\begin{proof}
By Lemma \ref{lem:adjointSO}, the $L$-packet $\w{\Pi}_{\phi}^{\G}$ consists of two members $\pi^{\G}_{\xi, 0,a,\zeta}, \pi^{\G}_{\xi,1,a\epsilon^{-1},\zeta}$ for some $\xi\in\{\pm1\}$, $a\in k^{\times}$, and $\zeta\in\{\pm1\}$.
Our goal is to determine the parameters $\xi$, $a$, and $\zeta$.

By Theorem \ref{thm:SECR}, we have the following equality:
\[
C\cdot\bigl(\Theta^{\G}_{\xi,0,a,\zeta} - \Theta^{\G}_{\xi,1,a\epsilon^{-1},\zeta}\bigr)(g)
\]
\[
=
\sum_{\begin{subarray}{c}(h_{-},h_{+})\\\leftrightarrow g/\sim\end{subarray}} 
\Delta_{\mathbf{H},\G}\bigl((h_{-},h_{+}),g\bigr)
\frac{D_{\mathbf{H}_{-}}(h_{-})^{2}D_{\mathbf{H}_{+}}(h_{+})^{2}}{D_{\G}(g)^{2}}
\Theta_{\xi_{-},a_{-}}^{\mathbf{H}_{-}}(h_{-}),
\tag{$\ast$}
\]
where the sum is over the stable conjugacy classes of norms $(h_{-}, h_{+})\in H_{-}\times H_{+}$ of $g\in G$ (note that the character of the $H_{+}$-part is trivial since now it is the character of the trivial representation $\mathbbm{1}$).
Here note that $C$ is a sign which is defined by
\[
C=
\begin{cases}
1 & \text{if $\pi_{\xi,0,a,\zeta}^{\G}$ is $\mathfrak{w}_{\G}$-generic},\\
-1 & \text{if $\pi_{\xi,1,a\epsilon^{-1},\zeta}^{\G}$ is $\mathfrak{w}_{\G}$-generic}.
\end{cases}
\]
In particular, $C$ does not depend on $g\in G$.

We let $u\in k^{\times}$ be an element satisfying $E_{\mu}=F(\sqrt{(-1)^{n-1}\varpi u})$ and consider an element $\tilde{g}_{u}^{\G}\in G$ defined in Sections \ref{subsec:tran-SO} and \ref{subsec:tran-SO-ur}.
Here recall that we defined $\nu_{\mu}\in\{0,1\}$ so that $\mu$ corresponds to the quadratic extension $F(\sqrt{\varpi\epsilon^{\nu_{\mu}}})$ (see Definition \ref{def:nu} and the comment below Definition \ref{def:nu}).
Thus the above condition on $u$ is equivalent to the condition that 
\[
u\in(-1)^{n-1}\epsilon^{\nu_{\mu}}k^{\times2}.
\]
As investigated in Sections \ref{subsec:tran-SO} and \ref{subsec:tran-SO-ur}, for the element $\tilde{g}_{u}^{\G}\in G$, the sum of the endoscopic character relation (stable conjugacy classes of norms of $\tilde{g}_{u}^{\G}$ in $H$) runs over two elements $(g_{u}^{\mathbf{H}_{-}},h_{+}^{0})$ and $(w_{\mu}g_{u}^{\mathbf{H}_{-}}w_{\mu}^{-1},h_{+}^{0})$ with an element $h_{+}^{0}\in H_{+}$.
In the following, we write $h_{-}^{(1)}:=g_{u}^{\mathbf{H}_{-}}$ and $h_{-}^{(-1)}:=w_{\mu}g_{u}^{\mathbf{H}_{-}}w_{\mu}^{-1}$.

Let us compute the endoscopic character relation at $\tilde{g}_{u}^{\G}$.
Since every member of our $L$-packet is $\Sigma(\G)$-stable, we have
\[
\bigl(\Theta^{\G}_{\xi,0,a,\zeta} - \Theta^{\G}_{\xi,1,a\epsilon^{-1},\zeta}\bigr)(\tilde{g}_{u}^{\G})
=
\bigl(\Theta^{\G}_{\xi,0,a,\zeta} - \Theta^{\G}_{\xi,1,a\epsilon^{-1},\zeta}\bigr)(g_{u}^{\G}).
\]
As the simple affine components of $g_{u}^{\G}$ are given by
\[
\begin{cases}
\bigl(-1, \underbrace{2, \ldots, 2}_{n-2}, -1, 2, (-1)^{n-1}2u \bigr) & \text{if } \G=\SO_{2n+2},\\
\bigl(1, \underbrace{2, \ldots, 2}_{n-2}, 2, (-1)^{n-1}2u \bigr) & \text{if } \G=\SO_{2n+2}^{\ur}
\end{cases}
\]
(see Propositions \ref{prop:affgen-SO} and \ref{prop:affgen-SOur}), we have 
\[
\bigl(\Theta^{\G}_{\xi,0,a,\zeta} - \Theta^{\G}_{\xi,1,a\epsilon^{-1},\zeta}\bigr)(g_{u}^{\G})
\]
\[
= 
\begin{cases}
\omega_{0}(-2)\cdot \Kl_{(-1)^{n+1}2^{2n-2}au}^{n+2}(\psi; w_{\G},\chi_{\G}) & \text{if } \G=\SO_{2n+2},\\
\Kl_{(-1)^{n+1}2^{2n-1}au}^{n;1}(\psi; w_{\G},\chi_{\G}) & \text{if } \G=\SO_{2n+2}^{\ur},
\end{cases}
\]
by Corollaries \ref{cor:pmSOspl} and \ref{cor:pmSOur}.
Here 
\[
w_{\G}=
\begin{cases}
(1,\underbrace{2,\ldots,2}_{n-2},1,1,1) & \text{if } \G=\SO_{2n+2},\\
(1,\underbrace{2,\ldots,2}_{n-2},1;1) & \text{if } \G=\SO_{2n+2}^{\ur},
\end{cases}
\]
\[
\chi_{\G}=
\begin{cases}
(\underbrace{\mathbbm{1},\ldots,\mathbbm{1}}_{n-1},\omega_{0},\omega_{0},\mathbbm{1}) & \text{if } \G=\SO_{2n+2},\\
(\mathbbm{1},\ldots,\mathbbm{1};\omega_{0}) & \text{if } \G=\SO_{2n+2}^{\ur}.
\end{cases}
\]
On the other hand, we have
\[
\Delta_{\mathbf{H},\G}\bigl((h_{-}^{(\pm1)},h_{+}^{0}),g_{u}^{\G}\bigr)
\frac{D_{\mathbf{H}_{-}}(h_{-}^{(\pm1)})^{2}D_{\mathbf{H}_{+}}(h_{+}^{0})^{2}}{D_{\G}(\tilde{g}_{u}^{\G})^{2}}
=
\begin{cases}
\omega_{0}(2)\cdot q & \text{if } \G=\SO_{2n+2},\\
-\omega_{0}(-1)\cdot q & \text{if } \G=\SO_{2n+2}^{\ur},
\end{cases}
\]
by Propositions \ref{prop:4SOspl}, \ref{prop:tranSOspl}, \ref{prop:4SOur}, and \ref{prop:tranSOur}.
We recall that both of $h_{-}^{(1)}$ and $h_{-}^{(-1)}$ are affine generic and that their simple affine components are given by
\[
\bigl(\underbrace{2, \ldots, 2}_{n-1}, \pm2v \bigr)
\]
(see Proposition \ref{prop:affgen-SOram} and the comment below Proposition \ref{prop:affgen-SOram}),
where $v$ is an element of $k^{\times}$ satisfying 
\[
v^{2}\epsilon^{\nu_{\mu}}=(-1)^{n-1}u.
\]
Therefore, by Proposition \ref{prop:charSOram}, we have 
\[
\sum_{\begin{subarray}{c}(h_{-},h_{+})\\\leftrightarrow \tilde{g}_{u}^{\G}/\sim\end{subarray}}\Theta_{\xi_{-},a_{-}}^{\mathbf{H}_{-}}(h_{-})
=
\Kl_{\pm2^{n}a_{-}v}^{n}(\psi).
\]
Thus, for every $u\in(-1)^{n-1}\epsilon^{\nu_{\mu}}k^{\times2}$, we get 
\begin{multline*}
\begin{cases}
C\cdot\Kl_{(-1)^{n+1}2^{2n-2}au}^{n+2}(\psi; w_{\G},\chi_{\G})\\
C\cdot\Kl_{(-1)^{n+1}2^{2n-1}au}^{n;1}(\psi; w_{\G},\chi_{\G})
\end{cases}\\
=
\begin{cases}
\omega_{0}(-1)\cdot q\cdot\Kl_{\pm2^{n}a_{-}v}^{n}(\psi) & \text{if } \G=\SO_{2n+2},\\
-\omega_{0}(-1)\cdot q\cdot\Kl_{\pm2^{n}a_{-}v}^{n}(\psi) & \text{if } \G=\SO_{2n+2}^{\ur}.
\end{cases}
\tag{$\dagger$}
\end{multline*}

For a multiplicative character $\chi$ of $k^{\times}$, we take the sum of this equality twisted by $\chi(u)$ over $u\in (-1)^{n-1}\epsilon^{\nu_{\mu}}k^{\times2}$.
When $\G=\SO_{2n+2}$, the $\chi(u)$-twisted sum of the left-hand side of $(\dagger)$ is given by
\begin{align*}
&
C\sum_{u\in(-1)^{n-1}\epsilon^{\nu_{\mu}}k^{\times2}}\chi(u)\Kl_{(-1)^{n+1}2^{2n-2}au}^{n+2}(\psi; w_{\G},\chi_{\G})\\
&=
C\sum_{u\in k^{\times}}\frac{1+\omega_{0}\bigl((-1)^{n-1}\epsilon^{\nu_{\mu}}u\bigr)}{2}\chi(u)
\Kl_{(-1)^{n+1}2^{2n-2}au}^{n+2}(\psi; w_{\G},\chi_{\G})\\
&=
\frac{C}{2}\sum_{u\in k^{\times}}\chi(u)\Kl_{(-1)^{n+1}2^{2n-2}au}^{n+2}(\psi; w_{\G},\chi_{\G})\\
&\qquad+
\frac{C\omega_{0}\bigl((-1)^{n-1}\epsilon^{\nu_{\mu}}\bigr)}{2}\sum_{u\in k^{\times}}\omega_{0}\chi(u)\Kl_{(-1)^{n+1}2^{2n-2}au}^{n+2}(\psi; w_{\G},\chi_{\G})\\
&=
\frac{C}{2}\cdot\chi\bigl((-1)^{n+1}2^{2n-2}a\bigr)^{-1}\sum_{u'\in k^{\times}}\chi(u')\Kl_{u'}^{n+2}(\psi; w_{\G},\chi_{\G})\\
&\qquad+
\frac{C\omega_{0}\bigl((-1)^{n-1}\epsilon^{\nu_{\mu}}\bigr)}{2}\omega_{0}\chi\bigl((-1)^{n+1}2^{2n-2}a\bigr)^{-1}\sum_{u'\in k^{\times}}\omega_{0}\chi(u')\Kl_{u'}^{n+2}(\psi; w_{\G},\chi_{\G})\\
&=
\frac{C}{2}\cdot\chi\bigl((-1)^{n+1}2^{2n-2}a\bigr)^{-1}\sum_{u'\in k^{\times}}\chi(u')\Kl_{u'}^{n+2}(\psi; w_{\G},\chi_{\G})\\
&\qquad+
\frac{C\omega_{0}(\epsilon^{\nu_{\mu}}a^{-1})}{2}\chi\bigl((-1)^{n+1}2^{2n-2}a\bigr)^{-1}\sum_{u'\in k^{\times}}\omega_{0}\chi(u')\Kl_{u'}^{n+2}(\psi; w_{\G},\chi_{\G}).
\end{align*}
Thus, by Proposition \ref{prop:Kl} (1), this equals
\[
C\cdot\frac{1+\omega_{0}(\epsilon^{\nu_{\mu}}a^{-1})}{2}\cdot \chi\bigl((-1)^{n+1}2^{2n-2}a\bigr)^{-1}\cdot G(\chi,\psi)^{2}\cdot G(\chi\omega_{0},\psi)^{2}\cdot G(\chi^{2},\psi)^{n-2}.
\]
When $\G=\SO_{2n+2}^{\ur}$, by almost the same computation, the $\chi(u)$-twisted sum of the left-hand side of $(\dagger)$ is given by
\[
-C\cdot\frac{1+\omega_{0}(2\epsilon^{\nu_{\mu}}a^{-1})}{2}\cdot \chi\bigl((-1)^{n+1}2^{2n-1}a\bigr)^{-1}\cdot G(\chi,\psi)^{2}\cdot G(\chi\omega_{0},\psi)^{2}\cdot G(\chi^{2},\psi)^{n-2}
\]
(we use Lemma \ref{lem:HD} (3) in addition to Proposition \ref{prop:Kl} (1)).

On the other hand, when $u$ runs through all elements of $(-1)^{n-1}\epsilon^{\nu_{\mu}}k^{\times2}$, $v$ runs through all elements of $k^{\times}$.
Therefore, again by Proposition \ref{prop:Kl} (1), we have
\begin{align*}
\sum_{u\in (-1)^{n-1}\epsilon^{\nu_{\mu}}k^{\times2}} \chi(u)\Kl_{\pm2^{n}a_{-}v}^{n}(\psi)
&=
\sum_{v\in k^{\times}} \chi\bigl((-1)^{n-1}v^{2}\epsilon^{\nu_{\mu}}\bigr)\Kl_{2^{n}a_{-}v}^{n}(\psi)\\
&=
\chi\bigl((-1)^{n-1}\epsilon^{\nu_{\mu}}2^{-2n}a_{-}^{-2}\bigr)
\sum_{v'\in k^{\times}} \chi^{2}(v')\Kl_{v'}^{n}(\psi)\\
&=
\chi\bigl((-1)^{n-1}\epsilon^{\nu_{\mu}}2^{-2n}a_{-}^{-2}\bigr)G(\chi^{2},\psi)^{n}.
\end{align*}
Hence the $\chi(u)$-twisted sum of the right-hand side of $(\dagger)$ is given by
\[
\chi\bigl((-1)^{n-1}\epsilon^{\nu_{\mu}}2^{-2n}a_{-}^{-2}\bigr)G(\chi^{2},\psi)^{n}\cdot
\begin{cases}
\omega_{0}(-1)\cdot q & \text{if } \G=\SO_{2n+2},\\
-\omega_{0}(-1)\cdot q & \text{if } \G=\SO_{2n+2}^{\ur}.
\end{cases}
\]


Therefore, for every $\chi$, we have
\begin{multline*}
C\cdot \frac{1+\omega_{0}(2^{t_{\G}}\epsilon^{\nu_{\mu}}a^{-1})}{2}\cdot\chi\bigl((-1)^{n+1}2^{2n-2+t_{\G}}a\bigr)^{-1}
\cdot G(\chi,\psi)^{2}\cdot G(\chi\omega_{0},\psi)^{2}\cdot G(\chi^{2},\psi)^{n-2}\\
=
\omega_{0}(-1)\cdot q \cdot\chi\bigl((-1)^{n-1}\epsilon^{\nu_{\mu}}2^{-2n}a_{-}^{-2}\bigr)
G(\chi^{2},\psi)^{n}.
\end{multline*}
Since the Gauss sum is not zero, we have
\begin{multline*}
C\cdot\frac{1+\omega_{0}(2^{t_{\G}}\epsilon^{\nu_{\mu}}a^{-1})}{2}\cdot\chi(a)^{-1}
\cdot G(\chi,\psi)^{2}\cdot G(\chi\omega_{0},\psi)^{2}\\
=
\omega_{0}(-1)\cdot q \cdot\chi\bigl(2^{t_{\G}-2}a_{-}^{-2}\epsilon^{\nu_{\mu}}\bigr)
\cdot G(\chi^{2},\psi)^{2}.
\end{multline*}
By the Hasse-Davenport product relation (Lemma \ref{lem:HD} (2))
\[
G(\chi^{2},\psi)\cdot G(\omega_{0},\psi)=G(\chi,\psi)\cdot G(\chi\omega_{0},\psi)\cdot \chi(4)
\]
and the relation 
\[
G(\omega_{0},\psi)^{2}=\omega_{0}(-1)\cdot q
\]
in Lemma \ref{lem:HD} (1), we get
\[
C\cdot\frac{1+\omega_{0}(2^{t_{\G}}\epsilon^{\nu_{\mu}}a^{-1})}{2}\cdot\chi(a)^{-1}
=
\chi\bigl(2^{t_{\G}+2}a_{-}^{-2}\epsilon^{\nu_{\mu}}\bigr).
\]
By noting that the right-hand side is not zero, we first see that $\frac{1+\omega_{0}(2^{t_{\G}}\epsilon^{\nu_{\mu}}a^{-1})}{2}$ is necessarily nonzero, hence $\frac{1+\omega_{0}(2^{t_{\G}}\epsilon^{\nu_{\mu}}a^{-1})}{2}=1$.
Then, since the equality holds for every $\chi$, we get
\[
a=\frac{a_{-}^{2}}{2^{t_{\G}+2}\epsilon^{\nu_{\mu}}} \quad\text{and}\quad
C=1.
\]

We next determine the data $\xi$ of the central character.
We take $g$ in the endoscopic character relation $(\ast)$ to be $-\tilde{g}_{u}^{\G}$.
Then the sum of the endoscopic character relation runs over $(-h_{-}^{(1)}, -h_{+})$ and $(-h_{-}^{(-1)}, -h_{+})$.
Thus we have
\[
\bigl(\Theta^{\G}_{\xi,0,a,\zeta} - \Theta^{\G}_{\xi,1,a\epsilon^{-1},\zeta}\bigr)(-\tilde{g}_{u}^{\G})
= 
\xi\cdot\bigl(\Theta^{\G}_{\xi,0,a,\zeta} - \Theta^{\G}_{\xi,1,a\epsilon^{-1},\zeta}\bigr)(g_{u}^{\G}), \text{ and }
\]
\[
\sum_{\begin{subarray}{c}(h_{-},h_{+})\\\leftrightarrow -\tilde{g}_{u}^{\G}/\sim\end{subarray}}\Theta_{\xi_{-},a_{-}}^{\mathbf{H}_{-}}(h_{-})
=
\xi_{-}\cdot\sum_{\begin{subarray}{c}(h_{-},h_{+})\\\leftrightarrow \tilde{g}_{u}^{\G}/\sim\end{subarray}}\Theta_{\xi_{-},a_{-}}^{\mathbf{H}_{-}}(h_{-}).
\]
Moreover, by Propositions \ref{prop:4SOspl}, \ref{prop:tranSOspl}, \ref{prop:4SOur}, and \ref{prop:tranSOur} we have
\[
\Delta_{\mathbf{H},\G}\bigl((-h_{-}^{(\pm1)},-h_{+}^{0}),-\tilde{g}_{u}^{\G}\bigr)
\frac{D_{\mathbf{H}_{-}}(-h_{-}^{(\pm1)})^{2}D_{\mathbf{H}_{+}}(-h_{+}^{0})^{2}}{D_{\G}(-\tilde{g}_{u}^{\G})^{2}}
\]
\[
=
\Delta_{\mathbf{H},\G}\bigl((h_{-}^{(\pm1)},h_{+}^{0}),\tilde{g}_{u}^{\G}\bigr)
\frac{D_{\mathbf{H}_{-}}(h_{-}^{(\pm1)})^{2}D_{\mathbf{H}_{+}}(h_{+}^{0})^{2}}{D_{\G}(\tilde{g}_{u}^{\G})^{2}}\cdot
(-1)^{t_{\G}}\omega_{0}(-1).
\]
Therefore, by the same computation as above, we get
\[
\xi=(-1)^{t_{\G}}\omega_{0}(-1)\cdot\xi_{-}.
\]
 
We finally determine the data $\zeta$.
To do this, again taking $u\in(-1)^{n-1}\epsilon^{\nu_{\mu}}k^{\times2}$, we consider the element $(\tilde{g}_{u}^{\G})'$ defined in Sections \ref{subsec:tran-SO} and \ref{subsec:tran-SO-ur}.
Recall that this element is $\rmO_{2n+2}^{(\ur)}(F)$-conjugate to
\[
g_{u}^{\G}\varphi_{(-2)^{-1+t_{\G}},2u(-1)^{n-1+t_{\G}}}^{\G}.
\]
Note that the squared element $(g_{u}^{\G}\varphi_{(-2)^{-1+t_{\G}},2u(-1)^{n-1+t_{\G}}}^{\G})^{2}$ is an affine generic element of $I_{\G}^{+}$.
Indeed, since the simple affine components of $g_{u}^{\G}$ are given by
\[
\begin{cases}
\bigl(-1, \underbrace{2, \ldots, 2}_{n-2}, -1, 2, 2u(-1)^{n-1}\bigr) & \text{if } \G=\SO_{2n+2},\\
\bigl(1, \underbrace{2, \ldots, 2}_{n-2}, 2, 2u(-1)^{n-1} \bigr) & \text{if } \G=\SO_{2n+2}^{\ur},
\end{cases}
\]
the simple affine components of $(g_{u}^{\G}\varphi_{(-2)^{-1+t_{\G}},2u(-1)^{n-1+t_{\G}}}^{\G})^{2}$ are given by
\[
\begin{cases}
\bigl(-2, \underbrace{4, \ldots, 4}_{n-2}, -2, 4, 4u(-1)^{n-1} \bigr) & \text{if } \G=\SO_{2n+2},\\
\bigl(2, \underbrace{4, \ldots, 4}_{n-2}, 4, 4u(-1)^{n-1} \bigr) & \text{if } \G=\SO_{2n+2}^{\ur},
\end{cases}
\]
according to Remarks \ref{rem:affgen-square} and \ref{rem:affgen-square-ur}.
As we assume that the residual characteristic $p$ is not equal to $2$, any simple affine component is nonzero, hence $(g_{u}^{\G}\varphi_{(-2)^{-1+t_{\G}},2u(-1)^{n-1+t_{\G}}}^{\G})^{2}$ is affine generic.
Thus we can utilize Propositions \ref{prop:charSOspl2} and \ref{prop:charSOur2} to compute the character 
\[
\Theta^{\G}_{\xi,\kappa,a\epsilon^{-\kappa},\zeta}\bigl((\tilde{g}_{u}^{\G})'\bigr)
=
\Theta^{\G}_{\xi,\kappa,a\epsilon^{-\kappa},\zeta}\bigl(g_{u}^{\G}\varphi_{(-2)^{-1+t_{\G}},2u(-1)^{n-1+t_{\G}}}^{\G}\bigr).
\]
First, this is not zero only if
\[
\omega_{0}(-2)^{-1+t_{\G}}=(-1)^{\kappa}
\quad\text{and}\quad
\omega_{0}\bigl(2u(-1)^{n-1+t_{\G}}\bigr)=\omega_{0}(-a^{-1}\epsilon^{\kappa})
\]
according to Propositions \ref{prop:charSOspl2} and \ref{prop:charSOur2}.
Since we have $a=a_{-}^{2}2^{-2-t_{\G}}\epsilon^{-\nu_{\mu}}$, the second condition is equivalent to the first one.
Thus, the character is not zero only if $\omega_{0}(-2)^{-1+t_{\G}}=(-1)^{\kappa}$.
For such a $\kappa$, the index $\pm\delta$ of the Kloosterman sums in Propositions \ref{prop:charSOspl2} and \ref{prop:charSOur2} are given by
\[
\pm\delta=\pm2^{n}a_{-}v.
\]
Here $v$ is an element of $k^{\times}$ satisfying $v^{2}\epsilon^{\nu_{\mu}}=(-1)^{n+1}u$.
Thus we get
\[
\bigl(\Theta^{\G}_{\xi,0,a,\zeta} - \Theta^{\G}_{\xi,1,a\epsilon^{-1},\zeta}\bigr)\bigl((\tilde{g}_{u}^{\G})'\bigr)
=
\omega_{0}(-2)^{-1+t_{\G}}\cdot\zeta\cdot
\Kl^{n}_{\pm2^{n}a_{-}v}(\psi).
\]
On the other hand, by Propositions \ref{prop:4SOspl}, \ref{prop:tranSOspl}, \ref{prop:4SOur}, and \ref{prop:tranSOur}, we have
\[
\Delta_{\mathbf{H},\G}\bigl((h_{-}^{(\pm1)},-h_{+}^{0}),(\tilde{g}_{u}^{\G})'\bigr)
\frac{D_{\mathbf{H}_{-}}(h_{-}^{(\pm1)})^{2}D_{\mathbf{H}_{+}}(-h_{+}^{0})^{2}}{D_{\G}((\tilde{g}_{u}^{\G})')^{2}}
=
\omega_{0}(-2)^{1-t_{\G}}.
\]
Moreover, by Proposition \ref{prop:charSOram}, we have
\[
\sum_{\begin{subarray}{c}(h_{-},h_{+})\leftrightarrow \\ (\tilde{g}_{u}^{\G})'/\sim\end{subarray}}\Theta_{\xi_{-},a_{-}}^{\mathbf{H}_{-}}(h_{-})
=
\Kl_{\pm2^{n}a_{-}v}^{n}(\psi).
\]
Thus, by combining these results, we get
\[
\omega_{0}(-2)^{1-t_{\G}}\cdot\zeta\cdot
\Kl^{n}_{\pm2^{n}a_{-}v}(\psi)
=
\omega_{0}(-2)^{1-t_{\G}}\cdot
\Kl^{n}_{\pm2^{n}a_{-}v}(\psi).
\]
By taking the sum of this equality over $u\in(-1)^{n-1}\epsilon^{\nu_{\mu}}k^{\times2}$, hence over $v\in k^{\times}$, we get
\[
\zeta\sum_{v'\in k^{\times}} \Kl_{v'}^{n}(\psi)
=
\sum_{v'\in k^{\times}} \Kl_{v'}^{n}(\psi).
\]
Since we have
\[
\sum_{v'\in k^{\times}} \Kl_{v'}^{n}(\psi)
=
G(\mathbbm{1},\psi)^{n}
=
(-1)^{n}
\]
by Proposition \ref{prop:Kl} (1), we get $\zeta=1$.
\end{proof}

\subsection{Construction of simple supercuspidal $L$-packets with $\zeta=-1$}\label{subsec:zeta=-1}
We next construct $L$-packets of $\G$ consisting of simple supercuspidal representations whose final parameters $\zeta$ are given by $-1$.
To do this, we consider the twist of the $L$-packets in the previous subsection via the spinor norm of $G$:
\[
\spin\colon G\rightarrow F^{\times}/F^{\times2}.
\]
We do not recall the definition of the spinor norm in this paper.
For example, see \cite[Section 24]{MR2665139} for the definition of the spinor norm.

\begin{lem}\label{lem:spin-image}
For any element $(\alpha,\beta)$ of 
\[
\begin{cases}
k^{\times}\times k^{\times} & \text{if $\G=\SO_{2n+2}$,}\\
\Nr^{1}\times k^{\times} & \text{if $\G=\SO_{2n+2}^{\ur}$,}
\end{cases}
\]
we have
\[
\mu_{\ur}\circ\spin(\varphi_{\alpha,\beta}^{\G})=-1.
\]
\end{lem}

\begin{proof}
We first note that the image of the hyperspecial maximal compact subgroup $\G(\mcO)$, which consists of matrices whose entries belong to $\mcO$, under the spinor norm is contained in $\mcO^{\times}/\mcO^{\times2}$ (see, for example, \cite[Theorem 1.9]{MR2262936}).
In particular, the image of $\G(\mcO)$ under $\mu_{\ur}\circ\spin$ is trivial.

On the other hand, we have
\[
\diag(\varpi,1,\ldots,1,\varpi^{-1})\cdot\varphi_{\alpha,\beta}^{\G}
\in
\G(\mcO).
\]
Thus, by noting that the restriction of the spinor norm on any Levi subgroup $\GL_{r}\times\SO_{2n+2-2r}^{(\ur)}$ is given by $\det\boxtimes\spin$ (see, for example, \cite[Lemma 4.9]{MR2027702}), we get
\[
\mu_{\ur}\circ\spin(\varphi_{\alpha,\beta}^{\G})
=\mu_{\ur}\circ\spin\bigl(\diag(\varpi,1,\ldots,1,\varpi^{-1})^{-1}\bigr)=\mu_{\ur}(\varpi^{-1})
=-1.
\]\end{proof}

Let $\phi$ be the $L$-parameter of $\G$ considered in the previous subsection.
We consider its twist $\phi\otimes\mu_{\ur}$ via $\mu_{\ur}$ (see Section \ref{sec:spin} for the details of this twist).

\begin{thm}\label{thm:packetSOspl2}
The $L$-packet of $\G$ for $\phi\otimes\mu_{\ur}$ is  given by
\[
\bigl\{ \pi^{\G}_{\xi, 0,a,-1}, \pi^{\G}_{\xi,1,a\epsilon^{-1},-1} \bigr\}
\]
where
\[
\xi=
\begin{cases}
\omega_{0}(-1)\cdot\xi_{-} & \text{if } \G=\SO_{2n+2},\\
-\omega_{0}(-1)\cdot\xi_{-} & \text{if } \G=\SO_{2n+2}^{\ur},
\end{cases}
\quad\text{and}\quad
a=
\begin{cases}
\frac{a_{-}^{2}}{4\epsilon^{\nu_{\mu}}} & \text{if } \G=\SO_{2n+2},\\
\frac{a_{-}^{2}}{8\epsilon^{\nu_{\mu}}} & \text{if } \G=\SO_{2n+2}^{\ur}.
\end{cases}
\]
\end{thm}

\begin{proof}
By Theorem \ref{thm:packetSOspl} and Proposition \ref{prop:spin-twist}, the $L$-packet of $\phi\otimes\mu_{\ur}$ is given by
\[
\bigl\{ \pi^{\G}_{\xi, 0,a,1}, \pi^{\G}_{\xi,1,a\epsilon^{-1},1} \bigr\}\otimes(\mu_{\ur}\circ\spin).
\]
As a consequence of the Frobenius reciprocity, the representation $\pi^{\G}_{\xi,\kappa,b,1}\otimes (\mu_{\ur}\circ\spin)$ is given by the compact induction of $\tilde{\chi}^{\G}_{\xi,\kappa,b,1}\otimes(\mu_{\ur}\circ\spin)$.

Since the image of $\pm I_{\G}^{+}\subset\G(\mcO)$ under $\spin$ is contained in $\mcO^{\times}/ \mcO^{\times2}$, hence contained in the kernel of $\mu_{\ur}$, the twist via $\mu_{\ur}\circ\spin$ does not affect the first three data $(\xi,\kappa,b)$.
On the other hand, by Lemma \ref{lem:spin-image}, we have
\[
\bigl(\tilde{\chi}^{\G}_{\xi,\kappa,b,1}\otimes(\mu_{\ur}\circ\spin)\bigr)(\varphi_{\alpha,\beta}^{\G})=-1
\]
for any $(\alpha,\beta)$.
Thus we can conclude
\[
\tilde{\chi}^{\G}_{\xi,\kappa,b,1}\otimes(\mu_{\ur}\circ\spin)
\cong
\tilde{\chi}^{\G}_{\xi,\kappa,b,-1}
\]
and get the result.
\end{proof}

In summary, we get the following:
\begin{thm}\label{thm:liftSOspltoGL}
Let $(\xi,0,a,\zeta)\in\SSC(\G)$.
Then
\[
\bigl\{ \pi^{\G}_{\xi,0,a,\zeta}, \pi^{\G}_{\xi,1,a\epsilon^{-1},\zeta} \bigr\}
\]
is an $L$-packet of $\G$.
Moreover, its endoscopic lift to $\GL_{2n+2}$ is given by the parabolic induction of 
\[
\begin{cases}
\pi_{\omega_{0},b,\eta}^{\GL_{2n}}\boxtimes(\omega_{\omega_{0},b,\eta}^{\GL_{2n}}\cdot\mu_{\G})\boxtimes\mathbbm{1} & \text{if }\zeta=1,\\
\pi_{\omega_{0},b,-\eta}^{\GL_{2n}}\boxtimes(\omega_{\omega_{0},b,\eta}^{\GL_{2n}}\cdot\mu_{\ur}\cdot\mu_{\G})\boxtimes\mu_{\ur} & \text{if }\zeta=-1.
\end{cases}
\]
Here $\omega_{\omega_{0},b,\eta}^{\GL_{2n}}$ is the central character of $\pi_{\omega_{0},b,\eta}^{\GL_{2n}}$ (note that $\omega_{\omega_{0},b,\eta}^{\GL_{2n}}=\omega_{\omega_{0},b,-\eta}^{\GL_{2n}}$) and 
\[
\eta=(-1)^{t_{\G}}q^{-\frac{1}{2}}G(\omega_{0},\psi)\omega_{0}(-1)\xi
\quad\text{and}\quad
b=(-1)^{n-1}2^{2+t_{\G}}a.
\]
\end{thm}

\begin{proof}
We first consider the case where $\zeta=1$.
By Theorem \ref{thm:packetSOspl}, the finite set in the assertion is an $L$-packet of $\G$.
Moreover, this $L$-packet is the endoscopic lift of an $L$-packet of $\mathbf{H}=\mathbf{H}_{-}\times\mathbf{H}_{+}$  which contains $\pi_{\xi_{-},a_{-}}^{\mathbf{H}_{-}}\boxtimes\mathbbm{1}$ and corresponds to an $L$-parameter $\phi_{-}\oplus\mathbbm{1}\oplus\mu\cdot\mu_{\G}$.
Here $\mathbf{H}_{-}$ corresponds to a quadratic character $\mu$ satisfying $a\in2^{-t_{\G}}\epsilon^{-\nu_{\mu}}k^{\times2}$.
The data $(\xi_{-},a_{-})$ is an element of $\SSC(\mathbf{H}_{-})$ satisfying the following relation:
\[
\xi=(-1)^{t_{\G}}\omega_{0}(-1)\xi_{-}
\quad\text{and}\quad
a=\frac{a_{-}^{2}}{2^{2+t_{\G}}\epsilon^{\nu_{\mu}}},
\]
and $\phi_{-}$ is the $L$-parameter of $\pi_{\xi_{-},a_{-}}^{\mathbf{H}_{-}}$.

On the other hand, by Theorem \ref{thm:liftSOramtoGL}, the endoscopic lift $\pi_{\phi_{-}}^{\GL_{2n}}$ of $\pi_{\xi_{-},a_{-}}^{\mathbf{H}_{-}}$ to $\GL_{2n}$ is given by $\pi_{\omega_{0},b,\eta}^{\GL_{2n}}$, where 
\[
\eta=q^{-\frac{1}{2}}G(\omega_{0},\psi)\xi_{-}
\quad\text{and}\quad
b=(-1)^{n-1}\frac{a_{-}^{2}}{\epsilon^{\nu_{\mu}}}.
\]
By combining the above relations, we get
\[
\eta=(-1)^{t_{\G}}q^{-\frac{1}{2}}G(\omega_{0},\psi)\omega_{0}(-1)\xi
\quad\text{and}\quad
b=(-1)^{n-1}2^{2+t_{\G}}a.
\]
Thus we get the assertion.

We next consider the case where $\zeta=-1$.
However, since we can check easily that 
\[
\pi_{\omega_{0},b,\eta}^{\GL_{2n}}\otimes(\mu_{\ur}\circ\det)\cong\pi_{\omega_{0},b,-\eta}^{\GL_{2n}},
\]
the assertion for this case follows from that for the case where $\zeta=1$.
\end{proof}

\section{Application: Formal degree conjecture for simple supercuspidal $L$-packets for some bad primes}\label{sec:FDC}

In this section, as an application of our results on a description of the structure of simple supercuspidal $L$-packets and their $L$-parameters, we prove that the Hiraga--Ichino--Ikeda formal degree conjecture holds for simple supercuspidal $L$-packets of quasi-split classical groups, under some restriction on the residual characteristic $p$.

\subsection{Formal degree conjecture of Hiraga--Ichino--Ikeda}\label{subsec:FDC}

We first recall the formal degree conjecture.
Let $\mathbf{G}$ be a quasi-split classical group over $F$.
Then the formal degree conjecture predicts that the following two objects are related under the local Langlands correspondence for $\G$:
\begin{description}
\item[Formal degrees of discrete series representations]
We fix a Haar measure $dg$ of $G$.
Then, for an irreducible discrete series representation $\pi\in\Pi_{2}(\G)$ of $G$, we can define the \textit{formal degree} $\deg(\pi)$ of $\pi$ with respect to $dg$ (see Section \ref{subsec:deg-Sp} for the definition of $\deg(\pi)$).

\item[Adjoint $\gamma$-factors of discrete $L$-parameters]
For a discrete $L$-parameter $\phi\in\Phi_{2}(\G)$ of $\G$, the \textit{adjoint $\gamma$-factor} of $\phi$ is defined by
\[
\gamma(s, \Ad\circ\phi, \psi_{0})
:=\varepsilon(s, \Ad\circ\phi, \psi_{0})\cdot\frac{L(1-s, \Ad\circ\phi)}{L(s, \Ad\circ\phi)}.
\]
Here $\psi_{0}$ is a nontrivial additive character of $F$ of level zero and $\Ad$ is the adjoint representation of ${}^{L}\G$ on its Lie algebra $\widehat{\mathfrak{g}}:=\Lie(\widehat{\G})$:
\[
\Ad\colon{}^{L}\G \rightarrow \GL(\widehat{\mathfrak{g}}).
\]
We note that $\gamma(s, \Ad\circ\phi, \psi_{0})$ is holomorphic and not zero at $s=0$ by the assumption that $\phi$ is discrete (see \cite[Lemma 1.2]{MR2350057}).
\end{description}

Now we state the Hiraga--Ichino--Ikeda formal degree conjecture.
Here we follow Gross--Reeder's formulation:
\begin{conj}[{\cite[Conjecture 1.4]{MR2350057}}, {\cite[Conjecture 7.1 (5)]{MR2730575}}]\label{conj:FDC}
Let $\mathrm{St}_{\G}\in\Pi_{2}(\G)$ be the Steinberg representation of $G$, and $\phi_{\G}$ its corresponding discrete $L$-parameter.
We take $\phi\in\Phi_{2}(\G)$ and $\pi\in\Pi_{\phi}^{\G}\subset\Pi_{2}(\G)$.
Then we have
\[
\frac{\deg(\pi)}{\deg(\mathrm{St}_{\G})}
=
\frac{|\pi_{0}(S^{\G}_{\phi_{\G}})|}{|\pi_{0}(S^{\G}_{\phi})|}\cdot
\frac{|\gamma(0, \Ad\circ\phi, \psi_{0})|}{|\gamma(0, \Ad\circ\phi_{\G}, \psi_{0})|}.
\]
\end{conj}

\begin{rem}
In the above formulation, the ambiguity of $dg$ is cancelled by taking the ratio of formal degrees.
On the other hand, the conjecture is formulated in the absolute form in the original paper of Hiraga--Ichino--Ikeda (\cite{MR2350057}) by taking $dg$ explicitly:
\[
\deg(\pi)
=
\frac{1}{|\pi_{0}(S^{\G}_{\phi})|}\cdot|\gamma(0, \Ad\circ\phi, \psi_{0})|.
\]
However, as the absolute version of the conjecture for the Steinberg representation is showed in \cite[Section 3.3]{MR2350057}, these two formulations are equivalent.
See also \cite[Section 7]{MR2730575}.
\end{rem}

The formal degree conjecture (Conjecture \ref{conj:FDC}) is still open in the cases of symplectic and even orthogonal groups.
The aim of this section is to check this conjecture for the simple supercuspidal $L$-packets of these groups under some assumption on the residual characteristic.
The following is the main theorem of this section.

\begin{thm}\label{thm:FDC}
Let $\G$ be one of the following quasi-split classical groups over $F$:
\[
\Sp_{2n},\quad
\SO_{2n}^{\mu},\quad
\SO_{2n+2}^{(\ur)},
\]
where $\mu$ is a ramified quadratic character of $F^{\times}$.
We assume either
\begin{itemize}
\item
that $p$ does not divide $2n$, or
\item
that $2n=p^{e}\cdot n'$, where $e\in\Z_{\geq1}$ and $n'\in\Z_{>0}$ satisfying $n' \mid (p-1)$.
\end{itemize}
Then the simple supercuspidal $L$-packets of $\G$ satisfy the formal degree conjecture (Conjecture \ref{conj:FDC}).
\end{thm}

In fact, it is known that the formal degree conjecture is consistent with the theta lifting by \cite{MR3166215}.
In particular, the cases where $\G$ is $\SO_{2n}^{\mu}$ or $\SO_{2n+2}^{(\ur)}$ follow from the case where $\G=\Sp_{2n}$.

To be more precise, let us recall that $(\Sp_{2n},\rmO_{2n}^{\mu})$ can be realized as a reductive dual pair of type (D) with invariant $l=-1$ in the sense of \cite[Section 3]{MR3166215} (we take $(H(V), G(W))$ to be $(\Sp_{2n},\rmO_{2n}^{\mu})$).
Let $\phi_{0}$ be the $L$-parameter of a simple supercuspidal representation of $\SO_{2n}^{\mu}$.
We define an $L$-parameter $\phi$ of $\Sp_{2n}$ by $\phi:=\phi_{0}\oplus\mu$.
As explained in the proof of Proposition \ref{prop:ssc-descent}, if we regard an element $\tilde{\sigma}$ of the simple supercuspidal $L$-packet $\widetilde{\Pi}_{\phi_{0}\otimes\mu}^{\SO_{2n}^{\mu}}$ as an irreducible smooth representation $\sigma$ of $\rmO_{2n}^{\mu}(F)$, then it maps to a member $\pi$ of the $L$-packet $\Pi_{\phi}^{\Sp_{2n}}$, which consists of two simple supercuspidal representations, by the theta lift.
We apply \cite[Theorem 15.1 (ii)]{MR3166215} to these representations $\sigma$ and $\pi$.
Here note that the invariant $l$ must be non-negative in \cite[Theorem 15.1 (ii)]{MR3166215}.
Thus, when we utilize \cite[Theorem 15.1 (ii)]{MR3166215}, we regard $(\Sp_{2n},\rmO_{2n}^{\mu})$ as a dual pair of type $(C'')$ with invariant $l=1$ by swapping $H(V)$ and $G(W)$.
Then we get an equality
\[
\frac{\deg(\pi)}{\deg(\sigma)}
=
|\gamma(0,\sigma,\mu^{-1},\psi_{0})|
=
|\gamma(0,\phi_{0}\otimes\mu,\psi_{0})|.
\]
On the other hand, as a representation of $W_{F}\times\SL_{2}(\C)$, we have
\[
\Ad\circ\phi
\cong
\wedge^{2}\phi
\cong
(\wedge^{2}\phi_{0})\oplus(\phi_{0}\otimes\mu).
\]
Thus, if (the absolute version of) the formal degree conjecture for the simple supercuspidal representation of $\pi$ of $\Sp_{2n}$ and its $L$-parameter $\phi$ is true, i.e., the equality
\[
\deg(\pi)
=
\frac{|\gamma(0,\Ad\circ\phi,\psi_{0})|}{|\pi_{0}(S_{\phi}^{\Sp_{2n}})|}
\]
holds, then we have
\[
\deg(\sigma)
=
\frac{|\gamma(0,\Ad\circ\phi,\psi_{0})|}{|\pi_{0}(S_{\phi}^{\Sp_{2n}})|}\cdot|\gamma(0,\sigma,\mu^{-1},\psi_{0})|^{-1}
=
\frac{|\gamma(0,\wedge^{2}\phi_{0},\psi_{0})|}{|\pi_{0}(S_{\phi}^{\Sp_{2n}})|}.
\]
By noting that 
\begin{itemize}
\item
$\Ad\circ(\phi_{0}\otimes\mu)\cong\wedge^{2}\phi_{0}$,
\item
$|\pi_{0}(S_{\phi_{0}\otimes\mu}^{\SO_{2n}^{\mu}})|=|\pi_{0}(S_{\phi}^{\Sp_{2n}})|=2$, and
\item
$\deg(\tilde{\sigma})=\deg(\sigma)$ under the choice of Haar measures according to \cite{MR3166215} (see \cite[Lemma 13.1]{MR3166215} and note that $\tilde{\sigma}$ consists of two element),
\end{itemize}
we get the equality
\[
\deg(\tilde{\sigma})
=
\frac{|\gamma(0,\Ad\circ(\phi_{0}\otimes\mu),\psi_{0})|}{|\pi_{0}(S_{\phi_{0}\otimes\mu}^{\SO_{2n}^{\mu}})|},
\]
which is the formal degree conjecture for the simple supercuspidal representation belonging to the $L$-packet for $\phi_{0}\otimes\mu$.
(Note that the formal degree is independent of the representative of the $\rmO_{2n}^{\mu}(F)$-orbit $\tilde{\sigma}$, so we loosely write ``$\deg(\tilde{\sigma})$'' above.)

The case of $\SO_{2n+2}^{(\ur)}$ can be derived from the case of $\G=\Sp_{2n}$ by the same argument.
We only remark the following points in this case (we put $\G:=\SO_{2n+2}^{(\ur)}$):
\begin{itemize}
\item
Let $\phi'=\phi_{-}\otimes\mu_{\G}\oplus\mu$ be the $L$-parameter of a simple supercuspidal representation of $\Sp_{2n}$.
Then the $L$-parameter of the theta lifts of the members of $\Pi_{\phi'}^{\Sp_{2n}}$ to $\rmO_{2n+2}^{(\ur)}$ is given by $\theta(\phi'):=\phi_{-}\oplus\mu\cdot\mu_{\G}\oplus\mathbbm{1}$.
Thus we have $\pi_{0}(S_{\phi'}^{\Sp_{2n}})=2$ and $\pi_{0}(S_{\theta(\phi')}^{\G})=4$.
\item
Any $\tilde{\pi}\in\Pi_{\theta(\phi')}^{\G}$ consists of only one element, thus we have $\deg(\tilde{\pi})=2\deg(\pi)$ by \cite[Lemma 13.1]{MR3166215}.
\item
In contrast to the case of $\SO_{2n}^{\mu}$, we cannot get all simple supercuspidal $L$-packets of $\SO_{2n+2}^{(\ur)}$ by the theta lift from $\Sp_{2n}$.
As we explained it in the previous section, in order to obtain simple supercuspidal $L$-packets whose ``$\zeta$'' equals $-1$, we have to consider twists of simple supercuspidal $L$-packets whose ``$\zeta$'' equals $1$ via a quadratic character.
Since such twisting does not affect the formal degree or the absolute value of the adjoint gamma factor, the formal degree conjecture for simple supercuspidal $L$-packets with $\zeta=-1$ follows from that for simple supercuspidal $L$-packets with $\zeta=1$.
\end{itemize}

Thus, in the following, we will show Theorem \ref{thm:FDC} only in the case of $\Sp_{2n}$.

\subsection{Outline of the proof of the formal degree conjecture}\label{subsec:outline}

Since the proof of Theorem \ref{thm:FDC} is slightly long and complicated, we explain the outline of the proof here.

\subsubsection{Outline of the proof: the case where $p$ does not divide $2n$}

In Section \ref{subsec:nmid} we first treat the case where $p\nmid2n$.
In this case, Kaletha constructs and investigates $L$-packets consisting of simple supercuspidal representations in his works \cite{MR3001735} and \cite{MR3402796}.
Especially, he associated to each simple supercuspidal representation an $L$-parameter satisfying a condition called ``simple-wildness'' and furthermore proved that the formal degree conjecture is true for them.
Thus it is enough to show that Kaletha's explicit construction of the local Langlands correspondence for simple supercuspidal representations is consistent with ours.
To be more precise, let $\Pi_{\phi}^{\Sp_{2n}}$ be a simple supercuspidal $L$-packet of $\Sp_{2n}$ with the $L$-parameter $\phi$ in the sense of Arthur, which is investigated in the main part of this paper.
If we can check that this $L$-parameter is simple wild in the sense of Kaletha, then we would get Kaletha's $L$-packet $\Pi_{\phi,\mathrm{Kal}}^{\Sp_{2n}}$ according to \cite{MR3001735} and \cite{MR3402796}.
Furthermore, the formal degree conjecture holds for $\Pi_{\phi,\mathrm{Kal}}^{\Sp_{2n}}$ and $\phi$.
The point here is that both $\Pi_{\phi}^{\Sp_{2n}}$ and $\Pi_{\phi,\mathrm{Kal}}^{\Sp_{2n}}$ consist of simple supercuspidal representations of $\Sp_{2n}(F)$.
In particular, all members of $\Pi_{\phi}^{\Sp_{2n}}$ and $\Pi_{\phi,\mathrm{Kal}}^{\Sp_{2n}}$ have the same formal degree.
Hence we can conclude that the conjecture is true also for $\Pi_{\phi}^{\Sp_{2n}}$ and $\phi$.
(Note that thus we do not show the coincidence of $\Pi_{\phi}^{\Sp_{2n}}$ with $\Pi_{\phi,\mathrm{Kal}}^{\Sp_{2n}}$ in this proof.)

\subsubsection{Outline of the proof: the second case}
The main part of this section is about the second case where $p\mid2n$, which is more complicated.
We treat this case by appealing to an explicit description of the $L$-parameters of simple supercuspidal representations of general linear groups due to Imai--Tsushima (\cite{Imai:2015aa}).

To be more precise, let us first recall how the $\gamma$-factor of a local Galois representation is described in general.
Let $\rho\colon W_{F}\rightarrow\GL(V)$ be a finite-dimensional complex representation which is admissible (i.e., $\Ker(\rho)$ is open in $W_{F}$ and any lift of the Frobenius maps to a semisimple element), or, equivalently, an $L$-parameter of $\GL_{n}$ (where $n=\dim(V)$) whose $\SL_{2}(\C)$-part is trivial.
Then, as already appeared in the beginning of Section \ref{subsec:FDC} ($\rho$ is taken to be $\Ad\circ\phi$ there), the $\gamma$-factor of $\rho$ is defined by
\[
\gamma(s, \rho, \psi_{0})
:=\varepsilon(s, \rho, \psi_{0})\cdot\frac{L(1-s, \rho^{\vee})}{L(s, \rho)}.
\]
Here the $L$-factor $L(s, \rho)$ of a representation $\rho$ is defined by
\[
L(s, \rho)
:=
\det\bigl(1-q^{-s}\cdot\rho(\Frob)\mid V^{\rho(I_{F})}\bigr)^{-1}.
\]
The definition of the $\varepsilon$-factor is more difficult.
However, when $\rho$ is self-dual, the absolute value of its special value at $s=0$ can be simply described in terms of the Artin conductor or the Swan conductor as follows (see \cite[Section 2]{MR2730575} for the details).
Since $\rho$ is admissible, its restriction $\rho|_{I_{F}}$ to the inertia subgroup $I_{F}$ factors through a finite quotient $D$, which is regarded as the Galois group of a finite Galois extension of the maximal unramified extension $F^{\ur}$ of $F$.
Then we can equip the group $D$ with the lower ramification filtration
\[
D_{0} \trianglerighteq D_{1} \trianglerighteq D_{2} \trianglerighteq \cdots
\]
indexed by non-negative integers (we will review the definition in Section \ref{subsec:wild}).
With this notation, the Artin conductor and the Swan conductor of $\rho$ are given by
\[
\Artin(\rho)= \sum_{j\geq0} \dim_{\C}(V/V^{D_{j}}) \frac{|D_{j}|}{|D_{0}|},\quad
\Swan(\rho)= \sum_{j\geq1} \dim_{\C}(V/V^{D_{j}}) \frac{|D_{j}|}{|D_{0}|}
\]
(note that these are non-negative integers independent of the choice of a finite quotient $D$ which $\rho$ factors through, see, e.g., \cite[Chapter IV]{MR554237}).
Then we have 
\[
|\varepsilon(0,\rho,\psi_{0})|
=
q^{\frac{1}{2}\Artin(\rho)}.
\]
(See \cite[Section 2, 435 page]{MR2730575}. Note that the constant $w(V)$ in \cite[Section 2, 435 page]{MR2730575} is a root of unity since $\rho$ is self-dual here, hence its absolute value can be ignored; see also \cite[Section 29.4, Proposition]{MR2234120}.)
Therefore, if we try to determine the special value of the $\gamma$-factor at $s=0$ according to this description, then it is important to understand how $\rho$ is described as a representation of a finite quotient group $D$ and how the lower ramification filtration of $D$ is given.

Let us go back to our situation.
By our result in this paper (Corollary \ref{cor:decomp} and \ref{thm:liftSptoGL}), the $L$-parameter $\phi$ of a simple supercuspidal representation of $\Sp_{2n}(F)$ is of the form $\phi_{0}\oplus\det\circ\phi_{0}$ (as a $(2n+1)$-dimensional representation of $W_{F}$), where $\phi_{0}$ is the $L$-parameter of a self-dual simple supercuspidal representation of $\GL_{2n}(F)$.
As the composition of $\phi$ with the adjoint representation on the Lie algebra of $\widehat{\Sp_{2n}}=\SO_{2n+1}(\C)$ is nothing but the exterior square product, we get
\[
\Ad\circ\phi
\cong
\wedge^{2}\phi
\cong
(\wedge^{2}\phi_{0})\oplus(\phi_{0}\otimes\det\circ\phi_{0}).
\]
Hence our tasks basically consist of the following:
\begin{itemize}
\item[(0)]
Determine (find) a finite quotient $D$ which $\phi_{0}$ factors through and describe $\phi_{0}$ as a representation of $D$ explicitly. 
\item[(1)]
Compute the lower ramification filtration of $D$.
\item[(2)]
Describe $\wedge^{2}\phi_{0}$ as a representation of $D$ explicitly. 
\end{itemize}

Among these, a complete answer to (0) is given in \cite{Imai:2015aa}.
In fact, we can take $D$ to be a certain finite group $\mathcal{Q}$ which has a Heisenberg group $Q'$ as the subgroup corresponding to the wild inertia subgroup $P_{F}$ (namely, $D_{1}=Q'$).
We review how these groups are described in the beginning of Section \ref{subsec:wild}.
Based on the structure of $\mathcal{Q}$, then $\phi_{0}$ is described in terms of representation theory of finite Heisenberg groups, especially the Stone--von Neumann theorem (Theorems \ref{thm:SvN} and \ref{thm:IT}).

Thus what we must do by ourselves are about (1) and (2).
Concerning (1), we will determine the lower ramification filtration of $\mathcal{Q}$ in Proposition \ref{prop:ramif}.
On the other hand, in fact the representation $\Ad\circ\phi$ does not have a non-trivial $I_{F}$-invariant part, hence the $L$-factor is trivial (Proposition \ref{prop:L}) and the difference between $\Artin(\Ad\circ\phi)$ and $\Swan(\Ad\circ\phi)$ is simply given by $\dim(\Ad\circ\phi)$.
Therefore, in order to determine $|\varepsilon(0,\Ad\circ\phi,\psi_{0})|=q^{\frac{1}{2}\Artin(\Ad\circ\phi)}$, it suffices to investigate the restriction of $\wedge^{2}\phi_{0}$ to the subgroup $Q'=D_{1}$.
This is done by a simple application of representation theory of finite Heisenberg groups (results in Section \ref{subsec:Heisen} and Proposition \ref{prop:exterior}).

\subsection{The case where $p$ does not divide $2n$}\label{subsec:nmid}
We first consider the case where $p$ does not divide $2n$.
In this case, in \cite{MR3001735} and \cite{MR3402796}, Kaletha attached an $L$-packet consisting of epipelagic representations to each \textit{epipelagic} $L$-parameter (note that simple supercuspidal representations are epipelagic).
Here we recall that, for a quasi-split tamely ramified connected reductive group $\mathbf{G}$ over $F$, an $L$-parameter $\phi$ of $\G$ is called \textit{epipelagic} if it satisfies the following conditions:
\begin{description}
\item[$(1)_{\G}$]
The centralizer $\Cent_{\widehat{\G}}(\phi(P_{F}))$ of the image of the wild inertia subgroup $P_{F}$ of $W_{F}$ in $\widehat{\G}$ is a maximal torus $\mathcal{T}$ belonging to a $\Gal(\ol{F}/F)$-fixed splitting of $\widehat{\G}$.
Note that, by this condition, the image $\phi(I_{F})$ of $I_{F}$ lies in $N_{\widehat{\G}}(\mathcal{T})\rtimes I_{F}$, where $N_{\widehat{\G}}(\mathcal{T})$ is the normalizer of $\mathcal{T}$ in $\widehat{\G}$.
\item[$(2)_{\G}$]
The image of $\phi(I_{F})$ in $\Omega_{\widehat{\G}}(\mathcal{T})\rtimes I_{F}$ is generated by a regular elliptic element $w$ of order $m$ (see, e.g., \cite[Section 5]{MR3164986} for the definition of the regular ellipticity), where $\Omega_{\widehat{\G}}(\mathcal{T})$ is the Weyl group of $\mathcal{T}$ in $\widehat{\G}$.
\item[$(3)_{\G}$]
For any $\sigma\in I_{F}^{\frac{1}{m}+}$, we have $\phi(\sigma)=1\rtimes\sigma\in{}^{L}\G$.
Here $I_{F}^{\bullet}$ is the ramification filtration of $I_{F}$.
\end{description}
Here we note that a Coxeter element of $\Omega_{\widehat{\G}}(\mathcal{T})\rtimes I_{F}$ is a typical example of regular elliptic element (see \cite[Section 7]{MR3000483}).
If the element $w$ in the above condition for an epipelagic $L$-parameter is a Coxeter element, then $\phi$ is a \textit{simple wild} $L$-parameter in the sense of Kaletha (see \cite[Section 2]{MR3001735}).
We also note that, in \cite{MR3001735}, only the case where $\G$ is split semisimple simply connected is treated.

Now we put $\phi$ to be the $L$-parameter of a simple supercuspidal representation of $\Sp_{2n}(F)$.
Then, by Corollary \ref{cor:decomp}, it is of the form 
\[
\phi_{0}\oplus\det\circ\phi_{0}
\]
as a representation of $W_{F}$ (note that the $\SL_{2}(\C)$-part of this $L$-parameter is trivial).
Here $\phi_{0}$ is the $L$-parameter of a simple supercuspidal representation of $\GL_{2n}(F)$ which is lifted from a ramified even special orthogonal group $\SO_{2n}^{\mu}$ by Theorem \ref{thm:liftSptoGL}.
Hence the image of $\phi$ in $\widehat{\Sp_{2n}}=\SO_{2n+1}(\C)$ is contained in 
\[
\bigl(\rmO_{2n}(\C)\times\rmO_{1}(\C)\bigr)^{\det=1}\subset \SO_{2n+1}(\C).
\]
Here we regard the left-hand side as a subgroup of the right-hand side by the following embedding:
\[
\biggl(
\begin{pmatrix}
A&B\\
C&D
\end{pmatrix},z\biggr)
\mapsto
\begin{pmatrix}
A&0&B\\
0&z&0\\
C&0&D
\end{pmatrix}.
\]
We note that, since we assume that $p$ is odd, $\det\circ\phi_{0}$ is trivial on $P_{F}$ and the image $\phi_{0}(P_{F})$ of $P_{F}$ under $\phi_{0}$ is contained in $\SO_{2n}(\C)$.
Namely, $\phi(P_{F})$ is contained in
\[
\SO_{2n}(\C)\times\{1\} \subset\bigl(\rmO_{2n}(\C)\times\rmO_{1}(\C)\bigr)^{\det=1}.
\]

To show the formal degree conjecture, we will prove that $\phi$ is a simple wild $L$-parameter of $\Sp_{2n}$.
In \cite[Section 7]{MR3402796}, Kaletha showed that his construction coincides with the local Langlands correspondence of Harris--Taylor in the case of general linear groups.
By his observation in \cite[Section 7]{MR3402796}, we can conclude that $\phi_{0}$ is an epipelagic $L$-parameter of $\GL_{2n}$ such that the image of $\phi_{0}(I_{F})$ in the Weyl group is generated by a Coxeter element.
Thus the centralizer $\Cent_{\GL_{2n}(\C)}(\phi_{0}(P_{F}))$ of the image of $P_{F}$ under $\phi_{0}$ is a maximal torus of $\GL_{2n}$.
Let us write $\mathcal{T}_{\GL_{2n}}$ for this maximal torus.
Then, since we have
\[
\phi_{0}(P_{F})
\subset\Cent_{\GL_{2n}(\C)}\bigl(\Cent_{\GL_{2n}(\C)}(\phi_{0}(P_{F}))\bigr)
=\Cent_{\GL_{2n}(\C)}(\mathcal{T}_{\GL_{2n}})
=\mathcal{T}_{\GL_{2n}},
\]
the image $\phi_{0}(P_{F})$ is commutative.

\begin{lem}\label{lem:diag}
The centralizer $\Cent_{\SO_{2n}(\C)}(\phi_{0}(P_{F}))$ is a torus of $\SO_{2n}(\C)$.
\end{lem}

\begin{proof}
Since we have
\[
\Cent_{\SO_{2n}(\C)}(\phi_{0}(P_{F}))
=
\Cent_{\GL_{2n}(\C)}(\phi_{0}(P_{F}))\cap\SO_{2n}(\C)
=\mathcal{T}_{\GL_{2n}}\cap\SO_{2n}(\C),
\]
the identity component $\Cent_{\SO_{2n}(\C)}(\phi_{0}(P_{F}))^{0}$ is a torus of $\SO_{2n}(\C)$.
Thus it is enough to show the connectedness of the centralizer $\Cent_{\SO_{2n}(\C)}(\phi_{0}(P_{F}))$.

The connectedness of the centralizer group follows from the same argument as in the proof of \cite[Lemma 5.2.2]{MR4013740}.
For the sake of completeness, we recall his argument.
As $\phi_{0}$ is smooth, the image $\phi_{0}(P_{F})$ of $P_{F}$ under $\phi_{0}$ is finite.
We put $\phi_{0}(P_{F})=\{x_{1},\ldots,x_{r}\}$.
Here, as $P_{F}$ is a pro-$p$ group, $\phi_{0}(P_{F})$ is a finite $p$-group.
Hence the order of every $x_{i}$ is a power of $p$.
In particular, the order of every $x_{i}$ does not divide the order of the fundamental group of $\SO_{2n}(\C)$, which is given by $2$.
Then, by combining \cite[Proposition A.7]{MR2431235} and \cite[Corollary II.4.6]{MR0268192}, we can conclude that $\Cent_{\SO_{2n}(\C)}(x_{1})$ is connected and a Levi subgroup of $\SO_{2n}(\C)$.
Since $x_{1}$ and $x_{2}$ commute, $x_{2}$ is contained in $\Cent_{\SO_{2n}(\C)}(x_{1})$.
Thus we have
\[
\Cent_{\SO_{2n}(\C)}(\{x_{1},x_{2}\})
=
\Cent_{\Cent_{\SO_{2n}(\C)}(x_{1})}(x_{2}).
\]
As $\Cent_{\SO_{2n}(\C)}(x_{1})$ is a Levi subgroup of $\SO_{2n}(\C)$, the order of the fundamental group of its derived subgroup is given by $1$ or $2$.
Hence we can again use the results \cite[Proposition A.7]{MR2431235} and \cite[Corollary II.4.6]{MR0268192}, and know that $\Cent_{\SO_{2n}(\C)}(\{x_{1},x_{2}\})$ is connected and a Levi subgroup of $\SO_{2n}(\C)$.
By repeating this procedure, we finally conclude that $\Cent_{\SO_{2n}(\C)}(\phi_{0}(P_{F}))$ is a connected Levi subgroup of $\SO_{2n}(\C)$.
\end{proof}

Let $\mathcal{T}_{\SO_{2n}}$ be a maximal torus of $\SO_{2n}(\C)$ containing the torus $\Cent_{\SO_{2n}(\C)}(\phi_{0}(P_{F}))$.
Here, by taking a conjugation via $\SO_{2n}(\C)$, we may assume that $\mathcal{T}_{\SO_{2n}}$ is the diagonal maximal torus of $\SO_{2n}(\C)$.
Then we note that $\mathcal{T}_{\GL_{2n}}$ is a diagonal maximal torus of $\GL_{2n}(\C)$.
Indeed, since $\phi_{0}(P_{F})$ is contained in $\mathcal{T}_{\SO_{2n}}$, we get
\[
\Cent_{\GL_{2n}(\C)}(\phi_{0}(P_{F}))
\supset
\Cent_{\GL_{2n}(\C)}(\mathcal{T}_{\SO_{2n}}).
\]
As the left-hand side equals $\mathcal{T}_{\GL_{2n}}$ and the right-hand side is given by the diagonal maximal torus of $\GL_{2n}(\C)$, they coincide.

\begin{lem}\label{lem:epip-cent}
We have $\Cent_{\GL_{2n+1}(\C)}(\phi(P_{F}))=\mathcal{T}_{\GL_{2n+1}}$.
Here $\mathcal{T}_{\GL_{2n+1}}$ is the diagonal maximal torus of $\GL_{2n+1}(\C)$.
\end{lem}

\begin{proof}
Since we have $\Cent_{\GL_{2n}(\C)}(\phi_{0}(P_{F}))=\mathcal{T}_{\GL_{2n}}$, we know that every element of $\GL_{2n}(\C)$ which is outside the diagonal maximal torus does not commute with some $g\in \phi_{0}(P_{F})$.
In particular, the following holds:
\[
\text{for every $1\leq i<j \leq2n$, some $\diag(t_{1},\ldots,t_{2n})\in \phi_{0}(P_{F})$ satisfies $t_{i}\neq t_{j}$} \tag{$\ast$}.
\]

Now we show the assertion.
The inclusion $\Cent_{\GL_{2n+1}(\C)}(\phi(P_{F}))\supset\mathcal{T}_{\GL_{2n+1}}$ is trivial.
Therefore, in order to show the claim, we have to show that every element of $\GL_{2n+1}(\C)$ outside the diagonal maximal torus does not commute with some $g\in \phi_{0}(P_{F})$.
Namely, it suffices to show that, for every $1\leq i<j \leq2n+1$, there exists 
\[
\diag(s_{1},\ldots,s_{2n+1})=\diag(t_{1},\ldots,t_{n},1,t_{n+1},\ldots,t_{2n})\in \phi(P_{F}) \subset \GL_{2n+1}(\C)
\]
such that $s_{i}\neq s_{j}$.
Then, by the condition $(\ast)$, it suffices to show that, for every $1\leq i\leq 2n$, there exists $\diag(t_{1},\ldots,t_{2n})\in \phi_{0}(P_{F})$ such that $t_{i}\neq1$.
Here we note that $\phi_{0}(P_{F})$ is a subgroup of the diagonal maximal torus of $\SO_{2n}(\C)$, hence every element $\diag(t_{1},\ldots,t_{2n})$ of $\phi_{0}(P_{F})$ satisfies $t_{i}=t_{2n+1-i}^{-1}$.
By applying $(\ast)$ to $j=2n+1-i$, there exists $\diag(t_{1},\ldots,t_{2n})\in \phi_{0}(P_{F})$ satisfying $t_{i}^{2}\neq1$.
In particular, we have $t_{i}\neq1$.
\end{proof}

Recall that the Weyl group $\Omega_{\SO_{2n+1}(\C)}(\mathcal{T}_{\SO_{2n+1}})$ of $\mathcal{T}_{\SO_{2n+1}}$ in $\SO_{2n+1}(\C)$ can be identified with the set of fixed points of the Weyl group $\Omega_{\GL_{2n+1}(\C)}(\mathcal{T}_{\GL_{2n+1}})$ of $\mathcal{T}_{\GL_{2n+1}}$ in $\GL_{2n+1}(\C)$ under the isomorphism $\hat{\theta}$ (see \cite[Section 1.1]{MR1687096} for the details).
Here $\hat{\theta}$ is the isomorphism of $\Omega_{\GL_{2n+1}(\C)}(\mathcal{T}_{\GL_{2n+1}})$ induced from the involution of $\GL_{2n+1}(\C)$ defined by $g\mapsto J_{2n+1}{}^{t}\!g^{-1}J_{2n+1}^{-1}$.
If we identify $\Omega_{\GL_{2n+1}(\C)}(\mathcal{T}_{\GL_{2n+1}})$ with the symmetric group $\mathfrak{S}_{2n+1}$ in a usual way, then $\hat{\theta}$ acts on $\mathfrak{S}_{2n+1}$ as follows:
\[
\hat{\theta}(w)(i)=2n+2-w(2n+2-i)\quad\text{for every $1\leq i\leq 2n+1$}.
\]
In the following, we simply write $\Omega_{\SO_{2n+1}}$ for $\Omega_{\SO_{2n+1}(\C)}(\mathcal{T}_{\SO_{2n+1}})$.

\begin{lem}\label{lem:Cox}
We take an element $w$ of $\Omega_{\SO_{2n+1}}$ and regard it as an element of $\mathfrak{S}_{2n+1}$.
If $w$ is a cyclic permutation of length $2n$, then $w$ is a Coxeter element of $\Omega_{\SO_{2n+1}}$.
\end{lem}

\begin{proof}
Let $\mathfrak{S}_{2n+1}(2n)$ denote the subset of $\mathfrak{S}_{2n+1}$ consisting of the cyclic permutations of length $2n$.
We put $\mathrm{Cox}(\SO_{2n+1})$ to be the set of Coxeter elements of $\Omega_{\SO_{2n+1}}$.
Then $\mathrm{Cox}(\SO_{2n+1})$ is a conjugacy classes of $\Omega_{\SO_{2n+1}}$ (see, for example, \cite[Proposition 3.16]{MR1066460}).
Moreover, we can easily check that if we regard a Coxeter element of $\Omega_{\SO_{2n+1}}$ as an element of $\mathfrak{S}_{2n+1}$, then it is cyclic of length $2n$.
Namely, $\mathrm{Cox}(\SO_{2n+1})$ is contained in $\mathfrak{S}_{2n+1}(2n)^{\hat{\theta}}$.
In order to prove the assertion, it is enough to show that these two sets are in fact equal.
To prove it, we compute the cardinalities of $\mathrm{Cox}(\SO_{2n+1})$ and $\mathfrak{S}_{2n+1}(2n)^{\hat{\theta}}$.

First, by \cite[Corollary of Proposition 30]{MR0318337}, the order of $\mathrm{Cox}(\SO_{2n+1})$ is given by
\[
\frac{|\Omega_{\SO_{2n+1}}|\cdot n}{2N},
\]
where $2N$ is the number of roots of $\SO_{2n+1}$.
Thus it equals
\[
\frac{n!\cdot 2^{n}\cdot n}{2n^{2}}=(n-1)!\cdot 2^{n-1}.
\]

We next compute the order of $\mathfrak{S}_{2n+1}(2n)^{\hat{\theta}}$.
Let $w$ be an element of $\mathfrak{S}_{2n+1}(2n)$.
We consider the condition that $w$ is fixed by $\hat{\theta}$.
By the description of the action $\hat{\theta}$ above the assertion of this lemma, $\hat{\theta}(w)$ equals $w$ if and only if we have
\[
w(i)+w(2n+2-i)=2n+2 \quad \text{for every $1\leq i\leq 2n+1$}.
\]
In this case, $w$ stabilizes $n+1$ and induces a permutation $\sigma$ of indices of sets
\[
A_{1}:=\{1,2n+1\},\quad A_{2}:=\{2,2n\},\quad \ldots,\quad A_{n}:=\{n,n+2\}.
\]
Since $w$ is cyclic of length $2n$, $\sigma$ is cyclic of length $n$.
Moreover, for each $1\leq i\leq n$, $w$ induces a bijection $w_{i}$ from $A_{i}$ to $A_{\sigma(i)}$.
Note that, as $w$ is cyclic of length $2n$, the map $w_{\sigma^{n-1}(1)}\colon A_{\sigma^{n-1}(1)}\rightarrow A_{1}$ has to satisfy 
\[
(w_{\sigma^{n-1}(1)}\circ\cdots\circ w_{\sigma(1)} \circ w_{1}) (1)=2n+1.
\tag{$\ast$}
\]
Conversely, if we take a cyclic permutation $\sigma\in\mathfrak{S}_{n}$ of length $n$ and a bijection $w_{i}$ from $A_{i}$ to $A_{\sigma(i)}$ for each $1\leq i \leq n$ satisfying the condition $(\ast)$, then we get an element $w$ of $\mathfrak{S}_{2n+1}(2n)$ defined by
\[
w(i):=
\begin{cases}
w_{i}(i) & \text{for $1\leq i \leq n$},\\
n+1 & \text{for $i=n+1$},\\
w_{2n+2-i}(i) & \text{for $n+2\leq i \leq 2n+1$}.
\end{cases}
\]
Moreover, since we have
\begin{align*}
w(i)+w(2n+2-i)
&=
\begin{cases}
w_{i}(i)+w_{i}(2n+2-i) & \text{if $1\leq i\leq n$}\\
(n+1)+(n+1) & \text{if $i=n+1$}\\
w_{2n+2-i}(i)+w_{2n+2-i}(2n+2-i) & \text{if $n+2\leq i\leq 2n+1$}
\end{cases}\\
&=2n+2
\end{align*}
for any $1\leq i\leq 2n+1$, $w$ is invariant under $\hat{\theta}$.
Therefore, in summary, $\mathfrak{S}_{2n+1}(2n)^{\hat{\theta}}$ bijectively corresponds to the choices of $(\sigma, w_{1},\ldots,w_{\sigma^{n-2}(1)})$ (by the condition $(\ast)$, $w_{\sigma^{n-1}(1)}$ is automatically determined uniquely by $w_{1},\ldots,w_{\sigma^{n-2}(1)}$).
The number of ways of choosing such $\sigma\in\mathfrak{S}_{n}$ and $w_{1},\ldots,w_{\sigma^{n-2}(1)}$ is given by $(n-1)!\cdot 2^{n-1}$.
This completes the proof.
\end{proof}

\begin{prop}\label{prop:simple-wild}
The $L$-parameter $\phi$ is a simple wild $L$-parameter of $\Sp_{2n}$.
\end{prop}

\begin{proof}
By Lemma \ref{lem:epip-cent}, we have
\[
\Cent_{\SO_{2n+1}(\C)}(\phi(P_{F}))
=\Cent_{\GL_{2n+1}(\C)}(\phi(P_{F}))\cap\SO_{2n+1}(\C)
=\mathcal{T}_{\GL_{2n+1}}\cap\SO_{2n+1}(\C).
\]
Thus this is equal to the diagonal maximal torus $\mathcal{T}_{\SO_{2n+1}}$ of $\SO_{2n+1}(\C)$.
In particular, the condition $(1)_{\Sp_{2n}}$ is satisfied for $\phi$.

We consider the condition $(2)_{\Sp_{2n}}$.
As we want to show that $\phi$ is a simple wild $L$-parameter, we have to prove that the image of $\phi(I_{F})$ in $\Omega_{\SO_{2n+1}}$ is generated by a Coxeter element of $\SO_{2n+1}(\C)$.
By \cite[Section 7]{MR3402796}, the image of $\phi_{0}(I_{F})$ in the Weyl group $\Omega_{\GL_{2n}(\C)}(\mathcal{T}_{\GL_{2n}})$ is generated by a Coxeter element.
Here recall that if we regard $\Omega_{\GL_{2n}(\C)}(\mathcal{T}_{\GL_{2n}})$ as $\mathfrak{S}_{2n}$ in a usual way, then Coxeter elements can be characterized as cyclic permutations of length $2n$.
Thus, by our way of regarding $\SO_{2n}(\C)$ as a subgroup of $\GL_{2n+1}(\C)$, we can conclude that the image of $\phi(I_{F})$ in $\Omega_{\SO_{2n+1}}$ is generated by an element $w$ which is cyclic of length $2n$ as an element of $\mathfrak{S}_{2n+1}$, which is the Weyl group $\Omega_{\GL_{2n+1}(\C)}(\mathcal{T}_{\GL_{2n+1}})$.
By Lemma \ref{lem:Cox}, such an element $w$ is a Coxeter element of $\Omega_{\SO_{2n+1}}$.

Finally, we consider the condition $(3)_{\Sp_{2n}}$.
Since $\phi_{0}$ is an epipelagic $L$-parameter of $\GL_{2n}$ and the image of $\phi_{0}(I_{F})$ in the Weyl group $\Omega_{\GL_{2n}(\C)}(\mathcal{T}_{\GL_{2n}})$ is generated by a Coxeter element, whose order is $2n$, $\phi_{0}$ is trivial on $I_{F}^{\frac{1}{2n}+}$.
Thus also $\phi$ is trivial on $I_{F}^{\frac{1}{2n}+}$.
As we saw in the previous paragraph, the image of $\phi(I_{F})$ in $\Omega_{\SO_{2n+1}(\C)}(\mathcal{T}_{\SO_{2n+1}})$ is generated by a Coxeter element, whose order is $2n$.
Thus the condition $(3)_{\Sp_{2n}}$ holds for $\phi$.
\end{proof}

Now we can show the formal degree conjecture for $\phi$ in the case where $p\nmid 2n$:

\begin{proof}[Proof of Theorem \ref{thm:FDC} in the case where $p\nmid 2n$]
By Proposition \ref{prop:simple-wild}, $\phi$ is a simple wild $L$-parameter in the sense of Kaletha.
For such an $L$-parameter, Kaletha attached an $L$-packet (finite set of irreducible admissible representations) $\Pi_{\phi,\mathrm{Kal}}^{\Sp_{2n}}$ consisting of simple supercuspidal representations in \cite{MR3001735}.
On the other hand, in \cite{MR3402796}, he furthermore constructed an $L$-packet corresponding to epipelagic $L$-parameters and checked that the formal degree conjecture holds for them (\cite[Section 5.4]{MR3402796}).
Thus, since a simple supercuspidal $L$-parameter is a special kind of epipelagic parameters, the formal degree conjecture holds for $\phi$ and $\Pi_{\phi,\mathrm{Kal}}^{\Sp_{2n}}$.
As every simple supercuspidal representation of $\Sp_{2n}(F)$ has the same formal degree, it follows that the equality predicted by the formal degree conjecture holds also for $\phi$ and $\Pi_{\phi}^{\Sp_{2n}}$.
\end{proof}

\subsection{Some wildly ramified extensions of $p$-adic fields}\label{subsec:wild}
From now on, we treat the case where $p$ divides $2n$.
We put 
\[
2n=p^{e}\cdot n'
\]
for $e\in\Z_{\geq1}$ and $n'\in \Z_{\geq1}$ satisfying $(n',p)=1$.
Before we consider the formal degree conjecture of the case where $p$ divides $2n$, we introduce a finite tower of ramified extensions of $F$ which will be needed in Imai--Tsushima's description of the $L$-parameters of simple supercuspidal representations of $\GL_{2n}$.
See Section 2.2 in \cite{Imai:2015aa} for the details of the arguments in this subsection.

We fix $\varphi$, $\alpha$, $\beta$, and $\gamma \in \overline{F}$ satisfying
\[
\varphi^{n'}=\varpi,\quad
\alpha^{p^{e}+1}=-\varphi,\quad 
\beta^{p^{2e}}+\beta=-\alpha^{-1},\quad\text{and}\quad
\gamma^{p}-\gamma=\beta^{p^{e}+1},
\]
and set
\[
E:=F(\varphi),\quad
T:=E^{\ur}(\alpha)=F^{\ur}(\varphi,\alpha),\quad 
M:=T(\beta),\quad\text{and}\quad
N:=M(\gamma),
\]
where $F^{\ur}$ and $E^{\ur}$ are the maximal unramified extensions of $F$ and $E$ in $\overline{F}$, respectively.

Then the extension $N/E^{\ur}$ is Galois and its Galois group is given by a finite group $Q$ of order $p^{2e+1}(p^{e}+1)$ defined by
\[
Q:=\bigl\{(a,b,c)\in \ol{k}\times\ol{k}\times\ol{k} \,\big\vert\, a^{p^{e}+1}=1,\, b^{p^{2e}}+b=0,\, c^{p}-c+b^{p^{e}+1}=0 \bigr\}
\]
with the multiplication
\[
(a_{1}, b_{1}, c_{1})\cdot (a_{2}, b_{2}, c_{2})
:=\biggl(a_{1}a_{2}, b_{1}+a_{1}b_{2}, c_{1}+c_{2}+\sum_{i=0}^{e-1}(a_{1}b_{1}^{p^{e}}b_{2})^{p^{i}}\biggr).
\]
To be more precise, for $\sigma\in I_{E}=\Gal(\overline{F}/E^{\ur})$, we put
\[
a_{\sigma}:= \sigma(\alpha)/ \alpha,\quad
b_{\sigma}:= a_{\sigma}\sigma(\beta)-\beta,\quad
c_{\sigma}:= \sigma(\gamma)-\gamma+\sum_{i=0}^{e-1}\bigl(b_{\sigma}^{p^{e}}(\beta+b_{\sigma})\bigr)^{p^{i}}.
\]
Then we can check that the map
\[
\Theta\colon I_{E} \twoheadrightarrow Q;\quad
\sigma\mapsto(\ol{a_{\sigma}}, \ol{b_{\sigma}}, \ol{c_{\sigma}})
\]
is surjective and induces a bijection from $\Gal(N/E^{\ur})$ to $Q$ (see \cite[Lemma 2.3]{Imai:2015aa}).
Here $\ol{a_{\sigma}}$, $\ol{b_{\sigma}}$, and $\ol{c_{\sigma}}$ denote the reductions of $a_{\sigma}$, $b_{\sigma}$, and $c_{\sigma}\in\mcO_{\overline{F}}$, respectively.

\begin{rem}
The well-definedness of the map $\Theta$ can be checked as follows.
Let $\sigma$ be an element of $I_{E}$.
\begin{itemize}
\item[(a)]
Since we have $\sigma(\varphi)=\varphi$, $a_{\sigma}$ is a $(p^{e}+1)$-th root of unity.
In particular, $a_{\sigma}$ belongs to $\mcO_{\overline{F}}$ and we have $\overline{a_{\sigma}}^{p^{e}+1}=1$.
\item[(b)]
We write $\val_{F}$ for the valuation of $\overline{F}$ normalized so that $\val_{F}(\varpi)=1$ ($\varpi$ is the fixed uniformizer of $F$).
We first note that the valuation $\val_{F}(\beta)$ of $\beta$ is given by $-\frac{1}{n'p^{2e}(p^{e}+1)}$.
Let us consider the $p$-th power of $b_{\sigma}=a_{\sigma}\sigma(\beta)-\beta$:
\[
b_{\sigma}^{p}
=
\sum_{i=0}^{p}\binom{p}{i}\bigl(a_{\sigma}\sigma(\beta)\bigr)^{i}\cdot(-\beta)^{p-i}.
\]
By noting that 
\[
\val_{F}\Bigl(\bigl(a_{\sigma}\sigma(\beta)\bigr)^{i}\cdot(-\beta)^{p-i}\Bigr)
=-\frac{1}{n'p^{2e-1}(p^{e}+1)}
>-1
\]
and that $\binom{p}{i}$ is divisible by $p$ when $0<i<p$, we have
\[
b_{\sigma}^{p}
=
\sum_{i=0}^{p}\binom{p}{i}\bigl(a_{\sigma}\sigma(\beta)\bigr)^{i}\cdot(-\beta)^{p-i}
\equiv
\bigl(a_{\sigma}\sigma(\beta)\bigr)^{p}+(-\beta)^{p}
\pmod{\mfp_{\overline{F}}}.
\]
By applying the same argument to $((a_{\sigma}\sigma(\beta))^{p}, (-\beta)^{p})$ instead of $(a_{\sigma}\sigma(\beta), -\beta)$, we get
\[
b_{\sigma}^{p^{2}}
\equiv
\bigl(a_{\sigma}\sigma(\beta)\bigr)^{p^{2}}+(-\beta)^{p^{2}}
\pmod{\mfp_{\overline{F}}}.
\]
Repeating this procedure, we eventually get
\[
b_{\sigma}^{p^{2e}}
\equiv
\bigl(a_{\sigma}\sigma(\beta)\bigr)^{p^{2e}}+(-\beta)^{p^{2e}}
\pmod{\mfp_{\overline{F}}}.
\]
Hence, by noting that $a_{\sigma}^{p^{2e}}=a_{\sigma}$, we get
\begin{align*}
b_{\sigma}^{p^{2e}}+b_{\sigma}
&\equiv \bigl(a_{\sigma}\sigma(\beta)\bigr)^{p^{2e}}+(-\beta)^{p^{2e}} + a_{\sigma}\sigma(\beta)-\beta \pmod{\mfp_{\overline{F}}}\\
&=a_{\sigma} \sigma(\beta^{p^{2e}}+\beta)-(\beta^{p^{2e}}+\beta)\\
&=a_{\sigma} \sigma(-\alpha^{-1})+\alpha^{-1}\\
&=0.
\end{align*}
This implies that $b_{\sigma}$ is integral over $\mcO_{\overline{F}}$, hence belongs to $\mcO_{\overline{F}}$, and that $\ol{b_{\sigma}}^{p^{2e}}+\ol{b_{\sigma}}=0$.
\item[(c)]
We have $\val_{F}(\gamma)=-\frac{1}{n'p^{2e+1}}$.
Thus, by the same argument as above, we can justify the equality
\[
c_{\sigma}^{p}
\equiv
\sigma(\gamma)^{p}-\gamma^{p}+\sum_{i=0}^{e-1}\bigl(b_{\sigma}^{p^{e}}(\beta+b_{\sigma})\bigr)^{p^{i+1}} \pmod{\mfp_{\overline{F}}}.
\]
Hence, modulo $\mfp_{\ol{F}}$, we have 
\begin{align*}
&c_{\sigma}^{p}-c_{\sigma}+b_{\sigma}^{p^{e}+1}\\
&\equiv
\sigma(\gamma)^{p}-\gamma^{p}+\sum_{i=0}^{e-1}\bigl(b_{\sigma}^{p^{e}}(\beta+b_{\sigma})\bigr)^{p^{i+1}}\\
&\qquad\qquad-\sigma(\gamma)+\gamma-\sum_{i=0}^{e-1}\bigl(b_{\sigma}^{p^{e}}(\beta+b_{\sigma})\bigr)^{p^{i}} + b_{\sigma}^{p^{e}+1}\\
&=
\sigma(\gamma^{p}-\gamma)-(\gamma^{p}-\gamma)+\bigl(b_{\sigma}^{p^{e}}(\beta+b_{\sigma})\bigr)^{p^{e}}-b_{\sigma}^{p^{e}}(\beta+b_{\sigma})+ b_{\sigma}^{p^{e}+1}\\
&=
\sigma(\beta^{p^{e}+1})-\beta^{p^{e}+1}+b_{\sigma}^{p^{2e}}(\beta+b_{\sigma})^{p^{e}}-b_{\sigma}^{p^{e}}\beta.
\end{align*}
Note that we can show the identity
\[
b_{\sigma}^{p^{e}}\beta
\equiv
\Bigl(\bigl(a_{\sigma}\sigma(\beta)\bigr)^{p^{e}}+(-\beta)^{p^{e}}\Bigr)\beta
\pmod{\mfp_{\overline{F}}}
\]
in the same manner as in the proof of the identity $b_{\sigma}^{p^{e}}
\equiv(a_{\sigma}\sigma(\beta))^{p^{e}}+(-\beta)^{p^{e}}$.
Thus, by recalling that $b_{\sigma}^{p^{2e}}+b_{\sigma}\equiv0$ and $\beta+b_{\sigma}=a_{\sigma}\sigma(\beta)$, we have
\begin{align*}
&\sigma(\beta^{p^{e}+1})-\beta^{p^{e}+1}+b_{\sigma}^{p^{2e}}(\beta+b_{\sigma})^{p^{e}}-b_{\sigma}^{p^{e}}\beta\\
&\equiv\sigma(\beta^{p^{e}+1})-\beta^{p^{e}+1}-b_{\sigma}(a_{\sigma}\sigma(\beta))^{p^{e}}-\Bigl(\bigl(a_{\sigma}\sigma(\beta)\bigr)^{p^{e}}+(-\beta)^{p^{e}}\Bigr)\beta\\
&=\sigma(\beta^{p^{e}+1})-b_{\sigma}a_{\sigma}^{p^{e}}\sigma(\beta)^{p^{e}}-a_{\sigma}^{p^{e}}\sigma(\beta)^{p^{e}}\beta\\
&=\sigma(\beta^{p^{e}+1})-(a_{\sigma}\sigma(\beta)-\beta)a_{\sigma}^{p^{e}}\sigma(\beta)^{p^{e}}-a_{\sigma}^{p^{e}}\sigma(\beta)^{p^{e}}\beta\\
&=0.
\end{align*}
This implies that $c_{\sigma}$ is integral over $\mcO_{\overline{F}}$, hence belongs to $\mcO_{\overline{F}}$, and that $\ol{c_{\sigma}}^{p}-\ol{c_{\sigma}}+\ol{b_{\sigma}}^{p^{e}+1}=0$.
\end{itemize}
\end{rem}

Furthermore, the Galois groups for the intermediate extensions $N/T$ and $N/M$ are described as follows.
Since the restriction of $\sigma\in I_{E}$ to $T=E^{\ur}(\alpha)$ is the identity if and only if $\sigma$ fixes $\alpha$, the Galois group $\Gal(N/T)$ of the extension $N/T$ is realized as a subgroup $Q'$ of $Q$ given by
\[
Q':=\{(a,b,c)\in Q \mid a=1\},
\]
which is of order $p^{2e+1}$.
Similarly, the Galois group $\Gal(N/M)$ of the extension $N/M$ is realized as a subgroup $Q''$ of $Q'$ given by
\[
Q'':=\{(a,b,c)\in Q \mid a=1, b=0\} \cong \F_{p},
\]
which is of order $p$.

\begin{rem}\label{rem:Heisen}
The group $Q'$ is a Heisenberg group with center $Q''$ in the sense of \cite[Section 1]{Imai:2015aa}.
Recall that, according to \cite[Section 1]{Imai:2015aa}, a finite group $G$ with center $Z$ is called a Heisenberg group if 
\begin{enumerate}
\item
the quotient group $G/Z$ is an elementary abelian $p$-group, and
\item
for any $g\in G\smallsetminus Z$, taking the commutator $g'\mapsto[g,g']$ defines a surjective map $G\rightarrow Z$.
\end{enumerate}
By looking at the multiplication law of $Q'$, we immediately see that $Q''$ is contained in the center of $Q'$.
Conversely, for any $g\in Q'\smallsetminus Q''$, the commutator map $g'\mapsto[g,g']$ is non-trivial by \cite[Lemma 2.4]{Imai:2015aa}.
Thus, in particular, such a $g$ does not belong to the center of $Q'$.
Hence $Q''$ is indeed the center of $Q'$.
Then the condition (1) can be easily checked for $Q'/Q''$ and the condition (2) follows again from \cite[Lemma 2.4]{Imai:2015aa}.
\end{rem}

Now we show the following technical lemma:
\begin{lem}\label{lem:Galois}
If $n'$ divides $p-1$, then the extension $N/F^{\ur}$ is Galois.
\end{lem}

\begin{proof}
Since we have $N=F^{\ur}(\beta,\gamma)$, it is enough to check that $N$ contains the all roots of the minimal polynomials of $\beta$ and $\gamma$ over $F^{\ur}$.
We check this only for the case of $\gamma$ (we can show the assertion for $\beta$ in a similar, but simpler, computation).

By taking the $(p^{e}+1)$-th power of the equality $\beta^{p^{2e}}+\beta=-\alpha^{-1}$, we have
\[
\beta^{p^{e}+1}\bigl(\beta^{(p^{e}+1)(p^{e}-1)}+1\bigr)^{p^{e}+1}=-\varphi^{-1}.
\]
By replacing $\beta^{p^{e}+1}$ in this equality with $\gamma^{p}-\gamma$, we get
\[
(\gamma^{p}-\gamma)\bigl((\gamma^{p}-\gamma)^{p^{e}-1}+1\bigr)^{p^{e}+1}=-\varphi^{-1}.
\]
Thus the minimal polynomial of $\gamma$ over $E^{\ur}$ divides
\[
(x^{p}-x)\bigl((x^{p}-x)^{p^{e}-1}+1\bigr)^{p^{e}+1}+\varphi^{-1}\in E^{\ur}[x].
\tag{$\ast$}
\]
By noting that
\begin{itemize}
\item
the degree of this polynomial is given by $p^{2e+1}$,
\item
the degree of the extension $N/E^{\ur}$ is given by $p^{2e+1}(p^{e}+1)$, and
\item
the degree of $\beta$ over $E^{\ur}(\gamma)$ is not greater than $p^{e}+1$,
\end{itemize}
the polynomial $(\ast)$ is a minimal polynomial of $\gamma$ over $E^{\ur}$.
Thus, since $N$ is Galois over $E^{\ur}$, $N$ contains the all roots of the polynomial $(\ast)$.

On the other hand, by taking the $n'$-th power of the identity $(\gamma^{p}-\gamma)((\gamma^{p}-\gamma)^{p^{e}-1}+1)^{p^{e}+1}=-\varphi^{-1}$, we know that the minimal polynomial of $\gamma$ over $F^{\ur}$ divides 
\[
(x^{p}-x)^{n'}\bigl((x^{p}-x)^{p^{e}-1}+1\bigr)^{n'(p^{e}+1)}-\varpi^{-1}\in F^{\ur}[x]
\tag{$\star$}
\]
(note that $n'$ is even).
We show that $N$ contains the all roots of this polynomial.
If we put 
\[
p_{\zeta_{n'}}(x)
:=
(x^{p}-x)\bigl((x^{p}-x)^{p^{e}-1}+1\bigr)^{p^{e}+1}-\zeta_{n'}\varphi^{-1}
\]
for an $n'$-th root $\zeta_{n'}$ of unity, then we have
\[
(x^{p}-x)^{n'}\bigl((x^{p}-x)^{p^{e}-1}+1\bigr)^{n'(p^{e}+1)}-\varpi^{-1}
=
\prod_{\zeta_{n'}} p_{\zeta_{n'}}(x),
\]
where the product runs over the set of $n'$-th roots of unity.
By the assumption that $n'$ divides $p-1$, we have a bijection from the set of roots of $p_{-1}$ to that of $p_{-\zeta_{n'}}$ given by $x_{0}\mapsto x_{0}\zeta_{n'}$.
Since $N$ contains the all roots of $p_{-1}$ by the argument in the previous paragraph, we can conclude that $N$ contains the all roots of the polynomial $(\star)$.
\end{proof}

From now on, we further assume
\[
n'\mid (p-1)
\]
and we put
\[
\mathcal{Q}:=\Gal(N/F^{\ur}).
\]
The relations between the fields and their Galois groups are summarized as follows:
\[
\xymatrix{
\mathcal{Q}=\Gal(N/F^{\ur})\ar@{}[d]|{\bigcup}^-{\quad\text{index } n'}&&F^{\ur}\ar@{}[d]|{\bigcap}&\\
Q=\Gal(N/E^{\ur})\ar@{}[d]|{\bigcup}^-{\quad\text{index } p^{e}+1}&&E^{\ur}:=F^{\ur}(\varphi)\ar@{}[d]|{\bigcap}&\varphi^{n'}=\varpi\\
Q'=\Gal(N/T)\ar@{}[d]|{\bigcup}^-{\quad\text{index } p^{2e}}&&T:=E^{\ur}(\alpha)\ar@{}[d]|{\bigcap}&\alpha^{p^{e}+1}=-\varphi\\
Q''=\Gal(N/M)\ar@{}[d]|{\bigcup}^-{\quad\text{index } p}&&M:=T(\beta)\ar@{}[d]|{\bigcap}&\beta^{p^{2e}}+\beta=-\alpha^{-1}\\
1=\Gal(N/N)&&N:=M(\gamma)&\gamma^{p}-\gamma=\beta^{p^{e}+1}
}
\]
Since the extension $N/F^{\ur}$ is Galois, the conjugate action of the group $I_{F}=\Gal(\ol{F}/F^{\ur})$ on itself induces an action of $I_{F}$ on $\mathcal{Q}=\Gal(N/F^{\ur})$.
As $E^{\ur}/F^{\ur}$ is Galois, the action of $I_{F}$ on $\mathcal{Q}$ preserves the subgroup $Q$.
Furthermore, by noting that $Q'$ is the unique $p$-Sylow subgroup of $Q$ and $Q''$ is the center of $Q'$ (Remark \ref{rem:Heisen}), we know that the action of $I_{F}$ on $\mathcal{Q}$ preserves $Q'$ and $Q''$.
Note that the induced action $I_{F}$ on $Q''$ factors through the quotient $I_{F}/I_{E}$ since $Q''$ is central also in $Q$.

\begin{lem}\label{lem:Fp}
Let $\{\gamma_{0}=\gamma, \ldots,\gamma_{p-1}\}$ be the set of roots of the equation $x^{p}-x=\beta^{p^{e}+1}$.
Then, for every $i\neq j$, $\gamma_{i}-\gamma_{j}$ belongs to the ring $\mcO_{\ol{F}}$ of integers in $\ol{F}$ and its reduction belongs to $\F_{p}^{\times}$.
\end{lem}

\begin{proof}
We take $i\neq j$.
Then we have $\gamma_{i}^{p}-\gamma_{i}=\gamma_{j}^{p}-\gamma_{j}$.
Thus, if we put $\gamma_{ij}:=\gamma_{i}-\gamma_{j}$,
then we get
\[
(\gamma_{ij}+\gamma_{j})^{p}-(\gamma_{ij}+\gamma_{j})=\gamma_{j}^{p}-\gamma_{j}
\]
\[
\Longleftrightarrow\quad
\gamma_{ij}^{p}+\sum_{k=1}^{p-1}\begin{pmatrix}p\\k\end{pmatrix}\gamma_{j}^{k}\gamma_{ij}^{p-k}-\gamma_{ij}=0.
\tag{$\ast$}
\]
Since $\gamma_{i}$ is a root of the equation in the claim, its valuation with respect to $F^{\ur}$ is given by $-\frac{1}{n'p^{2e+1}}$ (note that the valuation of $\beta$ is $-\frac{1}{n'(p^{e}+1)p^{2e}}$).
In particular, the coefficient of $\gamma_{ij}^{p-k}$ in $(\ast)$ has a positive valuation.
Therefore $\gamma_{ij}$ is an integral element of $\ol{F}$.
Moreover, by noting that $\gamma_{ij}\neq0$ and considering the reduction of $(\ast)$, we get an equality $\ol{\gamma_{ij}}^{p-1}-1=0$.
This implies that $\ol{\gamma_{ij}}$ belongs to $\F_{p}^{\times}$.
\end{proof}

By Lemma \ref{lem:Fp}, we may assume that, for every $1\leq i\leq p-1$, the following equality holds:
\[
\gamma_{i}-\gamma_{0}\equiv i \mod \mfp_{\ol{F}},
\]
where $\mfp_{\ol{F}}$ is the maximal ideal of the ring $\mcO_{\ol{F}}$ of integers in $\ol{F}$.
Then we can regard $Q''=\Gal(N/M)$ as $\F_{p}$ via
\[
\F_{p}\cong\Gal(N/M);\quad j \mapsto \sigma_{j},
\]
where $\sigma_{j}$ is an $M$-automorphism of $N$ defined by $\sigma_{j}(\gamma_{i})=\gamma_{i+j}$.

On the other hand, if we fix a primitive $n'$-th root $\zeta_{n'}$ of unity, then we can identify $I_{F}/I_{E}\cong\Gal(E^{\ur}/F^{\ur})$ with $\Z/n'\Z$ via
\[
\Z/n'\Z\cong \Gal(E^{\ur}/F^{\ur});\quad 1\mapsto \sigma,
\]
where $\sigma$ is an $F^{\ur}$-automorphism of $E^{\ur}$ given by $\varphi\mapsto\varphi\zeta_{n'}$.
We note that, by the assumption that $n'$ divides $p-1$, we can regard $\zeta_{n'}$ as an element of $\F_{p}^{\times}$.

\begin{lem}\label{lem:free}
Under the above identifications, the action of $\sigma\in I_{F}/I_{E}$ on $Q''\cong\F_{p}$ is given by the multiplication by $\zeta_{n'}$.
\end{lem}

\begin{proof}
We fix a primitive $n'(p^{e}+1)$-th root $\zeta_{n'(p^{e}+1)}$ of unity so that $\zeta_{n'}=\zeta_{n'(p^{e}+1)}^{p^{e}+1}$.
By using this root of unity, we define a lift of $\sigma\in I_{F}/I_{E}\cong\Gal(E^{\ur}/F^{\ur})$ to $\Gal(N/F^{\ur})$ by $\sigma(\beta)=\beta\zeta_{n'(p^{e}+1)}^{-1}$ and $\sigma(\gamma)=\gamma\zeta_{n'}^{-1}$.
Then we can compute the image of $\gamma_{0}$ under $\sigma\sigma_{j}\sigma^{-1}$ as follows:
\begin{align*}
\sigma\sigma_{j}\sigma^{-1}(\gamma_{0})-\gamma_{0}
&=
\sigma\sigma_{j}(\gamma_{0}\zeta_{n'})-\sigma(\gamma_{0}\zeta_{n'})\\
&=
\sigma\bigl(\sigma_{j}(\gamma_{0}\zeta_{n'})-\gamma_{0}\zeta_{n'}\bigr)\\
&\equiv
\sigma(j\cdot\zeta_{n'})
=j\cdot\zeta_{n'}
\mod \mfp_{\ol{F}}.
\end{align*}
Namely $\sigma$ acts on $Q''\cong\F_{p}$ by the multiplication by $\zeta_{n'}\in\F_{p}^{\times}$.
\end{proof}

Finally, we determine the lower ramification filtration of $\mathcal{Q}$.
Recall that, for a finite Galois extension $L/K$ of complete discrete valuation fields with perfect residue fields, the $i$-th ($i\in\Z_{\geq0}$) lower ramification filtration of the Galois group $\Gal(L/K)$ is defined by
\[
\Gal(L/K)_{i}
:=
\{\sigma\in\Gal(L/K) \mid \text{$\val_{L}(\sigma(x)-x)\geq i+1$ for any $x\in\mcO_{L}$}\}.
\]
See \cite[Chapter IV]{MR554237} for several equivalent definitions and the fundamental properties of the lower ramification filtration.
Note that, since the completion does not change the Galois group, we can consider the lower ramification filtration for our Galois groups such as $\mathcal{Q}$ although the field $F^{\ur}$ is not complete.
The following proposition will be used to compute the Swan conductor of the $L$-parameter of simple supercuspidal representations.

\begin{prop}\label{prop:ramif}
The lower ramification filtration of $\mathcal{Q}$ is given by
\[
\mathcal{Q}_{0}
\supsetneq\mathcal{Q}_{1}
\supsetneq\mathcal{Q}_{2}=\cdots=\mathcal{Q}_{p^{e}+1}
\supsetneq\mathcal{Q}_{p^{e}+2}=\cdots=1, 
\]
and we have
\[
\mathcal{Q}_{0}=\mathcal{Q},\quad
\mathcal{Q}_{1}=Q', \text{ and}\quad
\mathcal{Q}_{2}=\cdots=\mathcal{Q}_{p^{e}+1}=Q''.
\]
\end{prop}

\begin{proof}
To determine the lower ramification filtration of $\mathcal{Q}=\Gal(N/F^{\ur})$, we compute the Herbrand function $\varphi_{N/F^{\ur}}$ of $N/F^{\ur}$.
Recall that, for a finite Galois extension $L/K$ of complete discrete valuation fields with perfect residue fields, its Herbrand function is defined by
\[
\varphi_{L/K}(u):=\int_{0}^{u}\frac{dt}{\bigl(\Gal(L/K)_{0}:\Gal(L/K)_{t}\bigr)},
\]
where we put $\Gal(L/K)_{t}:=\Gal(L/K)_{\lceil t\rceil}$ for a real number $t\in\R_{\geq0}$ (see \cite[Section IV.3]{MR554237} for details).
Since the Herbrand function is associative with respect to the chain of extensions of fields (see \cite[Chapter IV.3, Proposition 15]{MR554237}), we have $\varphi_{N/F^{\ur}}=\varphi_{T/F^{\ur}}\circ\varphi_{M/T}\circ\varphi_{N/M}$.
By noting this, we first compute the individual Herbrand functions $\varphi_{T/F^{\ur}}$, $\varphi_{M/T}$, and $\varphi_{N/M}$.

\textbf{The case of $T/F^{\ur}$:}
As $T/F^{\ur}$ is totally tamely ramified extension, we have
\begin{align*}
\Gal(T/F^{\ur})_{0}&=\Gal(T/F^{\ur}), \text{ and}\\
\Gal(T/F^{\ur})_{1}&=\Gal(T/F^{\ur})_{2}=\cdots=1.
\end{align*}
Thus we have
\[
\varphi_{T/F^{\ur}}(u)=\frac{1}{n'(p^{e}+1)}u.
\]

\textbf{The case of $M/T$:}
By \cite[Section IV.2 Proposition 5]{MR554237}, for any uniformizer $\varpi_{M}$ of $M$, the $i$-th lower ramification filtration $\Gal(M/T)_{i}$ (for $i>0$) is given by 
\[
\{\sigma\in\Gal(M/T)\mid \val_{M}(\sigma(\varpi_{M})-\varpi_{M})\geq i+1\}.
\]
Hence, since $\beta^{-1}$ is an uniformizer of $M$, it is enough to compute the valuation of
\[
\sigma(\beta^{-1})-\beta^{-1}
\]
for every $\sigma\in\Gal(M/T)$.

By the same arguments as in Lemma \ref{lem:Fp} and the paragraph after Lemma \ref{lem:Fp}, we can identify $\Gal(M/T)$ with the abelian subgroup 
\[
\{x\in\ol{\F}_{p}\mid x^{p^{2e}}+x=0\}
\]
of $\ol{\F}_{p}$.
More precisely, for $x\in\ol{\F}_{p}$ satisfying $x^{p^{2e}}+x=0$, the corresponding element $\sigma_{x}\in\Gal(M/T)$ is characterized by the following condition:
\[
\sigma_{x}(\beta)-\beta\equiv x \mod \mfp_{\ol{F}}.
\]
Hence we have
\[
\val_{M}\bigl(\sigma_{x}(\beta^{-1})-\beta^{-1}\bigr)
=\val_{M}\biggl(-\frac{\sigma_{x}(\beta)-\beta}{\sigma_{x}(\beta)\beta}\biggr)
=2
\]
whenever $x\neq0$.
Therefore the lower ramification filtration of $\Gal(M/T)$ is given by
\begin{align*}
\Gal(M/T)_{0}&=\Gal(M/T)_{1}=\Gal(M/T), \text{ and}\\
\Gal(M/T)_{2}&=\cdots=1,
\end{align*}
and we have
\[
\varphi_{M/T}(u)=
\begin{cases}
u& \text{for } 0\leq u\leq1,\\
\frac{1}{p^{2e}}(u-1)+1& \text{for } 1\leq u.
\end{cases}
\]

\textbf{The case of $N/M$:}
First, since $N/M$ is totally ramified of degree $p$, we have $\val_{N}(\beta^{p^{e}+1})=-p(p^{e}+1)$.
By combining this with the equality $\gamma^{p}-\gamma=\beta^{p^{e}+1}$, we get 
\[
\val_{N}(\gamma)=-(p^{e}+1).
\]
Let $\varpi_{N}$ be a uniformizer of $N$ and we put
\[
\gamma=\varpi_{N}^{-(p^{e}+1)}u,
\]
where $u$ is a unit of $\mcO_{N}$.
Similarly to the previous case, our task is to compute
the valuation of
\[
\sigma(\varpi_{N})-\varpi_{N}
\]
for every $\sigma\in\Gal(N/M)$.

By Lemma \ref{lem:Fp}, the Galois group $\Gal(N/M)$ is identified with $\F_{p}$ and, for $i\in\F_{p}$, the corresponding element $\sigma_{i}\in\Gal(N/M)$ is characterized by
\[
\sigma_{i}(\gamma)-\gamma\equiv i \mod \mfp_{\ol{F}}.
\]
Hence we have $\val_{N}(\sigma_{i}(\gamma)-\gamma)=0$ whenever $i\neq0$.

Now we put 
\[
\sigma_{i}(\varpi_{N})-\varpi_{N}=\varpi_{N}^{s+1}v,
\]
where $s\in\Z$ and $v\in\mcO_{\ol{F}}^{\times}$.
Namely, $s$ is the unique integer such that $\sigma_{i}\in\Gal(N/M)_{s}\smallsetminus\Gal(N/M)_{s+1}$, and note that $s>0$ since $N/M$ is totally wildly ramified.
Let us determine $s$.
Since we have
\[
\sigma_{i}(\gamma)-\gamma
=\sigma_{i}\bigl(\varpi_{N}^{-(p^{e}+1)}u\bigr)-\varpi_{N}^{-(p^{e}+1)}u
=\frac{\varpi_{N}^{p^{e}+1}\sigma_{i}(u)-\sigma_{i}(\varpi_{N}^{p^{e}+1})u}{\sigma_{i}(\varpi_{N}^{p^{e}+1})\varpi_{N}^{p^{e}+1}},
\]
the valuation of $\varpi_{N}^{p^{e}+1}\sigma_{i}(u)-\sigma_{i}(\varpi_{N}^{p^{e}+1})u$ is given by $2(p^{e}+1)$ whenever $i\neq0$.

On the other hand, since $\sigma_{i}\in\Gal(N/M)_{s}$, we have
\[
\sigma_{i}(u)=u+\varpi_{N}^{s+1}v'
\]
for some $v'\in\mcO_{\ol{F}}$.
Therefore we have
\begin{align*}
\varpi_{N}^{p^{e}+1}\sigma_{i}(u)-\sigma_{i}(\varpi_{N}^{p^{e}+1})u
&=\varpi_{N}^{p^{e}+1}(u+\varpi_{N}^{s+1}v'\bigr)-\bigl(\varpi_{N}+\varpi_{N}^{s+1}v)^{p^{e}+1}u\\
&=\varpi_{N}^{p^{e}+s+2}v'-\sum_{k=1}^{p^{e}+1}\binom{p^{e}+1}{k}\varpi_{N}^{(p^{e}+1-k)+(s+1)k}v^{k}u.
\end{align*}
Here, since $s>0$, the valuations of the terms in the right-hand side are given by the following:
\begin{align*}
\val_{N}\bigl(\varpi_{N}^{p^{e}+s+2}v'\bigr)&\geq p^{e}+s+2, \text{ and}\\
\val_{N}\Biggl(\binom{p^{e}+1}{k}\varpi_{N}^{(p^{e}+1-k)+(s+1)k}v^{k}u\Biggr)&
\begin{cases}
=p^{e}+s+1& \text{if } k=1,\\
>p^{e}+s+1& \text{if } k>1.
\end{cases}
\end{align*}
Hence the valuation of $\varpi_{N}^{p^{e}+1}\sigma_{i}(u)-\sigma_{i}(\varpi_{N}^{p^{e}+1})u$ is given by $p^{e}+s+1$.

By comparing it with the above computations, for $i\neq0$, we get
\[
2(p^{e}+1)=p^{e}+s+1.
\]
Thus $s$ is equal to $p^{e}+1$.

Therefore the lower ramification filtration of $\Gal(N/M)$ is given by 
\begin{align*}
\Gal(N/M)_{0}&=\cdots=\Gal(N/M)_{p^{e}+1}=\Gal(N/M), \text{ and}\\
\Gal(N/M)_{p^{e}+2}&=\cdots=1,
\end{align*}
and we have
\[
\varphi_{N/M}(u)=
\begin{cases}
u& \text{for } 0\leq u\leq p^{e}+1,\\
\frac{1}{p}(u-p^{e}-1)+p^{e}+1& \text{for } p^{e}+1\leq u.
\end{cases}
\]

Finally, we compose the above three functions.
Then we get
\[
\varphi_{N/F^{\ur}}(u)=
\frac{1}{n'(p^{e}+1)}\times
\begin{cases}
u& \text{for } 0\leq u\leq 1,\\
\frac{1}{p^{2e}}(u-1)+1& \text{for } 1\leq u\leq p^{e}+1,\\
\frac{1}{p^{2e+1}}(u+p^{e+1}-p^{e}-1)+1& \text{for } p^{e}+1\leq u.
\end{cases}
\]
In particular,  the ramification filtration of $\mathcal{Q}$ jumps at three points:
\[
\mathcal{Q}_{0} \supsetneq\mathcal{Q}_{1},\quad
\mathcal{Q}_{1} \supsetneq\mathcal{Q}_{2},\quad\text{and}\quad
\mathcal{Q}_{p^{e}+1} \supsetneq\mathcal{Q}_{p^{e}+2}.
\]

On the other hand, since $N/F^{\ur}$ is totally ramified and $Q'$ is the image of the wild inertia subgroup $P_{F}$ of $I_{F}$, we have $\mathcal{Q}_{0} =\mathcal{Q}$ and $\mathcal{Q}_{1} =Q'$.
Thus our final task is to determine $\mathcal{Q}_{2}=\cdots=\mathcal{Q}_{p^{e}+1}$.
By the above description of the Herbrand function, the order of this subgroup $\mathcal{Q}_{2}$ is equal to $p$.
Moreover, since the lower ramification filtration is compatible with that of any subgroup, we have
\[
\mathcal{Q}_{2}=\Gal(N/F^{\ur})_{2}\supset\Gal(N/M)_{2}=\Gal(N/M)=Q''.
\] 
Therefore we get $\mathcal{Q}_{2}=Q''$.
\end{proof}

\subsection{Some representation theory of finite Heisenberg groups}\label{subsec:Heisen}
The group $Q'$ is a finite Heisenberg group of order $p^{2e+1}$ with center $Q''$ in the sense of \cite[Section 1]{Imai:2015aa} (see Remark \ref{rem:Heisen}).
The following is the fundamental fact in representation theory of such groups (see, for example, \cite[(8.3.3)]{MR701540}):

\begin{thm}[Stone--von Neumann's theorem]\label{thm:SvN}
Let $\psi_{0}$ be a nontrivial character of $Q''\cong\F_{p}$.
Then there exists a unique irreducible representation $\tau_{\psi_{0}}$ of $Q'$ whose central character is given by $\psi_{0}$.
In particular, irreducible representations of $Q'$ are exhausted by $\{\tau_{\psi_{0}}\mid\psi_{0}\}$ and the characters of $Q'/Q''$.
Moreover, the dimension of $\tau_{\psi_{0}}$ is given by $p^{e}$.
\end{thm}

By using this fact, we show some lemmas which will be used to compute the exterior product of the $L$-parameter of a simple supercuspidal representation of $\GL_{2n}(F)$.

\begin{lem}\label{lem:i=-j}
For every nontrivial character $\psi_{0}$ of $Q''$, we have
\[
\tau_{\psi_{0}}\otimes\tau_{\psi_{0}^{-1}}
\cong
\bigoplus_{\eta\in(Q'/Q'')^{\vee}} \eta.
\]
\end{lem}

\begin{proof}
The central character of $\tau_{\psi_{0}}\otimes\tau_{\psi_{0}^{-1}}$ is trivial.
On the other hand, the dimension of $\tau_{\psi_{0}}\otimes\tau_{\psi_{0}^{-1}}$ is equal to $p^{2e}$.
Thus, since the quotient group $Q'/Q''$ is abelian of order $p^{2e}$, it suffices to show that every character $\eta$ of $Q'/Q''$ is contained in the left-hand side.

As the central character of $\tau_{\psi_{0}}^{\vee}$ is given by $\psi_{0}^{-1}$, by Stone--von Neumann's theorem (Theorem \ref{thm:SvN}), we have $\tau_{\psi_{0}}^{\vee}\cong\tau_{\psi_{0}^{-1}}$.
Thus the canonical $Q'$-invariant pairing 
\[
(-,-)\colon \tau_{\psi_{0}}\times\tau_{\psi_{0}}^{\vee}\rightarrow\C
\]
induces a non-zero $Q'$-invariant pairing on $\tau_{\psi_{0}}\times\tau_{\psi_{0}^{-1}}$.
This implies that $\tau_{\psi_{0}}\otimes\tau_{\psi_{0}^{-1}}$ contains the trivial representation of $Q'$.

On the other hand, for every character $\eta$ of $Q'/Q''$, the $\eta$-twist does not change the central character of $\tau_{\psi_{0}}$.
Thus, again by Stone--von Neumann's theorem (Theorem \ref{thm:SvN}), we have $\tau_{\psi_{0}}\otimes\eta \cong \tau_{\psi_{0}}$.
Therefore, for every character $\eta$ of $Q'/Q''$, we have
\[
\eta=\eta\otimes\mathbbm{1}
\subset
\eta\otimes\tau_{\psi_{0}}\otimes\tau_{\psi_{0}^{-1}}
\cong
\tau_{\psi_{0}}\otimes\tau_{\psi_{0}^{-1}}.
\]
This completes the proof.
\end{proof}

\begin{lem}\label{lem:tensor-ext}
Let $\psi_{i}$ and $\psi_{j}$ be nontrivial additive characters of $\F_{p}$ satisfying $\psi_{i}\psi_{j}\neq\mathbbm{1}$.
Then we have
\[
\wedge^{2}\tau_{\psi_{i}}
\cong
\tau_{\psi_{i}^{2}}^{\oplus \frac{p^{e}-1}{2}}
\quad
\text{and}
\quad
\tau_{\psi_{i}}\otimes\tau_{\psi_{j}}
\cong
\tau_{\psi_{i}\psi_{j}}^{\oplus p^{e}}.
\]
\end{lem}

\begin{proof}
Since the dimension of the representation $\tau_{\psi_{i}}$ is equal to $p^{e}$ (Theorem \ref{thm:SvN}), that of $\wedge^{2}\tau_{\psi_{i}}$ is given by $\binom{p^{e}}{2}=\frac{p^{e}(p^{e}-1)}{2}$.
On the other hand, the central character of $\wedge^{2}\tau_{\psi_{i}}$ is given by $\psi_{i}^{2}$.
Since we assume $p$ is not equal to $2$, $\psi_{i}^{2}$ is a nontrivial character of $Q''$.
Therefore, by Stone--von Neumann's theorem (Theorem \ref{thm:SvN}), we have
\[
\wedge^{2}\tau_{\psi_{i}}
\cong
\tau_{\psi_{i}^{2}}^{\oplus \frac{p^{e}-1}{2}}.
\]

By noting that the dimension of $\tau_{\psi_{i}}\otimes\tau_{\psi_{j}}$ equals $p^{2e}$ and that the central character of $\tau_{\psi_{i}}\otimes\tau_{\psi_{j}}$ is given by $\psi_{i}\psi_{j}$, we can show the latter assertion by the same argument.
\end{proof}

\subsection{The second case where $p$ divides $2n$}\label{subsec:mid}
Now we consider the formal degree conjecture for simple supercuspidal representations in the case where $p$ divides $2n$.
Recall that we put $2n=p^{e}\cdot n'$ for $e\in\Z_{\geq1}$ and $n'\in \Z_{\geq1}$ satisfying $(n',p)=1$.

To compute the adjoint $\gamma$-factor of the $L$-parameters of simple supercuspidal representations explicitly, we utilize a result of Imai--Tsushima in \cite{Imai:2015aa}.

For $(\omega,a,\zeta)\in\SSC(\GL_{2n})$, we consider the simple supercuspidal representation $\pi_{\omega,a,\zeta}^{\GL_{2n}}$ of $\GL_{2n}(F)$ defined in Section \ref{subsec:ssc-GL} 
(note that, in \cite{Imai:2015aa}, this representation is denoted by $\pi_{a^{-1},\omega,-\zeta}$).
Here, by replacing the fixed uniformizer $\varpi$ with $a^{-1}\varpi$, we may assume $a=1$.

We define a representation $\tau_{2n, 1, \omega, -\zeta}$ of $W_{E}$ by 
\[
\tau_{2n, 1, \omega, -\zeta}:=
\tau_{2n,1}\otimes(\omega\circ\lambda_{1})\otimes\phi_{-\zeta}.
\]
Here we do not explain the definitions of the representations in the right-hand side (see \cite[Section 2.2]{Imai:2015aa}).
%
However, the key properties of this representation are the following:
\begin{itemize}
\item
We can define an action of $\Z$ on $Q$ and extend the map $\Theta\colon I_{E}\rightarrow Q$ to $W_{E}\rightarrow Q\rtimes\Z$ (again denoted by $\Theta$).
Then $\tau_{2n,1}$ is the pull-back of a representation $\tau_{2n}$ of $Q\rtimes\Z$ via $\Theta$.
\item
$\tau_{2n}$ is an irreducible representation of $Q\rtimes\Z$ satisfying $\tau_{2n}|_{Q'}\cong\tau_{\psi_{0}}$ for a nontrivial character $\psi_{0}$ of $\F_{p}$.
\item
$\omega\circ\lambda_{1}$ is a character of $W_{E}$ which is trivial on the wild inertia subgroup $P_{E}$.
\item
$\phi_{-\zeta}$ is a character of $W_{E}$ which is trivial on the inertia subgroup $I_{E}$.
\end{itemize}

\begin{thm}[{\cite[Theorem 4.1]{Imai:2015aa}}]\label{thm:IT}
The $L$-parameter corresponding to the simple supercuspidal representation $\pi_{\omega,1,\zeta}^{\GL_{2n}}$ of $\GL_{2n}(F)$ is given by 
\[
\phi_{\omega,1,\zeta}:=\Ind_{W_{E}}^{W_{F}}\tau_{2n, 1, \omega, -\zeta}.
\]
\end{thm}


In the following, we compute the special value (at $s=0$) of the adjoint $\gamma$-factor of the $L$-parameter of a simple supercuspidal representation of $\Sp_{2n}(F)$ by using this description.
From now on, we denote $\phi_{\omega,1,\zeta}$ and $\tau_{2n, 1, \omega, -\zeta}$ simply by $\phi_{0}$ and $\tau_{0}$, respectively.

We first compute $\phi_{0}|_{P_{E}}$.

\begin{cor}\label{cor:distinct}
The representation $\phi_{0}$ of $W_{F}$ is trivial on $\Gal(\ol{F}/N)$.
In particular, we can regard $\phi_{0}|_{P_{E}}$ as a representation of $Q'$.
Moreover, as representations of $Q'$, we have
\[
\phi_{0}|_{P_{E}} \cong
\tau_{\psi_{1}}\oplus\cdots\oplus\tau_{\psi_{n'}},
\]
where $\psi_{1}, \ldots, \psi_{n'}$ are nontrivial additive characters of $\F_{p}$ which are distinct from each other.
\end{cor}

\begin{proof}
By Mackey's theorem, we have
\[
\phi_{0}|_{P_{E}}
=
\bigl(\Ind_{W_{E}}^{W_{F}} \tau_{0} \bigr)\big|_{P_{E}}
\cong
\bigoplus_{s\in W_{E}\backslash W_{F}/P_{E}}\Ind_{P_{E}\cap s^{-1}W_{E}s}^{P_{E}} \tau_{0}^{s}.
\]
Since the extension $E/F$ is totally tamely ramified, we have $P_{F}=P_{E}$, hence $P_{E}$ is normal in $W_{F}$.
Thus we have $W_{E}\backslash W_{F}/P_{E}\cong W_{E}\backslash W_{F}$.
Furthermore, as $W_{E}$ is normal in $W_{F}$ (this follows from the assumption that $n'|(p-1)$ by Lemma \ref{lem:Galois}), the totally-ramifiedness of $E/F$ implies that the natural map $I_{E}\backslash I_{F}\rightarrow W_{E}\backslash W_{F}$ is bijective.
Also, for any $s\in I_{F}$, we have $P_{E}\cap s^{-1}W_{E}s=P_{E}$.
Therefore the above direct sum is rewritten as 
\[
\bigoplus_{s\in  I_{E}\backslash I_{F}} \tau_{0}^{s}.
\]
As the $I_{F}$-conjugation preserves $\Gal(\ol{F}/N)$ and $\tau_{0}$ is trivial on $\Gal(\ol{F}/N)$, we get the first assertion.
Now we determine $\tau_{0}^{s}$ as a representation of $Q'$.
Since $\tau_{0}$ is equivalent to $\tau_{\psi_{0}}$ for some nontrivial character $\psi_{0}$ of $\F_{p}$ as a representation of $Q'$, $\tau_{0}^{s}$ is an irreducible representation of $Q'$ with central character $\psi_{0}^{s}$.
Namely, by Theorem \ref{thm:SvN}, we have $\tau_{0}^{s}\cong\tau_{\psi_{0}^{s}}$.
By combining this with Lemma \ref{lem:free}, we get the second assertion.
\end{proof}

We next compute the exterior product of $\phi_{0}$ as a representation of $Q'$.

\begin{prop}\label{prop:exterior}
As representations of $Q'$, we have
\[
\wedge^{2}\phi_{0}
\cong
\biggl(\bigoplus_{i=1}^{m}\tau_{i}\biggr)
\oplus
\biggl(\bigoplus_{\eta\in(Q'/Q'')^{\vee}} \eta^{\oplus\frac{n'}{2}}\biggr),
\]
where $m:=\frac{n'(2n-p^{e}-1)}{2}$ and each $\tau_{i}$ is an irreducible representation of $Q'$ with nontrivial central character.
\end{prop}

\begin{proof}
By Corollary \ref{cor:distinct}, as representations of $Q'$, we have
\[
\wedge^{2} \phi_{0}
\cong
\biggl(\bigoplus_{i=1}^{n'} \wedge^{2}\tau_{\psi_{i}} \biggr) \oplus
\biggl(\bigoplus_{1\leq i<j\leq n'}\tau_{\psi_{i}}\otimes\tau_{\psi_{j}}\biggr).
\]
By Lemma \ref{lem:tensor-ext}, we have 
\[
\wedge^{2}\tau_{\psi_{i}}
\cong
\tau_{\psi_{i}^{2}}^{\oplus \frac{p^{e}-1}{2}}
\]
for every $\psi_{i}$.
Thus our task is to compute $\tau_{\psi_{i}}\otimes\tau_{\psi_{j}}$.
Since we have $\{\psi_{1}, \ldots, \psi_{n'}\}=\{\psi_{0}^{s^{-1}} \mid s \in I_{F}/I_{E}\}$ and the finite group $I_{F}/I_{E}$ acts on $Q'$ as in Lemma \ref{lem:free}, for each $i$, there exists exactly one $j\neq i$ satisfying $\psi_{i}\psi_{j}=\mathbbm{1}$.
Thus we have 
\begin{align*}
|\{(i,j) \mid 1\leq i<j\leq n',\, \psi_{i}\psi_{j}=\mathbbm{1}\}| &= \frac{n'}{2}, \text{ and}\\
|\{(i,j) \mid 1\leq i<j\leq n',\, \psi_{i}\psi_{j}\neq\mathbbm{1}\}| &= \binom{n'}{2}-\frac{n'}{2} = \frac{n'(n'-2)}{2}.
\end{align*}
Hence, by Lemmas \ref{lem:i=-j} and \ref{lem:tensor-ext}, $\bigoplus_{1\leq i<j\leq n'}\tau_{\psi_{i}}\otimes\tau_{\psi_{j}}$ is the direct sum of 
\begin{itemize}
\item $\frac{p^{e}n'(n'-2)}{2}$ irreducible representations of $Q'$ with nontrivial central characters, and
\item $\eta^{\oplus\frac{n'}{2}}$, for every character $\eta$ of $Q'/Q''$.
\end{itemize}
This completes the proof (note that we have $2n=p^{e}n'$).
\end{proof}

\begin{prop}\label{prop:Swan}
\begin{enumerate}
\item
For an irreducible representation $\tau_{\psi_{0}}$ of $Q'$ with nontrivial central character $\psi_{0}$, we have
\[
\sum_{j\geq1} \dim\bigl(\tau_{\psi_{0}}/\tau_{\psi_{0}}^{\mathcal{Q}_{j}}\bigr) \frac{|\mathcal{Q}_{j}|}{|\mathcal{Q}_{0}|}
=
\frac{1}{n'}.
\]
\item
For a character $\eta$ of $Q'/Q''$, we have
\[
\sum_{j\geq1} \dim\bigl(\eta/\eta^{\mathcal{Q}_{j}}\bigr) \frac{|\mathcal{Q}_{j}|}{|\mathcal{Q}_{0}|}
=
\begin{cases}
\frac{1}{n'(p^{e}+1)}& \text{if } \eta\neq\mathbbm{1},\\
0&\text{if } \eta=\mathbbm{1}.
\end{cases}
\]
\end{enumerate}
\end{prop}

\begin{proof}
Recall that, by Proposition \ref{prop:ramif}, the ramification filtration of $\mathcal{Q}$ is given by
\[
\mathcal{Q}_{0}
\supsetneq\mathcal{Q}_{1}
\supsetneq\mathcal{Q}_{2}=\cdots=\mathcal{Q}_{p^{e}+1}
\supsetneq\mathcal{Q}_{p^{e}+2}=\cdots=1, 
\]
and we have
\[
\mathcal{Q}_{0}=\mathcal{Q},\quad
\mathcal{Q}_{1}=Q', \text{ and}\quad
\mathcal{Q}_{2}=\cdots=\mathcal{Q}_{p^{e}+1}=Q''.
\]
Here the orders of these groups are given by
\[
|\mathcal{Q}|=n'p^{2e+1}(p^{e}+1),\quad
|Q'|=p^{2e+1},\text{ and}\quad
|Q''|=p.
\]

We show (1).
Let us take an irreducible representation $\tau_{\psi_{0}}$ of $Q'$ with nontrivial central character $\psi_{0}$.
Then, since we have
\[
\tau_{\psi_{0}}^{\mathcal{Q}_{j}}
=
\begin{cases}
0& \text{for }1\leq j \leq p^{e}+1,\\
\tau_{\psi_{0}}& \text{for } p^{e}+2\leq j, 
\end{cases}
\] 
we get
\begin{align*}
 &\sum_{j\geq1} \dim\bigl(\tau_{\psi_{0}}/\tau_{\psi_{0}}^{\mathcal{Q}_{j}}\bigr) \frac{|\mathcal{Q}_{j}|}{|\mathcal{Q}_{0}|}\\
&=\underbrace{p^{e}\frac{p^{2e+1}}{n'p^{2e+1}(p^{e}+1)}}_{j=1}+\underbrace{p^{e}\frac{p}{n'p^{2e+1}(p^{e}+1)}+\cdots+p^{e}\frac{p}{n'p^{2e+1}(p^{e}+1)}}_{j=2,\ldots, p^{e}+1}\\
&=\frac{p^{e}}{n'(p^{e}+1)}+\frac{1}{n'(p^{e}+1)}\\
&=\frac{1}{n'}.
\end{align*}

We next show (2).
We take a character $\eta$ of $Q'/Q''$.
If $\eta=\mathbbm{1}$, then we immediately get
\[
\sum_{j\geq1} \dim\bigl(\mathbbm{1}/\mathbbm{1}^{\mathcal{Q}_{j}}\bigr) \frac{|\mathcal{Q}_{j}|}{|\mathcal{Q}_{0}|}
=
0.
\]
Hence it is enough to consider the case where $\eta\neq\mathbbm{1}$.
In this case we have
\[
\eta^{\mathcal{Q}_{j}}
=
\begin{cases}
0& \text{for } j=1,\\
\eta& \text{for } 2\leq j. 
\end{cases}
\] 
Thus we get
\[
\sum_{j\geq1} \dim\bigl(\eta/\eta^{\mathcal{Q}_{j}}\bigr) \frac{|\mathcal{Q}_{j}|}{|\mathcal{Q}_{0}|}
=\underbrace{\frac{p^{2e+1}}{n'p^{2e+1}(p^{e}+1)}}_{j=1}
=\frac{1}{n'(p^{e}+1)}.
\]
\end{proof}

Recall that the $L$-parameter $\phi$ of a simple supercuspidal $L$-packet of $\Sp_{2n}$ is of the form $\phi_{0}\oplus\det\circ\phi_{0}$ as a $(2n+1)$-representation of $W_{F}$, where $\phi_{0}$ is the $L$-parameter of a simple supercuspidal representation of $\GL_{2n}(F)$ by Corollary \ref{cor:decomp} and Theorem \ref{thm:liftSptoGL}.

We compute the $\gamma$-factor of $\Ad\circ\phi$ for $\Ad\colon{}^{L}\!\Sp_{2n}\rightarrow\GL(\mathfrak{so}_{2n+1}(\C))$.
%

\begin{prop}\label{prop:Swan2}
We have
\[
\Swan\bigl(\Ad\circ\phi\bigr)
=
n.
\]
\end{prop}

\begin{proof}
Since we have
\[
\Ad\circ\phi
\cong
\wedge^{2}\phi
\cong
(\wedge^{2}\phi_{0})\oplus(\phi_{0}\otimes\det\circ\phi_{0}),
\]
we get
\[
\Swan(\Ad\circ\phi)
=
\Swan(\wedge^{2}\phi_{0})+\Swan(\phi_{0}\otimes\det\circ\phi_{0}).
\]

By Propositions \ref{prop:exterior} and \ref{prop:Swan}, the first term is computed as follows:
\begin{align*}
\Swan(\wedge^{2}\phi_{0})
&=\sum_{i=1}^{m}\sum_{j\geq1} \dim\bigl(\tau_{i}/\tau_{i}^{\mathcal{Q}_{j}}\bigr) \frac{|\mathcal{Q}_{j}|}{|\mathcal{Q}_{0}|}+\frac{n'}{2}\sum_{\eta\in(Q'/Q'')^{\vee}}\sum_{j\geq1} \dim\bigl(\eta/\eta^{\mathcal{Q}_{j}}\bigr) \frac{|\mathcal{Q}_{j}|}{|\mathcal{Q}_{0}|}\\
&=\frac{m}{n'}+\frac{n'}{2}(p^{2e}-1)\frac{1}{n'(p^{e}+1)}\\
&=\frac{2n-p^{e}-1}{2}+\frac{p^{e}-1}{2}\\
&=n-1.
\end{align*}
On the other hand, since $\det\circ\phi_{0}$ is quadratic, $\phi_{0}\otimes\det\circ\phi_{0}$ is the $L$-parameter of a simple supercuspidal representation of $\GL_{2n}(F)$ (recall the argument in the proof of Lemma \ref{lem:Lpartwist}).
Thus we can apply Corollary \ref{cor:distinct} to $\phi_{0}\otimes\det\circ\phi_{0}$.
Hence, again by Proposition \ref{prop:Swan}, we get
\[
\Swan(\phi_{0}\otimes\det\circ\phi_{0})
=\sum_{i=1}^{n'}\sum_{j\geq1} \dim\bigl(\tau_{\psi_{i}}/\tau_{\psi_{i}}^{\mathcal{Q}_{j}}\bigr) \frac{|\mathcal{Q}_{j}|}{|\mathcal{Q}_{0}|}
=\frac{n'}{n'}=1.
\]
Thus we get the assertion.
\end{proof}

\begin{prop}\label{prop:L}
The adjoint $L$-factor $L(s, \Ad\circ\phi)$ is trivial.
\end{prop}

\begin{proof}
By the definition of the $L$-factor, it suffices to show that $\Ad\circ\phi$ does not contain the trivial representation of $I_{F}$.
Since we have $\Ad\circ\phi\cong(\wedge^{2}\phi_{0})\oplus(\phi_{0}\otimes\det\circ\phi_{0})$ and $(\phi_{0}\otimes\det\circ\phi_{0})^{I_{F}}=0$ (for example, this easily follows from Corollary \ref{cor:distinct}), our task is to show that $(\wedge^{2}\phi_{0})^{I_{F}}=0$.
Namely, it is enough to show that $\phi_{0}$ does not have an $I_{F}$-invariant symplectic form.
Therefore, by noting that $\phi_{0}$ is orthogonal, it suffices to show that $\phi_{0}$ is irreducible as a representation of $I_{F}$.

By Mackey's theorem, we have
\[
\phi_{0}|_{I_{F}}
=
\bigl(\Ind_{W_{E}}^{W_{F}} \tau_{0} \bigr)\big|_{I_{F}}
\cong
\bigoplus_{s\in W_{E}\backslash W_{F}/I_{F}}\Ind_{I_{F}\cap s^{-1}W_{E}s}^{I_{F}} \tau_{0}^{s}
\cong
\Ind_{I_{E}}^{I_{F}} \tau_{0}.
\]
Thus, if $\tau_{0}^{s}\not\cong\tau_{0}$ for every $s\in (I_{E}\backslash I_{F})\smallsetminus\{1\}$, $\phi_{0}|_{I_{F}}$ is irreducible by the irreducibility criterion of induced representations (\cite[Proposition 23, Section 7.4]{MR0450380}).
However, by the proof of Corollary \ref{cor:distinct}, $\tau_{0}^{s}$ is not equivalent to $\tau_{0}$ as representations of $P_{E}$.
Hence $\tau_{0}^{s}$ is not equivalent to $\tau_{0}$ as representations of $I_{E}$.
\end{proof}

\begin{proof}[Proof of Theorem \ref{thm:FDC} in the second case]
Let $\pi$ be a simple supercuspidal representation of $\Sp_{2n}(F)$ with $L$-parameter $\phi$.
By the equality (72) of \cite{MR2730575} (see \cite[Section 7.1]{MR2730575} for the definition of the measure used in \cite[(72)]{MR2730575}), we have
\[
\frac{\deg(\pi)}{\deg(\mathrm{St}_{\G})}
=
\frac{q^{N+l}}{|Z_{\G}|\cdot|\gamma(0,\Ad\circ\phi_{\G},\psi_{0})|}.
\]
Here $|Z_{\G}|$ is the order of the center of $G=\Sp_{2n}(F)$, $N$ is the number of positive roots of $\G$, and $l$ is the rank of $\G$.
Thus, in our situation, $|Z_{\G}|$, $N$, and $l$ are given by $2$, $n^{2}$, and $n$, respectively.
Hence the right-hand side is equal to
\[
\frac{q^{n^{2}+n}}{2\cdot|\gamma(0,\Ad\circ\phi_{\G},\psi_{0})|}.
\]

Our task is to show that this equals
\[
\frac{|\mathcal{S}^{\G}_{\phi_{\G}}|}{|\mathcal{S}^{\G}_{\phi}|}\cdot
\frac{|\gamma(0, \Ad\circ\phi, \psi_{0})|}{|\gamma(0, \Ad\circ\phi_{\G}, \psi_{0})|}.
\]
Since we have $|\mathcal{S}^{\G}_{\phi_{\G}}|=1$ and $|\mathcal{S}^{\G}_{\phi}|=2$, it suffices to show that
\[
|\gamma(0, \Ad\circ\phi, \psi_{0})|=q^{n^{2}+n}.
\]
By Proposition \ref{prop:L}, the adjoint $L$-factor of $\phi$ is trivial.
On the other hand, we have
\[
|\varepsilon(0, \Ad\circ\phi, \psi_{0})|
=
q^{\frac{1}{2}\mathrm{Artin}(\Ad\circ\phi)}
\]
(see the equality (10) and Proposition 2.3 in \cite{MR2730575}).
Hence it suffices to show
\[
\mathrm{Artin}(\Ad\circ\phi)=2n^{2}+2n.
\]

By the proof of Proposition \ref{prop:L}, the space of $\Ad\circ\phi(I_{F})$-fixed vectors in $\widehat{\mathfrak{g}}=\mathfrak{so}_{2n+1}(\C)$ is zero.
Thus, by Proposition \ref{prop:Swan2}, we get
\begin{align*}
\Artin(\Ad\circ\phi)
&=
\Swan(\Ad\circ\phi)+\dim_{\C}(\widehat{\mathfrak{g}}/\widehat{\mathfrak{g}}^{D_{0}})\\
&=
n+\dim_{\C}\mathfrak{so}_{2n+1}(\C)\\
&=
n+n(2n+1)
=2n^{2}+2n.
\end{align*}
\end{proof}

\appendix
\section{Kloosterman sums and Gauss sums}\label{sec:Kl}
In this section, we summarize some properties of Kloosterman sums and Gauss sums, which are exponential sums defined by using characters of finite fields.
We first define the Kloosterman sum.
Let $k$ be a finite field, $\tilde{k}$ its quadratic extension, and $\psi$ a nontrivial additive character on $k$.

\begin{defn}\label{defn:Kl}
We take the following data:
\begin{itemize}
\item
non-negative integers $n$ and $m$,
\item
positive integers $k_{1},\ldots,k_{n},l_{1},\ldots, l_{m}$, and
\item
multiplicative characters $\chi_{1},\ldots,\chi_{n},\eta_{1},\ldots,\eta_{m}$ of $k^{\times}$.
\end{itemize}
We define the Kloosterman sum with respect to the above data, which is a function on $u \in k$, as follows: 
\[
\Kl_{u}^{n;m} \bigl(\psi; (k_{1},\ldots,k_{n};l_{1},\ldots, l_{m}), (\chi_{1},\ldots,\chi_{n};\eta_{1},\ldots,\eta_{m})\bigr) 
\]
\[
:= \sum_{\begin{subarray}{c} x_1, \ldots, x_n \in k\\ y_1, \ldots, y_{m} \in \tilde{k}\\  \prod_{i=1}^{n}\prod_{j=1}^{m}x_i^{k_i}\Nr(y_j)^{l_i} = u \end{subarray}} 
\psi(x_1 + \cdots + x_n)
\tilde\psi(y_1 + \cdots + y_{m})
\chi_{1}(x_{1})\cdots\chi_{n}(x_{n})
\tilde\eta_{1}(y_{1})\cdots\tilde\eta_{m}(y_{m}).
\] 
\end{defn}

When the above set $(n,m,(k_{1},\ldots,k_{n};l_{1},\ldots, l_{m}), (\chi_{1},\ldots,\chi_{n};\eta_{1},\ldots,\eta_{m}))$ includes trivial data, we omit them from the notation.
For example, if $m=0$, $(k_{1},\ldots,k_{n})=(1,\ldots,1)$, and $(\chi_{1},\ldots,\chi_{n})=(\mathbbm{1},\ldots,\mathbbm{1})$, then we simply write 
\[
\Kl_{u}^{n}(\psi)
:=\Kl_{u}^{n;0} \bigl(\psi; (1,\ldots,1), (\mathbbm{1},\ldots,\mathbbm{1})\bigr).
\]
We remark that the Kloosterman sums whose $m=0$ are defined in \cite[Chapter 4]{MR955052}.
In this paper, in order to treat the characters of simple supercuspidal representations of the unramified quasi-split even special orthogonal group $\SO_{2n}^{\ur}$, we have to consider the more general Kloosterman sums as above.

Now we define the Gauss sum with respect to $(\chi,\psi)$ by
 \[
G(\chi, \psi) := \sum_{t \in k^{\times}} \chi(t)\psi(t).
\]
Then we can write the Fourier transforms of the Kloosterman sums in terms of the Gauss sums:

\begin{prop}\label{prop:Kl}
\begin{enumerate}
\item
For a multiplicative character $\chi$ of $k^{\times}$, we have
 \[
 \sum_{u \in k^{\times}} \chi(u) 
 \Kl_{u}^{n;m} \bigl(\psi; (k_{1},\ldots,k_{n};l_{1},\ldots, l_{m}), (\chi_{1},\ldots,\chi_{n};\eta_{1},\ldots,\eta_{m})\bigr) 
\]
\[
=\prod_{i=1}^{n} G(\chi_{i}\chi^{k_{i}}, \psi)
 \prod_{j=1}^{m} G(\tilde{\eta}_{j}\tilde{\chi}^{l_{j}}, \tilde\psi).
 \]
\item
The Kloosterman sum $\Kl_{u}^{n;m} (\psi; (k_{1},\ldots,k_{n};l_{1},\ldots, l_{m}), (\chi_{1},\ldots,\chi_{n};\eta_{1},\ldots,\eta_{m}))$ is not constant on $u \in k^{\times}$.
In particular, for some $u \in k^{\times}$, the sum $\Kl_{u}^{n;m} (\psi; (k_{1},\ldots,k_{n};l_{1},\ldots, l_{m}), (\chi_{1},\ldots,\chi_{n};\eta_{1},\ldots,\eta_{m}))$ is not zero.
\end{enumerate}
\end{prop}

\begin{proof}
We can show the assertions by the same computation as in the case of usual Kloosterman sums.
See, e.g., \cite[Proposition A.4, Corollary A.5]{MR3904769}.
\end{proof}

\begin{prop}\label{prop:FT}
Let $a, b \in k^{\times}$ and $c \in \C^{\times}$.\\
If we have
\[
 \Kl_{ua}^{n;m} \bigl(\psi; (k_{1},\ldots,k_{n};l_{1},\ldots, l_{m}), (\chi_{1},\ldots,\chi_{n};\eta_{1},\ldots,\eta_{m})\bigr)
 \]
 \[
= c\Kl_{ub}^{n;m} \bigl(\psi; (k_{1},\ldots,k_{n};l_{1},\ldots, l_{m}), (\chi_{1},\ldots,\chi_{n};\eta_{1},\ldots,\eta_{m})\bigr) 
\]
 for every $u \in k^{\times}$, 
then we have $c=1$ and $a=b$.
\end{prop}

\begin{proof}
By using Proposition \ref{prop:Kl} (1), the claim is reduced to the non-vanishing of the Gauss sums.
See \cite[Proposition A.6]{MR3904769} for details.
\end{proof}

We next show a lemma on a vanishing property of a Kloosterman sum of the form
\[
 \Kl_{u}^{n+1} \bigl(\psi; (\underbrace{2,\ldots,2}_{n-1},1,1), (\underbrace{\mathbbm{1},\ldots,\mathbbm{1}}_{n},\omega_{0})\bigr),
\]
where $n\in\Z_{>0}$ and $\omega_{0}$ is the nontrivial quadratic character of $k^{\times}$.
Note that the Kloosterman sum of this type arises as the characters of simple supercuspidal representations of symplectic groups (see Proposition \ref{prop:charSp}).

\begin{lem}\label{lem:Klvan}
Let $u$ be an element of $k^{\times}$.
If $u$ does not belong to $k^{\times2}$, then we have
\[
\Kl_{u}^{n+1} \bigl(\psi; (2,\ldots,2,1,1), (\mathbbm{1},\ldots,\mathbbm{1},\omega_{0})\bigr)=0.
\]
\end{lem}

\begin{proof}
By the definition of the Kloosterman sum, we have
\[
\Kl_{u}^{n+1} \bigl(\psi; (2,\ldots,2,1,1), (\mathbbm{1},\ldots,\mathbbm{1},\omega_{0})\bigr)
\]
\[
=
\sum_{t_{1}\in k^{\times}}\psi(t_{1})\sum_{t_{2}\in k^{\times}}\psi(t_{2})\cdots\sum_{t_{n-1}\in k^{\times}}
\psi(t_{n-1})\Kl^{2}_{ut_{1}^{-2}\cdots t_{n-1}^{-2}}\bigl(\psi;(\mathbbm{1},\omega_{0})\bigr).
\]
Thus it is enough to show the claim in the case where $n=1$.

In this case, we have
\begin{align*}
\Kl_{u}^{2}\bigl(\psi;(\mathbbm{1},\omega_{0})\bigr)
&=
\sum_{\begin{subarray}{c}t_{1},t_{2}\in k^{\times}\\ t_{1}t_{2}=u\end{subarray}} \psi(t_{1}+t_{2})\omega_{0}(t_{2})\\
&=
\sum_{x\in k^{\times}}\psi(x+ux^{-1})\omega_{0}(x)\\
&=\frac{1}{2}\sum_{x\in k^{\times}}\psi(x+ux^{-1})\bigl(\omega_{0}(x)+\omega_{0}(ux^{-1})\bigr).
\end{align*}
Since $u$ does not belong to $k^{\times2}$, we have
\[
\omega_{0}(x)+\omega_{0}(ux^{-1})=0
\]
for every $x\in k^{\times}$.
Hence we have
\[
\Kl_{u}^{2}\bigl(\psi;(\mathbbm{1},\omega_{0})\bigr)=0.
\]
\end{proof}

We finally introduce well-known properties about Gauss sums.
The following lemma will play a key role in a comparison of characters of simple supercuspidal representations.
\begin{lem}\label{lem:HD}
Let $q$ be the order of the finite field $k$.
\begin{enumerate}
\item We have
\[
G(\omega_{0}, \psi)^{2}=q \cdot \omega_{0}(-1).
\]
\item (Hasse-Davenport product relation)
If the characteristic of $k$ is not equal to $2$, then we have
\[
G(\chi\omega_{0},\psi)G(\chi,\psi)=G(\chi^{2},\psi)G(\omega_{0},\psi)\chi(4)^{-1}
\]
for every multiplicative character $\chi$ of $k^{\times}$.
\item (Hasse-Davenport lifting relation)
For every multiplicative character $\chi$ of $k^{\times}$, we have
\[
G(\tilde\chi,\tilde\psi)
=
-G(\chi,\psi)^{2}.
\]
\end{enumerate}
\end{lem}

\begin{proof}
We can check (1) easily as follows:
\begin{align*}
q
=G(\omega_{0}, \psi)\cdot\ol{G(\omega_{0}, \psi)}
&=\biggl(\sum_{x\in k^{\times}}\omega_{0}(x)\psi(x)\biggr)\cdot\biggl(\sum_{y\in k^{\times}}\omega_{0}(y^{-1})\psi(-y)\biggr)\\
&=\biggl(\sum_{x\in k^{\times}}\omega_{0}(x)\psi(x)\biggr)\cdot\omega_{0}(-1)\cdot\biggl(\sum_{y\in k^{\times}}\omega_{0}(y)\psi(y)\biggr)\\
&=\omega_{0}(-1)G(\omega_{0},\psi)^{2}.
\end{align*}

The assertion (2) follows from basic properties of Gauss sums and Jacobi sums.
See, for example, \cite[Chapter 8, Theorem 1 and Exercise 6]{MR1070716}.

See, for example, \cite[503 page]{MR0029393} or \cite[Chapter 11, Theorem 1]{MR1070716} for a proof of the final assertion.
\end{proof}

\section{Depth of simple supercuspidal representations}\label{sec:depth}
The aim of this section is to characterize the simple supercuspidal representations via the notion of depth.

\subsection{Terminologies from Bruhat--Tits theory}\label{subsec:BT}
Let us first review the basic terminologies from Bruhat--Tits theory following Tits' expository article \cite{MR546588} and also Yu's supplementary note \cite{MR2508720}.

Let $\G$ be a connected reductive group over $F$.
We take a maximal $F$-split torus $\bfS_{\G}$ of $\G$.
Let $\T_{\G}$ (resp.\ $\bfN_{\G}$) denote the centralizer (resp.\ the normalizer) of $\bfS_{\G}$ in $\G$.
We write $\Phi_{\G}$ for the set of relative roots of $\bfS_{\G}$ in $\G$.
For each root $a \in \Phi_{\G}$, we let $\bfU_{a}$ denote the corresponding root subgroup of $\G$.

\subsubsection{Apartment}\label{subsubsec:apartment}
We put $V:=X_{\ast}(\bfS)\otimes_{\Z}\R$.
Let $\nu$ be the homomorphism
\[
\nu\colon T_{\G} \rightarrow V
\]
($T_{\G}=\T_{\G}(F)$) defined by the equation $\chi(\nu(t)) = -\val_{F}(\chi(t))$ for any $\chi\in X^{\ast}(\T_{\G})$.
We set $T_{\G,c}:=\Ker(\nu)$.
We consider an affine space $\mathcal{A}$ under $V$ (i.e., a $V$-torsor) equipped with a homomorphism
\[
\tilde{\nu}\colon N_{\T}/T_{\G,c}\rightarrow\Aut_{\aff}(\mathcal{A})
\]
making the following diagram commutative:
\[
\xymatrix{
1\ar[r]&T_{\G}/T_{\G,c}\ar[r]\ar^-{\nu}[d]&N_{\G}/T_{\G,c}\ar[r]\ar^-{\tilde{\nu}}[d]&N_{\G}/T_{\G}\ar[r]\ar[d]&1\\
1\ar[r]&V\ar[r]&\Aut_{\aff}(\mathcal{A})\ar[r]&\GL_{\R}(V)\ar[r]&1,
}
\]
where
\begin{itemize}
\item
$\Aut_{\aff}(\mathcal{A})$ is the group of affine transformations of $\mathcal{A}$ (recall that an affine transformation $\varphi$ of $\mathcal{A}$ is a bijective map from $\mathcal{A}$ to itself such that there exists $\dot{\varphi}\in\GL_{\R}(V)$ satisfying $\varphi(\mathbf{x}+v)=\varphi(\mathbf{x})+\dot{\varphi}(v)$ for any $\mathbf{x}\in\mathcal{A}$ and $v\in V$),
\item
the right vertical map $N_{\G}/T_{\G}\rightarrow\GL_{\R}(V)$ is the map naturally induced from the action of $N_{\G}/T_{\G}$ on $\bfS_{\G}$, hence on $X_{\ast}(\bfS_{\G})$, 
\item
the map $V\rightarrow\Aut_{\aff}(\mathcal{A})$ is given by the translation by $V$, i.e., the image of $v\in V$ in $\Aut_{\aff}(\mathcal{A})$ is given by the map $\mathbf{x}\mapsto\mathbf{x}+v$, and
\item
the map $\Aut_{\aff}(\mathcal{A})\rightarrow\GL_{\R}(V)$ is given by $\varphi\mapsto \dot{\varphi}$.
\end{itemize}

When $\G$ is semisimple, an affine space as above exists uniquely up to unique isomorphism.
We let $\mathcal{A}(\G,\bfS_{\G})$ denote the unique affine space.
When $\G$ is not semisimple, we consider its derived subgroup $\G_{\der}$ and the inverse image $\bfS_{\G,\der}$ of $\bfS_{\G}$ in $\G_{\der}$.
We put $\mathcal{A}_{\red}(\G,\bfS_{\G}):=\mathcal{A}(\G_{\der},\bfS_{\G,\der})$.
If we let $\mathbf{A}_{\G} (\subset\bfS_{\G})$ be the maximal split central torus of $\G$, then the product $\mathcal{A}_{\red}(\G,\bfS_{\G})\times X_{\ast}(\mathbf{A}_{\G})\otimes_{\Z}\R$ has a structure of an affine space under $V$ as above.
We put $\mathcal{A}(\G,\bfS_{\G}):=\mathcal{A}_{\red}(\G,\bfS_{\G})\times X_{\ast}(\mathbf{A}_{\G})\otimes_{\Z}\R$.
We call $\mathcal{A}(\G,\bfS_{\G})$ (resp.\ $\mathcal{A}_{\red}(\G,\bfS_{\G})$) the apartment (resp.\ the reduced apartment) of $\bfS_{\G}$ .
(See \cite[Section 1.2]{MR546588} and \cite[Sections 2.2.1 and 2.2.2]{MR2508720} for the details.)

\begin{rem}
We can also understand the reduced apartment as the set of ``valuation of root data'', which are families (indexed by relative roots $\Phi_{\G}$) of filtrations (indexed by $\R$) on root subgroups satisfying certain axioms (see \cite[Section 1.5]{MR546588}).
See \cite[Sections 2.2.5]{MR2508720} for how to compare these two languages.
\end{rem}

\subsubsection{Affine function and its gradient}\label{subsubsec:grad}
Let $\alpha$ be an affine function on $\mathcal{A}(\G,\bfS_{\G})$, that is, a function $\mathcal{A}(\G,\bfS_{\G})\rightarrow\R$ equipped with an $\dot{\alpha}\in\Hom_{\R}(V,\R)$ satisfying $\alpha(\mathbf{x}+v)=\alpha(\mathbf{x})+\dot{\alpha}(v)$ for any $\mathbf{x}\in\mathcal{A}(\G,\bfS_{\G})$ and $v\in V$.
(Note that obviously $\dot{\alpha}$ is uniquely determined by $\alpha$.)
Then, since we have a canonical identification $\Hom_{\R}(V,\R)\cong X^{\ast}(\bfS_{\G})\otimes_{\Z}\R$ induced from the natural pairing $X^{\ast}(\bfS_{\G})\times X_{\ast}(\bfS_{\G})\rightarrow\Z$, we may regard $\dot{\alpha}$ as an element of $X^{\ast}(\bfS_{\G})\otimes_{\Z}\R$.
We call $\dot{\alpha}\in X^{\ast}(\bfS_{\G})\otimes_{\Z}\R$ the gradient of $\alpha$.

\subsubsection{Filtrations of root subgroups}\label{subsubsec:filtration}
Let $a\in\Phi_{\G}$ and $u\in\bfU_{a}(F)\smallsetminus\{1\}$.
Then there exists a unique element belonging to the intersection $\bfU_{-a}u\bfU_{-a}\cap N_{\G}$.
We write $m(u)$ for this unique element.
Let $\alpha(a,u)$ denote the affine function on $\mathcal{A}(\G,\bfS_{\G})$ determined by the conditions that
\begin{itemize}
\item
the gradient $\dot{\alpha}(a,u) \in X^{\ast}(\bfS_{\G})\otimes_{\Z}\R$ is given by $a\in\Phi_{\G}$ and 
\item
the vanishing hyperplane $\alpha^{-1}(0)\subset\mathcal{A}(\G,\bfS_{\G})$ is given by the set of fixed points of $\tilde{\nu}(m(u))$.
\end{itemize}
Let $\Phi'_{\G}$ be the set of affine functions on $\mathcal{A}(\G,\bfS_{\G})$ whose gradients belong to $\Phi_{\G}$.
For $\alpha\in\Phi'_{\G}$ with gradient $\dot{\alpha}=a\in\Phi_{\G}$, we put
\[
U_{\alpha}:=\{u\in\bfU_{a}(F) \mid \text{$u=1$ or $\alpha(a,u)\geq\alpha$}\},
\]
where $\alpha(a,u)\geq\alpha$ denotes the inequality as functions, that is, $\alpha(a,u)(\mathbf{x})\geq\alpha(\mathbf{x})$ holds for any $\mathbf{x}\in\mathcal{A}(\G,\bfS_{\G})$.
For any affine function $\alpha$ on $\mathcal{A}(\G,\bfS_{\G})$ whose gradient $\dot{\alpha}$ does not belong to $\Phi_{\G}$, we set $U_{\alpha}:=1$.
We put $\bar{U}_{\alpha}:=U_{\alpha}/U_{\alpha+\varepsilon}$ for a sufficiently small positive number $\varepsilon$.
(See \cite[Section 1.4]{MR546588} and \cite[Sections 2.2.5 and 2.2.6]{MR2508720} for the details.)

\subsubsection{Affine roots}\label{subsubsec:affine-root}
For $\alpha\in\Phi'_{\G}$, we have the following chain of subgroups:
\[
U_{2\alpha} \subset U_{\alpha} \subset \bfU_{\dot{\alpha}}(F).
\]
This induces an injective map
\[
\bar{U}_{2\alpha}\hookrightarrow\bar{U}_{\alpha}.
\]
We say that $\alpha\in\Phi'_{\G}$ is an affine root of $\bfS_{\G}$ if the quotient $\bar{U}_{\alpha}/\bar{U}_{2\alpha}$ is not zero.
We write $\Psi_{\G}$ for the set of affine roots of $\bfS_{\G}$.
(See \cite[Section 1.6]{MR546588} for the details.)

\subsubsection{Bruhat--Tits building}\label{subsubsec:building}
We write $\mathcal{B}_{\red}(\G,F)$ for the unique (up to unique isomorphism) set having an action of $\G(F)$ and containing $\mathcal{A}_{\red}(\G,\bfS_{\G})$ such that the following conditions hold:
\begin{itemize}
\item
we have $\mathcal{B}_{\red}(\G,F)=\bigcup_{g\in \G(F)} g\cdot\mathcal{A}_{\red}(\G,\bfS_{\G})$,
\item
$N_{\G,\der}(F)$ stabilizes $\mathcal{A}_{\red}(\G,\bfS_{\G})$ and acts on it through $\tilde{\nu}$, where $N_{\G,\der}$ is the normalizer group of $\bfS_{\G,\der}$ in $\G_{\der}$, and
\item 
the group $U_{\alpha}$ fixes the set $\alpha^{-1}([0,\infty))\cap\mathcal{A}_{\red}(\G,\bfS_{\G})$ pointwise.
\end{itemize}
We define $\mathcal{B}(\G,F):=\mathcal{B}_{\red}(\G,F)\times X_{\ast}(\mathbf{A}_{\G})\otimes_{\Z}\R$.
We call $\mathcal{B}(\G,F)$ the Bruhat--Tits building of $\G$ and $\mathcal{B}_{\red}(\G,F)$ the reduced Bruhat--Tits building of $\G$.
(See \cite[Section 2.1]{MR546588} and \cite[Sections 2.2.2]{MR2508720} for the details.)

\subsection{Minimal positivity of depth of simple supercuspidal representations}\label{subsec:depth}
In the following, we assume that $\G$ is quasi-split and tamely ramified over $F$, namely splits over a tamely ramified extension of $F$.
Furthermore, we put the assumption that $\G$ is $F$-simple, i.e., the local Dynkin diagram (see \cite[Section 1.8]{MR546588}) is connected.

The set of affine roots $\Psi_{\G}$ gives the reduced apartment $\mathcal{A}_{\red}(\G,\bfS_{\G})$ a structure of an affine root system.
Accordingly, $\mathcal{A}_{\red}(\G,\bfS_{\G})$ gets a structure of a simplicial set.
If we fix a set $\Pi_{\G}$ of simple affine roots of the affine roots $\Psi_{\G}$, then we get the alcove $\mathcal{C}$ of the apartment $\mathcal{A}_{\red}(\G,\bfS_{\G})$ corresponding to $\Pi_{\G}$.
Namely, $\mathcal{C}$ is the facet of maximal dimension in the simplicial set $\mathcal{A}_{\red}(\G,\bfS_{\G})\subset\mathcal{B}_{\red}(\G,F)$ given by
\[
\mathcal{C}
=
\{\mathbf{x}\in\mathcal{A}_{\red}(\G,\bfS_{\G})\mid \text{$\alpha(\mathbf{x})>0$ for any $\alpha\in\Pi_{\G}$}\}.
\]

We put 
\[
\Pi_{\G}=\{\alpha_{0},\ldots,\alpha_{l}\}.
\]
Then there exists a unique $(l+1)$-tuple of positive integers $(b_{0}, b_{1}, \ldots, b_{l})$, one of which is equal to $1$, and $e\in\Z_{>0}$ satisfying 
\[
\sum_{i=0}^{l} b_{i} \alpha_{i} =\frac{1}{e}.
\]
By using these constants, we can define the \textit{twisted Coxeter number} $h$ of $\G$ by
\[
h=e\sum_{i=0}^{l} b_{i}
\]
(see \cite[Sections 2.6 and 3.3]{MR3164986} for details).

For example, when the residual characteristic of $F$ is odd, every quasi-split classical group treated in \cite{MR3904769}, \cite{MR4025003}, and this paper satisfies the assumptions in the beginning of this section.
For such groups, the twisted Coxeter numbers are given by
\[
\begin{cases}
N & \text{if $\G$ is $\GL_{N}$, $\mathrm{Res}_{E_{\ur}/F}\GL_{N}$, or $\U_{E_{\ur}/F}(N)$},\\
2n &\text{if $\G$ is $\SO_{2n+1}$, $\Sp_{2n}$, $\SO^{\mu}_{2n}$, or $\SO_{2n+2}^{(\ur)}$}.
\end{cases}
\]


By the assumption of the $F$-simplicity of $\G$, the affine root system with respect to $(\G,\bfS_{\G})$ is irreducible and we have
\[
\dim\mathcal{A}_{\red}(\G,\bfS_{\G})=l.
\]
Therefore there exists a unique point $\mathbf{b}\in\mathcal{C}\subset\mathcal{A}_{\red}(\G,\bfS_{\G})$ such that 
\[
\alpha_{0}(\mathbf{b})=\cdots=\alpha_{l}(\mathbf{b}).
\]
We call this point the \textit{barycenter} of the alcove $\mathcal{C}$.

We next recall the definition of the Moy--Prasad filtrations of parahoric groups, and the notion of the depth of representations.
We take $\mathbf{x}\in\mathcal{A}_{\red}(\mathbf{G},\bfS_{\G})$ and a non-negative real number $r\in\R_{\geq0}$.
Then the subgroups $G_{\mathbf{x},r}$ and $G_{\mathbf{x},r+}$ of $G=\mathbf{G}(F)$ are defined by
\begin{align*}
G_{\mathbf{x},r} &:= \langle T_{\G}^{r}, U_{\alpha}\mid \alpha\in\Psi_{\G}, \alpha(\mathbf{x})\geq r \rangle, \text{ and}\\
G_{\mathbf{x},r+} &:= \langle T_{\G}^{r+\varepsilon}, U_{\alpha}\mid \alpha\in\Psi_{\G}, \alpha(\mathbf{x})> r \rangle.
\end{align*}
Here 
\begin{itemize}
\item
$\T_{\G}$ is the centralizer of $\bfS_{\G}$ in $\G$ (note that this is a maximal torus since $\G$ is quasi-split), 
\item
$T_{\G}^{0}$ is the unique parahoric subgroup of $T_{\G}:=\T_{\G}(F)$,
\item
for $r\in\R_{>0}$, $T_{\G}^{r}$ is the $r$-th Moy--Prasad filtration of $T_{\G}^{0}$:
\[
T_{\G}^{r}
:=
\{t\in T_{\G}^{0} \mid \text{$\val(\chi(t)-1)\geq r$ for any $\chi\in X^{\ast}(\T_{\G})$}\},
\]
\item
$U_{\alpha}$ is the affine root subgroup attached to the affine root $\alpha$, and
\item
$\varepsilon$ is a sufficiently small positive number.
\end{itemize}

\begin{rem} 
Note that, in \cite{MR1371680}, the Moy--Prasad filtration is defined as the restriction of an filtration of $\G(F^{\ur})$ to $\G(F)$.
However, since we assume that $\mathbf{G}$ is quasi-split, our definition coincides with their one.
This follows from a fundamental property on the behavior of the valuation of root datum for a quasi-split reductive group over a $p$-adic field under an unramified extension.
More precisely, for a quasi-split reductive group $\G$ over a $p$-adic field $F$, the valuation of root datum of $\G$ over $F$ coincides with the descent of that of $\G$ over $F^{\ur}$.
See \cite[5.1.20, Remarques (2)]{MR756316} and \cite[Remark 3.3]{MR2532460} for details.
\end{rem}

\begin{rem}
In this paper and \cite{MR3164986}, we use the valuation of $F$ to define the valuation of root datum of $\G$ and the above Moy--Prasad filtrations.
We note that this normalization differs from the original definition in \cite{MR1371680} (in \cite{MR1371680}, the valuation of the splitting field of $\G$ is used to define Moy--Prasad filtrations).
\end{rem}

\begin{rem}\label{rem:Iwahori}
The subgroups $I_{\G}$, $I_{\G}^{+}$, and $I_{\G}^{++}$ of $G$ defined in Section \ref{subsec:Iwahori} are the first three steps of the Moy--Prasad filtration of the parahoric subgroup of $G$ attached to the barycenter $\mathbf{b}$.
More precisely, we have the following:
\begin{align*}
I_{\G}&=G_{\mathbf{b},0},\\
I_{\G}^{+}&=G_{\mathbf{b},r} \quad\text{for $0<r\leq\frac{1}{h}$, and}\\
I_{\G}^{++}&=G_{\mathbf{b},r} \quad\text{for $\frac{1}{h}<r\leq\frac{2}{h}$.}
\end{align*}
\end{rem}

\begin{defn}\label{defn:depth}
Let $\pi$ be an irreducible smooth representation of $G$.
Then we define the depth of $\pi$ by
\[
\depth(\pi)
:=
\inf \bigr\{r\in\R_{\geq0} \,\big\vert\, \pi^{G_{\mathbf{x}, r+}}\neq0 \text{ for some } \mathbf{x}\in\mathcal{B}_{\red}(\mathbf{G},F)\bigr\}\in\R_{\geq0}.
\]
\end{defn}
Note that the depth of representations are non-negative by the definition.

\begin{thm}[{\cite[Theorem 3.5]{MR1371680}}]\label{thm:depth}
For every irreducible smooth representation $\pi$ of $G$, its depth is attained by a point of $\mathcal{B}_{\red}(\mathbf{G},F)$, and is a rational number.
\end{thm}

\begin{lem}\label{lem:hyperplane}
For every point $\mathbf{x}$ of the closure of the fixed alcove $\overline{\mathcal{C}}$, there exists $\mathbf{y}\in\overline{\mathcal{C}}$ satisfying the following condition for every $0\leq i \leq l$:
\[
\alpha_{i}(\mathbf{x})<\frac{1}{h} \implies \alpha_{i}(\mathbf{y})=0.
\]
\end{lem}

\begin{proof}
We first note that the cardinality of the set
\[
\{i \mid \alpha_{i}(\mathbf{x})<1/h \}
\]
is at most $l$.
Indeed, if it is $l+1$, then we have
\[
\frac{1}{e}=\sum_{i=0}^{l} b_{i} \alpha_{i}(\mathbf{x})
< \sum_{i=0}^{l} b_{i}\cdot \frac{1}{h} = \frac{1}{e}
\]
and this is a contradiction.

Since the dimension of $\mathcal{A}_{\red}(\G,\bfS_{\G})$ is given by $l=|\Pi_{\G}|-1$, we can take a point $\mathbf{y}\in\overline{\mathcal{C}}$ which is contained in the zero locus of every $\alpha_{i}\in \Pi_{\G}$ satisfying $\alpha_{i}(\mathbf{x})<1/h$.
Then $\mathbf{y}$ is as desired.
\end{proof}

\begin{lem}\label{lem:0-h}
For every $\mathbf{x}\in\mathcal{A}_{\red}(\G,\bfS_{\G})$, there exists $\mathbf{y}\in\mathcal{A}_{\red}(\G,\bfS_{\G})$
satisfying
\[
G_{\mathbf{y}, 0+} \subset G_{\mathbf{x}, 1/h}.
\]
\end{lem}

\begin{proof}
By taking a conjugation, we may assume that $\mathbf{x}$ belongs to the closure of the fixed alcove $\overline{\mathcal{C}}$.
Then, by Lemma \ref{lem:hyperplane}, we can take a point $\mathbf{y}\in\overline{\mathcal{C}}$ such that, for every $0\leq i \leq l$, we have
\[
\alpha_{i}(\mathbf{x})<\frac{1}{h} \implies \alpha_{i}(\mathbf{y})=0. \tag{$\ast$}
\]

Now, for this point $\mathbf{y}$, we show that
\[
G_{\mathbf{y}, 0+} \subset G_{\mathbf{x}, 1/h}.
\]
To show this, by the definitions of $G_{\mathbf{y},0+}$ and $G_{\mathbf{x},1/h}$, it is enough to check that, for every $\alpha\in\Psi_{\G}$, we have
\[
\alpha(\mathbf{y})>0 \implies \alpha(\mathbf{x})\geq1/h.
\]

Let us take an affine root $\alpha\in\Psi_{\G}$ satisfying $\alpha(\mathbf{y})>0$.
We first note that we can write
\[
\alpha=\sum_{i=0}^{l} a_{i}\alpha_{i},
\]
where $a_{i}$ are integers which are either all non-negative or all non-positive.
Since $\mathbf{y}$ belongs to the closure of the fixed alcove $\overline{\mathcal{C}}$, we have $\alpha_{i}(\mathbf{y})\geq0$ for every $\alpha_{i}$.
Therefore, by the assumption that $\alpha(\mathbf{y})>0$, we have 
\begin{itemize}
\item
$a_{i}\geq0$ for every $i$, and
\item
$a_{i_{0}}\geq1$ and $\alpha_{i_{0}}(\mathbf{y})>0$ (hence $\alpha_{i_{0}}(\mathbf{x})\geq1/h$ by $(\ast)$) for at least one $i_{0}$.
\end{itemize}
Thus we can conclude that
\[
\alpha(\mathbf{x})
=
\sum_{i=0}^{l}a_{i}\alpha_{i}(\mathbf{x})
\geq
a_{i_{0}}\alpha_{i_{0}}(\mathbf{x})
\geq
\frac{1}{h}.
\]
\end{proof}

\begin{prop}\label{prop:min}
If the depth of an irreducible smooth representation of $G$ is smaller than $1/h$, then it equals zero.
\end{prop}

\begin{proof}
Let $\pi$ be an irreducible smooth representation of $G$, and $d$ the depth of $\pi$.
Then, by Theorem \ref{thm:depth}, there exists a point $\mathbf{x}\in\mathcal{A}_{\red}(\G,\bfS_{\G})$ such that $\pi$ has a non-zero $G_{\mathbf{x},d+}$-fixed vector.

Now we suppose that $d$ is smaller than $1/h$.
Then we have
\[
G_{\mathbf{x}, 1/h} \subset G_{\mathbf{x}, d+}.
\]
On the other hand, by Lemma \ref{lem:0-h}, there exists $\mathbf{y}\in\mathcal{A}_{\red}(\G,\bfS_{\G})$ such that
\[
G_{\mathbf{y}, 0+} \subset G_{\mathbf{x}, 1/h}.
\]
Thus we have
\[
\pi^{G_{\mathbf{y}, 0+}} \supset \pi^{G_{\mathbf{x}, d+}} \neq0.
\]
Therefore, by the definition of the depth, we have $d=0$.
\end{proof}

\begin{cor}\label{cor:min}
The number $1/h$ is the minimal positive depth of irreducible smooth representations of $G$.
\end{cor}

\begin{proof}
The depths of simple supercuspidal representations are given by $1/h$ (see \cite[Section 2.6]{MR3164986}).
Thus, by Proposition \ref{prop:min}, the number $1/h$ is the minimal positive depth of irreducible smooth representations of $G$.
\end{proof}

\begin{lem}\label{lem:non-aff-gen}
Let $\mathbf{b}$ be the barycenter of the fixed alcove $\mathcal{C}$.
Let $\alpha_{i}\in\Pi_{\G}$ be a simple affine root.
Then there exists a point $\mathbf{y}\in\overline{\mathcal{C}}$ such that
\[
\lan G_{\mathbf{b},1/h+}, U_{\alpha_{i}}\ran \supset G_{\mathbf{y}, 0+}.
\]
\end{lem}

\begin{proof}
Let $\mathbf{b}$ be the barycenter of the alcove $\mathcal{C}$, namely the unique point satisfying $\alpha_{j}(\mathbf{b})=1/h$ for every $0\leq j \leq l$.
We take $\mathbf{y}$ to be a point of $\overline{\mathcal{C}}$ contained in the zero locus of $\alpha_{j}$ for every $j\neq i$.
We show that $\mathbf{y}$ satisfies the desired condition.
To show this, it suffices to check that, for every $\alpha\in\Psi_{\G}$ satisfying $\alpha(\mathbf{y})>0$, we have either 
\begin{itemize}
\item
$\alpha=\alpha_{i}$, or
\item
$\alpha(\mathbf{b})>1/h$.
\end{itemize}

Let $\alpha\in\Psi_{\G}$ be an affine root satisfying $\alpha(\mathbf{y})>0$.
Then, since $\mathbf{y}$ belongs to $\overline{\mathcal{C}}$, the condition $\alpha(\mathbf{y})>0$ implies that $\alpha$ is a positive affine root (see the proof of Lemma \ref{lem:0-h}).
Now we suppose that we have $\alpha(\mathbf{b})\leq 1/h$.
Then, since $\alpha$ is positive and $\alpha_{j}(\mathbf{b})$ equals $1/h$ for every $0\leq j \leq l$, $\alpha$ is a simple affine root.
On the other hand, by the assumption on $\mathbf{y}$, we have $\alpha_{j}(\mathbf{y})=0$ for every $j\neq i$.
This implies that $\alpha=\alpha_{i}$.
\end{proof}


\begin{prop}\label{prop:depth-ssc}
For an irreducible smooth representation $\pi$ of $G$, the following are equivalent:
\begin{enumerate}
\item $\pi$ is a simple supercuspidal representation of $G$.
\item $\pi$ has the minimal positive depth $1/h$.
\end{enumerate}
\end{prop}

\begin{proof}
By Corollary \ref{cor:min}, it suffices to show that (2) implies (1).
Let us take an irreducible smooth representation $\pi$ of $G$ whose depth is given by $1/h$.
Then, by Theorem \ref{thm:depth}, there exists a point $\mathbf{x}\in\mathcal{A}_{\red}(\G,\bfS_{\G})$, satisfying
\[
\pi^{G_{\mathbf{x},1/h+}}\neq0.
\]
By taking a conjugation, we may assume that this point $\mathbf{x}$ belongs to $\overline{\mathcal{C}}$.

If this point $\mathbf{x}$ is the barycenter of the alcove $\mathcal{C}$, then $\pi$ is simple supercuspidal.
Indeed, since we have $\pi^{G_{\mathbf{x},1/h+}}\neq0$ and the quotient $V_{\G}:=G_{\mathbf{x},1/h}/G_{\mathbf{x},1/h+}$ is abelian, the representation $\pi$ contains a character $\chi$ of $V_{\G}$.
Here recall that, by Proposition \ref{prop:I-quot}, we have an isomorphism
\[
V_{\G}
\cong \prod_{i=0}^{l} V_{\G}(\dot{\alpha_{i}}),
\]
where $V_{\G}(\dot{\alpha_{i}})$ is the $\dot{\alpha_{i}}$-isotypic part of $V_{\G}$ with respect to the $\bfS_{\G}(k)$-action.
Suppose for the sake of contradiction that $\chi$ is not an affine generic character.
Then $\pi$ has a non-zero $\lan G_{\mathbf{x},1/h+}, U_{\alpha_{i}}\ran$-fixed vector for some simple affine root $\alpha_{i}\in\Pi_{\G}$ by the definition of the affine genericity and the description of $V_{\G}(\dot{\alpha_{i}})$ (see the proof of Proposition \ref{prop:I-quot}).
Then, by Lemma \ref{lem:non-aff-gen}, $\pi$ has a non-zero $G_{\mathbf{y},0+}$-fixed vector for some point $\mathbf{y}$.
Therefore, by the definition of the depth, the depth of $\pi$ is given by zero.
This is a contradiction.
Thus $\chi$ is affine generic.
Hence, by the Frobenius reciprocity and the definition of simple supercuspidal representations, $\pi$ is simple supercuspidal.

Now our task is to show that the point $\mathbf{x}$ is the barycenter of the alcove $\mathcal{C}$.
If $\mathbf{x}$ is not the barycenter, then the cardinality of the set
\[
\{i \mid \alpha_{i}(\mathbf{x})\leq1/h\}
\]
is at most $l$.
Indeed, if every simple affine root $\alpha_{i}$ satisfies $\alpha_{i}(\mathbf{x})\leq1/h$, then we have 
\[
\frac{1}{e}=\sum_{i=0}^{l} b_{i} \alpha_{i}(\mathbf{x})
\leq \sum_{i=0}^{l} b_{i}\cdot \frac{1}{h} = \frac{1}{e}.
\]
Hence we have $\alpha(\mathbf{x})=1/h$ for every $i$.
Namely $\mathbf{x}$ is the barycenter of $\mathcal{C}$, and this is a contradiction.

Then, by the same argument as in the proofs of Lemmas \ref{lem:hyperplane} and \ref{lem:0-h}, we can take a point $\mathbf{y}\in\overline{\mathcal{C}}$ such that 
\[
G_{\mathbf{y}, 0+} \subset G_{\mathbf{x},1/h+}.
\]
In particular, we have
\[
\pi^{G_{\mathbf{y},0+}}\neq 0.
\]
This implies that the depth of $\pi$ is zero and this is a contradiction.
\end{proof}

\section{Spinor twists for $L$-packets of even orthogonal groups}\label{sec:spin}

Let $\mathbf{H}$ be a quasi-split special orthogonal group of size $2n$ defined over $F$, where $n\geq2$.
Thus, in the notation of the main part of this paper, we have $\mathbf{H}=\SO_{2n}^{\bullet}$ for $\bullet\in\{\emptyset,\ur,\mu,\bar{\mu}\}$, where $\mu$ and $\bar{\mu}$ are distinct ramified quadratic characters of $F^{\times}$.
We write $\spin$ for the \textit{spinor norm} of $H=\mathbf{H}(F)$, which is a homomorphism from $H$ to $F^{\times}/F^{\times2}$:
\[
\spin \colon H \rightarrow F^{\times}/ F^{\times2}.
\]
See, for example, \cite[Section 24]{MR2665139} for the definition of the spinor norm.

Let $\phi$ be an $L$-parameter of $\mathbf{H}$:
\[
\phi\colon W_{F}\times\SL_{2}(\C)\rightarrow \SO_{2n}(\C)\rtimes W_{F}.
\]
We take a quadratic character $\chi$ of $F^{\times}$ and regard it as a homomorphism from $W_{F}$ to $\SO_{2n}(\C)$ by the local class field theory:
\[
\chi\colon W_{F}\rightarrow \{\pm1\}\subset\SO_{2n}(\C).
\]
Then we define the twist $\phi\otimes\chi$ of $\phi$ via $\chi$ as follows:
\begin{align*}
\phi\otimes\chi\colon W_{F}\times\SL_{2}(\C)&\rightarrow \SO_{2n}(\C)\rtimes W_{F}\\
(w,g)&\mapsto \phi(w,g)\chi(w).
\end{align*}
Note that if we regard $\phi\otimes\chi$ as a $2n$-dimensional representation $\iota\circ(\phi\otimes\chi)$ by using the $L$-embedding $\iota$ from ${}^{L}\mathbf{H}$ to ${}^{L}\GL_{2n}$ defined in Section \ref{sec:Arthur}, then it equals $(\iota\circ\phi)\otimes\chi$.

The aim of this section is to show the following compatibility between the local Langlands correspondence for $\mathbf{H}$ and the spinor twist:

\begin{prop}\label{prop:spin-twist}
Let $\phi$ be a discrete $L$-parameter of $\mathbf{H}$.
Then, for a quadratic character $\chi$ of $F^{\times}$, we have
\[
\widetilde{\Pi}^{\mathbf{H}}_{\phi\otimes\chi}=\widetilde{\Pi}^{\mathbf{H}}_{\phi}\otimes(\chi\circ\spin).
\]
\end{prop}

\begin{rem}
Once we establish this proposition, we may easily extend it to the general case where $\phi$ is tempered or even non-tempered by considering the relationship between the local Langlands correspondence and the parabolic induction.
Since the case where $\phi$ is discrete is enough for our purpose, we do not explain the proof in the general case.
\end{rem}

We prove Proposition \ref{prop:spin-twist} according to the method of \cite[Theorem C.5]{MR3166215} by means of Plancherel measures.

We first recall from \cite[Theorem 12.1]{MR3166215} the definition of the Plancherel measure briefly.
Let $r\in\Z_{\geq1}$ and put $\widetilde{\bfH}:=\SO_{2(n+r)}^{\bullet}$
Then the group $\mathbf{M}:=\GL_{r}\times\mathbf{H}$ is realized as a Levi subgroup of $\widetilde{\bfH}$.
Let $\tau$ be an irreducible smooth representation of $\GL_{r}(F)$ and $\pi$ an irreducible smooth representation of $H$.
For $s\in\C$, we put $\tau_{s}:=\tau\otimes|\det(-)|^{s}$.
By taking a parabolic subgroup $\mathbf{P}=\mathbf{M}\mathbf{U}$ with Levi component $\mathbf{M}$ and unipotent radical $\mathbf{U}$, we consider the normalized parabolic induction of $\tau_{s}\boxtimes\pi$ to $\widetilde{H}$:
\[
\nInd_{P}^{\widetilde{H}}(\tau_{s}\boxtimes\pi).
\]
If we let $\bar{\mathbf{P}}=\mathbf{M}\bar{\mathbf{U}}$ denote the parabolic subgroup opposite to $\mathbf{P}$ with unipotent radical $\bar{\mathbf{U}}$, we can define (at least formally) an intertwining operator 
\[
J_{\bar{P}|P}(\tau_{s}\boxtimes\pi)\colon \nInd_{P}^{\widetilde{H}}(\tau_{s}\boxtimes\pi)\rightarrow \nInd_{\bar{P}}^{\widetilde{H}}(\tau_{s}\boxtimes\pi)
\]
by
\[
\bigl(J_{\bar{P}|P}(\tau_{s}\boxtimes\pi)(f)\bigr)(g)
:=\int_{\bar{U}}f(\bar{u}g)\,d\bar{u}
\]
for $f\in\nInd_{P}^{\widetilde{H}}(\tau_{s}\boxtimes\pi)$ and $g\in\widetilde{H}$.
We also define an intertwining operator in the converse direction in a similar way:
\[
J_{P|\bar{P}}(\tau_{s}\boxtimes\pi)\colon \nInd_{\bar{P}}^{\widetilde{H}}(\tau_{s}\boxtimes\pi)\rightarrow \nInd_{P}^{\widetilde{H}}(\tau_{s}\boxtimes\pi).
\]
Then it is known that there exists a meromorphic function on $s\in\C$, which is denoted by $\mu_{\psi}(\tau_{s}\boxtimes\pi)$ and called the Plancherel measure, satisfying
\[
J_{P|\bar{P}}(\tau_{s}\boxtimes\pi)\circ J_{\bar{P}|P}(\tau_{s}\boxtimes\pi)
=\mu_{\psi}(\tau_{s}\boxtimes\pi)^{-1}\cdot\mathrm{id}.
\]
Note that, since the above intertwining operator depends on the choices of Haar measures on $U$ and $\bar{U}$, the Plancherel measure is well-defined only up to a scalar multiple unless we specify those Haar measures.
Here we use Haar measures determined by a fixed nontrivial additive character $\psi$ of $F$, and hence the above notation $\mu_{\psi}(\tau_{s}\boxtimes\pi)$ contains ``$\psi$''.
See \cite[Section B]{MR3166215} for the details.

\begin{lem}\label{lem:Plancherel}
For any irreducible smooth representations $\tau$ of $\GL_{r}(F)$ and $\pi$ of $\SO_{2n}^{\bullet}(F)$, we have
\[
\mu_{\psi}\bigl((\tau_{s}\otimes(\chi\circ\det))\boxtimes(\pi\otimes(\chi\circ\spin))\bigr)
=
\mu_{\psi}(\tau_{s}\boxtimes\pi).
\]
\end{lem}

\begin{proof}
The restriction of the character $\chi\circ\spin$ of $\widetilde{H}$ to $M=\GL_{r}\times H$ is given by the character $(\chi\circ\det)\boxtimes(\chi\circ\spin)$ (see \cite[Lemma 4.9]{MR2027702}).
Hence, for any parabolic subgroup $\mathbf{P}$ of $\widetilde{\bfH}$ with Levi subgroup $\mathbf{M}$, the normalized parabolic induction 
\[
\nInd_{P}^{\widetilde{H}} \bigl((\tau_{s}\otimes(\chi\circ\det))\boxtimes(\pi\otimes(\chi\circ\spin))\bigr)
\]
of $(\tau_{s}\otimes(\chi\circ\det))\boxtimes(\pi\otimes(\chi\circ\spin))$ is naturally isomorphic to the twist of the normalized induction of $\tau_{s}\boxtimes\pi$:
\[
\bigl(\nInd_{P}^{\widetilde{H}} (\tau_{s}\boxtimes\pi)\bigr)\otimes(\chi\circ\spin).
\]
Since the character twist does not change the representation space, two representations $\nInd_{P}^{\widetilde{H}} ((\tau_{s}\otimes(\chi\circ\det))\boxtimes(\pi\otimes(\chi\circ\spin)))$ and $\nInd_{P}^{\widetilde{H}} (\tau_{s}\boxtimes\pi)$ are realized on the same space.
From this observation, we see that intertwining operators $J_{\bar{P}|P}((\tau_{s}\otimes(\chi\circ\det))\boxtimes(\pi\otimes(\chi\circ\spin)))$ and $J_{P|\bar{P}}((\tau_{s}\otimes(\chi\circ\det))\boxtimes(\pi\otimes(\chi\circ\spin)))$ used to define the Plancherel measure $\mu_{\psi}((\tau_{s}\otimes(\chi\circ\det))\boxtimes(\pi\otimes(\chi\circ\spin)))$ is exactly the same as $J_{\bar{P}|P}(\tau_{s}\boxtimes\pi)$ and $J_{P|\bar{P}}(\tau_{s}\boxtimes\pi)$ used to define the Plancherel measure $\mu_{\psi}(\tau_{s}\boxtimes\pi)$.
Thus we get the assertion.
\end{proof}

The key ingredient is the following lemma, which enables us to recover an $L$-parameter from its $\gamma$-factor:

\begin{lem}[{\cite[Lemma 12.3]{MR2999299}}]\label{lem:Gan-Savin}
Let $\phi_{1}$ and $\phi_{2}$ are $L$-parameters of $\mathbf{H}$ which are multiplicity-free sums of irreducible orthogonal representations as representations of $W_{F}\times\SL_{2}(\C)$ (or, in other words, $\phi_{1}$ and $\phi_{2}$ are discrete in the sense of Section \ref{sec:Arthur}).
If we have
\[
\gamma(s,\phi_{\tau}\otimes\phi_{1},\psi)
\cdot
\gamma(-s,\phi_{\tau}^{\vee}\otimes\phi_{1},\overline{\psi})
=
\gamma(s,\phi_{\tau}\otimes\phi_{2},\psi)
\cdot
\gamma(-s,\phi_{\tau}^{\vee}\otimes\phi_{2},\overline{\psi})
\]
for any $L$-parameter $\phi_{\tau}$ of an irreducible supercuspidal representation $\tau$ of $\GL_{r}(F)$ (for any $r$), then we have $\phi_{1}\cong\phi_{2}$ as representations of $W_{F}\times\SL_{2}(\C)$ (or, in other words, $\phi_{1}$ equals $\phi_{2}$ as elements of $\widetilde{\Phi}(\mathbf{H})$ with the notation in Section \ref{sec:Arthur}).
\end{lem}

This lemma is proved in \cite[Lemma 12.3]{MR2999299} in the case where $\phi_{1}$ and $\phi_{2}$ are symplectic as representations of $W_{F}\times\SL_{2}(\C)$.
However, as noted in the proof of \cite[Theorem C.5]{MR3166215}, the completely same proof as in \cite[Lemma 12.3]{MR2999299} works in the above setting.

Now let us prove Proposition \ref{prop:spin-twist}.

\begin{proof}[Proof of Proposition \ref{prop:spin-twist}]
Let $\phi$ be a discrete $L$-parameter of $\mathbf{H}$.
Our goal is to show $\widetilde{\Pi}_{\phi\otimes\chi}^{\mathbf{H}}=\widetilde{\Pi}_{\phi}^{\mathbf{H}}\otimes(\chi\circ\spin)$.
In other words, by putting $\phi'$ to be the $L$-parameter of a(ny) member $\tilde{\pi}\otimes(\chi\circ\spin)$ of $\widetilde{\Pi}_{\phi}^{\mathbf{H}}\otimes(\chi\circ\spin)$ (here $\tilde{\pi}\in\widetilde{\Pi}_{\phi}^{\mathbf{H}}$), it suffices to show that $\phi'$ is equal to $\phi\otimes\chi$ as elements of $\widetilde{\Phi}(\mathbf{H})$.

Let $\pi\in\tilde{\pi}$ (recall that $\tilde{\pi}$ is a $\Sigma(\mathbf{H})$-orbit of irreducible smooth representations of $H$).
By Lemma \ref{lem:Plancherel}, we have
\[
\mu_{\psi}\bigl(\tau_{s}\boxtimes(\pi\otimes(\chi\circ\spin))\bigr)
=
\mu_{\psi}\bigl((\tau_{s}\otimes(\chi^{-1}\circ\det))\boxtimes\pi\bigr)
\]
for any irreducible supercuspidal representation $\tau$ of $\GL_{r}(F)$ (for any $r\in\Z_{\geq1}$).
We rewrite the both sides in terms of $\gamma$-factors of $L$-parameters.
In general, for any irreducible smooth representation $\rho$ of $H$ with $L$-parameter $\phi_{\rho}$, we have
\[
\mu_{\psi}(\tau_{s}\boxtimes\rho)
=
\gamma(s,\phi_{\tau}\otimes\phi_{\rho}^{\vee},\psi)\cdot
\gamma(-s,\phi_{\tau}^{\vee}\otimes\phi_{\rho},\overline{\psi})\cdot
\gamma(2s,\wedge^{2},\phi_{\tau},\psi)\cdot
\gamma(-2s,\wedge^{2},\phi_{\tau}^{\vee},\overline{\psi})
\]
(see \cite[Sections C.2]{MR3166215} or \cite[Theorem 12.8.1]{MR3709003}).
Thus we have
\begin{multline*}
\mu_{\psi}\bigl(\tau_{s}\boxtimes(\pi\otimes(\chi\circ\spin))\bigr)
=
\gamma(s,\phi_{\tau}\otimes\phi'^{\vee},\psi)\cdot
\gamma(-s,\phi_{\tau}^{\vee}\otimes\phi',\overline{\psi})\\
\cdot\gamma(2s,\wedge^{2},\phi_{\tau},\psi)\cdot
\gamma(-2s,\wedge^{2},\phi_{\tau}^{\vee},\overline{\psi})
\end{multline*}
and
\begin{multline*}
\mu_{\psi}\bigl((\tau_{s}\otimes(\chi^{-1}\circ\det))\boxtimes\pi\bigr)
=
\gamma(s,\phi_{\tau}\otimes\chi^{-1}\otimes\phi^{\vee},\psi)\cdot
\gamma(-s,\phi_{\tau}^{\vee}\otimes\chi\otimes\phi,\overline{\psi})\\
\gamma(2s,\wedge^{2},\phi_{\tau}\otimes\chi^{-1},\psi)\cdot
\gamma(-2s,\wedge^{2},\phi_{\tau}^{\vee}\otimes\chi,\overline{\psi}).
\end{multline*}
Thus, by noting that $\chi$ is a quadratic character and that $\phi$ and $\phi'$ are self-dual, we get
\[
\gamma(s,\phi_{\tau}\otimes\phi',\psi)\cdot
\gamma(-s,\phi_{\tau}^{\vee}\otimes\phi',\overline{\psi})
=
\gamma(s,\phi_{\tau}\otimes\chi\otimes\phi,\psi)\cdot
\gamma(-s,\phi_{\tau}^{\vee}\otimes\chi\otimes\phi,\overline{\psi}).
\]
Then Lemma \ref{lem:Gan-Savin} implies that $\phi'$ is equal to $\phi\otimes\chi$ as elements of $\widetilde{\Phi}(\mathbf{H})$.
\end{proof}


\end{document}